\documentclass[11pt]{article}

\usepackage{xcolor}
\usepackage[margin=1in]{geometry}
\usepackage[T1]{fontenc}
\usepackage{lmodern}
\usepackage{amsmath,amssymb,amsthm,mathtools,bm}
\usepackage[numbers,sort&compress]{natbib}
\usepackage[hidelinks]{hyperref}
\usepackage{authblk}

\numberwithin{equation}{section}

\usepackage{enumerate}
\usepackage{etoc}

 \bibpunct[, ]{[}{]}{,}{n}{}{,}%

\newcommand{\mathbbm}[1]{\mathbf{#1}}

\newcommand{\cY}{\mathcal{Y}}
\newcommand{\bE}{\mathbb{E}}
\newcommand{\eq}{\mathcal{E}\mathcal{Q}_{\Te}}

\newcommand{\cK}{\mathcal{K}}

\newcommand{\frakc}{\mathfrak{c}}

\newcommand{\lfavg}{L_{f,\mathrm{avg.}}}
\newcommand{\LS}{L_{\mathrm{Sk}}}
\newcommand{\fm}{x}
\newcommand{\sm}{\theta}

\newcommand{\cS}{\mathcal{S}}

\newcommand{\lp}{L_p} 
\newcommand{\lf}{L_f} 
\newcommand{\lhx}{L_{h,x}}
\newcommand{\lhte}{L_{h,\theta}}
\newcommand{\lmu}{L_{\mu}} 
\newcommand{\lfp}{L_{fp}}
\newcommand{\cfpoi}{C_{f,\mathrm{Pois.}}} 
\newcommand{\chpoi}{C_{h,\mathrm{Pois.}}} 
\newcommand{\cfastm}{C_{\mathrm{mg}}^{(x)}}
\newcommand{\etem}{E^{(\theta,\tilde{\theta})}_{m}}
\newcommand{\fraka}{\mathfrak{a}}
\newcommand{\frakb}{\mathfrak{b}}
\newcommand{\no}{N_{0}}

\newcommand{\um}{^{(m)}}

\newcommand{\favg}{f^{(\mathrm{av.})}}
\newcommand{\havg}{h^{(\mathrm{av.})}}
\newcommand{\proj}{\operatorname{proj}_{\Theta}}
\newcommand{\bN}{\mathbb{N}}

\newcommand{\cX}{\mathcal{X}}
\newcommand{\tte}{\tilde{\te}}
\newcommand{\tv}{\tilde{v}}
\newcommand{\ust}{^{\star}}

\newcommand{\bR}{\mathbb{R}}
\newcommand{\te}{\theta}

\newcommand{\bP}{\mathbb{P}}
\newcommand{\cE}{\mathcal{E}}
\newcommand{\cG}{\mathcal{G}}

\newcommand{\cB}{\mathcal{B}}
\newcommand{\eps}{\epsilon}

\newcommand{\cA}{\mathcal{A}}

\newcommand{\cF}{\mathcal{F}}

\newcommand{\Te}{\Theta}
\newcommand{\ute}{^{(\theta)}}

\newcommand{\rateslack}{\varepsilon_{\rm rate}}
\newcommand{\nal}[1]{\begin{align*}#1\end{align*}}
\newcommand{\al}[1]{\begin{align}#1\end{align}}

\theoremstyle{plain}
\newtheorem{theorem}{Theorem}[section]
\newtheorem{lemma}{Lemma}[section]
\newtheorem{proposition}{Proposition}[section]
\newtheorem{corollary}{Corollary}[section]

\theoremstyle{definition}
\newtheorem{assumption}{Assumption}[section]
\newtheorem{definition}{Definition}[section]
\newtheorem{condition}{Condition}[section]

\theoremstyle{remark}
\newtheorem{remark}{Remark}[section]

\newcommand{\norm}[1]{\left\lVert#1\right\rVert}

\DeclarePairedDelimiter\autobracket{(}{)}
\newcommand{\br}[1]{\autobracket*{#1}}

\begin{document}

\title{Finite-Time Concentration and Convergence Rates for
Projected Two-Time-Scale Stochastic Approximation with Markov Noise}

\author[1]{Rahul Singh}
\author[2]{Vivek S. Borkar}
\author[1,3]{Eric Moulines}

\affil[1]{Laboratoire de Recherche de l'EPITA, Paris, France}
\affil[2]{Indian Institute of Technology Bombay}
\affil[3]{Mohamed bin Zayed University of Artificial Intelligence,
Abu Dhabi, United Arab Emirates}

\date{}
\maketitle

\begin{abstract}
We study projected two-time-scale stochastic approximation algorithms driven by a controlled Markov chain. The fast recursion is governed by an averaged contractive map, while the slow recursion is constrained to a compact convex polyhedron through a projection step. The limiting slow dynamics are described by a projected differential equation, whose vector field may be discontinuous on the boundary of the constraint set. This feature prevents a direct application of standard finite-time stochastic approximation arguments based on Lipschitz limiting vector fields.

We establish explicit blockwise high-probability bounds showing that the fast iterate tracks the moving fixed point $x^\star(\theta_n)$ of the averaged fast recursion and that the step embedding of the slow iterate tracks the associated projected ODE. The bounds separate fast and slow martingale deviations, Poisson residuals generated by Markovian sampling, and the deterministic two-time-scale bias. Under a Lipschitz Lyapunov function satisfying a uniform one-block decrease condition for the projected ODE, these estimates yield almost-sure convergence of the last-iterate Lyapunov error and explicit rates when the decrease modulus has a power lower bound near zero. Uniform one-block contraction is the linear-decrease special case.

A uniform decrease condition for a Lipschitz Lyapunov function over fixed time intervals yields almost-sure convergence of the slow Lyapunov error, with explicit last-iterate rates when the decrease admits a power lower bound near zero. Under uniform Lyapunov contraction, suitably chosen polynomial step sizes give joint fast-tracking and slow Lyapunov-error exponents arbitrarily close to \(1/3\), while a slow step size of order \(1/n\) attains the fast-tracking exponent \(1/3\). When the reduced slow update map is also a Euclidean contraction, logarithmically separated step sizes improve the joint rate to \(O_{\mathrm{a.s.}}(n^{-1/2}\log n)\), including when the equilibrium lies on the boundary. The same rate holds under a distinct geometric condition: the slow equilibrium set is a strictly attracting face of a box, and the fast equilibrium is constant on that face.

We apply the results to an actor–critic algorithm, obtaining an almost-sure value-gap rate of \(O(n^{-1}\log n)\) relative to the optimum within the constrained policy class. Further applications include projected TD(0) and projected stochastic gradient descent. We also extend the analysis to projection of the fast recursion under Euclidean contractivity.
\end{abstract}


\etocdepthtag{main}
\section{Introduction}\label{sec:introduction}
Stochastic approximation algorithms are iterative schemes for solving fixed-point, root-finding, and optimization problems when only noisy observations of the underlying operator are available. They play a central role in stochastic control, reinforcement learning, simulation-based optimization, and adaptive systems. Many modern applications involve coupled recursions evolving on different time scales, with one component adapting rapidly and another changing more slowly. In addition, constraints on the iterates are often enforced through projection steps.~This paper studies finite-time behavior of a projected two-time-scale stochastic approximation scheme driven by a controlled Markov chain.

The algorithm consists of a fast recursion $x_n$ and a slow projected
recursion $\te_n$. The slow iterate is projected onto a compact convex
polyhedron $\Te$, while the transition probabilities of the controlled
Markov chain are determined by the current value of $\te_n$. Because
the slow variable changes on a smaller time scale, the fast recursion
is expected to equilibrate near the fixed point $x\ust(\te_n)$ of the
averaged fast-time-scale map. We make this intuition quantitative by
proving high-probability bounds on the tracking error
$x_n-x\ust(\te_n)$.

The slow recursion is then compared with the solution of an associated projected ODE. This comparison is nonstandard because the projection may make the limiting vector field discontinuous on the boundary of $\Te$. 

We overcome this difficulty by representing each projected path through the Skorokhod map. The Lipschitz property of this map then allows us to control differences between projected paths by differences between their unconstrained driving paths. This yields finite-time concentration bounds for the slow iterate around the projected ODE trajectory. If the associated projected ODE admits a Lipschitz Lyapunov function satisfying a uniform one-block decrease condition along its trajectories, these bounds imply almost-sure convergence of the last-iterate Lyapunov error to zero. Explicit rates follow when the decrease modulus has a power lower bound near zero, with uniform one-block contraction as the linear-decrease special case.

Consider the coupled stochastic recursions
\begin{align}
x_{n+1}
&=x_n+\alpha_n[h(Y_n,x_n,\theta_n)-x_n+M_{n+1}], \label{def:x_update}\\
\theta_{n+1}
&=\operatorname{proj}_{\Theta}\{\theta_n+\beta_n[f(Y_n,x_n,\theta_n)+M'_{n+1}]\}.\label{def:te_update}
\end{align}
where the step sizes satisfy \(\beta_n/\alpha_n\to0\) as \(n\to\infty\),
$\Theta$ is a compact convex polyhedron, and \(\operatorname{proj}_{\Theta}\) denotes projection onto \(\Theta\).
This separation of step sizes gives the two-time-scale structure: the \(x_n\)-recursion evolves on the faster time scale, while the \(\theta\)-recursion evolves on the slower time scale.
The processes $\{M_n\}$ and
$\{M'_n\}$ are martingale-difference noise terms. The sequence
$\{Y_n\}$ is a controlled Markov chain taking values in the finite set
$\cY$, and its transition kernel depends on the current slow iterate
$\theta_n$.

For each $\theta\in\Theta$, let $\mu^{(\theta)}$ be the stationary distribution of the controlled Markov chain induced by $\theta$, and define
\[
h^{(\mathrm{av.})}(x;\theta):=\sum_{y\in\cY}\mu^{(\theta)}(y)h(y,x,\theta).
\]
The unique fixed point of $h^{(\mathrm{av.})}(\cdot;\theta)$ is denoted by $x^\star(\theta)$. We prove that the fast iterate tracks the moving fixed point $x^\star(\theta_n)$ with high probability. We then show that the right-continuous step embedding of the slow iterate tracks the solution of the projected ODE
\[
\dot z(t)=\Pi_\Theta\left(z(t),\sum_{y\in\cY}\mu^{(z(t))}(y)f(y,x^\star(z(t)),z(t))\right).
\]
Here $\Pi_\Theta(z,v)$ denotes the projected vector at $z$ in direction
$v$; its formal definition is given in~\eqref{def:pi}.

Using these concentration bounds, we obtain iteration-wise almost-sure
last-iterate rates for
\(\alpha_n=(\no+n)^{-\fraka}\) and
\(\beta_n=(\no+n)^{-\frakb}\), where
\(1/2<\fraka<\frakb\le1\). Under uniform one-block contraction and
\(\frakb<1\), Corollary~\ref{coro:optimized_rate} gives joint
fast-tracking and slow Lyapunov-error rates with polynomial exponent
arbitrarily close to \(1/3\), up to a
\(\sqrt{\log(\no+n)}\) factor. At the harmonic endpoint
\(\frakb=1\), for which \(\beta_n=(\no+n)^{-1}\), taking
\(\fraka=2/3\) gives
\(\norm{x_n-x\ust(\te_n)}
=O_{\rm a.s.}((\no+n)^{-1/3}\sqrt{\log(\no+n)})\).
The polynomial exponent governing the slow Lyapunov error is then
\(\min\{\lambda_T,1/3\}\), with the precise logarithmic factors given
in the same corollary; here \(\lambda_T\) is defined in
\eqref{def:effective_block_exponent}.

More generally, suppose that the decrease in the Lyapunov function
over one fixed block of projected-ODE time is bounded below by a
positive constant times its value at the start of the block raised
to a power \(p_J>1\), whenever that value is sufficiently small.
Corollary~\ref{coro:optimized_power_rate} gives the optimized
\(\mathfrak{b}<1\) rate
\[
J(\te_n)
=
O_{\rm a.s.}\!\left(
(\no+n)^{-1/(4p_J-1)}
\{\log(\no+n)\}^{1/(2p_J)}
\right),
\]
whereas its harmonic-endpoint rate is
\(O_{\rm a.s.}(\{\log(\no+n)\}^{-1/(p_J-1)})\).
In Appendices~\ref{sec:ec_projected_td}--\ref{sec:one_scale_specialise}
of the Online Companion, we specialize the framework to the case
where the fast recursion is absent, obtain sharper rate bounds
for the remaining projected recursion, and give applications to
projected TD(0) and projected SGD.

Additional equilibrium geometry improves these polynomial exponents.
If the slow equilibrium set is a strictly attracting box face and the fast
equilibrium is constant on that face, Proposition~\ref{prop:face_rates}
gives a joint rate $O_{\rm a.s.}(n^{-1/2}\log n)$ with
$\alpha_n=A\log(N_0+n)/(N_0+n)$ and $\beta_n=B/(N_0+n)$.
Actor--critic algorithm from reinforcement learning satisfies
these conditions and also achieves a value-gap rate of
\(O_{\rm a.s.}(n^{-1}\log n)\).

The projection step introduces a difficulty that is absent in standard
finite-time analyses of stochastic approximation. The limiting dynamics
are described by a projected differential equation, whose vector field
need not be continuous at the boundary of the constraint set.
Consequently, standard perturbation arguments for stochastic
approximation with Lipschitz limiting vector fields do not directly
apply. We therefore follow the ODE method (\citet{borkar2023stochastic}):
on each finite block of accumulated slow time, we compare the slow step
embedding with the projected-ODE solution initialized at the block entrance.
We make this projected-ODE comparison quantitative: 
the Lipschitz continuity of the Skorokhod map on finite time
intervals yields explicit path-tracking concentration bounds,
and a separate one-block Lyapunov decrease
condition converts those bounds into last-iterate rates, with
contraction as its linear special case.

\subsection{Contributions}
Our main contributions are as follows.
\begin{enumerate}[(i)]
\item We establish explicit finite-time high-probability concentration
bounds showing that the fast iterate \(x_n\) tracks the moving fixed point
\(x\ust(\te_n)\) in projected two-time-scale stochastic approximation
driven by an iterate-controlled finite-state Markov chain. Only the
averaged fast map is assumed contractive; no contraction is imposed on
the slow recursion, which is projected onto the compact convex
polyhedron \(\Te\). The bounds separate the effects of martingale noise,
Markovian sampling, movement of the fast equilibrium, and time-scale
separation; see Proposition~\ref{prop:conc_v_v_te}.
\item We prove a finite-horizon supremum-norm concentration bound for
the slow iterates around the associated projected-ODE trajectory.
We use the Skorokhod map to handle the possible discontinuity
of the projected vector field at the boundary;
see Theorem~\ref{th:tube}.

\item
For the slow iterates, we prove almost-sure convergence to the
equilibrium set of the associated projected ODE, assuming that
this ODE admits a Lipschitz Lyapunov function that is nonincreasing
along trajectories and satisfies a uniform decrease condition
over a fixed time horizon \(T>0\).
The horizon can be chosen to establish this decrease, without
requiring contraction or nonexpansiveness of the reduced slow
update map.
This goes beyond the reduced slow-map contraction assumptions of
\citet{chandak2026Ok,chandakhaquebambos2025} and the nonexpansiveness
assumption of \citet{chandak2025non}.
The actor--critic application in Section~\ref{sec:applications}
illustrates this distinction: its reduced slow update map need not
be nonexpansive, while its projected ODE reaches the restricted-optimal
equilibrium face in uniformly bounded time and therefore satisfies
one-block Lyapunov contraction for a sufficiently large \(T\).
To quantify this convergence, we assume that, for sufficiently
small Lyapunov errors, the decrease over time \(T\) is bounded
below by a positive constant times a power of the current
Lyapunov error. This yields explicit optimized rates, with
one-block Lyapunov contraction as the linear-decrease special case.
Under this contraction condition and polynomial step sizes, the
joint fast-tracking and slow Lyapunov-error exponent can be taken
arbitrarily close to \(1/3\); see
Theorems~\ref{th:iter_rate_nonlinear} and~\ref{th:iter_rate}
and Corollaries~\ref{coro:optimized_rate}
and~\ref{coro:optimized_power_rate}.
When the slow equilibrium set is a face of a box constraint,
the averaged slow drift points strictly outward in each coordinate
fixed on that face, and \(x^\star(\theta)\) is constant on the face,
Proposition~\ref{prop:face_rates} gives the joint rate
\(O_{\rm a.s.}(n^{-1/2}\log n)\) with logarithmically separated
step sizes.

\item
We apply our results to an actor--critic algorithm in which the
current policy controls the Markov sampling process.
With logarithmically separated step sizes, the actor value gap
relative to the optimum within the constrained policy class is
\(O_{\rm a.s.}(n^{-1}\log n)\), and the critic error is
\(O_{\rm a.s.}(n^{-1/2}\log n)\).
The Online Companion contains the proof and applications to
projected TD(0) and projected SGD.
It also gives a saddle-point application in
Appendix~\ref{sec:ec_saddle_learning}, including examples with
vanishing curvature and rotational coupling from a bilinear term.
These examples fail both nonexpansiveness of the reduced forward
map in any fixed norm and uniform quadratic dissipation, but their
projected ODEs satisfy a power Lyapunov decrease condition;
see Remark~\ref{rem:sp_ac_frameworks}.
Finally, we extend the analysis to algorithms that project the
fast iterate onto a compact convex polyhedron at every step.
This makes the fast iterates bounded by construction, removing
Assumption~\ref{assum:xn_bound}, while retaining contraction of the
averaged fast map in the Euclidean norm and bounded noise before
projection.
\end{enumerate}

\section{Related work}
\label{sec:ec_extended_related_work}
Representative finite-time results include mean-square,
and high-probability concentration bounds due to \citet{chen2020finite,borkar2021concentration,chandak2022concentration,chen2022finite}.
In two-time-scale stochastic approximation (TTSA), one first studies the fast
averaged dynamics with the slow variable frozen. The corresponding
reduced slow dynamics are then obtained by replacing the fast variable
with its equilibrium \(x^\star(\theta)\). Existing analyses often impose
stability or contraction both on this frozen-slow fast subsystem and on
the resulting equilibrium-reduced slow subsystem; see, for example,
\citet{chandak2026Ok}. In the linear setting with i.i.d. 
noise,~\citet{konda2004convergence} characterize the asymptotic covariance of
two-time-scale stochastic approximation, and establish asymptotic
normality.~\citet{dalal2018finite} derive a concentration
bound for linear two-time-scale stochastic approximation, with applications to
reinforcement-learning algorithms.~\citet{doan2022nonlinear} prove almost-sure convergence and an
\(O(k^{-2/3})\) mean-square bound for nonlinear
TTSA under strong-monotonicity assumptions on the frozen-slow fast
operator and the equilibrium-reduced slow operator.
~\citet{chandakhaquebambos2025} obtain an
$O(n^{-2/3})$ mean-square rate under arbitrary-norm contractions and
Markov noise.~\citet{chandak2026Ok} obtains
$O(k^{-a})$ mean-square error for the two-time-scale schedule
$\alpha_k\asymp k^{-a}$, $\beta_k\asymp k^{-1}$, for
$a\in(1/2,1)$. These rates use contraction on both reduced
operators.~\citet{chandak2025non} allows
a nonexpansive slow map and a projected fast recursion, with a residual
mean-square rate $O(k^{-1/4+\varepsilon})$, where $\varepsilon>0$ is arbitrarily small.
Classical works on projected stochastic
approximation focus on asymptotic convergence, weak convergence, or
large-deviation principles; see, for example,
\citet{kushner2003stochastic,dupuis1987asymptotic}. \citet{borowski2025convergence} repair the gap in
\citet[Theorem~5.2.1]{kushner2003stochastic} for single-time-scale SA
with projection onto a hyperrectangle, proving
almost-sure convergence to the stationary set of its projected ODE.

\paragraph{Notation}
Throughout, $\|\cdot\|$ denotes the Euclidean norm for vectors and
the induced Euclidean operator norm for matrices. For vectors,
$\|\cdot\|_\infty$ denotes the sup norm. For matrices,
\[
 \|A\|_\infty:=\max_i\sum_j|A(i,j)|
\]
denotes the operator norm induced by the vector sup norm, equivalently
the maximum absolute row-sum norm. Boldface
$\mathbf 0$ and $\mathbf 1$ denote, respectively, the zero and all-ones
vectors of the appropriate dimension. If
$\cE$ is an event, then $\mathbbm{1}\{\cE\}$ is its indicator. If
$\cB$ is a finite nonempty set, then $\Delta(\cB)$ denotes the
probability simplex on $\cB$, viewed as a subset of $\bR^{|\cB|}$.
For a set $X$ and point $x$, $\operatorname{dist}(x,X)$ denotes the
distance from $x$ to $X$. For a
positive deterministic sequence \(a_n\), we write
\(Z_n=O_{\rm a.s.}(a_n)\) if, on an event of probability one, there
exist a finite nonnegative random constant \(C\) and a finite
integer-valued random index \(N\) such that
\(|Z_n|\le Ca_n\) for every \(n\ge N\), with absolute value replaced
by the Euclidean norm for vector-valued \(Z_n\).
For positive deterministic sequences \(\{a_n\}\) and
\(\{b_n\}\), we write \(a_n\asymp b_n\) when they are
comparable up to positive multiplicative constants for all
sufficiently large \(n\); that is, when
\(c b_n\le a_n\le C b_n\) eventually for some
\(0<c\le C<\infty\).
We write $\bN:=\{1,2,\ldots\}$ and write $\bR_+$ for the set of
nonnegative reals. For $d\ge1$, define
\[
\kappa_d:=\sup_{\|u\|_\infty\le1,\ u\in\bR^d}\|u\|.
\]
Thus $\|z\|\le\kappa_d\|z\|_\infty$ for every $z\in\bR^d$.

Sections~\ref{sec:pf}--\ref{sec:sketch} give the model, notation, and proof
strategy; Sections~\ref{sec:faster_iterations}--\ref{sec:conv_rate} establish
concentration and last-iterate rates; Section~\ref{sec:applications}
presents the actor--critic application; and Section~\ref{sec:conclusion}
concludes. Selected proofs appear in
Appendix~\ref{sec:article_selected_proofs}; the remaining proofs, the
one-time-scale specialization and applications, the fast-projection
extension, and auxiliary results are collected in the Online Companion.
\section{Problem Formulation}\label{sec:pf}
We study the recursions~\eqref{def:x_update}--\eqref{def:te_update}
introduced in Section~\ref{sec:introduction}.
Here $x_n \in \bR^{d_x}$ is the fast iterate,
$\te_n \in \Te \subset \bR^{d_{\te}}$ is the slow iterate, and $\Te$
is a nonempty compact convex polyhedron. The map
$\operatorname{proj}_{\Theta}(\cdot)$ denotes Euclidean projection onto
$\Te$. The initial slow iterate satisfies \(\te_1\in\Te\) almost
surely. The process $\{Y_n\}$ is a controlled Markov chain taking values
in the finite state space $\cY$, and its transition law depends on the
current slow iterate $\te_n$. The functions
$h: \cY \times \bR^{d_x} \times \Te \to \bR^{d_x}$ and
$f: \cY \times \bR^{d_x} \times \Te \to \bR^{d_{\te}}$ define the
fast and slow updates, respectively. The processes $\{M_n\}$ in
$\bR^{d_x}$ and $\{M'_n\}$ in $\bR^{d_{\te}}$ are noise processes.

In addition to the time-scale separation stated in
Section~\ref{sec:introduction}, the step sizes satisfy
\[
\sum_n \alpha_n=\sum_n \beta_n=\infty,
\qquad
\sum_n(\alpha_n^2+\beta_n^2)<\infty.
\]
We impose the following standing assumptions.
\begin{assumption}[Controlled Markov chain]
	\label{assum:markov_noise}
The process $\{Y_n\}$ is a controlled Markov process taking values in the finite state space $\cY$. Let
\[
\cF_n = \sigma\br{x_1,\te_1, \{Y_i\}_{i=1}^{n},\{M_i\}_{i=1}^{n}, \{M'_i\}_{i=1}^{n}}, n=1,2,\ldots.
\]
Then
\nal{
\bP(Y_{n+1} = y'|\cF_n ) = p^{(\te_n)}(Y_n,y'), \quad n \in \bN, \quad \text{a.s.},
}
where $\{p\ute(y,y')\}_{y,y'\in\cY}$ are the controlled transition probabilities of the Markov chain induced by $\te$.~For each $\te\in\Te$, the Markov chain with transition probabilities $p\ute$ has a unique stationary distribution, denoted by $\mu^{(\te)}$.
\end{assumption}

\begin{assumption}[Lipschitz transition kernel]
\label{assum:lipschitz_kernel}
There exists $\lp >0$ such that, for all $i\in\cY$ and all $\te,\te'\in\Te$,
\nal{
\sum_{j\in\cY}|p\ute(i,j) - p^{(\te')}(i,j)| \le \lp \|\te -\te'\|.
  }
Thus, the controlled transition kernel is Lipschitz in $\te$.
\end{assumption}
This regularity condition is standard in finite-time analyses of stochastic approximation; see, e.g., ~\citet{chandak2022concentration,metivier}.

\begin{assumption}[Martingale-difference noise]
	\label{assum:1_2}
The sequences $\{M_{n}\}$ and $\{M'_{n}\}$ are martingale-difference sequences with respect to $\{\cF_n\}$. Thus,
\nal{
\bE\br{M_{n+1}\mid \cF_n}  = \mathbf{0}, \qquad \bE\br{M'_{n+1}\mid \cF_n}  = \mathbf{0}, \quad \text{a.s.}
}
Moreover, there exist constants $\frakc_1,\frakc_2>0$ such that, almost surely,
\al{
			\| M_{n+1} \| \le \frakc_1 + \frakc_2 \|x_n\|, \qquad \| M'_{n+1} \| \le \frakc_1 + \frakc_2 \|x_n\|, \quad \forall n\in \bN.
		}
\end{assumption}
Martingale-difference noise assumptions of this form are standard in stochastic approximation; see, e.g.,~\citet{kushner2003stochastic,borkar2023stochastic}.

\begin{assumption}[Uniform contraction of the averaged fast map]
\label{assum:contraction} 
For each $\te\in\Theta$ define the averaged fast-time-scale map 
\al{
\havg(x;\te)  := \sum\limits_{y\in\cY} \mu\ute(y)  h(y,x,\te), \quad x\in\bR^{d_x}.\label{def:havg}
}
For a fixed \(\theta\), \(h^{(\mathrm{av.})}(x;\theta)\) is the
mean fast-update map obtained by freezing the slow parameter and averaging
\(h(y,x,\theta)\) with respect to the stationary distribution
\(\mu^{(\theta)}\) of the resulting Markov chain. Thus,
\(h^{(\mathrm{av.})}(x;\theta)-x\) is the corresponding averaged drift
of the fast recursion.~There exists $\gamma \in [0,1)$ such that 
\nal{
\Big\| \havg(x;\te) - \havg(x';\te) \Big\|  \le \gamma \| x -x' \|, \quad x,x' \in \bR^{d_x}, \te\in\Te.
}
The unique fixed point of $\havg(\cdot;\te)$ is denoted by $x\ust(\te)$, so that
		\nal{
			\havg(x\ust(\te);\te)  = x\ust(\te).
		}
\end{assumption}
\begin{remark}[Fast-scale contraction]
Assumption~\ref{assum:contraction} imposes a contraction condition only on the averaged fast-time-scale map.~No contractivity assumption on a discrete slow update map is needed for the concentration and path-tracking results; the slow dynamics are instead characterized through the projected ODE studied later. The constant $\gamma$ may be replaced by a parameter-dependent contraction factor $\gamma(\te)$, provided $\sup_{\te\in\Te} \gamma(\te)<1$.
\end{remark}
The extension to an arbitrary fixed norm for the fast coordinate is recorded in
Remark~\ref{rem:fast_coordinate_norm} of the Online Companion.
\begin{assumption}[Lipschitz update functions on bounded fast-state sets]
\label{assum:lipschitz}
For every $0\le R<\infty$, there exist finite constants
$\lf(R)$, $\lhx(R)$, and $\lhte(R)$ such that, for all $y\in\cY$,
$\te,\tilde{\te}\in\Te$, and $x,\tilde{x}\in\bR^{d_x}$ satisfying
$\norm{x},\norm{\tilde{x}}\le R$,
\nal{
	\|f(y,x,\te)-f(y,\tilde{x},\tilde{\te})\|  \le \lf(R) \br{\|x-\tilde{x}\| + \| \te -\tilde{\te}\|},
}			
and
\nal{
\|h(y,x,\te)-h(y,\tilde{x},\tilde{\te})\| \le \lhx(R) \|x-\tilde{x}\| + \lhte(R) \|\te - \tilde{\te}\|.
}
The constants may depend on $R$, but not on $y$, $x$, $\tilde x$,
$\te$, or $\tilde\te$ within the displayed sets.  The moduli may be
chosen nondecreasing in $R$.
\end{assumption}
For \(0\le R<\infty\), Assumption~\ref{assum:lipschitz},
finiteness of \(\cY\), and compactness of \(\Theta\) imply
\[
\mathsf B(R):=
\sup_{\substack{y\in\cY,\ \theta\in\Theta\\ \|x\|\le R}}
\max\{\|h(y,x,\theta)\|,\|f(y,x,\theta)\|\}<\infty.
\]
Write \(\mathsf B_0:=\mathsf B(0)\).
\begin{assumption}[Uniform hitting-time bound]
\label{assum:hitting_time}
There exists a reference state $y\ust \in \cY$ and $\frakc_5 <\infty$ such that, for every $\te\in\Te$ and every initial state $y\in\cY$,
\nal{
\bE^{(\te)}_y \br{\tau_{y\ust}} \le \frakc_5,
}
where $\tau_{y\ust} := \inf \left\{n\ge 1: Y_n = y\ust   \right\}$ for the Markov chain with transition probabilities $p\ute$.
The superscript \((\theta)\) on \(\mathbb E_y^{(\theta)}\)
indicates that the chain uses the fixed transition kernel
\(p^{(\theta)}\) at every step, while the subscript \(y\)
specifies the initial state \(Y_0=y\).
\end{assumption}
This condition provides uniform control of the Poisson-equation
solutions used to decompose the Markov noise;
see Lemmas~\ref{lemma:9} and~\ref{lemma:metivier_mtg}.
The next lemma records two regularity consequences of the preceding
controlled-Markov assumptions. Together,
Assumptions~\ref{assum:lipschitz_kernel}
and~\ref{assum:hitting_time} imply Lipschitz continuity of the
stationary-distribution map. Combined with
Assumptions~\ref{assum:lipschitz} and~\ref{assum:contraction}, this
also gives Lipschitz continuity of $x^\star(\cdot)$.
\begin{lemma}\label{lemma:lipschitz_stat_dist}
The stationary distribution map $\te \mapsto \mu\ute$ and the fixed-point
map $\te \mapsto x\ust(\te)$ are Lipschitz on $\Te$. That is, there exist
constants $\lmu,\lfp \ge 0$ such that
\[
\norm{\mu^{(\te)}-\mu^{(\te')}}_1
\le
L_\mu\norm{\te-\te'},
\qquad
\norm{x\ust(\te)-x\ust(\te')}
\le
L_{fp}\norm{\te-\te'},
\qquad
\te,\te'\in\Te.
\]
Here, for any signed measure \(\nu\) on \(\cY\),
\[
\norm{\nu}_1:=\sum_{y\in\cY}|\nu(y)|.
\]
In particular, $\sup_{\te\in\Te}\|x\ust(\te)\|<\infty$.
\end{lemma}
The proof is given in Appendix~\ref{sec:app_sec_pf} of the Online Companion. We henceforth write
\[
\frakc_8:=\sup_{\te\in\Te}\norm{x\ust(\te)}
\le\frac{\mathsf B_0}{1-\gamma}<\infty.
\]

\begin{assumption}[Stability of the fast iterate]
\label{assum:xn_bound}
There exists a deterministic constant $\frakc_6<\infty$ such that
\nal{
\sup_{n\in\bN} \|x_n\| \le \frakc_6~a.s.
}
\end{assumption}
Contraction of the averaged fast map does not by itself keep the
fast iterates bounded: individual sampled updates may still
produce large excursions. Assumption~\ref{assum:xn_bound} ensures that
the martingale increments, Poisson-equation solutions, and drift perturbations
used throughout the concentration proof are uniformly bounded.
It can be verified by finding a deterministic bounded set that
contains the initial iterate and is preserved by the recursion,
as illustrated in the Online Companion.
Alternatively, projecting the fast iterate onto a compact convex
polyhedron at every step ensures boundedness by construction.
Appendix~\ref{sec:ec_fast_projection_extension} develops this
extension, removing the need for Assumption~\ref{assum:xn_bound}.

Set
\[
\mathsf B_6:=\mathsf B(\frakc_6),\qquad
\mathsf B_8:=\mathsf B(\frakc_8),\qquad
R_{\rm loc}:=\max\{\frakc_6,\frakc_8\}.
\]
Hereafter, $\lf$, $\lhx$, and $\lhte$ denote, respectively, the
constants $\lf(R_{\rm loc})$, $\lhx(R_{\rm loc})$, and $\lhte(R_{\rm loc})$ from
Assumption~\ref{assum:lipschitz}. The ball of radius \(R_{\rm loc}\) contains every fast iterate and every fixed point \(x^\star(\theta)\).

Unless an alternative schedule is stated explicitly, we use polynomial
step sizes with a positive-integer shift \(N_0\ge1\):
\nal{
	\alpha_n = \frac{1}{(n+\no)^{\fraka}},\beta_n = \frac{1}{(n+\no)^{\frakb}}, n\in\bN,
}
where $\fraka\in(0,1)$ and $\frakb\in(0,1]$. The exponent $\fraka$ corresponds to the fast recursion and $\frakb$ to the slow recursion. Since the slow step size must decay faster, we require $\fraka<\frakb$.~The hyperparameters $\fraka,\frakb$ satisfy the following.
\begin{condition}[Step-size exponents]
\label{cond:hyper_1}
The step-size exponents satisfy
\nal{
\frac{1}{2} < \fraka < \frakb \le 1.
}  
\end{condition}
The shift \(N_0\) is part of the algorithm. For the polynomial
two-time-scale schedule above, we assume throughout, unless
explicitly stated otherwise, that
\begin{equation}
(N_0+1)^{1-\fraka}\ge4.
\label{cond:N0:alpha2}
\end{equation}
This choice ensures that the step-size bounds needed in the
analysis hold from index \(n=1\); see
Appendix~\ref{sec:N0_cond} of the Online Companion.

Alternatively, we can begin the analysis after a finite number
of iterations. In this case, we take \(N_0=1\) and run the
recursion from \(n=1\), but begin the block construction at
the first index \(n_0\ge1\) satisfying
\[
(n_0+1)^{1-\fraka} \ge 4.
\]
Appendix~\ref{sec:N0_cond} of the Online Companion verifies
that the analysis applies from \(n=n_0\), with convergence
rates expressed in terms of \(n+1\).
\section{Continuous-time step embeddings and projected dynamical systems}
\label{sec:notation}
We first introduce the cumulative time grid and the right-continuous
step paths used in the projected-ODE analysis. For \(n\ge m\), write
\[
a(m,n):=\sum_{k=m}^{n-1}\beta_k,
\qquad
b(n):=\sum_{k\ge n}\beta_k^2,
\]
and
\[
\chi(m,n):=
\begin{cases}
\displaystyle\prod_{k=m}^{n}(1-\alpha_k),&m\le n,\\
1,&m>n.
\end{cases}
\]
For \(\ell\ge1\), define the cumulative time grid associated with the
\(\theta\)-recursion by
\[
t(\ell):=\sum_{k=1}^{\ell-1}\beta_k.
\]
The empty sum is zero, so \(t(1)=0\). Moreover,
\[
t(n+1)-t(n)=\beta_n,
\qquad
t(n)-t(m)=a(m,n),
\qquad n\ge m\ge1.
\]

Define the right-continuous step embedding of the slow iterate by
\[
\te^{\circ}(r):=\te_n,
\qquad
r\in[t(n),t(n+1)).
\]
Thus,
\[
\te^{\circ}(t(n))=\te_n.
\]
By construction, the path \(\te^{\circ}\) is right-continuous with left
limits (càdlàg) and piecewise constant.
Fix a slow-time block horizon \(T>0\).
Let \(n_0\) be the starting index of the analysis specified in
Section~\ref{sec:pf}, with \(n_0=1\) under the default choice
of \(N_0\). Define recursively
\[
n_m
:=
\min\{n>n_{m-1}:t(n)\ge t(n_{m-1})+T\},
\qquad m\ge1.
\]
For \(m\ge0\), set
\[
T_m:=t(n_m),
\qquad
\cK_m:=\{n_m,n_m+1,\ldots,n_{m+1}-1\},
\]
and write
\[
K_m:=n_{m+1}-n_m,
\qquad
H_m:=T_{m+1}-T_m,
\]
We call \(\cK_m\) the \(m\)th discrete-time block;
its associated slow-time interval is \([T_m,T_{m+1}]\).
A local index \(k\) measures displacement
from the block entrance \(n_m\), with \(k=K_m\) corresponding to the
right endpoint \(n_{m+1}\).

For \(0\le\ell\le K_m\), define the shifted slow-time grid by
\[
t^{(m)}(\ell)
:=
t(n_m+\ell)-t(n_m).
\]
Then
\[
t^{(m)}(0)=0,
\qquad
t^{(m)}(K_m)=H_m,
\]
and
\[
t^{(m)}(\ell+1)-t^{(m)}(\ell)
=
\beta_{n_m+\ell}.
\]

For any discrete-time sequence \(\{u_n\}_{n\ge1}\), write
\[
u_\ell^{(m)}:=u_{n_m+\ell}.
\]
	For any sequence \(\{q_k^{(m)}\}_{0\le k\le K_m}\) indexed by the local
	position \(k\) relative to the entrance of block \(m\), define its right-continuous
block step path by
\[
q^{(m),\circ}(r)
:=
q_k^{(m)},
\qquad
r\in[t^{(m)}(k),t^{(m)}(k+1)),
\quad 0\le k<K_m,
\]
and set
\[
q^{(m),\circ}(H_m):=q_{K_m}^{(m)}.
\]
Let \(\Gamma\) denote the normal-reflection Skorokhod map on \(\Te\), and
let \(\LS\) denote a finite Lipschitz constant for this map, as stated in
Theorem~\ref{th:skrohod_lip} of the Online Companion. When a block path is
passed to the Skorokhod map, it is extended
constantly beyond \(H_m\); this does not affect any block estimate.
In particular,
\[
	\te^{(m),\circ}(r)
=
\te^{\circ}(T_m+r),
	\qquad 0\le r\le H_m.
	\]
We use parenthesized superscripts to label shifted objects. Thus,
\(u_\ell^{(m)}\) denotes a discrete-time sequence reindexed at the
block entrance \(n_m\), whereas \(z^{(T_m)}\) denotes the
projected-ODE trajectory initialized at the slow-time anchor \(T_m\).

\subsection{Projected dynamical systems}
Following~\citet{dupuis1993dynamical}, let
\(C\subset\mathbb R^d\) be a nonempty closed convex set,
and let \(\operatorname{proj}_C\) denote Euclidean projection
onto \(C\). For \(z\in C\) and \(v\in\mathbb R^d\), define
\begin{equation}
\Pi_C(z,v):=
\lim_{\varepsilon\downarrow0}
\frac{\operatorname{proj}_C(z+\varepsilon v)-z}{\varepsilon}.
\label{def:pi}
\end{equation}
For \(\varepsilon>0\), also define
\begin{equation}
\Pi_{C,\varepsilon}(z,v):=
\frac{\operatorname{proj}_C(z+\varepsilon v)-z}{\varepsilon}.
\label{def:pi_beta}
\end{equation}
We omit the set subscript when it is clear from the context.

\begin{definition}[\citet{dupuis1993dynamical}]
Let \(F:C\to\mathbb R^d\). A function
\(z:[0,\infty)\to C\) is a solution of
\(\dot z=\Pi_C(z,F(z))\) if it is absolutely continuous and
\[
\dot z(t)=\Pi_C(z(t),F(z(t)))
\]
for almost every \(t\ge0\).
\end{definition}
For the slow recursion, we apply these definitions with
\(C=\Theta\) and \(d=d_\theta\).
\section{Concentration proof sketch and overview}\label{sec:sketch}
The concentration arguments are developed in
Sections~\ref{sec:faster_iterations} and~\ref{sec:analysis_te_n};
selected proofs appear in Appendix~\ref{sec:article_selected_proofs}
and supporting proofs in the Online Companion. Since the fast recursion is contractive after averaging, the first step of the analysis is to show that $x_n$ tracks the moving fixed point $x\ust(\te_n)$ with high probability. This allows the slow recursion to be compared with an auxiliary projected recursion in which $x_n$ is replaced by $x\ust(\te_n)$. After averaging over the controlled Markov chain, the limiting slow dynamics are described by the differential equation
\al{
	\dot{z}   =  \Pi_{\Te}(z(t),\favg(z(t);z(t)) ), \label{eq:ode}
} 
where, for $\te,\te'\in\Te$, we define
\al{
\favg(\te;\te') :=  \sum_{y\in\cY} \mu^{(\te')}(y) f(y,x\ust(\te),\te).\label{def:favg}
}
In \(\favg(\te;\te')\), the first argument \(\te\) is
substituted into both the fixed-point map \(x\ust(\cdot)\) and the
parameter coordinate of \(f\), while the second argument \(\te'\)
determines the stationary distribution used to average over \(y\).
We refer to \(\favg\) as the two-parameter averaged field. Its evaluation
is called diagonal when the two arguments agree and off-diagonal otherwise.

We first establish a high-probability bound on the fast tracking
error \(\|x_n-x^\star(\theta_n)\|\);
see Proposition~\ref{prop:conc_v_v_te}.
Using this estimate, we then obtain a concentration bound on
the maximum deviation of the embedding of the slow iterates
from the projected-ODE trajectory over each finite slow-time
block; see Theorem~\ref{th:tube}.

For \(\tau\ge0\), let \(z^{(\tau)}(\cdot)\) denote the solution
to~\eqref{eq:ode} satisfying
\[
z^{(\tau)}(0)=\te^{\circ}(\tau).
\]
In particular, on the \(m\)th block, we compare
\[
\te^{(m),\circ}(r)
=
\te^{\circ}(T_m+r)
\]
with \(z^{(T_m)}(r)\) for \(0\le r\le H_m\).

Our analysis has three main components. First,
Section~\ref{sec:faster_iterations} establishes a high-probability
tracking bound for the fast iterate \(x_n\) around the moving
equilibrium \(x\ust(\te_n)\). Second, Section~\ref{sec:analysis_te_n} compares the step
embedding \(\te^{\circ}\) with the projected ODE over finite
time intervals. Third, Section~\ref{sec:conv_rate} combines this finite-time tracking estimate with
a Lipschitz Lyapunov function satisfying a uniform one-block decrease condition 
for the projected ODE to obtain almost-sure convergence. 
Explicit rates follow when, for sufficiently small Lyapunov
errors, the decrease over a fixed time horizon is bounded below
by a positive constant times a power of the current error.
One-block Lyapunov contraction is the linear-decrease special case.

The main technical obstacle is the projection. The projected vector
field in the limiting ODE may be discontinuous on the boundary of
\(\Te\), so standard stochastic-approximation concentration arguments
based on Lipschitz ODE fields do not apply directly. We instead use
right-continuous step embeddings of the discrete projected recursions
and their unconstrained drivers. More precisely, each projected step
path is exactly the image of its corresponding unconstrained driving
path under the Skorokhod map. The finite-horizon supremum-norm Lipschitz
property of that map then controls the difference between two projected
paths through the difference between their unconstrained drivers.

\section{Blockwise concentration of the fast iterate}
\label{sec:faster_iterations}
We analyze the fast recursion on the slow-time blocks introduced
in Section~\ref{sec:notation}. On block \(m\), our goal is to bound
the \emph{fast tracking error}, the distance between the fast
iterate and its moving equilibrium:
\[
\|x_{n_m+k}-x^\star(\theta_{n_m+k})\|,
\qquad 0\le k\le K_m.
\]
We bound this error through two comparisons: first, between the
original fast recursion and a block-restarted averaged recursion;
second, between this averaged recursion and the moving equilibrium.

Fix \(m\ge0\). For \(0\le k\le K_m\), write
\[
x_k^{(m)}:=x_{n_m+k},
\qquad
\te_k^{(m)}:=\te_{n_m+k},
\qquad
\alpha_k^{(m)}:=\alpha_{n_m+k},
\qquad
\beta_k^{(m)}:=\beta_{n_m+k}.
\]
On block \(m\), we compare the original fast iterate with an
averaged recursion initialized at the same value \(x_{n_m}\) and driven
by the same slow-parameter sequence
\(\{\te_k^{(m)}\}_{0\le k\le K_m}\). Consequently, the difference
between these two recursions vanishes at the block entrance and records
only the explicit martingale increments and the Markov noise terms
\[
h(Y_{n_m+k},x_{n_m+k},\te_k^{(m)})
-\havg(x_{n_m+k};\te_k^{(m)})
\]
accumulated within the block. The discrepancy already present at time
\(n_m\) enters through the distance between the averaged comparison
recursion and the moving fixed point \(x\ust(\te_k^{(m)})\).

\subsection{Block-restarted averaged recursion}
\label{sec:v_tilde}

Define the block-restarted averaged recursion
\(\{\widetilde x_k^{(m)}\}_{0\le k\le K_m}\) by
\[
\widetilde x_0^{(m)}:=x_{n_m},
\]
and, for \(0\le k<K_m\),
\begin{equation}
\widetilde x_{k+1}^{(m)}
=
\widetilde x_k^{(m)}
+
\alpha_k^{(m)}
\left[
\havg\bigl(\widetilde x_k^{(m)};\te_k^{(m)}\bigr)
-
\widetilde x_k^{(m)}
\right].
\label{def:x_tilde}
\end{equation}
Since $\norm{\havg(0;\te)}\le \mathsf B_0$, contraction gives
\[
\norm{\widetilde x_{k+1}^{(m)}}
\le\{1-(1-\gamma)\alpha_k^{(m)}\}
\norm{\widetilde x_k^{(m)}}+\alpha_k^{(m)}\mathsf B_0.
\]
Since \(\widetilde x_0^{(m)}=x_{n_m}\), induction gives
\[
\|\widetilde x_k^{(m)}\|
\le R_{\rm aux}
:=\max\left\{\frakc_6,\frac{\mathsf B_0}{1-\gamma}\right\}.
\]
The auxiliary iterates need not lie in the ball of radius
\(R_{\rm loc}\). At these iterates, we use the global contraction
of the averaged map; no bounded-set estimate for the unaveraged
maps \(h\) or \(f\) is needed. The recursion~\eqref{def:x_tilde} applies the averaged map \(\havg\)
from~\eqref{def:havg} to its own state
\(\widetilde x_k^{(m)}\), while using the same slow-parameter sequence
\(\{\te_k^{(m)}\}_{0\le k\le K_m}\). Thus it contains neither the
explicit noise increment \(M_{n_m+k+1}\) nor the sampled Markov state
\(Y_{n_m+k}\). It remains random only through its block initial value
and the random sequence \(\{\te_k^{(m)}\}\).

To quantify the first comparison---the discrepancy between the original
fast recursion and this averaged recursion---we define
for \(0\le k\le K_m\),
\[
D_k^{(m)}:=x_k^{(m)}-\widetilde x_k^{(m)}.
\]
By construction, \(D_0^{(m)}=\mathbf 0\).

\subsection{Poisson decomposition and local fast fluctuations}
The Markov noise
\[
h(Y_n,x_n,\te_n)-\havg(x_n;\te_n)
\]
is generally not a martingale difference because \(Y_n\) is Markovian.
The Poisson equation is used to decompose this error into a
martingale-difference term and a difference between two successive-state
evaluations of the Poisson solution.
Since the sampling error enters the fast update multiplied by
\(\alpha_n\), iterating the recursion produces weighted partial
sums of the two terms. We control the martingale contribution
using a concentration inequality and bound the remaining
Poisson residual by summation by parts, using the slow variation
of \((x_n,\theta_n)\).
\begin{definition}[Poisson equation for the fast recursion]
\label{def:poisson_fast}
Fix \(x\in\bR^{d_x}\) and \(\te\in\Te\). For the Markov chain on
\(\cY\) with transition matrix \(p\ute\), consider the Poisson equation
\begin{equation}
u(y)-\sum_{i\in\cY}p\ute(y,i)u(i)
=
h(y,x,\te)-\havg(x;\te),
\qquad y\in\cY.
\label{def:poisson}
\end{equation}
We use the particular solution, as in~\citet{chandak2022concentration}
\begin{equation}
u_h^{(x,\te)}(y)
:=
\bE^{(\te)}\left[
\sum_{r=0}^{\tau_{y\ust}-1}
\bigl(h(Y_r,x,\te)-\havg(x;\te)\bigr)
\,\middle|\,
Y_0=y
\right],
\label{def:poisson_soln_h}
\end{equation}
where \(\tau_{y\ust}\) is the hitting time of the reference state
\(y\ust\) from Assumption~\ref{assum:hitting_time}. The superscript
\((\te)\) on \(\bE^{(\te)}\) indicates that the controlled Markov chain
\((Y_r)\) is run with transition matrix \(p^{(\te)}\);
the arguments \(x\) and \(\theta\) in the summand are held
fixed as the chain evolves.

The expectation in~\eqref{def:poisson_soln_h} is finite by
Assumption~\ref{assum:hitting_time}. Moreover, by the definition of
\(\havg\),
\[
\sum_{y\in\cY}\mu^{(\te)}(y) \left\{h(y,x,\te)-\havg(x;\te)\right\} = \mathbf 0.
\]
Lemma~\ref{lemma:finite_state_cycle_formula} of the Online Companion,
applied coordinatewise, therefore gives
\(u_h^{(x,\te)}(y\ust)=\mathbf 0\). Conditioning on the first
transition then shows that~\eqref{def:poisson_soln_h}
solves~\eqref{def:poisson}.
\end{definition}

Let \(\chpoi<\infty\) denote the Poisson-sensitivity constant from
Lemma~\ref{lemma:poisson_sensitivity_1}. In particular, whenever
\(\norm{x},\norm{x'}\le\frakc_6\), \(\te,\te'\in\Te\), and
\(y\in\cY\),
\[
\norm{
u_h^{(x,\te)}(y)-u_h^{(x',\te')}(y)}
\le
\chpoi\bigl(\norm{x-x'}+\norm{\te-\te'}\bigr).
\]
\begin{lemma}[Poisson decomposition]
\label{lemma:9}
Define
\begin{equation}
M''_{n+1}
:=
u_h^{(x_n,\te_n)}(Y_{n+1})
-
\sum_{y\in\cY}p^{(\te_n)}(Y_n,y)
u_h^{(x_n,\te_n)}(y),
\qquad n\ge1.
\label{def:M''_n}
\end{equation}
Then \(\{M''_{n+1}\}_{n\ge1}\) is a martingale-difference sequence
with respect to \(\{\cF_n\}\), and, almost surely,
\begin{equation}
h(Y_n,x_n,\te_n)-\havg(x_n;\te_n)
=
M''_{n+1}
+
u_h^{(x_n,\te_n)}(Y_n)
-
u_h^{(x_n,\te_n)}(Y_{n+1}).
\label{eq:fast_poisson_pointwise}
\end{equation}
\end{lemma}
The sequence \(\{M''_{n+1}\}\) is defined on the original time axis and
therefore does not carry a block index. On block \(m\), the relevant
increments are simply
\[
M''_{n_m+k+1},\qquad 0\le k<K_m.
\]
For later use, define
\[
\begin{aligned}
\Delta_{n+1}^h
&:=
u_h^{(x_n,\te_n)}(Y_n)
   -u_h^{(x_n,\te_n)}(Y_{n+1}),\\
\Xi_{n+1}
&:=M_{n+1}+M''_{n+1},
\qquad n\ge1.
\end{aligned}
\]
Then~\eqref{eq:fast_poisson_pointwise} becomes
\[
h(Y_n,x_n,\te_n)-\havg(x_n;\te_n)+M_{n+1}
=
\Xi_{n+1}+\Delta_{n+1}^h.
\]

Set
\[
\cfastm
:=
\frakc_1+\frakc_2\frakc_6
+4\frakc_5\mathsf B_6,
\qquad
c_{\mathrm{mg}}^x
:=
\frac{1}{4\kappa_{d_x}^2(\cfastm)^2}.
\]
For \(0\le k\le K_m\), define the local weighted martingale sum
using \(\alpha_j^{(m)}=\alpha_{n_m+j}\), by
\begin{equation}
\mathcal M_{m,k}^x
:=
\sum_{j=0}^{k-1}
\alpha_j^{(m)}
\chi(n_m+j+1,n_m+k-1)
\Xi_{n_m+j+1},
\label{def:block_fast_martingale_transform}
\end{equation}
with \(\mathcal M_{m,0}^x:=\mathbf 0\).
For \(\delta>0\), define the blockwise fast-martingale event
\begin{equation}
\cG_x^{(m)}(\delta)
:=
\left\{
\max_{0\le k\le K_m}
\kappa_{d_x}\norm{\mathcal M_{m,k}^x}_{\infty}
<\delta
\right\}.
\label{eq:block_fast_good_event}
\end{equation}

\begin{lemma}[Blockwise fast-martingale tail bound]
\label{lem:block_fast_tail}
For every \(m\ge0\) and \(\delta>0\),
\begin{equation}
\bP\left(
\left[\cG_x^{(m)}(\delta)\right]^c
\right)
\le
2d_xK_m
\exp\left\{
-\frac{c_{\mathrm{mg}}^x \delta^2}{\alpha_{n_m}}
\right\}.
\label{eq:block_fast_tail}
\end{equation}
Equivalently,
\[
\bP\left(
\left[\cG_x^{(m)}(\delta)\right]^c
\right)
\le
2d_xK_m
\exp\left\{
-c_{\mathrm{mg}}^x \delta^2(\no+n_m)^{\fraka}
\right\}.
\]
\end{lemma}

\subsection{Comparison with the block-restarted averaged recursion}
The successive-state term in the Poisson decomposition produces a
stepsize-weighted telescoping residual. Summation by parts and the slow
variation of \((x_n,\te_n)\) bound this residual by
\(O(\alpha_{n_m})\). Contraction then converts the martingale and
residual terms into the following comparison. The residual calculation 
remains in Appendix~\ref{sec:appendix_proofs}; see
Sections~\ref{sec:app_proof_fast_aux_comparison}
and~\ref{sec:app_proof_moving_equilibrium} for the comparison proofs.

\begin{lemma}[Local comparison with the averaged recursion]
\label{lem:block_fast_aux_comparison}
There exists
\(C_{\rm aux}<\infty\), depending only on the problem parameters, such
that, on \(\cG_x^{(m)}(\delta)\),
\begin{equation}
\norm{x_k^{(m)}-\widetilde x_k^{(m)}}
\le
C_{\rm aux}\bigl(\delta+\alpha_{n_m}\bigr),
\qquad
0\le k\le K_m.
\label{eq:block_fast_aux_comparison}
\end{equation}
\end{lemma}

\subsection{Tracking the moving fixed point}
We next compare \(\widetilde x_k^{(m)}\) with the moving fixed point
\(x\ust(\te_k^{(m)})\). Set \(q:=1-\gamma\) and, for
\(0\le k\le K_m\), define
\[
\mathcal W_m(k)
:=
\prod_{\ell=0}^{k-1}
\left(1-q\alpha_{\ell}^{(m)}\right),
\]
with the empty product equal to one. Thus \(\mathcal W_m(k)\) is the
contraction accumulated from the block entrance to local index \(k\).
\begin{lemma}[Moving-equilibrium comparison on a block]
\label{lem:block_moving_equilibrium}
There exists \(C_{\rm mov}<\infty\), depending only on the problem
parameters, such that, almost surely, for every \(m\ge0\) and \(0\le k\le K_m\),
\begin{equation}
\begin{aligned}
\norm{\widetilde x_k^{(m)}-x\ust(\te_k^{(m)})}
&\le
\mathcal W_m(k)
\norm{x_{n_m}-x\ust(\te_{n_m})}\\
&\quad+
\frac{C_{\rm mov}}{q}
\frac{\beta_{n_m}}{\alpha_{n_m}}.
\end{aligned}
\label{eq:block_moving_equilibrium}
\end{equation}
\end{lemma}
The two terms in~\eqref{eq:block_moving_equilibrium} are, respectively,
the contracted block-entry error and the deterministic tracking lag of
order \(\beta_{n_m}/\alpha_{n_m}\).
\subsection{Blockwise fast tracking}
Finally, we combine the comparisons in
Lemmas~\ref{lem:block_fast_aux_comparison}
and~\ref{lem:block_moving_equilibrium}
using the triangle inequality.
This gives pointwise, stepsize-weighted, and block-endpoint
bounds for the fast tracking error. The block-endpoint estimate can
then be iterated over successive blocks.
\begin{proposition}[Blockwise fast tracking]
\label{prop:conc_v_v_te}
Let
\[
\varepsilon_n^{(x)}
:=
\norm{x_n-x\ust(\te_n)}.
\]
There exists
\(C_x<\infty\), depending only on the problem parameters, such that,
on \(\cG_x^{(m)}(\delta)\),
\begin{equation}
\varepsilon_{n_m+k}^{(x)}
\le
\mathcal W_m(k)\varepsilon_{n_m}^{(x)}
+
C_x
\left(
\delta
+
\alpha_{n_m}
+
\frac{\beta_{n_m}}{\alpha_{n_m}}
\right),
\qquad
0\le k\le K_m.
\label{eq:block_fast_pointwise}
\end{equation}
Moreover, there exists \(C_{x,T}<\infty\), depending only on \(T\)
and the problem parameters, such that, on the same event,
\begin{equation}
\sum_{r=n_m}^{n_{m+1}-1}
\beta_r
\norm{x_r-x\ust(\te_r)}
\le
C_{x,T}
\left(
\delta
+
\alpha_{n_m}
+
\frac{\beta_{n_m}}{\alpha_{n_m}}
\right).
\label{eq:block_fast_weighted}
\end{equation}
At the block endpoint,
\begin{equation}
\varepsilon_{n_{m+1}}^{(x)}
\le
\mathcal W_m(K_m)\varepsilon_{n_m}^{(x)}
+
C_x
\left(
\delta
+
\alpha_{n_m}
+
\frac{\beta_{n_m}}{\alpha_{n_m}}
\right).
\label{eq:block_fast_endpoint}
\end{equation}
\end{proposition}

\begin{proof}
See Appendix~\ref{sec:app_proof_fast_tracking}.
\end{proof}
For \(m\ge0\) and \(\delta>0\), set
\begin{equation}
g_m^{(x)}(\delta)
:=
\delta+\alpha_{n_m}
+\frac{\beta_{n_m}}{\alpha_{n_m}}.
\label{def:block_fast_envelope}
\end{equation}
\begin{corollary}[Fast tracking across successive blocks]
\label{coro:fast_across_blocks}
For every \(m\ge1\)
and every choice of positive real numbers
\(\delta_0,\ldots,\delta_{m-1}\), on
\(\bigcap_{j=0}^{m-1}\cG_x^{(j)}(\delta_j)\),
\begin{equation}
\varepsilon_{n_m}^{(x)}
\le
\left(
\prod_{\ell=0}^{m-1}\mathcal W_\ell(K_\ell)
\right)
\varepsilon_{n_0}^{(x)}
+
C_x
\sum_{j=0}^{m-1}
\left(
\prod_{\ell=j+1}^{m-1}\mathcal W_\ell(K_\ell)
\right)
g_j^{(x)}(\delta_j).
\label{eq:fast_across_blocks}
\end{equation}
Furthermore,
\begin{equation}
\bP\left(
\bigcup_{j=0}^{m-1}
\left[\cG_x^{(j)}(\delta_j)\right]^c
\right)
\le
\sum_{j=0}^{m-1}
2d_xK_j
\exp\left\{
-\frac{c_{\mathrm{mg}}^x \delta_j^2}{\alpha_{n_j}}
\right\}.
\label{eq:fast_across_blocks_probability}
\end{equation}
\end{corollary}

\section{Concentration of the slow iterate}
\label{sec:analysis_te_n}
We now establish a blockwise high-probability comparison between the
step embedding of the slow iterate and the projected
ODE~\eqref{eq:ode}. On each block, we introduce an auxiliary projected
recursion and decompose the comparison into two stages. First, we
compare the original slow recursion with the auxiliary recursion,
thereby isolating the fast-tracking error, the martingale noise in the
slow recursion, and the Markov noise. Second, we compare the
auxiliary recursion with the projected ODE. This second comparison controls the
time-grid discretization error and the mismatch arising because the
auxiliary recursion evaluates the first and second arguments of the averaged
field at its own current state and the original slow state, respectively,
whereas the ODE
evaluates both arguments at its current state.

Throughout this section, ``auxiliary recursion'' refers to the slow
projected recursion introduced below. It is distinct from the
block-restarted averaged fast recursion of
Section~\ref{sec:faster_iterations}.

Since the projection operator is nonlinear, both comparisons are
carried out through the corresponding unconstrained driving paths.
Bounds on the driving paths are then transferred to the projected
paths using the Lipschitz property of the Skorokhod map.

Fix \(m\ge0\) and use the block notation introduced in
Section~\ref{sec:notation}. In particular, for \(0\le n\le K_m\), write
\[
\te\um_n:=\te_{n_m+n},
\qquad
x\um_n:=x_{n_m+n},
\qquad
Y\um_n:=Y_{n_m+n}.
\]
For \(0\le n\le K_m-1\), define
\[
\beta\um_n:=\beta_{n_m+n},
\qquad
(M')\um_{n+1}:=M'_{n_m+n+1}.
\]
For \(0\le n\le K_m\), also define the block-shifted filtration and the
fast-tracking error by
\[
\cF\um_n:=\cF_{n_m+n},
\qquad
e\um_n:=x\um_n-x\ust(\te\um_n).
\]
Throughout this section, the parenthesized superscript \((m)\) denotes
reindexing relative to the \(m\)th block; it is neither an exponent nor
an abbreviation for ``martingale.'' 

Our goal is to control
\[
\sup_{0\le t\le H_m}
\norm{\te^{(m),\circ}(t)-z^{(T_m)}(t)},
\]
where \(z^{(T_m)}\) solves~\eqref{eq:ode} with
\[
z^{(T_m)}(0)
=
\te^{\circ}(T_m)
=
\te_{n_m}.
\]
\subsection{Auxiliary projected recursion}
\label{subsec:aux_seq}
The shifted slow recursion can be written as
\[
\te\um_{n+1}
=
\te\um_n
+
\beta\um_n
\Pi_{\Te,\beta\um_n}
\left(
\te\um_n,
f(Y\um_n,x\ust(\te\um_n),\te\um_n)
+
\mathcal E_{f,n}^{(m)}
+
(M')\um_{n+1}
\right),
\]
where
\[
\mathcal E_{f,n}^{(m)}
:=
f(Y\um_n,x\um_n,\te\um_n)
-
f(Y\um_n,x\ust(\te\um_n),\te\um_n).
\]
Thus, \(\mathcal E_{f,n}^{(m)}\) is the drift error caused by the fast
iterate not being at its moving fixed point.

Define the auxiliary projected recursion by setting
\[
\tte\um_0=\te\um_0
\]
and, for \(0\le n\le K_m-1\),
\begin{equation}
\tte\um_{n+1}
=
\tte\um_n
+
\beta\um_n
\Pi_{\Te,\beta\um_n}
\left(
\tte\um_n,
\favg(\tte\um_n;\te\um_n)
\right).
\label{def:te_n_prime}
\end{equation}
Equivalently,
\[
\tte\um_{n+1}
=
\proj\left(
\tte\um_n
+
\beta\um_n\favg(\tte\um_n;\te\um_n)
\right).
\]
In~\eqref{def:te_n_prime}, the two-parameter averaged field is evaluated
off diagonal:
\[
\favg(\tte\um_n;\te\um_n)
=
\sum_{y\in\cY}
\mu^{(\te\um_n)}(y)
f\bigl(y,x\ust(\tte\um_n),\tte\um_n\bigr).
\]
The auxiliary iterate \(\widetilde\theta_n^{(m)}\) determines
where the drift is evaluated. The original slow iterate
\(\theta_n^{(m)}\) determines the averaging distribution
\(\mu^{(\theta_n^{(m)})}\), since the Markov chain uses the
transition kernel \(p^{(\theta_n^{(m)})}\) at this iteration.
Replacing the second argument by \(\widetilde\theta_n^{(m)}\)
gives the diagonal field used by the projected ODE.
The error introduced by this replacement is bounded by a
constant times
\(\|\widetilde\theta_n^{(m)}-\theta_n^{(m)}\|\),
which we control by comparing the original and auxiliary
recursions.

Before comparing the projected paths, we compare the unprojected update
directions that drive their Skorokhod representations. The difference
between the original and auxiliary update directions has the exact
decomposition
\[
\begin{aligned}
&f(Y\um_n,x\um_n,\te\um_n)+(M')\um_{n+1}
-\favg(\tte\um_n;\te\um_n)\\
&\quad=
\bigl[
f(Y\um_n,x\um_n,\te\um_n)
-f(Y\um_n,x\ust(\te\um_n),\te\um_n)
\bigr]\\
&\qquad+
\bigl[
f(Y\um_n,x\ust(\te\um_n),\te\um_n)
-\favg(\te\um_n;\te\um_n)
\bigr]\\
&\qquad+
\bigl[
\favg(\te\um_n;\te\um_n)
-\favg(\tte\um_n;\te\um_n)
\bigr]
+(M')\um_{n+1}.
\end{aligned}
\]
The four terms are, respectively, the fast-tracking error, the centered
Markovian sampling error, the auxiliary-state mismatch, and the explicit
slow martingale noise.

To compare the projected recursions, introduce their unconstrained
driving sequences. Set
\[
v\um_0=\te\um_0,
\qquad
\tv\um_0=\tte\um_0,
\]
and, for \(0\le n\le K_m-1\), define
\begin{equation}
v\um_{n+1}
=
v\um_n
+
\beta\um_n
\left(
f(Y\um_n,x\um_n,\te\um_n)
+
(M')\um_{n+1}
\right),
\label{def:v_unconstrained}
\end{equation}
and
\begin{equation}
\tv\um_{n+1}
=
\tv\um_n
+
\beta\um_n
\favg(\tte\um_n;\te\um_n).
\label{def:vtilde_unconstrained}
\end{equation}
The step-path representation and Lipschitz comparison are proved in
Lemma~\ref{lemma:skor_1} of the Online Companion's Skorokhod section.
In particular, for every \(0\le n\le K_m\),
\[
\max_{0\le k\le n}
\norm{\te\um_k-\tte\um_k}
\le
\LS
\max_{0\le k\le n}
\norm{v\um_k-\tv\um_k}.
\]
Subtracting~\eqref{def:vtilde_unconstrained} from
\eqref{def:v_unconstrained} gives
\begin{align}
v\um_{n+1}-\tv\um_{n+1}
&=
v\um_n-\tv\um_n
+
\beta\um_n
\left\{
\favg(\te\um_n;\te\um_n)
-
\favg(\tte\um_n;\te\um_n)
\right\}
\notag\\
&\quad+
\beta\um_n
\left\{
f(Y\um_n,x\ust(\te\um_n),\te\um_n)
-
\favg(\te\um_n;\te\um_n)
\right\}
\notag\\
&\quad+
\beta\um_n
\left\{
f(Y\um_n,x\um_n,\te\um_n)
-
f(Y\um_n,x\ust(\te\um_n),\te\um_n)
\right\}
\notag\\
&\quad+
\beta\um_n(M')\um_{n+1}.
\label{eq:driver-difference-before-poisson}
\end{align}
The second term on the right-hand side of
\eqref{eq:driver-difference-before-poisson}, namely,
\al{
f(Y\um_n,x\ust(\te\um_n),\te\um_n)-\favg(\te\um_n;\te\um_n).\label{def:fast_mkv_noise}
}
is the Markov noise obtained by subtracting the stationary
average under \(\mu^{(\theta_n^{(m)})}\). To control it,
we use a Poisson-equation decomposition for the controlled
Markov chain.
\begin{definition}[Poisson equation for the slow drift]\label{def:poisson_slow}
For $\te\in\Te$, consider the Poisson equation
\begin{equation}
u(y)-\sum_{j\in\cY}p\ute(y,j)u(j)
=
f(y,x\ust(\te),\te)-\favg(\te;\te),\qquad y\in\cY.
\label{def:poisson_f}
\end{equation}
Let $\{Y_n\}$ be the Markov chain with transition matrix $p\ute$. For $y\in\cY$, define
\begin{equation}
u^{(\te)}_f(y):=\bE^{(\te)}\left[\sum_{n=0}^{\tau-1}\left(f(Y_n,x\ust(\te),\te)-\favg(\te;\te)\right) \mathrel{\bigg|} Y_0 = y \right],
\label{def:poisson_f_soln}
\end{equation}
where $\tau$ is the hitting time of the reference state $y\ust$ in Assumption~\ref{assum:hitting_time}. The expectation \(\mathbb E^{(\theta)}\) is taken under a Markov chain that uses the transition kernel \(p^{(\theta)}\)
at every step. The parameter \(\theta\) in the summand is
held fixed as the chain evolves.

The same argument as in Definition~\ref{def:poisson_fast}
shows that~\eqref{def:poisson_f_soln} solves
\eqref{def:poisson_f}, with
\(u_f^{(\theta)}(y^\star)=\mathbf 0\).
\end{definition}
Next we define a martingale-difference sequence.

\begin{lemma}[Poisson decomposition]\label{lemma:metivier_mtg}
Define, for $0\le n\le K_m-1$,
\begin{equation}
(M''')\um_{n+1}:=u^{(\te\um_n)}_f(Y\um_{n+1})-\sum_{y\in\cY}p^{(\te\um_n)}(Y\um_n,y)u^{(\te\um_n)}_f(y).
\label{def:M_3'}
\end{equation}
Then
\(\{(M''')\um_{n+1}\}_{0\le n\le K_m-1}\) is a
martingale-difference sequence with respect to
\(\{\cF\um_n\}_{0\le n\le K_m}\); more precisely,
\[
(M''')\um_{n+1}\ \text{is }\cF\um_{n+1}\text{-measurable},
\qquad
\bE\!\left[(M''')\um_{n+1}\mid\cF\um_n\right]=\mathbf 0,
\]
for every \(0\le n\le K_m-1\). Moreover, for every
\(0\le n\le K_m-1\), almost surely,
\begin{equation}
\begin{aligned}
&
f(Y\um_n,x\ust(\te\um_n),\te\um_n)
-
\favg(\te\um_n;\te\um_n)
\\
&\qquad=
u_f^{(\te\um_n)}(Y\um_n)
-
u_f^{(\te\um_n)}(Y\um_{n+1})
+
(M''')\um_{n+1}.
\end{aligned}
\label{eq:slow-poisson-decomposition}
\end{equation}
\end{lemma}
Let \(C_P\) and \(C_L\) denote the constants in
Lemma~\ref{lemma:adhoc_1}, and set
\[
R_m:=3C_P\beta_{n_m}+C_L\{b(n_m)-b(n_{m+1})\}.
\]
To handle the Markov noise~\eqref{def:fast_mkv_noise}, we use the
Poisson decomposition. Its martingale-difference term is controlled
later on the slow-martingale good event defined below, and its Poisson boundary term is
controlled by the deterministic residual bound $R_m$.

Substituting~\eqref{eq:slow-poisson-decomposition} into the Markov noise term
in~\eqref{eq:driver-difference-before-poisson} gives
\begin{align}
v\um_{n+1}-\tv\um_{n+1}
&=v\um_n-\tv\um_n+\beta\um_n\left\{\favg(\te\um_n;\te\um_n)-\favg(\tte\um_n;\te\um_n)\right\} \notag\\
&
+\beta\um_n\left\{u^{(\te\um_n)}_f(Y\um_n)-u^{(\te\um_n)}_f(Y\um_{n+1})\right\} \notag\\
&+\beta\um_n\left\{f(Y\um_n,x\um_n,\te\um_n)-f(Y\um_n,x\ust(\te\um_n),\te\um_n)\right\} \notag\\
&+\beta\um_n   \br{ (M''')\um_{n+1}  + (M')\um_{n+1} }.\label{decomp:v_vt}
\end{align}

The final term in~\eqref{decomp:v_vt} is the combined slow martingale:
\((M')\um_{n+1}\) is the explicit martingale noise in the slow
recursion, while \((M''')\um_{n+1}\) arises from the Poisson
decomposition of the slow Markov noise. For each block \(m\) and each
slow-martingale deviation threshold \(\delta_s>0\), define
\begin{equation}
\cG_{\te}^{(m)}(\delta_s)
:=
\left\{
\max_{1\le k\le K_m}
\norm{
\sum_{\ell=0}^{k-1}
\beta\um_{\ell}
\left[
(M')\um_{\ell+1}
+
(M''')\um_{\ell+1}
\right]
}
<
\delta_s
\right\}.
\label{def:G_slow_martingale}
\end{equation}
Thus, on \(\cG_{\te}^{(m)}(\delta_s)\), the cumulative
slow-martingale perturbation remains below \(\delta_s\) throughout
block \(m\). The subscript \(\te\) identifies this as the slow-coordinate
martingale event, while \(m\) is the block index.
Define
\[
C_{\mathrm{mg}}^{(\theta)} :=\frakc_1+\frakc_2\frakc_6+4\frakc_5\mathsf B_8.
\]

\begin{lemma}\label{lemma:mc_diarmid_mtilde}
For every \(m\ge0\) and \(\delta_s>0\),
\[
\bP\left(\left(\cG_{\te}^{(m)}(\delta_s) \right)^c\right)\le 2d_{\te}\sum_{k=1}^{K_m}\exp\left(-\frac{\delta_s^2}{2\kappa_{d_{\te}}^2(C_{\mathrm{mg}}^{(\theta)})^2\{b(n_m)-b(n_m+k)\}}\right).
\]
\end{lemma}

\begin{proposition}[Comparison between $\te\um$ and $\tte\um$]\label{prop:te_te1}
Fix $m\ge0$ and \(\delta_s>0\). On the event
$\cG_{\te}^{(m)}(\delta_s)$, for every $0\le n\le K_m$,
\al{
\sup_{0\le k\le n}\norm{\te\um_k-\tte\um_k}
& \le
\LS\left[
\lf \sum_{\ell=0}^{n-1}\beta\um_{\ell}\norm{e\um_{\ell}}
+  \delta_s +R_m
\right] \notag \\
& \times \exp\left(\lf (\lfp +1)\LS\,a(n_m,n_m+n)\right),
\label{ineq:te_tte_block}
}
Here \(R_m\) is the deterministic Poisson-residual bound proved in
Lemma~\ref{lemma:adhoc_1}.
\end{proposition}

\begin{proof}
See Appendix~\ref{sec:app_proof_te_te1}.
\end{proof}

Proposition~\ref{prop:te_te1} leaves the weighted fast-tracking term
\[
F_m^x
:=
\sum_{\ell=0}^{K_m-1}
\beta\um_\ell\norm{e\um_\ell}
=
\sum_{r=n_m}^{n_{m+1}-1}
\beta_r\norm{x_r-x\ust(\te_r)}
\]
explicit. For a fast-martingale deviation level \(\delta>0\), define
\begin{equation}
S_m(\delta)
:=
C_{x,T}
\left(
\delta
+
\alpha_{n_m}
+
\frac{\beta_{n_m}}{\alpha_{n_m}}
\right),
\label{def:Sm}
\end{equation}
where \(C_{x,T}\) is the constant in
Proposition~\ref{prop:conc_v_v_te}. By~\eqref{eq:block_fast_weighted},
on \(\cG_x^{(m)}(\delta)\),
\begin{equation}
F_m^x\le S_m(\delta).
\label{eq:weighted-fast-input-section6}
\end{equation}

For \(\delta_f,\delta_s>0\), set
\begin{align}
\etem(\delta_f,\delta_s)
&:=
\LS
\left[
\lf S_m(\delta_f)
+
\delta_s
+
R_m
\right]
\notag\\
&\quad\times
\exp\left
\{
\lf(\lfp+1)\LS\,a(n_m,n_{m+1})
\right\}.
\label{def:E_theta_ttheta_block}
\end{align}

\begin{corollary}[Original--auxiliary slow comparison]
\label{coro:te-tte}
Fix \(m\ge0\) and \(\delta_f,\delta_s>0\). On
\[
\cG_x^{(m)}(\delta_f)
\cap
\cG_{\te}^{(m)}(\delta_s),
\]
we have
\[
\max_{0\le k\le K_m}
\norm{\te\um_k-\tte\um_k}
\le
\etem(\delta_f,\delta_s).
\]
Equivalently,
\[
\sup_{0\le t\le H_m}
\norm{
\te^{(m),\circ}(t)-\tte^{(m),\circ}(t)
}
\le
\etem(\delta_f,\delta_s).
\]
\end{corollary}

\subsection{Comparison of the auxiliary recursion with the projected ODE}

We next compare the auxiliary recursion with the projected ODE
\[
\dot z(t)
=
\Pi_{\Te}
\left(
z(t),
\favg(z(t);z(t))
\right).
\]
Recall that \(z^{(T_m)}(0)=\te_{n_m}\). The auxiliary recursion uses
the off-diagonal evaluation
\(\favg(\tte\um_n;\te\um_n)\), whereas the ODE uses the diagonal
evaluation \(\favg(z;z)\). For any \(s\in[0,H_m]\) within block \(m\), add and
subtract \(\favg(\tte^{(m),\circ}(s);\tte^{(m),\circ}(s))\) to write
\[
\begin{aligned}
&\favg(z^{(T_m)}(s);z^{(T_m)}(s))
-\favg(\tte^{(m),\circ}(s);\te^{(m),\circ}(s))\\
&\quad=
\Bigl[
\favg(z^{(T_m)}(s);z^{(T_m)}(s))
-\favg(\tte^{(m),\circ}(s);\tte^{(m),\circ}(s))
\Bigr]\\
&\qquad+
\Bigl[
\favg(\tte^{(m),\circ}(s);\tte^{(m),\circ}(s))
-\favg(\tte^{(m),\circ}(s);\te^{(m),\circ}(s))
\Bigr].
\end{aligned}
\]
The first bracket is controlled by the difference between the ODE state
\(z^{(T_m)}(s)\) and the auxiliary state
\(\tte^{(m),\circ}(s)\). The second is the off-diagonal mismatch between
the stationary laws indexed by \(\tte^{(m),\circ}(s)\) and
\(\te^{(m),\circ}(s)\). To compare the auxiliary path \(\tte^{(m),\circ}\) with
the projected-ODE trajectory \(z^{(T_m)}\), we first bound
the difference between their unconstrained driving paths.
The Lipschitz property of the Skorokhod map then bounds
the difference between these two projected paths.
Lemma~\ref{lemma:skorohod_4} of the
Online Companion's Skorokhod section provides an absolutely continuous
driver \(w^{(\tau)}\) such that
\[
w^{(\tau)}(0)=z^{(\tau)}(0)=\te^{\circ}(\tau),\qquad
\dot w^{(\tau)}(t)=\favg(z^{(\tau)}(t);z^{(\tau)}(t))\ \text{a.e.},
\qquad z^{(\tau)}=\Gamma(w^{(\tau)}).
\]

We compare \(w^{(T_m)}\) with the step driver \(\tv^{(m),\circ}\), and Lemma~\ref{lemma:skor_1} bounds the difference between
the projected paths \(z^{(T_m)}\) and \(\tte^{(m),\circ}\)
in terms of the difference between their unconstrained
driving paths.

Let \(\lfavg\) denote the diagonal-field Lipschitz constant defined in
Lemma~\ref{lemma:mu_p_g}\eqref{lemma:mu_p_g-ii}.
\begin{proposition}[Auxiliary-recursion--ODE comparison]
\label{prop:te4_z}
Fix \(m\ge0\) and \(\delta_f,\delta_s>0\), and recall from
Section~\ref{sec:notation} that
\(H_m=t^{(m)}(K_m)=a(n_m,n_{m+1})\).
On
\[
\cG_x^{(m)}(\delta_f)\cap\cG_{\te}^{(m)}(\delta_s),
\]
define, for \(0\le t\le H_m\),
\[
\cA_m(t)
:=
\sup_{0\le s\le t}
\norm{
w^{(T_m)}(s)-\tv^{(m),\circ}(s)
}.
\]
Then
\begin{equation}
\cA_m(t)
\le
\exp\!\left\{\lfavg\LS t\right\}
\left[
\mathsf B_8\lmu\,
t\,\etem(\delta_f,\delta_s)
+
\mathsf B_8\beta_{n_m}
\right].
\label{ineq:w_tv}
\end{equation}
Consequently,
\[
\begin{aligned}
&\sup_{0\le s\le t}
\norm{
z^{(T_m)}(s)-\tte^{(m),\circ}(s)
}\\
&\qquad\le
\LS
\exp\!\left\{\lfavg\LS t\right\}
\left[
\mathsf B_8\lmu\,
t\,\etem(\delta_f,\delta_s)
+
\mathsf B_8\beta_{n_m}
\right].
\end{aligned}
\]
\end{proposition}

\begin{proof}
See Appendix~\ref{sec:app_proof_te4_z}.
\end{proof}

\subsection{Putting the estimates together}
For later use, define
\[
p_{\rm det}:=\frakb-\fraka,
\]
which is the decay exponent of the step-size ratio
\(\beta_n/\alpha_n\).
For \(m\ge0\), a fast-martingale deviation threshold \(\delta_f>0\),
and a slow-martingale deviation threshold \(\delta_s>0\),
define
\begin{align}
B_m(\delta_f,\delta_s)
&:=
\etem(\delta_f,\delta_s)
\notag\\
&\quad+
\LS
\exp\!\left\{
\lfavg\LS a(n_m,n_{m+1})
\right\}
\Big[
\mathsf B_8\lmu
a(n_m,n_{m+1})\etem(\delta_f,\delta_s)
\notag\\
&\hspace{9em}
+
\mathsf B_8\beta_{n_m}
\Big].
\label{def:B_bound}
\end{align}

\begin{theorem}[Localized blockwise tracking of the projected ODE]
\label{th:tube}
The deviation thresholds \(\delta_f\) and \(\delta_s\)
may be chosen arbitrarily for each block \(m\).
Section~\ref{sec:conv_rate} later substitutes the
rate-specific choices \(\delta_{\rm dev}^{(\fm)}(m)\) and
\(\delta_{\rm dev}^{(\sm)}(m)\). Here \({\rm dev}\) denotes a deviation threshold; \(\fm\) and
\(\sm\) denote the fast and slow martingales, respectively.
Fix \(m\ge0\) and \(\delta_f,\delta_s>0\).
On
\[
\cG_x^{(m)}(\delta_f)
\cap
\cG_{\te}^{(m)}(\delta_s),
\]
we have
\begin{equation}
\sup_{0\le t\le H_m}
\norm{
\te^{(m),\circ}(t)-z^{(T_m)}(t)
}
\le
B_m(\delta_f,\delta_s).
\label{eq:block_ode_tracking_generic}
\end{equation}
Moreover, there exists \(C_{B,T}<\infty\), depending only on \(T\)
and the problem parameters, such that
\begin{equation}
B_m(\delta_f,\delta_s)
\le
C_{B,T}
\left\{
\delta_f
+
\delta_s
+
\alpha_{n_m}
+
\frac{\beta_{n_m}}{\alpha_{n_m}}
\right\}.
\label{ineq:B_blockwise_rate}
\end{equation}
Consequently,
\begin{equation}
B_m(\delta_f,\delta_s)
\le
C_{B,T}
\left\{
\delta_f
+
\delta_s
+
(\no+n_m)^{-p_{\rm det}}
\right\}.
\label{ineq:B_blockwise_rate_pdet}
\end{equation}
\end{theorem}

\begin{proof}
See Appendix~\ref{sec:app_proof_tube}.
\end{proof}

\section{Almost-sure convergence and last-iterate rates}
\label{sec:conv_rate}
We now convert the localized blockwise projected-ODE tracking estimate
of Theorem~\ref{th:tube} into iteration-wise almost-sure rates. The
argument has three steps. First, we choose block-dependent fast- and
slow-martingale deviation thresholds whose failure probabilities are
summable over the blocks. The Borel--Cantelli lemma then implies that
the blockwise tracking estimate holds on every sufficiently large block
almost surely. Second, a uniform one-block Lyapunov decrease condition
yields a perturbed nonlinear recursion at the block endpoints. Third,
we analyze this recursion and use the within-block tracking estimate to
transfer the endpoint bound to every discrete iterate. 
Uniform one-block contraction is the special case in which
the Lyapunov value decreases by at least a fixed proportion
over each block.
Throughout this section, the standing assumptions and
Condition~\ref{cond:hyper_1} of Section~\ref{sec:pf}, together
with~\eqref{cond:N0:alpha2}, remain in force.

For $s\ge0$, let $\mathsf S_s$ denote the time-$s$ flow map of the projected ODE~\eqref{eq:ode}. The well-posedness of~\eqref{eq:ode} is established in Lemma~\ref{lemma:ode_soln_unique_exist}. Let $\eq$ denote the set of equilibrium points of the ODE~\eqref{eq:ode}:
\nal{
	\eq  := \left\{ z \in \Theta: \Pi_{\Te}(z,\favg(z;z))= \mathbf{0} \right\}.
} 

\begin{definition}[Lyapunov function]\label{def:lyapunov}
Consider the projected ODE~\eqref{eq:ode}, repeated here for convenience:
\[
\dot z(t)=\Pi_{\Te}\left(z(t),\favg(z(t);z(t))\right),\qquad z(0)=z_0\in\Te.
\]
A continuous function $J:\Te\to\bR_+$ is called a Lyapunov function for~\eqref{eq:ode} if $J(z)=0$ if and only if $z\in\eq$, and if, for every solution $z(\cdot)$ of~\eqref{eq:ode},
\[
J(z(t))=J(\mathsf S_{t-s}(z(s)))\le J(z(s)),\qquad 0\le s\le t<\infty.
\]
Thus $J$ is nonincreasing along
trajectories of the projected ODE.
\end{definition}
For any Lyapunov function \(J\) as in Definition~\ref{def:lyapunov}, write
\[
J_{\max}:=\max_{\te\in\Te}J(\te)<\infty,
\]
where finiteness follows from the continuity of \(J\) and the compactness of \(\Te\).
We use a Lyapunov function to measure convergence to the
equilibrium set.
\begin{assumption}\label{assum:ode_stability}
The projected ODE~\eqref{eq:ode} admits a Lyapunov function in the sense of Definition~\ref{def:lyapunov}.
\end{assumption}
Theorem~\ref{th:tube} bounds, on each slow-time block, the gap
between the path of the slow iterates \(\te^{(m),\circ}\)
and the projected-ODE trajectory \(z^{(T_m)}\) started at
\(\te_{n_m}\). We now choose block-dependent thresholds \(\delta_f\) and
\(\delta_s\) that vanish as \(m\to\infty\),
while their block-failure probabilities remain summable.
Smaller deviation levels sharpen this bound but increase the
block-failure probabilities; the logarithmic choices below make both
levels vanish while keeping those probabilities summable.

For \(m\ge 0\), set
\[
D_m:=b(n_m)-b(n_{m+1}),
\qquad
L_m^x:=\log\br{4d_xK_m(m+2)^2},
\qquad
L_m^\te:=\log\br{4d_\te K_m(m+2)^2}.
\]
Define
\[
\delta_{\rm dev}^{(\fm)}(m)
:=
\left(
\frac{\alpha_{n_m}L_m^x}{c_{\mathrm{mg}}^x}
\right)^{1/2},
\]
and
\[
\delta_{\rm dev}^{(\sm)}(m)
:=
\kappa_{d_\te}C_{\mathrm{mg}}^{(\theta)} \sqrt{2D_mL_m^\te}.
\]
Define the blockwise rate event
\[
\cG_m^{\rm rate}
:=
\cG_x^{(m)}
\bigl(\delta_{\rm dev}^{(\fm)}(m)\bigr)
\cap
\cG_{\te}^{(m)}\bigl(\delta_{\rm dev}^{(\sm)}(m)\bigr).
\]
On \(\cG_m^{\rm rate}\), both martingale perturbations stay below
their rate-calibrated levels, allowing the blockwise tracking estimate
below to be applied.
By Lemmas~\ref{lem:block_fast_tail} and
\ref{lemma:mc_diarmid_mtilde}, the chosen thresholds satisfy
\[
\bP\!\left(
 [\cG_x^{(m)}(\delta_{\rm dev}^{(\fm)}(m))]^c
\right)\le\frac{1}{2(m+2)^2},
\qquad
\bP\!\left(
 [\cG_{\te}^{(m)}(\delta_{\rm dev}^{(\sm)}(m))]^c
\right)\le\frac{1}{2(m+2)^2}.
\]
By the union bound,
\(\sum_{m\ge1}\bP((\cG_m^{\rm rate})^c)<\infty\), and the
Borel--Cantelli lemma implies that \(\cG_m^{\rm rate}\) holds eventually
almost surely.

Define
\[
d_m^{\rm rate}
:=
\delta_{\rm dev}^{(\fm)}(m)
+
\delta_{\rm dev}^{(\sm)}(m)
+
\alpha_{n_m}
+
\frac{\beta_{n_m}}{\alpha_{n_m}}.
\]
By Theorem~\ref{th:tube}, applied with
\[
\delta_f=\delta_{\rm dev}^{(\fm)}(m),
\qquad
\delta_s=\delta_{\rm dev}^{(\sm)}(m),
\]
on \(\cG_m^{\rm rate}\),
\begin{equation}
\sup_{0\le r\le H_m}
\norm{
\te^{(m),\circ}(r)-z^{(T_m)}(r)
}
\le
C_{B,T}d_m^{\rm rate}.
\label{ineq:localized-block-rate}
\end{equation}
The estimate~\eqref{ineq:localized-block-rate} controls the stochastic
perturbation accumulated within one block.
To propagate this estimate over successive blocks, we impose
a Lyapunov decrease condition over time \(T\) along the
projected-ODE flow, uniformly over initial states
\(\theta\in\Theta\).
\begin{assumption}[Uniform one-block Lyapunov decrease]
\label{assum:one_block_decrease}
Let \(J\) be the Lyapunov function in
Assumption~\ref{assum:ode_stability}.
Let \(T>0\) be the block horizon used in the definition of
\(\{n_m\}\). There exists a continuous function
\[
\mathcal D_J:[0,J_{\max}]\to\bR_+
\]
such that
\[
\mathcal D_J(0)=0,
\qquad
0<\mathcal D_J(u)\le u,
\quad 0<u\le J_{\max},
\]
and
\[
J(\mathsf S_T(\te))
\le
J(\te)-\mathcal D_J(J(\te)),
\qquad \te\in\Te.
\]
\end{assumption}

Since \(J\) is nonincreasing along the projected-ODE flow, for every \(s\ge T\),
\[
J(\mathsf S_s(\te))
\le
J(\mathsf S_T(\te))
\le
J(\te)-\mathcal D_J(J(\te)),
\qquad \te\in\Te.
\]
A decrease proportional to the current Lyapunov value
gives the following contraction condition.
\begin{assumption}[Uniform one-block Lyapunov contraction]
\label{assum:one_block_contraction}
Let \(J\) be the Lyapunov function in
Assumption~\ref{assum:ode_stability}, and let \(T>0\) be the block
horizon used in the definition of \(\{n_m\}\). There exists
\(\rho_T\in(0,1)\) such that
\[
J(\mathsf S_T(\te))\le \rho_T J(\te),
\qquad \te\in\Te.
\]
\end{assumption}
Assumption~\ref{assum:one_block_contraction} is the special case of
Assumption~\ref{assum:one_block_decrease} obtained by taking
\[
\mathcal D_J(u):=(1-\rho_T)u.
\]
Define
\[
m(n):=\max\{m\ge0:n_m\le n\},
\]
so that \(n_{m(n)}\le n<n_{m(n)+1}\).
\begin{theorem}[Fast tracking rate]
\label{th:fast_iter_rate}
Under Assumptions~\ref{assum:markov_noise},
\ref{assum:lipschitz_kernel}, \ref{assum:1_2},
\ref{assum:contraction}, \ref{assum:lipschitz},
\ref{assum:hitting_time},
and~\ref{assum:xn_bound}, the step-size exponent requirement in
Condition~\ref{cond:hyper_1}, and
Condition~\eqref{cond:N0:alpha2}, 
\[
\norm{x_n-x\ust(\te_n)}
=
O_{\rm a.s.}\!\left(
(\no+n)^{-r_x(\fraka,\frakb)}\sqrt{\log(\no+n)}
\right),
\qquad
r_x(\fraka,\frakb)
:=\min\left\{\frac{\fraka}{2},\frakb-\fraka\right\}.
\]
The exponent \(r_x(\fraka,\frakb)\) furnished by this bound is maximized over
\(1/2<\fraka<\frakb\le1\) at
\(\frakb=1\) and \(\fraka=2/3\), where \(r_x(2/3,1)=1/3\).
\end{theorem}

\begin{theorem}[Almost-sure convergence under one-block Lyapunov decrease]
\label{th:iter_rate_nonlinear}
Fix a Lyapunov function \(J\) as in
Assumption~\ref{assum:ode_stability}. Suppose that \(J\) is
\(L_J\)-Lipschitz on \(\Te\), and that
Assumption~\ref{assum:one_block_decrease} holds for the block horizon
\(T\) used to define \(\{n_m\}\). Define
\[
U_m:=J(\te_{n_m}),
\qquad
e_m^J:=L_JC_{B,T}d_m^{\rm rate}.
\]
Then, almost surely, there exists a finite random index
\(m_0=m_0(\omega)\) such that
\begin{equation}
U_{m+1}
\le
U_m-\mathcal D_J(U_m)+e_m^J,
\qquad m\ge m_0.
\label{ineq:nonlinear_block_recursion}
\end{equation}
Whenever \(m(n)\ge m_0\),
\begin{equation}
J(\te_n)
\le
U_{m(n)}+e_{m(n)}^J.
\label{ineq:nonlinear_within_block_transfer}
\end{equation}
Consequently,
\[
J(\te_n)\longrightarrow0,
\qquad
\operatorname{dist}(\te_n,\eq)\longrightarrow0
\quad\text{almost surely}.
\]
\end{theorem}

\begin{proof}
See Appendix~\ref{sec:app_proof_nonlinear}.
\end{proof}
When \(\frakb=1\), \(\beta_n=(N_0+n)^{-1}\), and both
the accumulated slow time and the number of completed blocks
\(m(n)\) grow as \(\log(N_0+n)\); for \(\frakb<1\), they
grow as \((N_0+n)^{1-\frakb}\).
Theorem~\ref{th:iter_rate_nonlinear} gives almost-sure
convergence in both cases. Explicit rates require a
quantitative lower bound on \(\mathcal D_J\), as imposed
in the next result.

\begin{theorem}[Rates under a power-law decrease condition]
\label{th:power_decrease_rate}
Assume the hypotheses of Theorem~\ref{th:iter_rate_nonlinear}. Suppose,
in addition, that there exist constants \(p_J>1\), \(\kappa_J>0\), and
\(u_J>0\) such that
\begin{equation}
\mathcal D_J(u)\ge\kappa_Ju^{p_J},
\qquad 0\le u\le u_J.
\label{ineq:power_lyapunov_decrease}
\end{equation}
If, for some \(r>0\) and \(C_r<\infty\),
\[
d_m^{\rm rate}
\le
C_r(N_0+n_m)^{-r}
\sqrt{\log(N_0+n_m)}
\]
for all sufficiently large \(m\), the following conclusions hold. If
\(\frakb<1\), then
\begin{equation}
J(\te_n)
= O_{\rm a.s.}\!\left(
(N_0+n)^{-\frac{1-\frakb}{p_J-1}}
+ (N_0+n)^{-\frac{r}{p_J}}
\bigl(\log(N_0+n)\bigr)^{\frac{1}{2p_J}}
\right).
\label{ineq:nonlinear_iteration_rate}
\end{equation}
If \(\frakb=1\), then
\begin{equation}
J(\te_n)
=
O_{\rm a.s.}\!\left(
\{\log(N_0+n)\}^{-\frac1{p_J-1}}
\right).
\label{ineq:nonlinear_iteration_rate_bone}
\end{equation}
\end{theorem}
The rate~\eqref{ineq:nonlinear_iteration_rate_bone} for
\(\frakb=1\) is sharp under the power-decrease hypothesis;
see Remark~\ref{rem:b1_power_sharpness} in the Online Companion.

Theorem~\ref{th:iter_rate_nonlinear} allows a general
lower bound on the Lyapunov decrease over one block.
In the contraction case, the block recursion can be unrolled explicitly,
giving the sharper estimate below.
\begin{theorem}[Slow-iterate Lyapunov error rates under one-block contraction]
\label{th:iter_rate}
Fix a Lyapunov function \(J\) as in
Assumption~\ref{assum:ode_stability}. Suppose that \(J\) is
\(L_J\)-Lipschitz on \(\Te\), and that
Assumption~\ref{assum:one_block_contraction} holds for the block
horizon \(T\) used to define \(\{n_m\}\).
Then, almost surely, there exists a finite random index \(m_0=m_0(\omega)\) such that, with \(d_j^{\rm rate}\) as defined above, the following inequality holds for every \(n\) satisfying \(m(n)\ge m_0\):
\[
J(\te_n)\le \rho_T^{m(n)-m_0}J_{\max}
+L_JC_{B,T}\sum_{j=m_0}^{m(n)-1}
 \rho_T^{m(n)-1-j}d_j^{\rm rate}
+L_JC_{B,T}d_{m(n)}^{\rm rate}.
\]
Consequently, suppose that, for some $r>0$ and $C_r<\infty$,
\[
d_m^{\rm rate}\le C_r(N_0+n_m)^{-r}\sqrt{\log(N_0+n_m)},\qquad m\ge m_0.
\]
If \(\frakb<1\), then
\[
J(\te_n)
=
O_{\rm a.s.}\!\left(
(N_0+n)^{-r}\sqrt{\log(N_0+n)}
\right).
\]
If \(\frakb=1\), define
\begin{equation}
\lambda_T:=-\frac{\log\rho_T}{T}>0.
\label{def:effective_block_exponent}
\end{equation}
Then
\begin{equation}
J(\te_n)
=
\begin{cases}
O_{\rm a.s.}\!\left((N_0+n)^{-r}\sqrt{\log(N_0+n)}\right),
&\lambda_T>r,\\[1mm]
O_{\rm a.s.}\!\left((N_0+n)^{-r}\{\log(N_0+n)\}^{3/2}\right),
&\lambda_T=r,\\[1mm]
O_{\rm a.s.}\!\left((N_0+n)^{-\lambda_T}\right),
&0<\lambda_T<r.
\end{cases}
\label{ineq:contraction_iteration_rate_bone}
\end{equation}
\end{theorem}

The block estimates in Lemmas~\ref{lem:block_geometry}
and~\ref{lem:harmonic_block_geometry} of the Online Companion imply that,
for the polynomial step sizes above under
Condition~\ref{cond:hyper_1},
\[
d_m^{\rm rate}
\le
C_T(\no+n_m)^{-r_x(\fraka,\frakb)}
\sqrt{\log(\no+n_m)}
\]
for all sufficiently large \(m\). Thus the rate-envelope conditions in
Theorem~\ref{th:power_decrease_rate} and in the rate part of
Theorem~\ref{th:iter_rate} hold with exponent \(r(\fraka,\frakb)\);
see the proof of
Corollary~\ref{coro:optimized_rate} in
Appendix~\ref{sec:ec_section7_proofs} of the Online Companion.
\begin{corollary}[Optimized two-time-scale rate under contraction]\label{coro:optimized_rate}
Assume the hypotheses of Theorem~\ref{th:iter_rate}. For
\(\frakb<1\),
\[
J(\te_n)
+\norm{x_n-x\ust(\te_n)}
=
O_{\rm a.s.}\!\left(
(N_0+n)^{-r_x(\fraka,\frakb)}
\sqrt{\log(N_0+n)}
\right).
\]
Moreover,
\[
\sup_{\frac12<\fraka<\frakb<1}
r_x(\fraka,\frakb)=\frac13.
\]
More precisely, for every
\[
0<\rateslack<\frac1{12},
\]
choose
\[
\frakb=1-3\rateslack,\qquad \fraka=\frac{2\frakb}{3}=\frac23-2\rateslack.
\]
Then the step-size conditions hold and
\[
J(\te_n)
+\norm{x_n-x\ust(\te_n)}
=
O_{\rm a.s.}\!\left(
(N_0+n)^{-1/3+\rateslack}\sqrt{\log(N_0+n)}
\right).
\]

At the harmonic endpoint, take $\frakb=1$ and $\fraka=2/3$.
The slow rate is~\eqref{ineq:contraction_iteration_rate_bone} with
$r=1/3$, and Theorem~\ref{th:fast_iter_rate} gives fast-tracking
rate $O_{\rm a.s.}((N_0+n)^{-1/3}\sqrt{\log(N_0+n)})$.
If \(J(\te)=\norm{\te-\te\ust}\) for a fixed equilibrium \(\te\ust\)
of the projected ODE~\eqref{eq:ode}, each rate stated for \(J(\te_n)\)
is the corresponding slow-iterate distance rate.
\end{corollary}

\begin{corollary}[Optimized polynomial exponent under power Lyapunov decrease]
\label{coro:optimized_power_rate}
Assume the hypotheses of Theorem~\ref{th:iter_rate_nonlinear}, together
with~\eqref{ineq:power_lyapunov_decrease} for some \(p_J>1\). For the
case \(\frakb<1\), choose
\[
\fraka=\frac{2p_J}{4p_J-1},
\qquad
\frakb=\frac{3p_J}{4p_J-1}.
\]
Then \(1/2<\fraka<\frakb<1\), and
\[
J(\te_n)
=
O_{\rm a.s.}\!\left(
(N_0+n)^{-\frac{1}{4p_J-1}}
\bigl(\log(N_0+n)\bigr)^{\frac{1}{2p_J}}
\right),
\]
and
\[
\norm{x_n-x\ust(\te_n)}
=
O_{\rm a.s.}\!\left(
(N_0+n)^{-\frac{p_J}{4p_J-1}}
\sqrt{\log(N_0+n)}
\right).
\]
If \(J(\te)=\norm{\te-\te\ust}\) for a fixed equilibrium \(\te\ust\)
of the projected ODE~\eqref{eq:ode}, the first display is also the
corresponding slow-iterate distance bound.

At the harmonic endpoint, take $\frakb=1$ and $\fraka=2/3$.
The slow rate is~\eqref{ineq:nonlinear_iteration_rate_bone}, and
Theorem~\ref{th:fast_iter_rate} gives fast-tracking rate
$O_{\rm a.s.}((N_0+n)^{-1/3}\sqrt{\log(N_0+n)})$.
\end{corollary}

\subsection{Faster rates at an attracting equilibrium face}
\label{sec:face_rates}
The general fast-tracking bound accounts for movement of
\(x^\star(\theta_n)\) as the slow iterate evolves. This contribution
can be reduced when the slow equilibrium set is a strictly attracting
face of a box and \(x^\star\) has the same value at every point of that
face. The resulting rates are given next.
\begin{proposition}[Rates near an equilibrium face]
\label{prop:face_rates}
For~\eqref{def:x_update}--\eqref{def:te_update}, retain the standing
assumptions of Section~\ref{sec:pf} on the chain, update maps, noise,
and fast-iterate stability, and the
assumptions on the Lipschitz Lyapunov function \(J\) in
Theorem~\ref{th:iter_rate_nonlinear}. Write \(t_n=N_0+n\).
In place of Condition~\ref{cond:hyper_1} and~\eqref{cond:N0:alpha2}, allow either
\begin{equation}
\begin{aligned}
 &\alpha_n=t_n^{-\fraka},\quad \beta_n=t_n^{-\frakb},
 &&\tfrac12<\fraka<\frakb\le1,\quad
       t_1^{1-\fraka}\ge4;\\
 &\alpha_n=A\log(t_n)/t_n,\quad \beta_n=B/t_n,
 &&A,B>0,\quad N_0\ge3,\quad
       \alpha_1\le1,\quad A\log t_1\ge\max\{4,B\}.
\end{aligned}
\label{eq:face_steps}
\end{equation}
Suppose that \(\Theta=\prod_{i=1}^{d_\theta}[\ell_i,u_i]\), with
\(\ell_i<u_i\), and that, for some
\(I\subseteq\{1,\ldots,d_\theta\}\), its projected-ODE equilibrium set is
\[
 \eq=\{\theta\in\Theta:\theta_i=v_i\ (i\in I)\},
 \qquad v_i\in\{\ell_i,u_i\}.
\]
For \(i\in I\), set \(s_i=-1\) at a lower endpoint and \(s_i=1\)
at an upper endpoint. Assume
\begin{equation}
 \inf_{\theta\in\eq}s_i\favg_i(\theta;\theta)>0,
 \qquad i\in I.
 \label{eq:face_strict_outward}
\end{equation}
Then
\begin{equation}
 \operatorname{dist}(\theta_n,\eq)+J(\theta_n)
 =O_{\rm a.s.}(\beta_n\log t_n).
 \label{eq:face_slow_rate}
\end{equation}
If, additionally, \(x^\star(\theta)=x^\dagger\) for every
\(\theta\in\eq\), for some fixed \(x^\dagger\), then
\begin{equation}
 \norm{x_n-x^\dagger}+\norm{x_n-x^\star(\theta_n)}
 =O_{\rm a.s.}\!\left(\sqrt{\alpha_n\log t_n}\right).
 \label{eq:face_fast_rate}
\end{equation}
Thus \(\frakb=1\), \(\fraka=1-2\varepsilon\), with
\(0<\varepsilon<1/4\), gives a joint rate
\(O_{\rm a.s.}(n^{-1/2+\varepsilon}\sqrt{\log n})\).
The second schedule in~\eqref{eq:face_steps} gives
\(O_{\rm a.s.}(n^{-1/2}\log n)\).
\end{proposition}
\begin{proof}
See Appendix~\ref{sec:ec_face_rates}.
\end{proof}
\subsection{Improved rates under contraction of the reduced slow map}
\label{sec:boundary_contraction_rate}
A different mechanism yields the \(n^{-1/2}\) polynomial exponent:
assume that the reduced slow update map is a Euclidean contraction.
Together with logarithmic separation of the two step sizes, this
controls the slow error directly, even when its equilibrium lies on
the boundary. Unlike Proposition~\ref{prop:face_rates}, the next
result permits an arbitrary compact convex polyhedron and does not
require an attracting box face or a common fast target along a face.
Its contraction assumption is stronger than the one-block
Lyapunov condition.
\begin{proposition}[Contraction rate with boundary equilibria]
\label{prop:boundary_contraction_rate}
For~\eqref{def:x_update}--\eqref{def:te_update}, retain the standing
assumptions of Section~\ref{sec:pf} on the chain, update maps, noise,
fast contraction, and fast-iterate stability. Suppose that
\begin{equation}
 \overline G(\theta):=\theta+\favg(\theta;\theta),\qquad
 \norm{\overline G(\theta)-\overline G(\vartheta)}
 \le\rho_s\norm{\theta-\vartheta},\quad
 \theta,\vartheta\in\Theta,
 \label{eq:bc_reduced_contraction}
\end{equation}
for some \(\rho_s\in[0,1)\), in the Euclidean slow-coordinate norm.
Let \(\theta^\star\) be the unique fixed point of
\(\operatorname{proj}_{\Theta}\circ\overline G\).
In place of Condition~\ref{cond:hyper_1} and~\eqref{cond:N0:alpha2}, take
\begin{equation}
 t_n=N_0+n,\qquad
 \alpha_n=\frac{A\log t_n}{t_n},\qquad
 \beta_n=\frac{B}{t_n},\qquad
 A>0,\quad B\ge\frac{2}{1-\rho_s},
 \label{eq:bc_steps}
\end{equation}
where \(N_0\ge3\), \(\alpha_1<1\), and
\(A\log t_1\ge\max\{4,B\}\).
Assumption~\ref{assum:xn_bound} is required for this selected schedule.
Then
\begin{equation}
 \norm{x_n-x^\star(\theta_n)}+\norm{\theta_n-\theta^\star}
 =O_{\rm a.s.}\!\left(\frac{\log t_n}{\sqrt{t_n}}\right).
 \label{eq:bc_rate}
\end{equation}
The same bound holds for \(\norm{x_n-x^\star(\theta^\star)}\)
and for \(J(\theta_n)\) whenever \(J\) is Lipschitz and
\(J(\theta^\star)=0\).
\end{proposition}
\begin{proof}
See Appendix~\ref{sec:ec_boundary_contraction}.
\end{proof}

The equilibrium drift may be nonzero and outward; neither interiority
nor strict complementarity is assumed. Condition~\eqref{eq:bc_reduced_contraction}
implies contraction of distance along the projected ODE and is stronger
than the one-block Lyapunov contraction used in Theorem~\ref{th:iter_rate}.

In the terminology of~\citet{chandak2026Ok}, ``true'' two-time-scale
stochastic approximation requires \(\beta_n/\alpha_n\to0\), not
specifically polynomial separation. The schedule~\eqref{eq:bc_steps}
satisfies this condition since \(\beta_n/\alpha_n=B/(A\log t_n)\).
Under the additional contraction
condition~\eqref{eq:bc_reduced_contraction}, the bound~\eqref{eq:bc_rate}
has polynomial exponent \(1/2\) for both iterate errors, with only
logarithmic separation of the step sizes. This improvement uses the
stronger contraction assumption together with the revised schedule and
proof; it does not follow from changing the schedule alone. The proof
adapts the auxiliary noise-averaging idea of that work to the projected
recursion and controlled Markov noise, using a feasible auxiliary sequence
to handle boundary equilibria. 

One setting in which the reduced-map condition can be checked
directly is projected gradient descent for a smooth strongly convex
objective. If
\(\overline G(\theta)=\theta-\eta\nabla F(\theta)\) and
\(mI\preceq\nabla^2F(\theta)\preceq LI\) on \(\Theta\), then, for
\(0<\eta<2/L\), the mean-value theorem gives
\[
\norm{\overline G(\theta)-\overline G(\vartheta)}
\le
\max\{|1-\eta m|,\ |1-\eta L|\}
\norm{\theta-\vartheta}.
\]
The displayed factor is less than one. Projection onto a compact
convex polyhedron preserves this contraction bound, including when
the constrained minimizer lies on a boundary face.
\section{Actor--Critic Application}
\label{sec:applications}
Actor--critic provides an example in which the slow iterate controls
both the sampling law and the projected update. We first specify the
single-trajectory algorithm and its constrained policy class, then
state the rate obtained from the attracting-face result. The
verification of the assumptions is in the Online Companion;
standalone projected TD(0) and projected-SGD applications appear in
Appendices~\ref{sec:ec_projected_td} and
\ref{sec:projected_stochastic_gradient} there.

Consider a finite-state, finite-action Markov decision process (MDP);
see, for example, \citet{puterman1994mdp}. Let \(\cS\) and \(\cA\)
be its state and action spaces, \(P(s'\mid s,a)\) its transition
probabilities, and \(c(s,a)\) its expected one-step cost. At time \(n\),
the current actor parameter determines a policy from which \(A_n\)
is drawn at state \(S_n\). The environment then produces
\(S_{n+1}\) and a realized one-step cost \(C_{n+1}\).
Let \(\cF_n^{\rm MDP}\) be the history immediately before \(A_n\)
is drawn. For every \(s'\in\cS\),
\[
\bP(S_{n+1}=s'\mid\cF_n^{\rm MDP},A_n)
=P(s'\mid S_n,A_n),
\qquad
\bE[C_{n+1}\mid\cF_n^{\rm MDP},A_n]
=c(S_n,A_n).
\]
The Online Companion, Appendix~\ref{sec:ec_section8_proofs},
expresses the critic and actor updates in the drift-plus-martingale
form of~\eqref{def:x_update}--\eqref{def:te_update} and checks the
assumptions used in the rate results.

Fix \(0<\gamma_{\rm disc}<1\). For a stationary policy \(\pi\),
its discounted total-cost value is
\[
V_\pi(s)
:=
\bE_s^\pi\!\left[
\sum_{k=0}^{\infty}\gamma_{\rm disc}^k C_{k+1}
\right],
\qquad s\in\cS.
\]
We analyze an online actor--critic recursion using one observed
state--action transition per iteration, motivated by
\citet[Algorithm~3]{konda1999actor}. In the original asynchronous
form, an update associated with a state--action pair uses a step
size indexed by that pair's visit count. Here the critic coordinate
for the observed state and the actor coordinate for the observed
state--action pair are updated using the global step sizes
\(\alpha_n\) and \(\beta_n\), indexed by the number \(n\) of
environment transitions. The rate below applies to this
global-clock recursion.

Assume \(|C_{n+1}|\le C_{\max}\). The actor parameter
\(\vartheta=(\vartheta(s,a))_{s,a}\) consists of state--action logits, also
called preference parameters.  Fix \(B_{\rm ac}>0\) and set
\[
 \Theta_{\rm ac}
 :=[-B_{\rm ac},B_{\rm ac}]^{|\cS||\cA|},
 \qquad
 \pi_\vartheta(a\mid s)
 :=\frac{e^{\vartheta(s,a)}}
          {\sum_{b\in\cA}e^{\vartheta(s,b)}},
\]
and draw \(A_n\sim\pi_{\vartheta_n}(\cdot\mid S_n)\) before the
time-\(n\) update. For \(s\in\cS\) and \(a\in\cA\), let
\(e_s\in\bR^{|\cS|}\) and \(e_{s,a}\in\bR^{|\cS||\cA|}\) denote the
corresponding standard coordinate vectors. With
\[
 \delta^{\rm ac}_{n+1}
 :=C_{n+1}+\gamma_{\rm disc}V_n(S_{n+1})-V_n(S_n),
\]
the recursion is
\begin{equation}
\label{eq:actor_critic_app}
\begin{aligned}
 V_{n+1}
 &=V_n+\alpha_ne_{S_n}\delta^{\rm ac}_{n+1},\\
 \vartheta_{n+1}
 &=\operatorname{proj}_{\Theta_{\rm ac}}
   \left(\vartheta_n-
   \beta_ne_{S_n,A_n}\delta^{\rm ac}_{n+1}\right).
\end{aligned}
\end{equation}

For \(\vartheta\in\Theta_{\rm ac}\), let \(d_\vartheta\) and
\(V_\vartheta\) be the stationary state distribution and discounted cost
value under \(\pi_\vartheta\). Assume that the controlled state process
induced by the policy family
\(\{\pi_\vartheta:\vartheta\in\Theta_{\rm ac}\}\) satisfies
Assumptions~\ref{assum:markov_noise}, \ref{assum:lipschitz_kernel},
and~\ref{assum:hitting_time} uniformly over
\(\vartheta\in\Theta_{\rm ac}\). Assume also that the directed graph on
\(\cS\), with an edge \(s\to s'\) whenever
\(P(s'\mid s,a)>0\) for some \(a\in\cA\), is strongly connected. As shown
in Appendix~\ref{sec:ec_section8_proofs} of the Online Companion, the full
support of the bounded-logit softmax policies and compactness of
\(\Theta_{\rm ac}\) then imply
\begin{equation}
\label{eq:actor_critic_coverage}
 d_{\min}^{\rm ac}
 :=\inf_{\vartheta\in\Theta_{\rm ac}}
   \min_{s\in\cS}d_\vartheta(s)>0.
\end{equation}
We refer to~\eqref{eq:actor_critic_coverage} as uniform stationary state
coverage. Set
\[
 \pi_{\min}^{\rm ac}:=\frac{e^{-2B_{\rm ac}}}{|\cA|},
 \qquad
 q_{\min}^{\rm ac}:=d_{\min}^{\rm ac}\pi_{\min}^{\rm ac}.
\]
Initialize \(\vartheta_1\in\Theta_{\rm ac}\) and choose
\(\norm{V_1}_\infty\le C_{\max}/(1-\gamma_{\rm disc})\).
The critic recursion preserves this cube, which also contains every
discounted value vector \(V_\vartheta\). On that bounded set, the
finite MDP and the softmax policy on the compact logit box give the
localized Lipschitz bounds required for the sampled updates.
In the fast-coordinate norm \(\norm{\cdot}_\infty\), the averaged
critic map is contractive on \(\bR^{|\cS|}\); cube invariance supplies
fast-iterate stability. The Online Companion verifies these claims.

Let \(V_{B_{\rm ac}}^\star\) be the optimal value among policies
representable, state by state, by logits in
\([-B_{\rm ac},B_{\rm ac}]^{|\cA|}\). We measure the actor's
aggregate value gap by
\[
J_{B_{\rm ac}}(\vartheta)
:=
\sum_{s\in\cS}
\{V_\vartheta(s)-V_{B_{\rm ac}}^\star(s)\}.
\]
For a fixed actor parameter \(\vartheta\), the stationary frequency
of the state--action pair \((s,a)\) is
\(q_\vartheta(s,a):=d_\vartheta(s)\pi_\vartheta(a\mid s)\).
These frequencies weight the coordinates of the averaged actor
drift. Appendix~\ref{sec:ec_section8_proofs} of the Online Companion
uses their uniform positive lower bound to establish Lyapunov
decrease, identify the restricted-optimal equilibrium face, and show
that the projected actor ODE reaches that face in uniformly bounded
time. On the face, the critic equilibrium is
\(V_{B_{\rm ac}}^\star\), and each coordinate fixed at an endpoint
has strictly outward drift. Hence
Proposition~\ref{prop:face_rates} applies.
\begin{corollary}[Actor--critic rate]
\label{coro:actor_critic_rate}
Suppose that the preceding MDP, policy, and initialization conditions hold,
and choose either step-size schedule in Proposition~\ref{prop:face_rates}.
Then
\begin{equation}
\label{eq:actor_critic_actor_rate}
 J_{B_{\rm ac}}(\vartheta_n)
 =O_{\rm a.s.}\!\left(\beta_n\log(N_0+n)\right),
\end{equation}
and
\begin{equation}
\label{eq:actor_critic_critic_rate}
 \norm{V_n-V_{\vartheta_n}}_\infty
 +\norm{V_n-V_{B_{\rm ac}}^\star}_\infty
 =O_{\rm a.s.}\!\left(\sqrt{\alpha_n\log(N_0+n)}\right).
\end{equation}
For \(\alpha_n=A\log(N_0+n)/(N_0+n)\) and
\(\beta_n=B/(N_0+n)\), with the shift specified in that proposition,
the actor value gap is \(O_{\rm a.s.}(n^{-1}\log n)\), and
\begin{equation}
\label{eq:actor_critic_rate}
 J_{B_{\rm ac}}(\vartheta_n)
 +\norm{V_n-V_{\vartheta_n}}_\infty
 =O_{\rm a.s.}\!\left(
 \frac{\log(N_0+n)}{\sqrt{N_0+n}}\right).
\end{equation}
This controls the aggregate value Lyapunov gap and the critic tracking error. 
If \(V^\star\) is the unrestricted optimal value, then
\[
 0\le\sum_{s\in\cS}
 \{V_{B_{\rm ac}}^\star(s)-V^\star(s)\}
 \le
 \frac{2|\cS|(|\cA|-1)C_{\max}e^{-2B_{\rm ac}}}
      {(1-\gamma_{\rm disc})^2}.
\]
Consequently, for any prescribed target tolerance
\(\varepsilon_{\rm pol}>0\), a fixed \(B_{\rm ac}\) can be chosen
sufficiently large that the displayed restriction bias is at most
\(\varepsilon_{\rm pol}\).  The unrestricted aggregate value gap is then at most
\(\varepsilon_{\rm pol}\) plus the vanishing term in
\eqref{eq:actor_critic_actor_rate}.
\end{corollary}

\section{Conclusion}
\label{sec:conclusion}
We obtained finite-time concentration bounds for projected
two-time-scale stochastic approximation driven by a controlled Markov chain.
On each slow-time block, the bounds compare the fast iterate with
its moving equilibrium and the projected slow path with its ODE
trajectory. The Skorokhod-map comparison handles boundary
discontinuities, while the estimates identify separately the
martingale, Markovian, and time-scale contributions.

A one-block Lyapunov decrease gives almost-sure convergence to
the slow equilibrium set; power decrease and contraction give
explicit last-iterate rates. For polynomial step sizes, the
optimized joint exponent under one-block contraction approaches
\(1/3\). With a strictly attracting equilibrium face and a common
fast target, logarithmically separated step sizes give the joint
\(O_{\rm a.s.}(n^{-1/2}\log n)\) rate. Euclidean contraction of the
reduced slow map gives the same joint order on a general compact
convex polyhedron, including boundary equilibria. In actor--critic,
the attracting-face result further gives an
\(O_{\rm a.s.}(n^{-1}\log n)\) constrained-policy value gap.
The Online Companion develops the actor--critic proof, projected
TD(0), projected SGD, and fast-coordinate projection.

\appendix

\section{Selected Proofs of the Main Results}
\label{sec:article_selected_proofs}
We give the path comparisons underlying the main concentration
bounds and the argument that carries Lyapunov control from block
endpoints to every iteration. These comprise the two comparisons
for the fast recursion, their fast-tracking consequence, the
original--auxiliary and auxiliary--ODE comparisons for the slow
recursion, and the localized projected-ODE comparison. The step-path
Skorokhod representation, supporting Poisson-equation and martingale
estimates, and the remaining proofs are collected in the Online
Companion.
\subsection{Proof of Proposition~\ref{prop:conc_v_v_te}}
\label{sec:app_proof_fast_tracking}

\begin{proof}
The triangle inequality and
Lemmas~\ref{lem:block_fast_aux_comparison} and
\ref{lem:block_moving_equilibrium} give, on
\(\cG_x^{(m)}(\delta)\) and for \(0\le k\le K_m\),
\[
\varepsilon_{n_m+k}^{(x)}
\le
\mathcal W_m(k)\varepsilon_{n_m}^{(x)}
+
C_{\rm aux}\bigl(\delta+\alpha_{n_m}\bigr)
+
\frac{C_{\rm mov}}{q}
\frac{\beta_{n_m}}{\alpha_{n_m}}.
\]
This proves~\eqref{eq:block_fast_pointwise}, for instance with
\[
C_x:=\max\left\{C_{\rm aux},\frac{C_{\rm mov}}{q}\right\}.
\]
Assumption~\ref{assum:xn_bound} and the definition of \(\frakc_8\) give
\(\varepsilon_{n_m}^{(x)}\le\frakc_6+\frakc_8\). Moreover, monotonicity
of \(\beta_n/\alpha_n\) and
\(\mathcal W_m(k+1)=\mathcal W_m(k)(1-q\alpha_{n_m+k})\) yield
\[
\sum_{k=0}^{K_m-1}\beta_{n_m+k}\mathcal W_m(k)
\le \frac{\beta_{n_m}}{q\alpha_{n_m}}
\bigl[1-\mathcal W_m(K_m)\bigr]
\le \frac{\beta_{n_m}}{q\alpha_{n_m}}.
\]
Since \(\sum_{r=n_m}^{n_{m+1}-1}\beta_r\le T+1\), summing the
pointwise bound proves~\eqref{eq:block_fast_weighted} with
\(C_{x,T}=(\frakc_6+\frakc_8)/q+(T+1)C_x\).
Equation~\eqref{eq:block_fast_endpoint} is
\eqref{eq:block_fast_pointwise} at \(k=K_m\).
\end{proof}

\subsection{Proof of Lemma~\ref{lem:block_fast_aux_comparison}}
\label{sec:app_proof_fast_aux_comparison}
We first control the discrepancy between the original block-reindexed
fast iterate \(x_k^{(m)}\) and the block-restarted averaged recursion
\(\widetilde x_k^{(m)}\). Both recursions use the same slow-parameter
path and are unprojected, so their comparison is controlled directly by
contraction in the fast state; no parameter-path or Skorokhod-map
comparison is required. Subtracting their updates and
using the Poisson decomposition separates this discrepancy into a
martingale transform, a Poisson residual, and a contractive
propagation term. The first term is controlled on
\(\cG_x^{(m)}(\delta)\), while the second is bounded deterministically.
\begin{proof}
Fix \(m\ge0\), write \(s:=n_m\), and recall \(q =1-\gamma\).
Let \(\mathcal R_{m,k}^x\) and \(C_{\rm res}^x\) be as in
\eqref{eq:block_fast_poisson_residual} and
Proposition~\ref{prop:block_fast_pois_residual}, respectively, and set
\[
c:=\delta+C_{\rm res}^x\alpha_s.
\]
Subtracting the averaged recursion \eqref{def:x_tilde} from the original update \eqref{def:x_update}, applying
the Poisson decomposition \eqref{eq:fast_poisson_pointwise}, and iterating from $D_0^{(m)}=0$ gives
\[
\begin{aligned}
D_k^{(m)}
={}&\mathcal M_{m,k}^x+\mathcal R_{m,k}^x\\
&+\sum_{j=0}^{k-1}\alpha_{s+j}\chi(s+j+1,s+k-1)
\bigl[
h^{(\mathrm{av.})}(x_j^{(m)};\te_j^{(m)})
-h^{(\mathrm{av.})}(\widetilde x_j^{(m)};\te_j^{(m)})
\bigr].
\end{aligned}
\]
Define
\[
H_k^{(m)}
:=\sum_{j=0}^{k-1}\alpha_{s+j}\chi(s+j+1,s+k-1)
\norm{D_j^{(m)}},
\qquad H_0^{(m)}:=0.
\]
The polynomial step sizes satisfy $0<\alpha_n<1$, so the weights
$\alpha_{s+j}\chi(s+j+1,s+k-1)$ are nonnegative. Hence the triangle
inequality and Assumption~\ref{assum:contraction}, which makes the averaged
fast map $\gamma$-contractive in its state argument, give
\[
\begin{aligned}
&\left\|
\sum_{j=0}^{k-1}\alpha_{s+j}\chi(s+j+1,s+k-1)
\bigl[
h^{(\mathrm{av.})}(x_j^{(m)};\te_j^{(m)})
-h^{(\mathrm{av.})}(\widetilde x_j^{(m)};\te_j^{(m)})
\bigr]
\right\|\\
&\qquad\le
\gamma\sum_{j=0}^{k-1}\alpha_{s+j}\chi(s+j+1,s+k-1)
\norm{D_j^{(m)}}
=\gamma H_k^{(m)}.
\end{aligned}
\]
On $\mathcal G_x^{(m)}(\delta)$, the definition of $\kappa_{d_x}$
and~\eqref{eq:block_fast_good_event} give
$\norm{\mathcal M_{m,k}^x}\le
\kappa_{d_x}\norm{\mathcal M_{m,k}^x}_\infty<\delta$, while
Proposition~\ref{prop:block_fast_pois_residual} gives
$\norm{\mathcal R_{m,k}^x}\le C_{\rm res}^x\alpha_{n_m}
=C_{\rm res}^x\alpha_s$. Combining these estimates and weakening the
strict martingale bound to a non-strict one gives
\begin{equation}
\norm{D_k^{(m)}}\le c+\gamma H_k^{(m)}.
\label{eq:block_D_pre_recursion}
\end{equation}
Moreover,
\[
\begin{aligned}
H_{k+1}^{(m)}
&=(1-\alpha_{s+k})H_k^{(m)}+\alpha_{s+k}\norm{D_k^{(m)}}\\
&\le(1-q\alpha_{s+k})H_k^{(m)}+\alpha_{s+k}c.
\end{aligned}
\]
Iteration and telescoping give
\[
\begin{aligned}
H_k^{(m)}
&\le c\sum_{j=0}^{k-1}\alpha_{s+j}
\prod_{\ell=j+1}^{k-1}(1-q\alpha_{s+\ell})\\
&=\frac{c}{q}
\left[1-\prod_{\ell=0}^{k-1}(1-q\alpha_{s+\ell})\right]
\le\frac{c}{q}.
\end{aligned}
\]
Substitution into~\eqref{eq:block_D_pre_recursion}, together with
\(1+\gamma/q=1/q\), yields
\[
\norm{D_k^{(m)}}
\le
\frac{c}{q}
=
\frac{\delta+C_{\rm res}^x\alpha_{n_m}}{q}
\le
C_{\rm aux}\bigl(\delta+\alpha_{n_m}\bigr),
\]
where
\[
C_{\rm aux}
:=
\frac{1}{q}\max\{1,C_{\rm res}^x\}.
\]
This proves~\eqref{eq:block_fast_aux_comparison}.
\end{proof}

\subsection{Proof of Lemma~\ref{lem:block_moving_equilibrium}}
\label{sec:app_proof_moving_equilibrium}
\begin{proof}
Set
\(C_{\rm mov}:=\lfp
(\frakc_1+\frakc_2\frakc_6+\mathsf B_6)\), and define
\[
d_k^{(m)}
:=
\norm{\widetilde x_k^{(m)}-x\ust(\te_k^{(m)})}.
\]
The fixed-point identity
\(\havg(x\ust(\te_k^{(m)});\te_k^{(m)})=x\ust(\te_k^{(m)})\),
Assumption~\ref{assum:contraction}, and~\eqref{def:x_tilde} give
\[
d_{k+1}^{(m)}
\le
(1-q\alpha_{n_m+k})d_k^{(m)}
+
\norm{x\ust(\te_{k+1}^{(m)})-x\ust(\te_k^{(m)})}.
\]
By Lemma~\ref{lemma:lipschitz_stat_dist}, nonexpansiveness of the
projection in~\eqref{def:te_update}, and
Assumptions~\ref{assum:1_2} and~\ref{assum:xn_bound},
\[
\norm{x\ust(\te_{k+1}^{(m)})-x\ust(\te_k^{(m)})}
\le
C_{\rm mov}\beta_{n_m+k}.
\]
Iteration of the scalar recursion yields
\[
\begin{aligned}
d_k^{(m)}
&\le
\mathcal W_m(k)d_0^{(m)}\\
&\quad+
C_{\rm mov}
\sum_{j=0}^{k-1}
\beta_{n_m+j}
\prod_{\ell=j+1}^{k-1}
(1-q\alpha_{n_m+\ell}).
\end{aligned}
\]
Since \(\beta_n/\alpha_n\) is nonincreasing,
\[
\begin{aligned}
&\sum_{j=0}^{k-1}
\beta_{n_m+j}
\prod_{\ell=j+1}^{k-1}
(1-q\alpha_{n_m+\ell})\\
&\quad\le
\frac{\beta_{n_m}}{\alpha_{n_m}}
\sum_{j=0}^{k-1}
\alpha_{n_m+j}
\prod_{\ell=j+1}^{k-1}
(1-q\alpha_{n_m+\ell})\\
&\quad=
\frac{\beta_{n_m}}{q\alpha_{n_m}}
\left[1-\mathcal W_m(k)\right]\\
&\quad\le
\frac{\beta_{n_m}}{q\alpha_{n_m}}.
\end{aligned}
\]
Indeed,
\[
\begin{aligned}
&q\sum_{j=0}^{k-1}
\alpha_{n_m+j}
\prod_{\ell=j+1}^{k-1}
(1-q\alpha_{n_m+\ell})\\
&\quad=
\sum_{j=0}^{k-1}
\left[
\prod_{\ell=j+1}^{k-1}(1-q\alpha_{n_m+\ell})
-
\prod_{\ell=j}^{k-1}(1-q\alpha_{n_m+\ell})
\right]\\
&\quad=
1-\prod_{\ell=0}^{k-1}
(1-q\alpha_{n_m+\ell})
=
1-\mathcal W_m(k).
\end{aligned}
\]
In the first step, we have used
\[
\beta_{n_m+j}
\le
\frac{\beta_{n_m}}{\alpha_{n_m}}\alpha_{n_m+j}.
\]
Finally,
\[
d_0^{(m)}
=
\norm{x_{n_m}-x\ust(\te_{n_m})},
\]
which proves~\eqref{eq:block_moving_equilibrium}.
\end{proof}
\subsection{Proof of Proposition~\ref{prop:te_te1}}
\label{sec:app_proof_te_te1}
\begin{proof}
For \(0\le n\le K_m\), let
\[
A_n:=\sup_{0\le k\le n}\norm{v\um_k-\tv\um_k},
\qquad L_0:=\lf(\lfp+1)\LS.
\]
Lemma~\ref{lemma:mu_p_g} and the Skorokhod comparison in
Lemma~\ref{lemma:skor_1} give
\[
\norm{\favg(\te\um_\ell;\te\um_\ell)
-\favg(\tte\um_\ell;\te\um_\ell)}\le L_0A_\ell.
\]
Moreover, since
\(e\um_\ell=x\um_\ell-x\ust(\te\um_\ell)\),
Assumption~\ref{assum:lipschitz} gives
\[
\norm{
f(Y\um_\ell,x\um_\ell,\te\um_\ell)
-f(Y\um_\ell,x\ust(\te\um_\ell),\te\um_\ell)
}
\le \lf\norm{e\um_\ell}.
\]
Iterating~\eqref{decomp:v_vt}, maximizing over terminal indices, and
using \(\cG_{\te}^{(m)}(\delta_s)\) and
Lemma~\ref{lemma:adhoc_1} therefore gives
\[
A_n\le L_0\sum_{\ell=0}^{n-1}\beta\um_\ell A_\ell+B_n,
\qquad
B_n:=\lf\sum_{\ell=0}^{n-1}\beta\um_\ell\norm{e\um_\ell}
+\delta_s+R_m.
\]
Because \(B_n\) is nondecreasing, discrete Gronwall yields
\[
A_n\le B_n\exp\{L_0a(n_m,n_m+n)\}.
\]
A final application of Lemma~\ref{lemma:skor_1} gives
\(\sup_{0\le k\le n}\norm{\te\um_k-\tte\um_k}\le\LS A_n\),
which is~\eqref{ineq:te_tte_block}.
\end{proof}

\subsection{Proof of Proposition~\ref{prop:te4_z}}
\label{sec:app_proof_te4_z}
\begin{proof}
Write \(F(\vartheta):=\favg(\vartheta;\vartheta)\) and
\(C_\star:=\mathsf B_8\). On
\([t^{(m)}(\ell),t^{(m)}(\ell+1))\), let
\(G_m(s):=\favg(\tte\um_\ell;\te\um_\ell)\), and set
\(G_m(H_m):=\mathbf0\). The ODE driver satisfies
\[
w^{(T_m)}(t)=\te_{n_m}
+\int_0^tF(z^{(T_m)}(s))\,ds.
\]
The auxiliary step-path driver accumulates only completed grid
increments. Consequently, its difference from
\(\te_{n_m}+\int_0^tG_m(s)\,ds\) has norm at most
\(C_\star\beta_{n_m}\): the difference is the contribution of the
single grid interval containing \(t\). Therefore
\[
\norm{w^{(T_m)}(t)-\tv^{(m),\circ}(t)}
\le
\int_0^t\norm{F(z^{(T_m)}(s))-G_m(s)}\,ds
+C_\star\beta_{n_m}.
\]
For almost every \(s\in[0,H_m]\), Lemma~\ref{lemma:mu_p_g}, the identities
\(z^{(T_m)}=\Gamma(w^{(T_m)})\) and
\(\tte^{(m),\circ}=\Gamma(\tv^{(m),\circ})\) from
Lemmas~\ref{lemma:skorohod_4} and~\ref{lemma:skor_1}, and
Corollary~\ref{coro:te-tte} give
\[
\norm{F(z^{(T_m)}(s))-G_m(s)}
\le \lfavg\LS\cA_m(s)
+C_\star\lmu\,\etem(\delta_f,\delta_s).
\]
Taking the supremum up to \(t\) therefore yields
\[
\cA_m(t)\le \lfavg\LS\int_0^t\cA_m(s)\,ds
+C_\star\lmu\,t\,\etem(\delta_f,\delta_s)
+C_\star\beta_{n_m}.
\]
Gronwall proves~\eqref{ineq:w_tv}. Applying the same finite-horizon
Skorokhod estimate to these drivers proves the final assertion.
\end{proof}

\subsection{Proof of Theorem~\ref{th:tube}}
\label{sec:app_proof_tube}
\begin{proof}
By the triangle inequality,
\[
\begin{aligned}
&\sup_{0\le t\le H_m}
\norm{
\te^{(m),\circ}(t)-z^{(T_m)}(t)
}\\
&\quad\le
\sup_{0\le t\le H_m}
\norm{
\te^{(m),\circ}(t)-\tte^{(m),\circ}(t)
}
+
\sup_{0\le t\le H_m}
\norm{
\tte^{(m),\circ}(t)-z^{(T_m)}(t)
}.
\end{aligned}
\]
Corollary~\ref{coro:te-tte} bounds the first term by
\(\etem(\delta_f,\delta_s)\). Proposition~\ref{prop:te4_z}, applied with
\(t=H_m=a(n_m,n_{m+1})\), bounds the second term by
\[
\LS
e^{\lfavg\LS a(n_m,n_{m+1})}
\left[
\mathsf B_8\lmu
a(n_m,n_{m+1})\etem(\delta_f,\delta_s)
+
\mathsf B_8\beta_{n_m}
\right].
\]
This proves~\eqref{eq:block_ode_tracking_generic} with
\(B_m(\delta_f,\delta_s)\) defined in~\eqref{def:B_bound}.

For the rate bounds, let \(C_{B,T}\) increase from line to line. Since
\(a(n_m,n_{m+1})\le T+1\), the definitions of \(B_m\) and \(\etem\)
imply
\[
B_m(\delta_f,\delta_s)
\le C_{B,T}\{S_m(\delta_f)+\delta_s+R_m+\beta_{n_m}\}.
\]
Using \(R_m\le C_{B,T}\beta_{n_m}\), \eqref{def:Sm}, and
\(\beta_{n_m}\le\alpha_{n_m}\) proves
\eqref{ineq:B_blockwise_rate}. Finally,
\[
\alpha_{n_m}
+
\frac{\beta_{n_m}}{\alpha_{n_m}}
\le
2(\no+n_m)^{-p_{\rm det}},
\]
which gives~\eqref{ineq:B_blockwise_rate_pdet}.
\end{proof}

\subsection{Proof of Theorem~\ref{th:iter_rate_nonlinear}}
\label{sec:app_proof_nonlinear}
\begin{proof}
By the Borel--Cantelli argument in Section~\ref{sec:conv_rate}, almost
surely there exists \(m_0<\infty\) such that \(\cG_m^{\rm rate}\)
holds for every \(m\ge m_0\). Fix such a sample path and write
\(U_m:=J(\te_{n_m})\) and
\(e_m^J:=L_JC_{B,T}d_m^{\rm rate}\).
At the block endpoint, \eqref{ineq:localized-block-rate} and the
Lipschitz continuity of \(J\), followed by nonincrease of \(J\) along
the ODE flow, \(H_m\ge T\), and
Assumption~\ref{assum:one_block_decrease}, give
\[
U_{m+1}\le J\bigl(z^{(T_m)}(H_m)\bigr)+e_m^J
\le U_m-\mathcal D_J(U_m)+e_m^J.
\]
This proves~\eqref{ineq:nonlinear_block_recursion}.

For \(n_m\le n<n_{m+1}\), let
\(r_n:=t(n)-T_m\in[0,H_m]\).
Then \(\te^{(m),\circ}(r_n)=\te_n\), and the within-block tracking
estimate gives
\[
J(\te_n)\le J\bigl(z^{(T_m)}(r_n)\bigr)+e_m^J
\le U_m+e_m^J.
\]
This proves~\eqref{ineq:nonlinear_within_block_transfer}.

It remains to prove convergence. If \(J_{\max}=0\), the conclusion is
immediate. Otherwise, fix \(0<\eta\le2J_{\max}\) and set
\(c_\eta:=\min_{u\in[\eta/2,J_{\max}]}\mathcal D_J(u)>0\).
The definition of \(d_m^{\rm rate}\) and
Lemmas~\ref{lem:block_geometry} and
\ref{lem:harmonic_block_geometry}, in their respective cases, give
\(d_m^{\rm rate}\to0\), and hence
\(e_m^J\to0\). Thus, for all sufficiently large \(m\),
\[
e_m^J\le\min\{c_\eta/2,\eta/2\},
\qquad
U_m\ge\eta/2\ \Longrightarrow\ U_{m+1}\le U_m-c_\eta/2.
\]
The sequence must therefore eventually enter \([0,\eta/2)\). Once it
does, the two preceding inequalities show that it remains in
\([0,\eta]\): below \(\eta/2\), its next increase is at most
\(e_m^J\), while on \([\eta/2,\eta]\) it decreases by at least
\(c_\eta/2\). Since \(\eta>0\) is arbitrary, \(U_m\to0\). The
within-block estimate and \(e_m^J\to0\) give \(J(\te_n)\to0\).
Finally, compactness of \(\Te\), continuity of \(J\), and
\(J^{-1}(\{0\})=\eq\) imply
\(\operatorname{dist}(\te_n,\eq)\to0\).
\end{proof}

\section*{Acknowledgments}
During the preparation of this manuscript, the authors used
large language models (LLMs) to improve its presentation and clarity.

\bibliographystyle{plainnat}

\bibliography{references}

\clearpage
\setcounter{section}{0}
\renewcommand{\thesection}{S.\arabic{section}}
\renewcommand{\theHsection}{supp.\arabic{section}}

\section*{Supplementary Appendix: Proofs and Additional Results}

\etocdepthtag{ec}
\begingroup
\etocsettagdepth{main}{none}
\etocsettagdepth{ec}{section}
\etocobeydepthtags
\etocsettocstyle
  {\par\medskip\noindent{\large\bfseries Contents}\par\smallskip}
  {\par\medskip}
\etocsetstyle{section}
  {\parindent0pt\parskip0pt\small}
  {}
  {\noindent\makebox[4em][l]{\bfseries\etocnumber}%
   \etocname\nobreak\leaders\hbox to .5em{\hss.\hss}\hfill
   \nobreak\etocpage\par}
  {}
\tableofcontents
\endgroup

\noindent\textbf{Guide to the Supplementary Appendix.}
Appendix~\ref{sec:ec_projected_td} gives a standalone projected on-policy
linear TD(0) application and its last-iterate rate.
Appendix~\ref{sec:projected_stochastic_gradient} gives the projected-SGD
application, including distance and objective-gap rates, and
Appendix~\ref{sec:one_scale_specialise} states the general one-time-scale
specializations under contraction and power Lyapunov decrease.

Appendix~\ref{sec:app_sec_pf} establishes regularity of the stationary law
and fast equilibrium. Appendices~\ref{sec:appendix_proofs}--
\ref{sec:ec_section7_proofs} contain, respectively, the remaining
fast-scale proofs and supporting results; the remaining slow-scale
martingale, Poisson, and projected-path arguments; and the remaining
convergence and rate proofs, including those for the one-time-scale
specializations. Appendix~\ref{sec:ec_section7_proofs} also proves
Proposition~\ref{prop:face_rates} and records the harmonic-endpoint
sharpness example. The seven selected proofs of the main results appear in
Sections~\ref{sec:app_proof_fast_tracking}--
\ref{sec:app_proof_nonlinear} of the main-paper appendix.
Appendix~\ref{sec:appendix_proofs} also records, in
Remark~\ref{rem:fast_coordinate_norm}, the extension of the unprojected
fast-scale analysis to an arbitrary fixed fast-coordinate norm.
Appendix~\ref{sec:ec_section8_proofs} verifies the actor--critic application,
including controlled-chain coverage, critic contraction, Lyapunov geometry,
restriction bias, and improved actor and critic rates.
Appendix~\ref{sec:ec_fast_projection_extension} gives the optional extension
with every-step compact-polyhedral Euclidean projection of the fast iterate,
including high-probability tracking and almost-sure rates.
Appendices~\ref{sec:aux_app}--\ref{app:strong_monotonicity} collect, in order,
the auxiliary fast-scale Poisson estimates, the auxiliary slow-scale Poisson
estimates, and the strong-monotonicity consequences for projected ODEs.
Appendix~\ref{sec:skorokhod} contains the
Skorokhod-map tools used for projection.
Appendix~\ref{sec:ec_stepsize} derives auxiliary results related to step-sizes.
Appendix~\ref{sec:ec_misc} contains several miscellaneous results, while 
Appendix~\ref{sec:N0_cond} discusses the shifted-step-size and finite-burn-in
alternatives.

For \(1/2<\frakb\le1\) and \(\lambda>0\), define the rate envelope
used in the one-time-scale contraction results below
(including the projected TD(0) and projected-SGD applications) by
\[
\mathfrak R_n(\frakb,\lambda)
:=
\begin{cases}
(N_0+n)^{-\frakb/2}\sqrt{\log(N_0+n)},
&\frac12<\frakb<1,\\[1mm]
(N_0+n)^{-1/2}\sqrt{\log(N_0+n)},
&\frakb=1,\ \lambda>\frac12,\\[1mm]
(N_0+n)^{-1/2}\{\log(N_0+n)\}^{3/2},
&\frakb=1,\ \lambda=\frac12,\\[1mm]
(N_0+n)^{-\lambda},
&\frakb=1,\ 0<\lambda<\frac12.
\end{cases}
\]
\section{Projected On-Policy Linear TD(0)}
\label{sec:ec_projected_td}

Projected linear TD(0) approximates a fixed policy's discounted value from one
Markov trajectory.~Fix a policy \(\pi\) in a finite discounted Markov reward process.  Let
\(P\) be the resulting irreducible transition matrix on \(\cS\), let
\(d\) be its stationary distribution, and set \(D:=\operatorname{diag}(d)\).
Given the natural pre-transition history \(\cF_n^{\rm td}\), suppose that
\((R_{n+1},S_{n+1})\) is drawn from a fixed joint reward--transition kernel
depending only on \(S_n\), with transition marginal \(P\), and that
\[
 r(s):=\bE[R_{n+1}\mid S_n=s],
 \qquad
 |R_{n+1}|\le R_{\max}\quad\text{a.s.}
\]
Fix \(0<\gamma_{\rm td}<1\), and let \(V^\pi\) denote the fixed policy's
value function, so that \(V^\pi=r+\gamma_{\rm td}PV^\pi\). Let
\(\Psi=(\psi(s)^\top)_{s\in\cS}\in\bR^{|\cS|\times d_w}\) be a
full-column-rank feature matrix; its columns span the linear class
\(V_w:=\Psi w\), so \(V_w(s)=\psi(s)^\top w\). Let \(W\) be a
nonempty compact convex polyhedron, assume \(w_1\in W\), and consider
\begin{equation}
\begin{aligned}
 w_{n+1}
 &=
 \operatorname{proj}_{W}\!\left[
 w_n+\eta_n\psi(S_n)
 \left\{
 R_{n+1}+\gamma_{\rm td}\psi(S_{n+1})^\top w_n
             -\psi(S_n)^\top w_n
 \right\}
 \right],\\
 \eta_n&=c_{\rm td}(N_0+n)^{-\frakb},
 \qquad c_{\rm td}>0,
 \qquad \frac12<\frakb\le1.
\end{aligned}
\label{eq:ec_projected_linear_td}
\end{equation}
Here \(c_{\rm td}>0\) is the multiplicative gain in the TD step size. Define
\begin{equation}
\begin{aligned}
 A_{\rm td}
 &:=
 \Psi^\top D(I-\gamma_{\rm td}P)\Psi,
 &q_{\rm td}&:=\Psi^\top Dr,\\
 \mu_{\rm td}
 &:=
 (1-\gamma_{\rm td})
 \lambda_{\min}(\Psi^\top D\Psi),
 &w_{\rm td}^\star&:=A_{\rm td}^{-1}q_{\rm td}.
\end{aligned}
\label{eq:ec_td_mean_field_objects}
\end{equation}
The proof below verifies explicitly that \(A_{\rm td}\) is nonsingular and
\(\mu_{\rm td}>0\). Assume that \(w_{\rm td}^\star\in W\).

\begin{corollary}[Projected linear TD(0)]
\label{cor:ec_projected_linear_td}
Under the preceding conditions, with
\(\lambda_{\rm td}:=c_{\rm td}\mu_{\rm td}\),
\[
 \norm{w_n-w_{\rm td}^\star}
 =
 O_{\rm a.s.}\!\left(
 \mathfrak R_n(\frakb,\lambda_{\rm td})
 \right).
\]
Thus, if \(1/2<\frakb<1\),
\[
 \norm{w_n-w_{\rm td}^\star}
 =
 O_{\rm a.s.}\!\left(
 (N_0+n)^{-\frakb/2}\sqrt{\log(N_0+n)}
 \right).
\]
At \(\frakb=1\), equivalently,
\[
 \norm{w_n-w_{\rm td}^\star}
 =
 \begin{cases}
 O_{\rm a.s.}\!\left((N_0+n)^{-1/2}\sqrt{\log(N_0+n)}\right),
 &\lambda_{\rm td}>1/2,\\[1mm]
 O_{\rm a.s.}\!\left((N_0+n)^{-1/2}
       \{\log(N_0+n)\}^{3/2}\right),
 &\lambda_{\rm td}=1/2,\\[1mm]
 O_{\rm a.s.}\!\left((N_0+n)^{-\lambda_{\rm td}}\right),
 &0<\lambda_{\rm td}<1/2.
 \end{cases}
\]
When \(\frakb=1\) and \(\lambda_{\rm td}>1/2\),
\[
 \norm{\Psi(w_n-w_{\rm td}^\star)}_D^2
 =
 O_{\rm a.s.}\!\left(
 \frac{\log(N_0+n)}{N_0+n}
 \right),
 \qquad
 \norm{v}_D^2:=v^\top Dv.
\]
\end{corollary}

\begin{proof}
Write \(\beta_n=(N_0+n)^{-\frakb}\) and define
\[
 g_{\rm td}(s,w)
 :=
 c_{\rm td}\psi(s)
 \left\{
 r(s)+\gamma_{\rm td}\sum_{s'}P(s,s')\psi(s')^\top w
       -\psi(s)^\top w
 \right\}.
\]
Set
\[
 N_{n+1}^{\rm td}
 :=
 c_{\rm td}\psi(S_n)
 \left\{
 R_{n+1}+\gamma_{\rm td}\psi(S_{n+1})^\top w_n
             -\psi(S_n)^\top w_n
 \right\}
 -g_{\rm td}(S_n,w_n).
\]
Then~\eqref{eq:ec_projected_linear_td} has the form
\[
 w_{n+1}
 =\operatorname{proj}_W\!\left[
 w_n+\beta_n\{g_{\rm td}(S_n,w_n)+N_{n+1}^{\rm td}\}
 \right],
\]
and
\[
 \bE[N_{n+1}^{\rm td}\mid\cF_n^{\rm td}]=\mathbf 0.
\]
Because \(W\) is compact, \(\cS\) is finite, and the rewards and features are
bounded, \(\{N_{n+1}^{\rm td}\}\) is uniformly bounded. Because the transition
kernel \(P\) does not depend on \(w\), the required Lipschitz condition on the
kernel holds with constant zero. Finite-state irreducibility gives the unique
stationary law and the required hitting-time bound, and
\(g_{\rm td}(s,\cdot)\) is uniformly Lipschitz on \(W\).  Moreover,
\[
 \sum_s d(s)g_{\rm td}(s,w)
 =
 c_{\rm td}\{q_{\rm td}-A_{\rm td}w\}.
\]

It remains to verify contraction of the projected mean ODE.  For
\(u\in\bR^{d_w}\), put \(v=\Psi u\).  Stationarity of \(d\) and
Jensen's inequality give
\[
 \norm{Pv}_D^2
 =
 \sum_s d(s)
 \left(\sum_{s'}P(s,s')v(s')\right)^2
 \le
 \sum_{s,s'}d(s)P(s,s')v(s')^2
 =
 \norm{v}_D^2.
\]
Consequently,
\begin{align}
 u^\top A_{\rm td}u
 &=
 \norm{v}_D^2-\gamma_{\rm td}\langle v,Pv\rangle_D\notag\\
 &\ge
 (1-\gamma_{\rm td})\norm{v}_D^2
 \ge
 \mu_{\rm td}\norm{u}^2.
\label{eq:ec_td_coercivity}
\end{align}
Here the first inequality also uses Cauchy--Schwarz and
\(\norm{Pv}_D\le\norm{v}_D\). Finite-state irreducibility gives
\(d(s)>0\) for every \(s\), and full column rank of \(\Psi\) therefore makes
\(\Psi^\top D\Psi\) positive definite. Hence \(\mu_{\rm td}>0\). If
\(A_{\rm td}u=0\), then~\eqref{eq:ec_td_coercivity} gives
\[
0
=u^\top A_{\rm td}u
\ge
\mu_{\rm td}\norm{u}^2,
\]
so \(u=0\). Thus \(A_{\rm td}\) is nonsingular.

Consider
\[
 \dot z
 =
 \Pi_W\!\left(
 z,c_{\rm td}\{q_{\rm td}-A_{\rm td}z\}
 \right).
\]
For every \(z,\bar w\in W\) and \(h\in\bR^{d_w}\), the tangent-cone
projection satisfies
\[
\left\langle z-\bar w,\Pi_W(z,h)-h\right\rangle\le0.
\]
Apply this inequality with \(\bar w=w_{\rm td}^\star\) and
\(h=c_{\rm td}\{q_{\rm td}-A_{\rm td}z\}\). Since
\(q_{\rm td}=A_{\rm td}w_{\rm td}^\star\),
equation~\eqref{eq:ec_td_coercivity} gives, for almost every \(t\),
\[
\begin{aligned}
\frac12\frac{\mathrm d}{\mathrm dt}
\norm{z(t)-w_{\rm td}^\star}^2
&=
\left\langle
z(t)-w_{\rm td}^\star,
\Pi_W\!\left(z(t),c_{\rm td}\{q_{\rm td}-A_{\rm td}z(t)\}\right)
\right\rangle\\
&\le
-c_{\rm td}
\bigl(z(t)-w_{\rm td}^\star\bigr)^\top
A_{\rm td}
\bigl(z(t)-w_{\rm td}^\star\bigr)\\
&\le
-c_{\rm td}\mu_{\rm td}
\norm{z(t)-w_{\rm td}^\star}^2.
\end{aligned}
\]
Thus its flow satisfies
\[
 \norm{\mathsf S_t^{(1)}(w)-w_{\rm td}^\star}
 \le
 e^{-c_{\rm td}\mu_{\rm td}t}
 \norm{w-w_{\rm td}^\star}.
\]
Proposition~\ref{prop:one_time_rate}, with
\(J(w)=\norm{w-w_{\rm td}^\star}\),
\(\rho_T=e^{-c_{\rm td}\mu_{\rm td}T}\), and
\(\lambda_T^{(1)}=c_{\rm td}\mu_{\rm td}\), gives the stated rates.
The squared value-error conclusion follows from norm equivalence in finite
dimensions.
\end{proof}

Let \(T_{\rm td}v:=r+\gamma_{\rm td}Pv\), and let
\[
 \Pi_D
 :=
 \Psi(\Psi^\top D\Psi)^{-1}\Psi^\top D
\]
be the \(D\)-orthogonal projector onto \(\operatorname{range}(\Psi)\).
Then \(\Psi w_{\rm td}^\star=\Pi_DT_{\rm td}(\Psi w_{\rm td}^\star)\), 
so \(\Psi w_{\rm td}^\star\) is the usual on-policy projected-Bellman solution; 
see~\citet{tsitsiklis1997analysis}.
This Bellman
projection differs from \(\operatorname{proj}_W\), which safeguards the
iterate. If \(W\) excludes
\(w_{\rm td}^\star\), the target becomes the constrained
variational-inequality solution for \(A_{\rm td}w-q_{\rm td}\); assuming
\(w_{\rm td}^\star\in W\) preserves the usual TD target.

Finite-time TD results differ in sampling model, projection, output, and
guarantee. \citet{bhandari2021finite} study expectation bounds for
projected TD driven by a Markov
trajectory, whereas \citet{chandak2026concentration} obtain
uniform-in-time concentration for raw unprojected online TD iterates.
The result above instead gives an eventual almost-sure envelope
for the raw iterate of explicitly projected TD as a direct specialization of
Proposition~\ref{prop:one_time_rate}.

\section{Optimization: Projected Stochastic Gradient Descent}
\label{sec:projected_stochastic_gradient}

We consider the constrained convex optimization problem
\(\Phi^\star:=\min_{\te\in\Te}\Phi(\te)\), where \(\Te\) is a compact
convex polyhedron and \(\Phi\) is continuously differentiable on an open
neighborhood of \(\Te\), convex on \(\Te\), and has an
\(L_\Phi\)-Lipschitz gradient there. Projected SGD for this problem is
\begin{equation}
\te_{n+1}
=
\operatorname{proj}_{\Te}
\left[
\te_n-\beta_nG_{n+1}
\right],
\qquad
\beta_n=(\no+n)^{-\frakb},
\qquad
\frac12<\frakb\le1,
\label{eq:projected_sgd}
\end{equation}
Let \(\cF_n^{\rm opt}\) denote the natural optimization history 
through time \(n\), including the current iterate and all oracle samples revealed before \(G_{n+1}\) is sampled.
Assume that, for some deterministic
\(B_{\rm sgd}<\infty\),
\[
\bE[G_{n+1}\mid\cF_n^{\rm opt}]
=
\nabla\Phi(\te_n),
\qquad
\norm{G_{n+1}-\nabla\Phi(\te_n)}
\le
B_{\rm sgd}
\quad\text{a.s.}
\]
This is a one-time-scale singleton-state application, so Markov noise is
absent. Its limiting projected gradient flow is
\begin{equation}
\dot z
=
\Pi_{\Te}
\left(
z,-\nabla\Phi(z)
\right).
\label{eq:projected_gradient_flow}
\end{equation}
\begin{corollary}[Almost-sure rates for projected SGD]
\label{cor:projected_sgd_rates}
Suppose that \(1/2<\frakb\le1\) and that \(\Phi\) is
\(m_\Phi\)-strongly convex on an open
neighborhood of \(\Te\), and let \(\te^\star\) be its unique minimizer over
\(\Te\). Then,
\begin{align}
\norm{\te_n-\te^\star}
&=
O_{\rm a.s.}
\left(
\mathfrak R_n(\frakb,m_\Phi)
\right),
\label{eq:projected_sgd_distance_rate}\\
\Phi(\te_n)-\Phi^\star
&=
O_{\rm a.s.}
\left(
\mathfrak R_n(\frakb,m_\Phi)
\right).
\label{eq:projected_sgd_gap_rate}
\end{align}
If, in addition, \(\te^\star\in\operatorname{int}(\Te)\), then
\begin{equation}
\Phi(\te_n)-\Phi^\star
=
O_{\rm a.s.}
\left(
\mathfrak R_n(\frakb,m_\Phi)^2
\right).
\label{eq:projected_sgd_interior_gap_rate}
\end{equation}
\end{corollary}
\begin{proof}
Write
\[
-G_{n+1}
=
-\nabla\Phi(\te_n)
+
\bigl\{\nabla\Phi(\te_n)-G_{n+1}\bigr\}.
\]
By the oracle assumptions following~\eqref{eq:projected_sgd}, the term in
braces is a uniformly
bounded martingale difference with respect to
\(\{\cF_n^{\rm opt}\}\). Hence~\eqref{eq:projected_sgd} has the
one-time-scale form in Proposition~\ref{prop:one_time_rate}, with drift
\(-\nabla\Phi\). Taking the Markov state to be a singleton makes
Assumptions~\ref{assum:markov_noise},
\ref{assum:lipschitz_kernel}, and
\ref{assum:hitting_time} automatic. The oracle condition gives the
bounded martingale difference required by
Assumption~\ref{assum:1_2}, while
\(g=-\nabla\Phi\) is \(L_\Phi\)-Lipschitz on \(\Te\), as required by
Assumption~\ref{assum:lipschitz}; compactness of \(\Te\) also makes this
drift bounded. The required flow contraction is verified next.

Under strong convexity,
\[
\left\langle
-\nabla\Phi(\te)+\nabla\Phi(\te^\star),
\te-\te^\star
\right\rangle
\le
-m_\Phi\norm{\te-\te^\star}^2.
\]
The first-order optimality condition
\(\langle\nabla\Phi(\te^\star),\te-\te^\star\rangle\ge0\) for every
\(\te\in\Te\) is equivalent to
\(\Pi_{\Te}(\te^\star,-\nabla\Phi(\te^\star))=0\); see
\citet[Lemma~1]{dupuis1993dynamical}. Thus \(\te^\star\) is an equilibrium of
\eqref{eq:projected_gradient_flow}.

Hence
Corollary~\ref{cor:strong_monotone_consequences} gives, with
\(\mathsf S_t^{(1)}\) denoting the flow in
\eqref{eq:projected_gradient_flow},
\[
\norm{\mathsf S_t^{(1)}(\te)-\te^\star}
\le e^{-m_\Phi t}\norm{\te-\te^\star},
\qquad t\ge0.
\]
Thus, for \(J(\te)=\norm{\te-\te^\star}\), one may take
\(\rho_T=e^{-m_\Phi T}\), so \(\lambda_T^{(1)}=m_\Phi\).
Proposition~\ref{prop:one_time_rate} then gives
\(J(\te_n)=O_{\rm a.s.}(\mathfrak R_n(\frakb,m_\Phi))\), proving
\eqref{eq:projected_sgd_distance_rate}.
Since \(\nabla\Phi\) is bounded on the compact set \(\Te\), the
fundamental theorem of calculus along the segment from \(\te^\star\) to
\(\te\), followed by Cauchy--Schwarz, gives
\[
0
\le
\Phi(\te)-\Phi^\star
\le
\left(
\sup_{u\in\Te}\norm{\nabla\Phi(u)}
\right)
\norm{\te-\te^\star}.
\]
Thus the objective-function gap is bounded by a finite deterministic multiple
of the distance to \(\te^\star\). Combining this inequality with the distance
rate~\eqref{eq:projected_sgd_distance_rate} proves the objective-gap
rate~\eqref{eq:projected_sgd_gap_rate}. If
\(\te^\star\in\operatorname{int}(\Te)\), then
\(\nabla\Phi(\te^\star)=0\), and the descent lemma gives
\[
0
\le
\Phi(\te_n)-\Phi^\star
\le
\frac{L_\Phi}{2}
\norm{\te_n-\te^\star}^2.
\]
Thus, in the interior case, the objective-function gap is controlled by a
finite deterministic multiple of the squared distance. Squaring the distance
rate in
\eqref{eq:projected_sgd_distance_rate} therefore proves
\eqref{eq:projected_sgd_interior_gap_rate}.

\end{proof}

\section{One-time-scale specialization}
\label{sec:one_scale_specialise}
Recursions without a fast variable can be embedded in the general scheme by
taking the fast coordinate, its update, and its noise identically zero. We
state the resulting one-time-scale guarantee as follows.
\begin{proposition}[One-time-scale contraction specialization]
\label{prop:one_time_rate}
Consider
\begin{equation}
\te_{n+1}=\operatorname{proj}_{\Te}
 \left(\te_n+\beta_n\{g(Y_n,\te_n)+N_{n+1}\}\right),
\qquad
\beta_n=(N_0+n)^{-\frakb},
\label{eq:one_time_generic}
\end{equation}
where \(1/2<\frakb\le1\). Suppose that \(\{Y_n\}\) is a controlled
Markov chain on the finite state space \(\cY\) whose conditional transition
kernel, given \(\cF_n\), is \(p^{(\te_n)}(Y_n,\cdot)\). Assume that each
\(p^{(\te)}\) has a unique stationary distribution \(\mu^{(\te)}\), that the
kernels depend Lipschitz-continuously on \(\te\), and that they satisfy the
uniform hitting-time bound. Suppose also that \(g(y,\cdot)\) is Lipschitz on
\(\Te\), uniformly in \(y\in\cY\), and that
\(\{N_{n+1}\}_{n\ge1}\) is a uniformly bounded martingale-difference sequence
with respect to \(\{\cF_n\}\). Define
\[
 \bar g(\te):=\sum_{y\in\cY}\mu^{(\te)}(y)g(y,\te)
\]
and consider the projected ODE
\[
 \dot z=\Pi_{\Te}(z,\bar g(z)).
\]
Let \(\mathsf S_s^{(1)}\) denote the time-\(s\) flow of this ODE.
Assume that a Lipschitz function \(J:\Te\to\bR_+\) is nonincreasing along
this ODE and that there exist \(T>0\) and \(\rho_T\in(0,1)\) such that
\[
 J(\mathsf S_s^{(1)}(\te))\le \rho_TJ(\te),
 \qquad \te\in\Te,\quad s\ge T.
\]
If \(\frakb<1\), then, almost surely, there are finite random variables
\(C\) and \(N\) such that
\[
 J(\te_n)\le C(N_0+n)^{-\frakb/2}
             \sqrt{\log(N_0+n)},
 \qquad n\ge N.
\]
If \(\frakb=1\), set
\[
\lambda_T^{(1)}:=-\frac{\log\rho_T}{T}.
\]
Then, almost surely, there are finite random variables \(C_1\) and
\(N_1\) such that, for \(n\ge N_1\),
\[
J(\te_n)
\le
C_1
\begin{cases}
(N_0+n)^{-1/2}\sqrt{\log(N_0+n)},
&\lambda_T^{(1)}>1/2,\\[1mm]
(N_0+n)^{-1/2}\{\log(N_0+n)\}^{3/2},
&\lambda_T^{(1)}=1/2,\\[1mm]
(N_0+n)^{-\lambda_T^{(1)}},
&0<\lambda_T^{(1)}<1/2.
\end{cases}
\]
\end{proposition}

\begin{proposition}[One-time-scale specialization under power Lyapunov decrease]
\label{prop:one_time_power_rate}
Consider~\eqref{eq:one_time_generic} under the stochastic regularity and
boundedness hypotheses of Proposition~\ref{prop:one_time_rate}, but
replace its one-block contraction hypothesis as follows. Let
\(\mathsf S_s^{(1)}\) denote the flow of the limiting projected ODE,
and suppose that \(J:\Te\to\bR_+\) is Lipschitz and nonincreasing along
this flow. Let
\[
J_{\max}:=\sup_{\te\in\Te}J(\te).
\]
Assume that there exist \(T>0\) and a continuous function
\(\mathcal D_J:[0,J_{\max}]\to\bR_+\) such that
\[
\mathcal D_J(0)=0,
\qquad
0<\mathcal D_J(u)\le u
\quad(0<u\le J_{\max}),
\]
and
\[
J(\mathsf S_T^{(1)}(\te))
\le
J(\te)-\mathcal D_J(J(\te)),
\qquad \te\in\Te.
\]
Then \(J(\te_n)\to0\) almost surely.

If, in addition, for some \(p_J>1\), \(\kappa_J>0\), and \(u_J>0\),
\[
\mathcal D_J(u)\ge\kappa_Ju^{p_J},
\qquad 0\le u\le u_J,
\]
then the following quantitative conclusions hold. If \(\frakb<1\),
\[
J(\te_n)
=
O_{\rm a.s.}\!\left(
(N_0+n)^{-\frac{1-\frakb}{p_J-1}}
+
(N_0+n)^{-\frac{\frakb}{2p_J}}
\bigl(\log(N_0+n)\bigr)^{\frac{1}{2p_J}}
\right).
\]
The optimized choice within the branch \(1/2<\frakb<1\) is
\[
\frakb=\frac{2p_J}{3p_J-1},
\]
for which
\[
J(\te_n)
=
O_{\rm a.s.}\!\left(
(N_0+n)^{-\frac{1}{3p_J-1}}
\bigl(\log(N_0+n)\bigr)^{\frac{1}{2p_J}}
\right).
\]
At the harmonic endpoint \(\frakb=1\), the same power-decrease
hypothesis gives instead
\[
J(\te_n)
=
O_{\rm a.s.}\!\left(
\{\log(N_0+n)\}^{-\frac1{p_J-1}}
\right).
\]
\end{proposition}

\section{Proofs from Main-Paper Section~\ref{sec:pf}}
\label{sec:app_sec_pf}

\begin{proof}[Proof of Lemma~\ref{lemma:lipschitz_stat_dist}]
The fixed-point identity, Assumption~\ref{assum:contraction}, and
$\norm{\havg(0;\te)}\le \mathsf B_0$ give
\[
\norm{x^\star(\te)}
=\norm{\havg(x^\star(\te);\te)}
\le\gamma\norm{x^\star(\te)}+\mathsf B_0.
\]
Hence $\norm{x^\star(\te)}\le \mathsf B_0/(1-\gamma)$ uniformly on $\Te$, so
$\frakc_8 =\sup_{\te\in\Te}\norm{x^\star(\te)}<\infty$.

We next prove Lipschitz continuity of $\te\mapsto\mu^{(\te)}$. Let
$\cY_0:=\cY\setminus\{y\ust\}$. If $\cY_0$ is empty,
then $\mu^{(\te)}=\delta_{y\ust}$ for every $\te$; take $\lmu=0$
and proceed directly to the fixed-point argument. Otherwise, let $Q^{(\te)}$ be the restriction of
$p^{(\te)}$ to $\cY_0\times\cY_0$. We regard distributions on
$\cY_0$ as row vectors. Under the global matrix-norm convention,
\[
\norm{w}_1:=\sum_{y\in\cY_0}|w(y)|,
\qquad
\norm{A}_\infty
=\max_{i\in\cY_0}\sum_{j\in\cY_0}|A(i,j)|
=\sup_{w\ne\mathbf 0}\frac{\norm{wA}_1}{\norm{w}_1}.
\]
In particular, $\norm{\cdot}_\infty$ is submultiplicative. For $i,j\in\cY_0$,
\[
(Q^{(\te)})^n(i,j)
=\bP_i^{(\te)}(Y_n=j,\tau_{y\ust}>n).
\]
Consequently,
\[
\sum_{n\ge0}\sum_{j\in\cY_0}(Q^{(\te)})^n(i,j)
=\bE_i^{(\te)}[\tau_{y\ust}]\le \frakc_5,
\]
and therefore
\[
Z^{(\te)}:=\sum_{n\ge0}(Q^{(\te)})^n=(I-Q^{(\te)})^{-1},
\qquad \norm{Z^{(\te)}}_\infty\le \frakc_5.
\]
Order the states as $y\ust$ followed by $\cY_0$ and write
\[
p^{(\te)}=
\begin{pmatrix}
p^{(\te)}(y\ust,y\ust)&r^{(\te)}\\
b^{(\te)}&Q^{(\te)}
\end{pmatrix}.
\]
Set $v^{(\te)}:=r^{(\te)}Z^{(\te)}$ and
$d_{\rm cyc}^{(\te)}:=1+v^{(\te)}\mathbf 1$, where $\mathbf 1$
is the column vector of ones indexed by $\cY_0$. In this state ordering,
$(1,v^{(\te)})$ is a nonnegative row vector with total mass
$d_{\rm cyc}^{(\te)}$. Since
$v^{(\te)}(I-Q^{(\te)})=r^{(\te)}$,
\[
r^{(\te)}+v^{(\te)}Q^{(\te)}=v^{(\te)}.
\]
Row stochasticity also gives $b^{(\te)}=\mathbf 1-Q^{(\te)}\mathbf 1$ and hence
\[
v^{(\te)}b^{(\te)}
=r^{(\te)}\mathbf 1
=1-p^{(\te)}(y\ust,y\ust).
\]
Thus $(1,v^{(\te)})/d_{\rm cyc}^{(\te)}$ is a stationary probability
vector. Uniqueness in Assumption~\ref{assum:markov_noise} gives
\[
\mu^{(\te)}=\frac{(1,v^{(\te)})}{d_{\rm cyc}^{(\te)}}.
\]

Put $\Delta_\te:=\norm{\te-\te'}$. Applying
Assumption~\ref{assum:lipschitz_kernel} to the rows indexed by
$\cY_0$ and to the reference-state row gives
\[
\norm{Q^{(\te)}-Q^{(\te')}}_\infty\le \lp\Delta_\te,
\qquad
\norm{r^{(\te)}-r^{(\te')}}_1\le \lp\Delta_\te.
\]
The resolvent identity yields
\[
Z^{(\te)}-Z^{(\te')}
=Z^{(\te)}\bigl(Q^{(\te)}-Q^{(\te')}\bigr)Z^{(\te')},
\]
so submultiplicativity and $\norm{Z^{(\vartheta)}}_\infty\le\frakc_5$ give
\[
\norm{Z^{(\te)}-Z^{(\te')}}_\infty
\le\frakc_5^2\lp\Delta_\te.
\]
Since $\norm{r^{(\te')}}_1\le1$,
\[
\begin{aligned}
\norm{v^{(\te)}-v^{(\te')}}_1
&\le
\norm{r^{(\te)}-r^{(\te')}}_1\norm{Z^{(\te)}}_\infty
+\norm{r^{(\te')}}_1\norm{Z^{(\te)}-Z^{(\te')}}_\infty\\
&\le \lp\frakc_5(1+\frakc_5)\Delta_\te.
\end{aligned}
\]
Also,
$|d_{\rm cyc}^{(\te)}-d_{\rm cyc}^{(\te')}|
\le\norm{v^{(\te)}-v^{(\te')}}_1$ and
$d_{\rm cyc}^{(\te)}\ge1$. Therefore,
\[
\begin{aligned}
\norm{\mu^{(\te)}-\mu^{(\te')}}_1
&\le
\frac{\norm{v^{(\te)}-v^{(\te')}}_1
+|d_{\rm cyc}^{(\te)}-d_{\rm cyc}^{(\te')}|}
{d_{\rm cyc}^{(\te)}}\\
&\le2\lp\frakc_5(1+\frakc_5)\norm{\te-\te'}.
\end{aligned}
\]
Thus one may take $\lmu:=2\lp\frakc_5(1+\frakc_5)$.

Finally, contraction of $\havg(\cdot;\te)$ gives
\[
\begin{aligned}
\norm{x^\star(\te)-x^\star(\te')}
&\le\gamma\norm{x^\star(\te)-x^\star(\te')}\\
&\quad+\norm{\havg(x^\star(\te');\te)
-\havg(x^\star(\te');\te')}.
\end{aligned}
\]
By Assumption~\ref{assum:lipschitz}, the preceding
stationary-distribution estimate, and
$\norm{x^\star(\te')}\le \frakc_8$, the last term is at most
\[
\bigl[L_{h,\te}(\frakc_8)+\lmu\mathsf B_8\bigr]\norm{\te-\te'}.
\]
Rearranging proves
\[
\norm{x^\star(\te)-x^\star(\te')}
\le\frac{L_{h,\te}(\frakc_8)+\lmu\mathsf B_8}{1-\gamma}
\norm{\te-\te'}.
\]
Thus the displayed coefficient is an admissible choice of $\lfp$.
\end{proof}

\section{Supporting Results and Proofs for Main-Paper Section~\ref{sec:faster_iterations}}
\label{sec:appendix_proofs}

\subsection{Proof of Lemma~\ref{lemma:9}}
\label{pf:lemma:9}

\begin{proof}
The controlled Markov property gives
\[
\bE\left[
u_h^{(x_n,\te_n)}(Y_{n+1})\mid\cF_n\right]
=
\sum_{y\in\cY}
p^{(\te_n)}(Y_n,y)
u_h^{(x_n,\te_n)}(y).
\]
Hence \(\bE[M''_{n+1}\mid\cF_n]=\mathbf 0\) almost surely. Since
\(u_h^{(x_n,\te_n)}\) solves~\eqref{def:poisson},
\[
\begin{aligned}
&h(Y_n,x_n,\te_n)-\havg(x_n;\te_n)\\
&\quad=
u_h^{(x_n,\te_n)}(Y_n)
-
\sum_{y\in\cY}
p^{(\te_n)}(Y_n,y)
u_h^{(x_n,\te_n)}(y).
\end{aligned}
\]
Adding and subtracting \(u_h^{(x_n,\te_n)}(Y_{n+1})\), and using
\eqref{def:M''_n}, proves~\eqref{eq:fast_poisson_pointwise}.
\end{proof}

\subsection{Proof of Lemma~\ref{lem:block_fast_tail}}
\label{sec:ec_fast_martingale_tail}

\begin{proof}
Fix \(m\ge0\), \(1\le k\le K_m\), and
\(i\in\{1,\ldots,d_x\}\). For these fixed indices, the partial sums in
\(j\) defining \(\mathcal M_{m,k}^x(i)\) form a scalar martingale. By
Assumptions~\ref{assum:1_2} and~\ref{assum:xn_bound}, and by
part~(\ref{bound:m_2_prime}) of
Lemma~\ref{lemma:poisson_sensitivity_1}, its \(j\)th increment is bounded in
absolute value by
\(\alpha_{n_m+j}\chi(n_m+j+1,n_m+k-1)\cfastm\).
Lemma~\ref{lemma:alpha2_chi_bound} bounds the corresponding squared-weight
sum by \(2\alpha_{n_m}\). Applying the Azuma--Hoeffding inequality in
Theorem~\ref{th:azuma}, using the coordinate threshold
\(\delta/\kappa_{d_x}\), and taking a union bound over
\(i=1,\ldots,d_x\) and \(k=1,\ldots,K_m\) proves
\eqref{eq:block_fast_tail} with the stated constant
\(c_{\mathrm{mg}}^x\).

Under the arbitrary-norm formulation of
Remark~\ref{rem:fast_coordinate_norm}, let \(\norm{\cdot}_x\) be the fixed
fast-coordinate norm and set
\(\kappa_x:=\sup_{\norm{u}_\infty\le1}\norm{u}_x\) and
\(\lambda_x:=\sup_{\norm{u}_x\le1}\norm{u}_\infty\). For every
\(r\ge1\), \(\norm{\Xi_{r+1}}_x\le\cfastm\), and hence
\(\norm{\Xi_{r+1}}_\infty\le\lambda_x\cfastm\). The good event uses
the coordinate threshold \(\delta/\kappa_x\). The same calculation
therefore replaces \(c_{\mathrm{mg}}^x\) by
\[
\frac{1}{4\kappa_x^2\lambda_x^2(\cfastm)^2},
\]
which is the asserted norm-robust constant.
\end{proof}

\subsection{Fast Poisson residual}

For the residual calculation, set
\[
A_x
:=
\frakc_1+\mathsf B_6+\frakc_6(1+\frakc_2),
\qquad
A_\te
:=
\frakc_1+\frakc_2\frakc_6+\mathsf B_6,
\]
\[
U_h
:=
2\frakc_5\mathsf B_6,
\qquad
C_\te
:=
\chpoi A_\te,
\qquad
C_{\rm res}^x
:=
8\chpoi A_x+3U_h.
\]
Here \(A_x\) and \(A_\te\) control the one-step changes of the fast and
slow iterates, respectively, while \(U_h\) bounds the fast Poisson
solution.

The definitions above and
Condition~\ref{cond:hyper_1} imply
\begin{equation}
C_\te\beta_1
\le
\chpoi A_x\alpha_1.
\label{cond:N0:dom}
\end{equation}
Indeed,
\(A_x=A_\te+\frakc_6\ge A_\te\),
\(C_\te=\chpoi A_\te\), and
\(\beta_1\le\alpha_1\).
Together with the monotonicity of \(\beta_n/\alpha_n\), this is the
domination used below to control changes in the Poisson solution.

Recall that
\[
\Delta_{n+1}^h
=
u_h^{(x_n,\te_n)}(Y_n)
-u_h^{(x_n,\te_n)}(Y_{n+1}).
\]
For \(m\ge0\) and \(0\le k\le K_m\), define the associated local Poisson residual by
\begin{equation}
\begin{aligned}
\mathcal R_{m,k}^x
&:=
\sum_{j=0}^{k-1}
\alpha_{n_m+j}
\chi(n_m+j+1,n_m+k-1)
\Delta_{n_m+j+1}^h,
\end{aligned}
\label{eq:block_fast_poisson_residual}
\end{equation}
with \(\mathcal R_{m,0}^x:=\mathbf 0\).

\begin{proposition}[Local fast Poisson residual]
\label{prop:block_fast_pois_residual}
Under~\eqref{cond:N0:alpha2}, with
\eqref{cond:N0:dom} as established above,
\begin{equation}
\norm{\mathcal R_{m,k}^x}
\le
C_{\rm res}^x\alpha_{n_m},
\qquad
m\ge0,\quad 0\le k\le K_m.
\label{eq:block_fast_poisson_residual_bound}
\end{equation}
\end{proposition}

\begin{proof}
The case $k=0$ is immediate. Fix $k\ge1$. For the remainder of this proof,
write
\[
s:=n_m,\qquad t:=n_m+k,\qquad
u_r:=u_h^{(x_r,\te_r)},\qquad
A_r:=\alpha_r\chi(r+1,t-1).
\]
Summation by parts in \eqref{eq:block_fast_poisson_residual} gives
\[
\begin{aligned}
\mathcal R_{m,k}^x
={}&A_su_s(Y_s)-A_{t-1}u_{t-1}(Y_t)\\
&+\sum_{r=s+1}^{t-1}A_r\bigl[u_r(Y_r)-u_{r-1}(Y_r)\bigr]\\
&+\sum_{r=s+1}^{t-1}(A_r-A_{r-1})u_{r-1}(Y_r),
\end{aligned}
\]
with empty sums interpreted as zero. By part~(\ref{lemma:B2.3}) of Lemma~\ref{lemma:poisson_sensitivity_1}, \eqref{cond:N0:dom}, and the
monotonicity of $\beta_n/\alpha_n$,
\[
\norm{u_r(Y_r)-u_{r-1}(Y_r)}
\le \chpoi A_x\alpha_{r-1}+C_\theta\beta_{r-1}
\le2\chpoi A_x\alpha_{r-1}.
\]
Since $\alpha_{r-1}\le2\alpha_r$, Lemma~\ref{lemma:alpha2_chi_bound} yields
\[
\left\|
\sum_{r=s+1}^{t-1}A_r\bigl[u_r(Y_r)-u_{r-1}(Y_r)\bigr]
\right\|
\le8\chpoi A_x\alpha_{t-1}.
\]
The subadditivity argument used in Lemma~\ref{lemma:weight_variation_chi} shows that $A_r$ is
nondecreasing. Hence
\[
\sum_{r=s+1}^{t-1}|A_r-A_{r-1}|
=A_{t-1}-A_s\le\alpha_{t-1}.
\]
Part~(\ref{poisson_h_bound}) of Lemma~\ref{lemma:poisson_sensitivity_1} bounds the second interior sum by $U_h\alpha_{t-1}$, and
the two boundary terms by $2U_h\alpha_s$. Since $\alpha_{t-1}\le\alpha_s$,
\[
\norm{\mathcal R_{m,k}^x}
\le\bigl(8\chpoi A_x+3U_h\bigr)\alpha_s
=C_{\rm res}^x\alpha_{n_m},
\]
which proves \eqref{eq:block_fast_poisson_residual_bound}.
\end{proof}

\subsection{Proof of Corollary~\ref{coro:fast_across_blocks}}

\begin{proof}
On \(\cG_x^{(j)}(\delta_j)\),
\eqref{eq:block_fast_endpoint} gives
\[
\varepsilon_{n_{j+1}}^{(x)}
\le
\mathcal W_j(K_j)\varepsilon_{n_j}^{(x)}
+
C_xg_j^{(x)}(\delta_j).
\]
Iterating this recursion for \(j=0,\ldots,m-1\) proves
\eqref{eq:fast_across_blocks}. The probability estimate
\eqref{eq:fast_across_blocks_probability} follows from the union bound
and Lemma~\ref{lem:block_fast_tail}.
\end{proof}

\subsection{Arbitrary fast-coordinate norms}

\begin{remark}[Choice of norm for the fast coordinate]
\label{rem:fast_coordinate_norm}
The fast-scale arguments in Sections~\ref{sec:faster_iterations}--
\ref{sec:conv_rate} remain valid if the Euclidean norm on
\(\bR^{d_x}\) is replaced by any fixed norm \(\norm{\cdot}_x\), provided
that the contraction, local Lipschitz, bounded-set, stability, and noise
bounds involving the fast coordinate are formulated in that norm.  In
$\mathsf B(R)$, the fast-state radius and the norm of $h$ are then
interpreted using $\norm{\cdot}_x$, while the norm of $f$ remains
Euclidean. Recall the two norm-equivalence
constants introduced in the proof of Lemma~\ref{lem:block_fast_tail}:
\[
  \kappa_x=\sup_{\norm{u}_\infty\le1}\norm{u}_x,
  \qquad
  \lambda_x=\sup_{\norm{u}_x\le1}\norm{u}_\infty.
\]
Thus
\(\norm{z}_x\le\kappa_x\norm{z}_\infty\) and
\(\norm{z}_\infty\le\lambda_x\norm{z}_x\). The first inequality converts
the coordinatewise martingale estimate to \(\norm{\cdot}_x\), while the
second converts an \(x\)-norm increment bound to a coordinatewise one.
	The corresponding fast-martingale event and tail constant are recorded in
	the proof of Lemma~\ref{lem:block_fast_tail} in
		Section~\ref{sec:ec_fast_martingale_tail} of the Online Companion. The slow projection and
Skorokhod-map arguments remain Euclidean.
\end{remark}

\section{Proofs from Main-Paper Section~\ref{sec:analysis_te_n}}
\label{sec:ec_section6_proofs}

\subsection{Proof of Lemma~\ref{lemma:metivier_mtg}}
\begin{proof}
The martingale-difference property follows from
\[
\bE\left[u^{(\te\um_n)}_f(Y\um_{n+1})\mid \cF\um_n \right]
=
\sum_{y\in\cY}p^{(\te\um_n)}(Y\um_n,y)u^{(\te\um_n)}_f(y).
\]
Since $u^{(\te\um_k)}_f$ solves~\eqref{def:poisson_f} with the parameter $\te$ equal to $\te\um_k$, we have
\[
f(Y\um_k,x\ust(\te\um_k),\te\um_k)-\favg(\te\um_k;\te\um_k)
=
u^{(\te\um_k)}_f(Y\um_k)-\sum_{y\in\cY}p^{(\te\um_k)}(Y\um_k,y)u^{(\te\um_k)}_f(y).
\]
Adding and subtracting $u^{(\te\um_k)}_f(Y\um_{k+1})$ gives~\eqref{eq:slow-poisson-decomposition}.
\end{proof}

\subsection{Proof of Lemma~\ref{lemma:mc_diarmid_mtilde}}
\begin{proof}
By Assumptions~\ref{assum:1_2} and~\ref{assum:xn_bound}, and
part~(\ref{lemma:Poisson_3}) of Lemma~\ref{lemma:Poisson},
the coordinate increments of the weighted martingale differences
\(\beta\um_\ell[(M')\um_{\ell+1}+(M''')\um_{\ell+1}]\) are bounded by
\(\beta\um_\ell C_{\mathrm{mg}}^{(\theta)}\). Apply the Azuma--Hoeffding inequality in
Theorem~\ref{th:azuma} coordinatewise, use
\(\sum_{\ell=0}^{k-1}(\beta\um_\ell)^2=b(n_m)-b(n_m+k)\), and take a
union bound over \(i=1,\ldots,d_{\te}\) and
\(k=1,\ldots,K_m\) to obtain the stated bound.
\end{proof}

\subsection{Proof of Corollary~\ref{coro:te-tte}}

\begin{proof}
On \(\cG_x^{(m)}(\delta_f)\),
Proposition~\ref{prop:conc_v_v_te}, specifically
\eqref{eq:block_fast_weighted}, gives
\[
\sum_{\ell=0}^{K_m-1}
\beta\um_\ell\norm{e\um_\ell}
=
\sum_{r=n_m}^{n_{m+1}-1}
\beta_r\norm{x_r-x\ust(\te_r)}
\le
S_m(\delta_f).
\]
On \(\cG_{\te}^{(m)}(\delta_s)\), Proposition~\ref{prop:te_te1} with
\(n=K_m\) therefore yields
\[
\sup_{0\le k\le K_m}
\norm{\te\um_k-\tte\um_k}
\le
\etem(\delta_f,\delta_s)
\]
by~\eqref{def:E_theta_ttheta_block}. Since both paths are right-continuous step embeddings on the same
slow-time grid,
\[
\sup_{0\le t\le H_m}
\norm{
\te^{(m),\circ}(t)-\tte^{(m),\circ}(t)
}
=
\max_{0\le k\le K_m}
\norm{\te\um_k-\tte\um_k}.
\]
The asserted step-path bound follows.
\end{proof}

\section{Proofs from Main-Paper Section~\ref{sec:conv_rate}}
\label{sec:ec_section7_proofs}

\subsection{Proof of Theorem~\ref{th:power_decrease_rate}}
\begin{proof}
Theorem~\ref{th:iter_rate_nonlinear} first gives \(U_m\to0\).
Consequently, the local power condition
\eqref{ineq:power_lyapunov_decrease} and
\eqref{ineq:nonlinear_block_recursion} imply, eventually almost surely,
\[
U_{m+1}
\le
U_m-\kappa_JU_m^{p_J}+e_m^J.
\]
Suppose first that \(\frakb<1\). By~\eqref{eq:block-index-growth},
\[
N_0+n_m\asymp(m+1)^{1/(1-\frakb)}.
\]
The assumed rate envelope for \(d_m^{\rm rate}\) therefore gives
\[
e_m^J
\le
C(m+2)^{-\frac{r}{1-\frakb}}
\{\log(m+2)\}^{1/2}.
\]
Applying Lemma~\ref{lem:nonlinear_block_recursion} with
\(\zeta=r/(1-\frakb)\) and \(\ell=1/2\) yields
\[
U_m
\le
C(\omega)
\left\{
(m+2)^{-1/(p_J-1)}
+
(m+2)^{-\frac{r}{p_J(1-\frakb)}}
\{\log(m+2)\}^{1/(2p_J)}
\right\}.
\]
Using~\eqref{eq:block-index-growth} once more gives
\[
U_m
\le
C(\omega)
\left\{
(N_0+n_m)^{-\frac{1-\frakb}{p_J-1}}
+
(N_0+n_m)^{-\frac r{p_J}}
\{\log(N_0+n_m)\}^{\frac1{2p_J}}
\right\}.
\]
For \(n_m\le n<n_{m+1}\), the within-block contribution in
\eqref{ineq:nonlinear_within_block_transfer} is of order
\((N_0+n_m)^{-r}\sqrt{\log(N_0+n_m)}\), which is eventually bounded
by its \(p_J\)th root. The comparison between \(N_0+n\) and
\(N_0+n_m\) used in the proof of Theorem~\ref{th:iter_rate} then proves
\eqref{ineq:nonlinear_iteration_rate}.

Suppose now that \(\frakb=1\). By
Lemma~\ref{lem:harmonic_block_geometry},
\(N_0+n_m\asymp e^{mT}\). Hence the assumed envelope gives
\[
e_m^J\le Ce^{-rTm}(m+1)^{1/2}.
\]
For every \(Q>0\), the right-hand side is eventually at most
\(C_Q(m+2)^{-Q}\). Choose \(Q>p_J/(p_J-1)\) and apply
Lemma~\ref{lem:nonlinear_block_recursion} with \(\zeta =Q\) and
\(\ell=0\). Then
\[
U_m=O_{\rm a.s.}\!\left((m+2)^{-1/(p_J-1)}\right).
\]
Here and below, the implicit constants in \(O_T(\cdot)\) may depend on the
fixed block horizon \(T\).
Since \(m=T^{-1}\log(N_0+n_m)+O_T(1)\), and since
Lemma~\ref{lem:harmonic_block_geometry} makes \(N_0+n\) comparable
with \(N_0+n_m\) throughout block \(m\), while \(e_m^J\) is
exponentially smaller, the within-block estimate yields
\[
J(\te_n)
=O_{\rm a.s.}\!\left(
\{\log(N_0+n)\}^{-1/(p_J-1)}
\right),
\]
which is~\eqref{ineq:nonlinear_iteration_rate_bone}.
\end{proof}

\begin{remark}[Sharpness of the harmonic-endpoint power rate]
\label{rem:b1_power_sharpness}
The logarithmic rate in
\eqref{ineq:nonlinear_iteration_rate_bone} is intrinsic to the assumed
power decrease. Indeed, for \(u_1>0\) and \(N_0\) sufficiently large,
the deterministic scalar recursion
\[
u_{n+1}=u_n-\frac{\kappa_J}{N_0+n}u_n^{p_J},
\qquad p_J>1,
\]
remains positive and satisfies
\[
\lim_{n\to\infty}
u_n\bigl[(p_J-1)\kappa_J\log n\bigr]^{1/(p_J-1)}
=1.
\]
To see this, set \(v_n:=u_n^{-(p_J-1)}\). A Taylor expansion gives

\[
v_{n+1}-v_n
=\frac{(p_J-1)\kappa_J}{N_0+n}
+O\!\left((N_0+n)^{-2}\right),
\]
and summation yields
\(v_n=(p_J-1)\kappa_J\log n+O(1)\).
Thus a polynomial iteration-wise rate cannot be guaranteed at
\(\frakb=1\) from the power-decrease condition alone. In particular, the
rate is \(O_{\rm a.s.}(1/\log n)\) when \(p_J=2\).
\end{remark}

\subsection{Proof of Theorem~\ref{th:iter_rate}}
\begin{proof}
Since
\[
\sum_{m\ge1}
\bP\!\left((\cG_m^{\rm rate})^c\right)<\infty,
\]
the Borel--Cantelli lemma implies that, almost surely, there exists
\(m_0<\infty\) such that \(\cG_m^{\rm rate}\) holds for every
\(m\ge m_0\). Fix such a sample path and write
\[
U_m:=J\!\left(\te^{(m),\circ}(0)\right).
\]
From \eqref{ineq:localized-block-rate}, the $L_J$-Lipschitz continuity of $J$, $H_m\ge T$, the ODE
contraction assumption, and
$\te^{(m),\circ}(H_m)=\te^{(m+1),\circ}(0)$, we obtain
\[
U_{m+1}\le \rho_TU_m+L_JC_{B,T}d_m^{\rm rate},
\qquad m\ge m_0.
\]
Iteration gives
\[
U_{m(n)}\le \rho_T^{m(n)-m_0}J_{\max}
+L_JC_{B,T}\sum_{j=m_0}^{m(n)-1}
\rho_T^{m(n)-1-j}d_j^{\rm rate}.
\]
If \(n\) lies in block \(m=m(n)\), then
\[
t(n)-T_m\in[0,H_m],
\qquad
\te^{(m),\circ}\!\left(t(n)-T_m\right)=\te_n.
\]
Hence \eqref{ineq:localized-block-rate}, the Lipschitz continuity of
\(J\), and the nonincrease of \(J\) along the ODE trajectory imply
\[
\begin{aligned}
J(\te_n)
&\le
J\!\left(z^{(T_m)}\!\left(t(n)-T_m\right)\right)
 +L_JC_{B,T}d_m^{\rm rate}\\
&\le U_m+L_JC_{B,T}d_m^{\rm rate}.
\end{aligned}
\]
Substituting the preceding bound for $U_m$ proves the first claim.
For the rate assertion, suppose first that \(\frakb<1\).
Applying Lemma~\ref{lem:geometric-convolution}, with the present
value of \(r\), to the assumed bound on \(d_m^{\rm rate}\) shows
that both the geometric convolution and the initial transient are
bounded by a constant multiple of
\[
(N_0+n_{m(n)})^{-r}
\sqrt{\log(N_0+n_{m(n)})}.
\]
Moreover, Lemma~\ref{lem:block_geometry} gives
\[
N_0+n_{m(n)}
\le N_0+n
\le N_0+n_{m(n)}
+C_T(N_0+n_{m(n)})^{\frakb}
\le 2(N_0+n_{m(n)})
\]
for all sufficiently large \(m(n)\). Thus \(N_0+n_{m(n)}\) and
\(N_0+n\), as well as their logarithms, are comparable. This proves
the \(\frakb<1\) assertion.

Suppose \(\frakb=1\), and recall \(\lambda_T\) from~\eqref{def:effective_block_exponent}.
Lemma~\ref{lem:harmonic_geometric_convolution}, with
\(\ell=1/2\) and \(\rho=\rho_T\), bounds the endpoint term \(U_m\)
and its initial transient. For \(n_m\le n<n_{m+1}\), the
within-block estimate established above gives
\[
J(\te_n)\le U_m+L_JC_{B,T}d_m^{\rm rate}.
\]
Writing \(X_m:=N_0+n_m\), Lemma~\ref{lem:harmonic_block_geometry}
gives
\[
X_m\le N_0+n<(e^T+1)X_m.
\]
Thus \(N_0+n\asymp_T X_m\), and their logarithms are comparable.
Moreover, the current-block term
\(d_m^{\rm rate}=O(X_m^{-r}\sqrt{\log X_m})\) is bounded by the
corresponding endpoint order in each of the three cases below.
Consequently,
\[
J(\te_n)
=
\begin{cases}
O_{\rm a.s.}((N_0+n)^{-r}\sqrt{\log(N_0+n)}),
&\lambda_T>r,\\[1mm]
O_{\rm a.s.}((N_0+n)^{-r}\{\log(N_0+n)\}^{3/2}),
&\lambda_T=r,\\[1mm]
O_{\rm a.s.}((N_0+n)^{-\lambda_T}),
&0<\lambda_T<r,
\end{cases}
\]
which proves~\eqref{ineq:contraction_iteration_rate_bone}.
\end{proof}

\subsection{Proof of Theorem~\ref{th:fast_iter_rate}}
\begin{proof}
Let
\[
\varepsilon_n^{(x)}:=\norm{x_n-x\ust(\te_n)},
\]
and recall $q = 1-\gamma$. Lemma~\ref{lem:block_fast_tail} and the definition of $\delta_{\rm dev}^{(\fm)}(m)$ give
\[
\bP\!\left(\left[\mathcal G_x^{(m)}
\bigl(\delta_{\rm dev}^{(\fm)}(m)\bigr)\right]^c\right)
\le\frac{1}{2(m+2)^2}.
\]
Thus, by Borel--Cantelli, these events hold for all sufficiently large $m$,
almost surely. On those blocks, \eqref{eq:block_fast_endpoint} gives
\[
\varepsilon_{n_{m+1}}^{(x)}
\le \mathcal W_m(K_m)\varepsilon_{n_m}^{(x)}+C_x g_m^{(x)}\!\left(\delta_{\rm dev}^{(\fm)}(m)\right).
\]
Moreover,
\[
\sum_{r=n_m}^{n_{m+1}-1}\alpha_r
\ge \frac{\alpha_{n_m}}{\beta_{n_m}}a(n_m,n_{m+1})
\ge T(\no+n_m)^{\frakb-\fraka},
\]
so
\[
\mathcal W_m(K_m)
\le\exp\{-qT(\no+n_m)^{\frakb-\fraka}\}.
\]

Suppose first that \(\frakb<1\). The preceding factor is at most
\(1/2\) for all sufficiently large \(m\). Hence, eventually almost
surely,
\[
\varepsilon_{n_{m+1}}^{(x)}
\le\tfrac12\varepsilon_{n_m}^{(x)}+C_x g_m^{(x)}\!\left(\delta_{\rm dev}^{(\fm)}(m)\right).
\]
Lemma~\ref{lem:block_geometry} gives
\[
g_m^{(x)}\!\left(\delta_{\rm dev}^{(\fm)}(m)\right)   
\le C_T(\no+n_m)^{-r_x(\fraka,\frakb)}
\sqrt{\log(\no+n_m)}.
\]
Applying Lemma~\ref{lem:geometric-convolution} to the endpoint
recursion yields
\[
\varepsilon_{n_m}^{(x)}
\le C(\omega)(\no+n_m)^{-r_x(\fraka,\frakb)}
\sqrt{\log(\no+n_m)}
\]
for all sufficiently large \(m\). If \(n_m\le n<n_{m+1}\), then
\eqref{eq:block_fast_pointwise} and
\(\mathcal W_m(n-n_m)\le1\) give
\[
\varepsilon_n^{(x)}\le\varepsilon_{n_m}^{(x)}+C_x g_m^{(x)}\!\left(\delta_{\rm dev}^{(\fm)}(m)\right).
\]
Finally, Lemma~\ref{lem:block_geometry} gives
\(K_m\le C_T(\no+n_m)^{\frakb}\), and \(\frakb<1\) implies
\(\no+n\le2(\no+n_m)\) for all sufficiently large \(m\). The
endpoint rate therefore holds at every iterate \(n\).

Suppose now that \(\frakb=1\), and put
\[
\varphi_m^x:=X_m^{-r_x(\fraka,1)}\sqrt{\log X_m},
\qquad X_m:=N_0+n_m.
\]
Lemma~\ref{lem:harmonic_block_geometry} gives
\( g_m^{(x)}\!\left(\delta_{\rm dev}^{(\fm)}(m)\right) \le C_T\varphi_m^x\). Choose any
\(\eta\in(0,e^{-r_x(\fraka,1)T})\). Since
\[
\mathcal W_m(K_m)
\le\exp\{-qTX_m^{1-\fraka}\},
\]
eventually \(\mathcal W_m(K_m)\le\eta\). The endpoint recursion and
Lemma~\ref{lem:harmonic_geometric_convolution}, with \(\rho=\eta\)
and \(\ell=1/2\), therefore give
\[
\varepsilon_{n_m}^{(x)}
=O_{\rm a.s.}(X_m^{-r_x(\fraka,1)}\sqrt{\log X_m}).
\]
The pointwise block estimate and
Lemma~\ref{lem:harmonic_block_geometry} transfer this bound to every
\(n\), proving the harmonic-endpoint assertion.
\end{proof}

\subsection{Proof of Corollary~\ref{coro:optimized_rate}}
\begin{proof}
Suppose first that \(\frakb<1\). We identify the rate of
\(d_m^{\rm rate}\). By
Lemma~\ref{lem:block_geometry}, for all sufficiently large \(m\),
\[
D_m\le(T+1)\beta_{n_m},
\qquad
K_m\le C_T(\no+n_m)^{\frakb},
\]
and
\[
L_m^x+L_m^\te
\le
C_T\log(\no+n_m).
\]
Consequently,
\[
\delta_{\rm dev}^{(\fm)}(m)
\le
C_T(\no+n_m)^{-\fraka/2}
\sqrt{\log(\no+n_m)},
\]
and
\[
\delta_{\rm dev}^{(\sm)}(m)
\le
C_T(\no+n_m)^{-\frakb/2}
\sqrt{\log(\no+n_m)}.
\]
Moreover,
\[
\alpha_{n_m}
+
\frac{\beta_{n_m}}{\alpha_{n_m}}
\le
2(\no+n_m)^{-p_{\rm det}}.
\]
Since
\[
r_x(\fraka,\frakb)
\le
p_{\rm det},
\]
all deterministic terms are covered by the same envelope. Hence
\[
d_m^{\rm rate}
\le
C_T(\no+n_m)^{-r_x(\fraka,\frakb)}
\sqrt{\log(\no+n_m)}.
\]
Theorem~\ref{th:iter_rate} then gives the same iteration-wise exponent for $J(\te_n)$, and Theorem~\ref{th:fast_iter_rate} gives the fast tracking bound.

It remains to optimize \(r_x(\fraka,\frakb)\).
For fixed \(\frakb>3/4\), its maximum is obtained by balancing
\(\fraka/2=\frakb-\fraka\), that is,
\(\fraka=2\frakb/3\), and equals \(\frakb/3\).
For \(1/2<\frakb\le3/4\), its supremum is
\(\frakb-1/2\), approached as \(\fraka\downarrow1/2\)
(and not attained under the strict constraint).
In either case, the supremum over
\(1/2<\fraka<\frakb<1\) is at most \(1/3\).
For the displayed choice
\[
\frakb=1-3\rateslack,\qquad \fraka=\frac{2\frakb}{3}=\frac23-2\rateslack,
\]
the restriction \(0<\rateslack<1/12\) gives \(\frakb>3/4\), so the balancing
branch above applies, and
\[
\frac12<\fraka<\frakb<1,
\]
and
\[
r_x(\fraka,\frakb)
=\min\left\{\frac13-\rateslack,\frac13-\rateslack\right\}
=\frac13-\rateslack.
\]
This proves the stated \(\frakb<1\) slow and fast optimized rates.

For the harmonic endpoint \(\frakb=1\), set \(X_m:=N_0+n_m\).
Lemma~\ref{lem:harmonic_block_geometry} gives
\[
D_m\le C_TX_m^{-1},
\qquad
K_m\le C_TX_m,
\qquad
L_m^x+L_m^\te\le C_T\log X_m.
\]
Consequently,
\[
d_m^{\rm rate}
\le
C_TX_m^{-r_x(\fraka,1)}\sqrt{\log X_m}.
\]
The slow assertions follow from Theorem~\ref{th:iter_rate}, and the
fast assertion follows from Theorem~\ref{th:fast_iter_rate}.
Finally, \(r_x(\fraka,1)\) is maximized by
\(\fraka/2=1-\fraka\), namely \(\fraka=2/3\), and its maximum is
\(1/3\). This proves the harmonic-endpoint claims.
\end{proof}

\begin{remark}[Comparison with recent mean-square bounds]\label{rem:rate_comparison}
For every fixed \(0<\rateslack<1/12\),
Corollary~\ref{coro:optimized_rate} gives an almost-sure distance
envelope of order
\(n^{-1/3+\rateslack}\sqrt{\log n}\), whose square has trajectory-wise scale
\(n^{-2/3+2\rateslack}\log n\). This observation does
\emph{not} imply a mean-square bound: passing from an almost-sure
envelope with a random prefactor to an expectation requires an
additional integrability argument. At the level of exponents, the
trajectory-wise squared scale agrees, up to logarithmic and arbitrarily
small polynomial losses, with the $O(n^{-2/3})$ mean-square rate for
contractive two-time-scale SA with arbitrary-norm contractions and
Markov noise in~\citet{chandakhaquebambos2025}. Other
$O(k^{-a})$ or $O(k^{-1})$ mean-square results have been obtained by
\citet{chandak2026Ok,doan2025fast} under
different assumptions or for modified recursions. Because their
algorithms, assumptions, and modes of convergence differ, those results
are not directly comparable with the almost-sure envelope above. Our result instead provides
blockwise high-probability tracking for projected-ODE slow dynamics
without assuming a contractive discrete slow update map. The comparison
in this remark concerns the linear-decrease branch, which uses uniform
one-block contraction of a Lipschitz Lyapunov function. The nonlinear
decrease result gives different exponents and is not being compared here
with the cited contractive mean-square bounds.
\end{remark}

\subsection{Proof of Corollary~\ref{coro:optimized_power_rate}}
\begin{proof}
For the case \(\frakb<1\), the proof of
Corollary~\ref{coro:optimized_rate} shows that
\[
d_m^{\rm rate}
\le
C_T(N_0+n_m)^{-r_x(\fraka,\frakb)}
\sqrt{\log(N_0+n_m)}.
\]
Theorem~\ref{th:power_decrease_rate} therefore gives the slow
polynomial exponent
\[
s(\fraka,\frakb)
:=
\min\left\{
\frac{1-\frakb}{p_J-1},
\frac{r_x(\fraka,\frakb)}{p_J}
\right\}.
\]
Fix \(\sigma>0\). If \(s(\fraka,\frakb)\ge \sigma\), then necessarily
\[
\fraka\ge 2p_J \sigma,
\qquad
\frakb-\fraka\ge p_J \sigma,
\qquad
\frakb\le1-(p_J-1)\sigma.
\]
Hence
\[
3p_J \sigma \le\frakb\le1-(p_J-1) \sigma,
\]
which implies
\[
\sigma \le\frac1{4p_J-1}.
\]
Equality is attained at
\[
\fraka=\frac{2p_J}{4p_J-1},
\qquad
\frakb=\frac{3p_J}{4p_J-1}.
\]
For this choice,
\[
r_x(\fraka,\frakb)=\frac{p_J}{4p_J-1}.
\]
The slow bound follows from
Theorem~\ref{th:power_decrease_rate}, while the fast bound follows
from Theorem~\ref{th:fast_iter_rate}.

At the harmonic endpoint \(\frakb=1\), take \(\fraka=2/3\). The
harmonic branch of Theorem~\ref{th:power_decrease_rate} gives
\[
J(\te_n)
=O_{\rm a.s.}\!\left(
\{\log(N_0+n)\}^{-1/(p_J-1)}
\right),
\]
while Theorem~\ref{th:fast_iter_rate} gives
\[
\norm{x_n-x\ust(\te_n)}
=O_{\rm a.s.}\!\left(
(N_0+n)^{-1/3}\sqrt{\log(N_0+n)}
\right).
\]
\end{proof}

\subsection{Proof of Proposition~\ref{prop:face_rates}}
\label{sec:ec_face_rates}

\begin{proof}
All bounds below hold for the iterates generated by the selected
schedule; in particular, the deterministic stability assumption is
retained for those iterates. We first establish convergence for both
schedules. This step is needed because the second schedule is outside
Condition~\ref{cond:hyper_1}.

\emph{A comparison on fixed fast-time intervals.}
For either schedule in~\eqref{eq:face_steps}, the sequences \(\alpha_n\),
\(\beta_n\), and \(\beta_n/\alpha_n\) are nonincreasing, with
\[
 0<\beta_n\le\alpha_n\le1,\qquad
 \frac{\beta_n}{\alpha_n}\longrightarrow0,
 \qquad t_n\alpha_n\longrightarrow\infty.
\]
Both step-size sums diverge.
For the second schedule these facts follow directly from
\[
 \frac{\beta_n}{\alpha_n}=\frac{B}{A\log t_n},
 \qquad
 \sum_{n<k}\alpha_n=\frac A2(\log t_k)^2+O(1),
 \qquad
 \sum_{n<k}\beta_n=B\log t_k+O(1).
\]
The shift conditions ensure that \(\alpha_n\), \(\beta_n\),
and \(\beta_n/\alpha_n\) are nonincreasing and that
\(0<\beta_n\le\alpha_n\le1\) for every \(n\ge1\).

Fix a fast-time block horizon \(T_x>0\), and define
deterministic block endpoints
\[
 k_0=1,\qquad
 k_{j+1}=\min\left\{k>k_j:
                      \sum_{n=k_j}^{k-1}\alpha_n\ge T_x\right\}.
\]
For either schedule,
\begin{equation}
 k_{j+1}-k_j=O(\alpha_{k_j}^{-1}),\qquad
 \frac{t_{k_{j+1}}}{t_{k_j}}\longrightarrow1,
 \qquad
 \sum_{n=k_j}^{k_{j+1}-1}\alpha_n\in[T_x,T_x+\alpha_{k_j}].
 \label{eq:face_fast_block_geometry}
\end{equation}
Indeed, \(t\alpha(t)\to\infty\), and each displayed step-size
function has logarithmic derivative of magnitude at most \(1/t\)
for \(t\ge t_1\). On intervals of length \(O(1/\alpha_{k_j})\),
its ratio to its initial value therefore tends to one. This proves
\eqref{eq:face_fast_block_geometry}, as well as uniform comparability
of the steps within each such interval.

Put
\[
 \xi_{n+1}=h(Y_n,x_n,\theta_n)
             -h^{(\mathrm{av.})}(x_n;\theta_n)+M_{n+1}.
\]
The Poisson equation~\eqref{def:poisson_soln_h} decomposes this term
as a bounded martingale difference plus
\(u_n(Y_n)-u_n(Y_{n+1})\), where
\(u_n=u_h^{(x_n,\theta_n)}\). The increment and Poisson-sensitivity
bounds of Lemma~\ref{lemma:poisson_sensitivity_1} use only the
standing assumptions and \(\beta_n\le\alpha_n\); they therefore hold
for both schedules. Summation by parts gives, uniformly for
\(k_j<r\le k_{j+1}\),
\begin{align*}
 &\left\|\sum_{n=k_j}^{r-1}
       \alpha_n\{u_n(Y_n)-u_n(Y_{n+1})\}\right\|\\
 &\quad\le C\left\{
       \alpha_{k_j}
       +\sum_{n=k_j+1}^{r-1}|\alpha_n-\alpha_{n-1}|
       +\sum_{n=k_j+1}^{r-1}\alpha_{n-1}
          (\norm{x_n-x_{n-1}}+\norm{\theta_n-\theta_{n-1}})
              \right\}
 \le C_{T_x}\alpha_{k_j}.
\end{align*}
Here and below deterministic constants may depend on \(T_x\).
The martingale part has bounded increments of order \(\alpha_n\),
and
\[
 \sum_{n=k_j}^{k_{j+1}-1}\alpha_n^2
 \le(T_x+1)\alpha_{k_j}.
\]
The maximal martingale inequality, applied coordinatewise, consequently
bounds its maximum partial sum above \(u\) with probability at most
\(C\exp\{-cu^2/\alpha_{k_j}\}\). Choosing a sufficiently large
multiple of \(\sqrt{\alpha_{k_j}\log t_{k_j}}\), the failure
probabilities are at most \(Ct_{k_j}^{-3}\), and are summable since
\(k_j\ge j+1\). Thus, almost surely,
\begin{equation}
 \max_{k_j\le r\le k_{j+1}}
 \left\|\sum_{n=k_j}^{r-1}\alpha_n\xi_{n+1}\right\|
 =O\!\left(\sqrt{\alpha_{k_j}\log t_{k_j}}\right).
 \label{eq:face_fast_noise_partial_sums}
\end{equation}

Restart an averaged recursion at each \(k_j\):
\[
 \widehat x_{k_j}^{(j)}=x_{k_j},\qquad
 \widehat x_{n+1}^{(j)}
 =(1-\alpha_n)\widehat x_n^{(j)}
   +\alpha_n h^{(\mathrm{av.})}(\widehat x_n^{(j)};\theta_n).
\]
The parameter sequence in this recursion is the actual \(\theta_n\).
The averaged drift \(h^{(\mathrm{av.})}(x;\theta)-x\) is
\((1+\gamma)\)-Lipschitz in \(x\), uniformly in \(\theta\).
Subtracting the two recursions, taking partial sums, and applying
discrete Gronwall over a fast-time interval of length at most
\(T_x+1\) to~\eqref{eq:face_fast_noise_partial_sums} yields
\begin{equation}
 \max_{k_j\le n\le k_{j+1}}
 \norm{x_n-\widehat x_n^{(j)}}
 =O_{\rm a.s.}\!\left(\sqrt{\alpha_{k_j}\log t_{k_j}}\right).
 \label{eq:face_fast_aux_local}
\end{equation}
This argument estimates martingale partial sums before making a
pathwise drift comparison; it introduces no martingale weights
containing future slow iterates.

\emph{Initial convergence.}
Let \(e_j : = \norm{x_{k_j}-x^\star(\theta_{k_j})}\), and recall \(q=1-\gamma>0\).
One step of the averaged recursion contracts the distance to the
current fast target by \(1-q\alpha_n\), while the target moves by
at most \(L_{fp}A_\theta\beta_n\).
Using~\eqref{eq:face_fast_aux_local} and monotonicity of
\(\beta_n/\alpha_n\), we obtain
\begin{equation}
 e_{j+1}\le e^{-qT_x}e_j+
 O_{\rm a.s.}\!\left(
    \sqrt{\alpha_{k_j}\log t_{k_j}}+
    \frac{\beta_{k_j}}{\alpha_{k_j}}\right).
 \label{eq:face_initial_fast_recursion}
\end{equation}
The deterministic envelope on the right tends to zero and has
consecutive ratio tending to one by~\eqref{eq:face_fast_block_geometry}.
Geometric comparison gives the same bound for \(e_j\), and the
within-interval comparison gives
\begin{equation}
 \norm{x_n-x^\star(\theta_n)}
 =O_{\rm a.s.}\!\left(
    \sqrt{\alpha_n\log t_n}+\frac{\beta_n}{\alpha_n}\right)
 \longrightarrow0.
 \label{eq:face_initial_fast_tracking}
\end{equation}
For clarity, the geometric comparison follows by dividing
\eqref{eq:face_initial_fast_recursion} by its envelope at \(k_{j+1}\).
The coefficient of the normalized previous error tends to
\(e^{-qT_x}<1\), so the normalized errors are eventually bounded.

Now use the deterministic slow-time blocks \(n_m\) with horizon
\(T\) from Assumption~\ref{assum:one_block_decrease}. Their lengths
in slow time lie in \([T,T+\beta_{n_m}]\). For either schedule,
\(t_{n_{m+1}}/t_{n_m}\) is bounded. In the second schedule it tends
to \(e^{T/B}\). Apply the slow Poisson decomposition
(Lemma~\ref{lemma:metivier_mtg}) to
\(f(y,x^\star(\theta),\theta)-\favg(\theta;\theta)\).
The same partial-sum argument,
now with weights \(\beta_n\), gives almost-sure martingale errors
of order \(\sqrt{\beta_{n_m}\log t_{n_m}}\) and Poisson residuals
of order \(\beta_{n_m}\). Their failure probabilities can again be
chosen at most \(Ct_{n_m}^{-3}\).
The forcing due to the fast error is bounded by
\((T+1)L_f\sup_{n_m\le n<n_{m+1}}
                 \norm{x_n-x^\star(\theta_n)}\).
Consequently, the deterministic Skorokhod-map and Gronwall comparison
used in the proof of Theorem~\ref{th:tube} gives a uniform
slow-path error on block \(m\) tending to zero. This uses only the
driver estimates just obtained and the vanishing largest step in the
block. In particular, it applies to the second schedule without
invoking any pure-power rate formula.

Lipschitz continuity of \(J\) and its one-block decrease therefore
give
\[
 J(\theta_{n_{m+1}})
 \le J(\theta_{n_m})-\mathcal D_J(J(\theta_{n_m}))+o(1)
 \quad\text{almost surely},
\]
with a within-block increase bounded by \(o(1)\).
Since \(\mathcal D_J\) is continuous and strictly positive away
from zero, the scalar convergence argument in
Theorem~\ref{th:iter_rate_nonlinear} gives \(J(\theta_n)\to0\).
Compactness of \(\Theta\) and \(J^{-1}(\{0\})=\eq\) then imply
\(\operatorname{dist}(\theta_n,\eq)\to0\) almost surely. These consistency statements
are the only inputs from this step needed for the sharper rates.

\emph{Distance from a required endpoint.}
Fix \(i\in I\), choose
\(0<\eta_i\le\inf_{\theta\in \eq}s_i\favg_i(\theta;\theta)\), and write
\[
 y_n:=s_i(v_i-\theta_{n,i}),\qquad
 X_n:=s_i\{f_i(Y_n,x_n,\theta_n)+M'_{n+1,i}\}.
\]
Thus \(0\le y_n\le u_i-\ell_i\), and coordinatewise box projection gives
the exact recursion
\begin{equation}
 y_{n+1}=\min\{u_i-\ell_i,[y_n-\beta_nX_n]_+\}.
 \label{eq:face_coordinate_recursion}
\end{equation}
The standing boundedness assumptions give a deterministic constant
\(D<\infty\) with \(|X_n|\le D\) for every \(n\).

We next obtain a positive conditional mean over a fixed number of
steps from the averaged condition~\eqref{eq:face_strict_outward}.
Let \(\Pi_{\eq}\) denote Euclidean projection onto \(\eq\), and define
\[
 A_i:=\sup_{y\in\mathcal Y,\,\zeta\in \eq}
          |f_i(y,x^\star(\zeta),\zeta)|,
 \qquad C_i:=L_\mu A_i,
\]
\[
 g_i(y,\theta)
 :=s_i f_i(y,x^\star(\Pi_{\eq}\theta),\Pi_{\eq}\theta)
      +C_i \operatorname{dist}(\theta,\eq).
\]
These functions are bounded and Lipschitz: the maps \(x^\star\),
\(\Pi_{\eq}\), and the distance function \(\operatorname{dist}(\cdot,\eq)\)
are Lipschitz, and the arguments of \(f\)
remain in a fixed bounded set. The stationary-distribution bound in
Lemma~\ref{lemma:lipschitz_stat_dist} gives
\begin{align*}
 \bar g_i(\theta)
 &:=\sum_y\mu^{(\theta)}(y)g_i(y,\theta)\\
 &\ge s_i\favg_i(\Pi_{\eq}\theta;\Pi_{\eq}\theta)
       -L_\mu A_i \operatorname{dist}(\theta,\eq)+C_i \operatorname{dist}(\theta,\eq)
 \ge\eta_i,
 \qquad\theta\in\Theta.
\end{align*}
The added term makes this lower bound valid throughout the box.
It vanishes as the iterates approach the face.

Put
\[
 X_n^0:=g_i(Y_n,\theta_n)+s_iM'_{n+1,i}.
\]
For the bounded-set Lipschitz constant \(L_f\) from the standing
assumptions,
\begin{equation}
 |X_n-X_n^0|
 \le L_f\norm{x_n-x^\star(\theta_n)}
    +\{L_f(1+L_{fp})+C_i\}\operatorname{dist}(\theta_n,\eq)
 \longrightarrow0
 \quad\text{almost surely}.
 \label{eq:face_increment_comparison}
\end{equation}
The variables \(X_n^0\) are uniformly bounded.

\emph{A conditional exponential estimate.}
The same finite-state Poisson argument as in
Lemma~\ref{lemma:poisson_sensitivity_1}, applied to the bounded
Lipschitz scalar observable \(g_i\), gives solutions
\[
 u_i^{(\theta)}-P^{(\theta)}u_i^{(\theta)}
   =g_i(\cdot,\theta)-\bar g_i(\theta)
\]
with a uniform bound \(U_i\) and a uniform Lipschitz constant
\(L_{u,i}\) in \(\theta\). The standing bounds and nonexpansiveness
of projection give \(\norm{\theta_{n+1}-\theta_n}\le A_\theta\beta_n\).
Telescoping the Poisson equation over \(k,\ldots,k+L-1\), and taking
conditional expectations, therefore gives
\begin{equation}
 \mathbb E\left[\sum_{j=k}^{k+L-1}X_j^0\,\middle|\,\mathcal F_k\right]
 \ge L\eta_i-2U_i-L_{u,i}A_\theta L\beta_k.
 \label{eq:face_block_mean}
\end{equation}
Indeed, the Poisson martingale and the original martingale increments
have zero conditional mean. The two endpoint terms contribute at
most \(2U_i\), and changing the Poisson parameter between consecutive
indices contributes at most
\(L_{u,i}\sum_{j=k}^{k+L-2}\norm{\theta_{j+1}-\theta_j}\).

Choose a positive integer \(L\) such that \(2U_i\le\eta_iL/4\).
For all sufficiently large deterministic \(k\), the last term in
\eqref{eq:face_block_mean} is at most \(\eta_iL/4\), and hence
\[
 \mathbb E\left[\sum_{j=k}^{k+L-1}X_j^0\,\middle|\,\mathcal F_k\right]
 \ge\eta_iL/2.
\]
Set \(W_n:=X_n^0-\eta_i/4\), and let \(H_i\) be a deterministic bound
on \(|W_n|\). For sufficiently small \(\lambda>0\), Taylor's inequality
gives, uniformly over all such \(k\),
\begin{align*}
 &\mathbb E\left[
   \exp\left\{-\lambda\sum_{j=k}^{k+L-1}W_j\right\}
                  \,\middle|\,\mathcal F_k\right]\\
 &\qquad\le 1-\lambda\eta_iL/4
       +\tfrac12\lambda^2L^2H_i^2e^{\lambda LH_i}
 \le\rho<1.
\end{align*}
Iterating over consecutive full blocks and bounding the last
incomplete block by \((L-1)H_i\) yields constants \(C,c>0\) such that
\begin{equation}
 \mathbb E\left[
   \exp\left\{-\lambda\sum_{j=k}^{n-1}W_j\right\}
                   \,\middle|\,\mathcal F_k\right]
 \le Ce^{-c(n-k)},\qquad n>k,
 \label{eq:face_interval_exponential}
\end{equation}
whenever \(k\) exceeds a fixed deterministic index.

\emph{Rescaling the endpoint distance.}
Let \(t_n=N_0+n\), \(z_n=y_n/\beta_n\), and
\(r_n=\beta_n/\beta_{n+1}\).
Both schedules satisfy \(r_n-1\le1/t_n\).
Equation~\eqref{eq:face_coordinate_recursion} implies
\[
 z_{n+1}\le[r_nz_n-r_nX_n]_+
           \le[z_n-X_n+e_n]_+,
\]
where
\[
 0\le e_n:=(r_n-1)(z_n+D)
 \le \frac{y_n}{t_n\beta_n}+D/t_n\longrightarrow0
 \quad\text{almost surely}.
\]
Here \(y_n\to0\) and \(1/(t_n\beta_n)\) is bounded: it equals
\(t_n^{\frakb-1}\) for the power schedule and \(1/B\) for the second schedule.
Together with~\eqref{eq:face_increment_comparison}, this proves that,
almost surely, for every sufficiently large \(n\),
\begin{equation}
 z_{n+1}\le[z_n-W_n]_+.
 \label{eq:face_queue_domination}
\end{equation}

To handle this random starting index, fix a deterministic \(N\)
beyond the starting index in~\eqref{eq:face_interval_exponential},
and define
\[
 Q_N=z_N,\qquad Q_{n+1}=[Q_n-W_n]_+,\quad n\ge N.
\]
On the event that~\eqref{eq:face_queue_domination} holds for every
\(n\ge N\), monotonicity gives \(z_n\le Q_n\). The union of these
events over deterministic \(N\) has probability one. Direct iteration
of the last recursion gives
\[
 Q_n=\max\left\{
 z_N-\sum_{j=N}^{n-1}W_j,
 \max_{N<k\le n}\left(-\sum_{j=k}^{n-1}W_j\right)
 \right\}.
\]
Since \(z_N\le(u_i-\ell_i)/\beta_N\) deterministically,
Markov's inequality,~\eqref{eq:face_interval_exponential}, and a union
bound over the suffixes give, for every \(u>0\),
\begin{align*}
 \mathbb P(Q_n>u)
 &\le Ce^{-\lambda u+\lambda(u_i-\ell_i)/\beta_N}e^{-c(n-N)}
       +Ce^{-\lambda u}\sum_{m\ge0}e^{-cm}\\
 &\le C_Ne^{-\lambda u}.
\end{align*}
Take \(u=K\log(N_0+n)\), with \(\lambda K>2\).
Borel--Cantelli yields \(Q_n=O_{\rm a.s.}(\log n)\) for every fixed
\(N\). Intersecting these probability-one events over the countable
set of \(N\), and using eventual domination, gives
\[
 y_n=O_{\rm a.s.}(\beta_n\log n).
\]
There are finitely many constrained coordinates, and
\(\operatorname{dist}(\theta_n,\eq)^2=\sum_{i\in I}y_{n,i}^2\).
This proves the distance bound. Lipschitz continuity of \(J\) and
\(J=0\) on \(\eq\) imply \(J(\theta_n)\le L_J\operatorname{dist}(\theta_n,\eq)\), proving
\eqref{eq:face_slow_rate}. If \(I\) is empty, this conclusion is
immediate because \(\eq=\Theta\).

\emph{The fast rate under a common target.}
Suppose \(x^\star(\theta)=x^\dagger\) on \(\eq\).
The fixed-point identity, contraction, and Lipschitz continuity of
\(x^\star\) give
\begin{equation}
 \norm{h^{(\mathrm{av.})}(x^\dagger;\theta)-x^\dagger}
 \le(1+\gamma)\norm{x^\star(\theta)-x^\dagger}
 \le(1+\gamma)L_{fp}\operatorname{dist}(\theta,\eq).
 \label{eq:face_fixed_target_residual}
\end{equation}
Return to the fixed fast-time intervals~\eqref{eq:face_fast_block_geometry}
and their averaged recursions. One averaged step contracts the distance
to \(x^\dagger\) by \(1-q\alpha_n\) and adds at most
\(\alpha_n(1+\gamma)L_{fp}\operatorname{dist}(\theta_n,\eq)\).
The sum of the resulting contraction weights is at most \(1/q\).
Combining this with~\eqref{eq:face_fast_aux_local} and the face bound
already proved yields
\begin{align*}
 \norm{x_{k_{j+1}}-x^\dagger}
 &\le e^{-qT_x}\norm{x_{k_j}-x^\dagger}\\
 &\quad+O_{\rm a.s.}\!\left(
       \sqrt{\alpha_{k_j}\log t_{k_j}}
       +\beta_{k_j}\log t_{k_j}\right).
\end{align*}
The deterministic envelope again has consecutive ratio tending to one.
Geometric comparison and the within-interval bound therefore give
\[
 \norm{x_n-x^\dagger}
 =O_{\rm a.s.}\!\left(
       \sqrt{\alpha_n\log t_n}+\beta_n\log t_n\right).
\]
For both schedules, \(\beta_n\log t_n=o(\sqrt{\alpha_n\log t_n})\).
Moreover,
\[
 \norm{x_n-x^\star(\theta_n)}
 \le\norm{x_n-x^\dagger}+L_{fp}\operatorname{dist}(\theta_n,\eq).
\]
This proves~\eqref{eq:face_fast_rate}. The two stated polynomial
exponents follow by substituting the corresponding schedules.
No lower bound on \(A\) relative to \(q\) is needed: each fixed
fast-time interval has a fixed contraction factor, and its forcing
envelope has consecutive ratio tending to one.
\end{proof}
\subsection{Proof of Proposition~\ref{prop:one_time_rate}}
\begin{proof}
Repeat the proof of Theorem~\ref{th:tube} with the fast component
absent. The term corresponding to \(F_m^x\) is absent.

For \(\vartheta,\te'\in\Te\), define the two-parameter
one-time-scale averaged field
\[
g^{(\mathrm{av.})}(\vartheta;\te')
:=
\sum_{y\in\cY}\mu^{(\te')}(y)g(y,\vartheta).
\]
Thus
\[
g^{(\mathrm{av.})}(\te;\te)=\bar g(\te).
\]
For \(\tau\ge0\), let \(z_g^{(\tau)}(\cdot)\) solve
\[
\begin{aligned}
\dot z_g^{(\tau)}(s)
&=
\Pi_{\Te}\!\left(
z_g^{(\tau)}(s),
g^{(\mathrm{av.})}
\bigl(z_g^{(\tau)}(s);z_g^{(\tau)}(s)\bigr)
\right)\\
&=
\Pi_{\Te}\!\left(
z_g^{(\tau)}(s),
\bar g\bigl(z_g^{(\tau)}(s)\bigr)
\right),
\qquad
z_g^{(\tau)}(0)=\te^\circ(\tau).
\end{aligned}
\]
Equivalently,
\[
z_g^{(\tau)}(s)
=
\mathsf S_s^{(1)}\bigl(\te^\circ(\tau)\bigr).
\]

Let \(u_g^{(\te)}\) be the hitting-time-normalized Poisson solution of
\[
u_g^{(\te)}(y)
-\sum_{j\in\cY}p^{(\te)}(y,j)u_g^{(\te)}(j)
=g(y,\te)-\bar g(\te),
\qquad y\in\cY,
\]
and define
\[
M^g_{n+1}
:=
u_g^{(\te_n)}(Y_{n+1})
-\sum_{j\in\cY}p^{(\te_n)}(Y_n,j)u_g^{(\te_n)}(j).
\]
Then \(\{M^g_{n+1}\}\) is a martingale-difference sequence and
\[
g(Y_n,\te_n)-\bar g(\te_n)
=
u_g^{(\te_n)}(Y_n)-u_g^{(\te_n)}(Y_{n+1})+M^g_{n+1}.
\]
The one-time-scale analogues of the standing bounded-noise and uniform
hitting-time assumptions, together with compactness of \(\Te\), give
constants \(C_N,C_P^{(1)},C_L^{(1)}<\infty\) such that
\[
\norm{N_{n+1}}\le C_N
\quad\text{almost surely for every }n\ge1,
\qquad
C_P^{(1)}:=\sup_{\te\in\Te,\,y\in\cY}
\norm{u_g^{(\te)}(y)}<\infty,
\]
and
\[
\sup_{y\in\cY}
\norm{u_g^{(\te)}(y)-u_g^{(\te')}(y)}
\le C_L^{(1)}\norm{\te-\te'},
\qquad \te,\te'\in\Te.
\]
Let
\[
A_g
:=
\sup_{y\in\cY,\,\te\in\Te}\norm{g(y,\te)}
+C_N
<\infty.
\]
Because \(\te_r\in\Te\) and Euclidean projection is
nonexpansive,
\[
\norm{\te_{r+1}-\te_r}
\le A_g\beta_r
\quad\text{almost surely},
\qquad r\ge1.
\]
For \(m\ge1\) and \(1\le k\le K_m\), define
\[
\mathcal R_{m,k}^{(1)}
:=
\sum_{\ell=0}^{k-1}
\beta_{n_m+\ell}
\left[
u_g^{(\te_{n_m+\ell})}(Y_{n_m+\ell})
-u_g^{(\te_{n_m+\ell})}(Y_{n_m+\ell+1})
\right].
\]
The summation-by-parts argument used in the proof of
Lemma~\ref{lemma:adhoc_1} gives
\[
\max_{1\le k\le K_m}
\norm{\mathcal R_{m,k}^{(1)}}
\le R_m^{(1)},
\]
where
\[
R_m^{(1)}
:=
3C_P^{(1)}\beta_{n_m}
+C_L^{(1)}A_gD_m
=O_T(\beta_{n_m}).
\]
Hence, with
\[
C_{\mathrm{mg}}^{(\theta,1)} :=C_N+2C_P^{(1)},
\]
we have, for every \(n\ge1\),
\[
\norm{N_{n+1}+M^g_{n+1}}
\le C_{\mathrm{mg}}^{(\theta,1)}
\quad\text{almost surely}.
\]
For \(m\ge1\) and \(\delta>0\), define the one-time-scale good event
\[
\cG_{\rm one}^{(m)}(\delta)
:=
\left\{
\max_{1\le k\le K_m}
\norm{
\sum_{\ell=0}^{k-1}
\beta_{n_m+\ell}
\bigl(N_{n_m+\ell+1}+M^g_{n_m+\ell+1}\bigr)
}
<\delta
\right\}.
\]
Coordinatewise Azuma--Hoeffding and a union bound give
\[
\bP\!\left(
\bigl[\cG_{\rm one}^{(m)}(\delta)\bigr]^c
\right)
\le
2d_{\te}\sum_{k=1}^{K_m}
\exp\!\left\{
-\frac{\delta^2}
{2\kappa_{d_{\te}}^2(C_{\mathrm{mg}}^{(\theta,1)})^2
\{b(n_m)-b(n_m+k)\}}
\right\}.
\]
Set
\[
\delta_{\rm dev}^{(1)}(m)
:=
\kappa_{d_{\te}}C_{\mathrm{mg}}^{(\theta,1)}
\sqrt{2D_mL_m^{\te}}.
\]
Since \(b(n_m)-b(n_m+k)\le D_m\), the definition of \(L_m^{\te}\)
implies
\[
\bP\!\left(
\bigl[\cG_{\rm one}^{(m)}
(\delta_{\rm dev}^{(1)}(m))\bigr]^c
\right)
\le \frac{1}{2(m+2)^2}.
\]
Thus the failure probabilities are summable. These are exactly the
substitutions
\[
f(y,x^\star(\te),\te)\mapsto g(y,\te),
\qquad
\favg(\te;\te')\mapsto g^{(\mathrm{av.})}(\te;\te'),
\qquad
M'_{n+1}\mapsto N_{n+1},
\qquad
M'''_{n+1}\mapsto M^g_{n+1}
\]
in~\eqref{def:poisson_f}--\eqref{def:G_slow_martingale}, with every fast
term omitted. The resulting version of Lemma~\ref{lemma:mc_diarmid_mtilde}
and the comparison proof of Theorem~\ref{th:tube} show that there exist
\(C_T,C_T'<\infty\) such that, on
\(\cG_{\rm one}^{(m)}(\delta_{\rm dev}^{(1)}(m))\),
\[
\begin{aligned}
\sup_{0\le r\le H_m}
\norm{
\te^{(m),\circ}(r)-z_g^{(T_m)}(r)
}
&\le C_T\left\{
\delta_{\rm dev}^{(1)}(m)
+R_m^{(1)}
+
\beta_{n_m}
\right\}\\
&\le C_T'\left\{
\delta_{\rm dev}^{(1)}(m)
+
\beta_{n_m}
\right\}.
\end{aligned}
\]
As before,
\[
D_m=b(n_m)-b(n_{m+1})
\le (T+1)\beta_{n_m},
\]
and Lemmas~\ref{lem:block_geometry} and
\ref{lem:harmonic_block_geometry}, in their respective cases, show
that the logarithmic union-bound factor is
\(O(\log(N_0+n_m))\). Therefore,
\[
\delta_{\rm dev}^{(1)}(m)
\le
C_T(N_0+n_m)^{-\frakb/2}
\sqrt{\log(N_0+n_m)}.
\]
Since \(\beta_{n_m}=(N_0+n_m)^{-\frakb}\) is of smaller order, the
Borel--Cantelli lemma yields, eventually almost surely,
\[
\sup_{0\le r\le H_m}
\norm{
\te^{(m),\circ}(r)-z_g^{(T_m)}(r)
}
\le
C_T(N_0+n_m)^{-\frakb/2}
\sqrt{\log(N_0+n_m)}.
\]
If \(\frakb<1\), the block-contraction argument in the proof of
Theorem~\ref{th:iter_rate}, together with
Lemma~\ref{lem:geometric-convolution}, gives the claimed
iteration-wise rate. If \(\frakb=1\), applying
Lemma~\ref{lem:harmonic_geometric_convolution} with
\(r=\ell=1/2\) and \(\rho=\rho_T\), and then using the within-block
comparison in Lemma~\ref{lem:harmonic_block_geometry}, gives the three
cases stated in Proposition~\ref{prop:one_time_rate}.
\end{proof}

\subsection{Proof of Proposition~\ref{prop:one_time_power_rate}}
\begin{proof}
The proof of Proposition~\ref{prop:one_time_rate} gives, eventually
almost surely,
\[
\sup_{0\le r\le H_m}
\norm{\te^{(m),\circ}(r)-z_g^{(T_m)}(r)}
\le
C_T(N_0+n_m)^{-\frakb/2}
\sqrt{\log(N_0+n_m)}.
\]
Set \(U_m:=J(\te_{n_m})\). Repeating the endpoint and within-block
argument in the proof of Theorem~\ref{th:iter_rate_nonlinear} gives
qualitative convergence and, under the power-decrease condition, the
eventual recursion
\[
U_{m+1}
\le
U_m-\kappa_JU_m^{p_J}
+
C_T(N_0+n_m)^{-\frakb/2}
\sqrt{\log(N_0+n_m)}.
\]
Suppose first that \(\frakb<1\). By~\eqref{eq:block-index-growth},
the perturbation on the right-hand side is bounded by
\[
C(m+2)^{-\frac{\frakb}{2(1-\frakb)}}
\{\log(m+2)\}^{1/2}.
\]
Applying Lemma~\ref{lem:nonlinear_block_recursion} with
\(\zeta=\frakb/\{2(1-\frakb)\}\) and \(\ell=1/2\), and then converting
block number back to iteration number gives
\[
J(\te_n)
=
O_{\rm a.s.}\!\left(
(N_0+n)^{-\frac{1-\frakb}{p_J-1}}
+
(N_0+n)^{-\frac{\frakb}{2p_J}}
\{\log(N_0+n)\}^{\frac1{2p_J}}
\right).
\]

To optimize the polynomial exponent, consider
\[
\min\left\{
\frac{1-\frakb}{p_J-1},
\frac{\frakb}{2p_J}
\right\}.
\]
The two terms are equal when
\[
\frakb=\frac{2p_J}{3p_J-1},
\]
and their common value is \(1/(3p_J-1)\), proving the final claim.

Suppose now that \(\frakb=1\). By
Lemma~\ref{lem:harmonic_block_geometry}, the perturbation is
\[
O\!\left(e^{-Tm/2}(m+1)^{1/2}\right),
\]
and hence is \(O((m+2)^{-Q})\) for every \(Q>0\). 

Taking \(Q>p_J/(p_J-1)\) and applying
Lemma~\ref{lem:nonlinear_block_recursion} with
\(\zeta=Q\) and \(\ell=0\) gives
\[
U_m=O_{\rm a.s.}\!\left((m+2)^{-1/(p_J-1)}\right).
\]
Since \(m=T^{-1}\log(N_0+n_m)+O_T(1)\), the within-block comparison
in Lemma~\ref{lem:harmonic_block_geometry} yields
\[
J(\te_n)
=O_{\rm a.s.}\!\left(
\{\log(N_0+n)\}^{-1/(p_J-1)}
\right),
\]
as claimed.
\end{proof}

\section{Proofs from Main-Paper Section~\ref{sec:applications} (Actor-Critic)}
\label{sec:ec_section8_proofs}
We begin by discussing the filtration and sampling conventions used in the
proof. For \(n\ge1\), let
\[
\cF_n^{\rm MDP}
:=
\sigma\br{
V_1,\vartheta_1,S_1,
\{(A_k,C_{k+1},S_{k+1})\}_{1\le k<n}
},
\]
the history immediately before \(A_n\) is drawn. Then
\(V_n\), \(\vartheta_n\), and \(Y_n=S_n\) are
\(\cF_n^{\rm MDP}\)-measurable. The filtration is taken before
\(A_n\) is drawn from the current policy; the bounded one-step cost and
finite-valued next state are then observed before the update. We now prove Corollary~\ref{coro:actor_critic_rate}.

\begin{proof}[Proof of Corollary~\ref{coro:actor_critic_rate}]

For \(\vartheta\in\Theta_{\rm ac}\), define
\[
 P_\vartheta(s,s')
 :=\sum_{a\in\cA}
   \pi_\vartheta(a\mid s)P(s'\mid s,a).
\]
With \(c(s,a)\) as in the finite-MDP description in the main paper, define
\[
 c_\vartheta(s)
 :=\sum_{a\in\cA}\pi_\vartheta(a\mid s)c(s,a).
\]
Take \(Y_n=S_n\), and take \(\cF_n^{\rm MDP}\) before \(A_n\) is
drawn. Conditional on this history,
\(A_n\sim\pi_{\vartheta_n}(\cdot\mid S_n)\); thereafter, for every
\(s'\in\cS\),
\[
\bE[C_{n+1}\mid\cF_n^{\rm MDP},A_n]=c(S_n,A_n),
\qquad
\bP(S_{n+1}=s'\mid\cF_n^{\rm MDP},A_n)
=P(s'\mid S_n,A_n).
\]
Thus the controlled kernel for \(Y_n\) is
\(P^{(\vartheta)}=P_\vartheta\).

For \(s\in\cS\), \(V\in\bR^{|\cS|}\), and
\(\vartheta\in\Theta_{\rm ac}\), set
\[
 \bar\delta_{\rm ac}(s,V,\vartheta)
 :=c_\vartheta(s)
   +\gamma_{\rm disc}\sum_{s'}P_\vartheta(s,s')V(s')-V(s),
\]
\[
 h_{\rm ac}(s,V,\vartheta)
 :=V+e_s\bar\delta_{\rm ac}(s,V,\vartheta),
\]
and
\[
 f_{\rm ac}(s,V,\vartheta)
 :=-\sum_{a\in\cA}\pi_\vartheta(a\mid s)e_{s,a}
 \left\{
 c(s,a)+\gamma_{\rm disc}\sum_{s'}P(s'\mid s,a)V(s')-V(s)
 \right\}.
\]
Define
\[
 M^V_{n+1}
 :=e_{S_n}\left\{
 \delta^{\rm ac}_{n+1}
 -\bar\delta_{\rm ac}(S_n,V_n,\vartheta_n)
 \right\}
\]
and
\[
 M^\vartheta_{n+1}
 :=-e_{S_n,A_n}\delta^{\rm ac}_{n+1}
   -f_{\rm ac}(S_n,V_n,\vartheta_n).
\]
The conditional cost and transition identities above imply
\[
\bE[\delta^{\rm ac}_{n+1}\mid\cF_n^{\rm MDP},A_n]
=
c(S_n,A_n)
+\gamma_{\rm disc}\sum_{s'\in\cS}
P(s'\mid S_n,A_n)V_n(s')
-V_n(S_n).
\]
Averaging this identity over
\(A_n\sim\pi_{\vartheta_n}(\cdot\mid S_n)\), using the tower property,
gives
\[
 \bE[M^V_{n+1}\mid\cF_n^{\rm MDP}]
 =\bE[M^\vartheta_{n+1}\mid\cF_n^{\rm MDP}]
 =\mathbf 0.
\]
Hence~\eqref{eq:actor_critic_app} is precisely
\eqref{def:x_update}--\eqref{def:te_update} under the identifications
\[
 x_n=V_n,\qquad \te_n=\vartheta_n,\qquad Y_n=S_n,\qquad
 h=h_{\rm ac},\qquad f=f_{\rm ac}.
\]

Let
\(D_\vartheta:=\operatorname{diag}((d_\vartheta(s))_{s\in\cS})\)
and define
\[
 T_\vartheta V:=c_\vartheta+\gamma_{\rm disc}P_\vartheta V.
\]
The averaged fast map is
\[
 h_{\rm ac}^{(\mathrm{av.})}(V;\vartheta)
 =V+D_\vartheta(T_\vartheta V-V),
\]
whose unique fixed point is \(V_\vartheta\).  For all \(V,V'\),
\begin{equation}
\label{eq:actor_critic_fast_contraction}
 \norm{
 h_{\rm ac}^{(\mathrm{av.})}(V;\vartheta)
 -h_{\rm ac}^{(\mathrm{av.})}(V';\vartheta)
 }_\infty
 \le
 \{1-(1-\gamma_{\rm disc})d_{\min}^{\rm ac}\}
 \norm{V-V'}_\infty.
\end{equation}
Thus \(h_{\rm ac}^{(\mathrm{av.})}(\cdot;\vartheta)\) is a contraction
in \(\norm{\cdot}_\infty\) with the common modulus
\(1-(1-\gamma_{\rm disc})d_{\min}^{\rm ac}<1\), uniformly over
\(\vartheta\in\Theta_{\rm ac}\).
The critic is bounded: if
\(B_V:=C_{\max}/(1-\gamma_{\rm disc})\) and
\(\norm{V_n}_\infty\le B_V\), then its updated coordinate is a convex
combination of \(V_n(S_n)\) and
\(C_{n+1}+\gamma_{\rm disc}V_n(S_{n+1})\), both in
\([-B_V,B_V]\), since \(0<\alpha_n\le1\).  Induction therefore gives
\[
 \sup_{n\ge1}\norm{V_n}_\infty\le B_V.
\]
We apply Remark~\ref{rem:fast_coordinate_norm} with
\(\norm{\cdot}_x=\norm{\cdot}_\infty\).  Thus
\(\kappa_x=\lambda_x=1\), and the fast contraction, local Lipschitz,
bounded-set, stability, and martingale bounds are interpreted in this norm; the slow
coordinate and its projection remain Euclidean.

We next verify the localized Lipschitz condition directly for the original
update maps. Since the softmax map \(\vartheta\mapsto\pi_\vartheta\) is
continuously differentiable and its Jacobian is continuous, its Jacobian is
uniformly bounded on the compact logit box. Moreover, \(c_\vartheta\) and
\(P_\vartheta\) are linear functions of \(\pi_\vartheta\). Hence, by the line-segment form of the mean-value theorem, there exist finite constants
\(L_\pi,L_c,L_P>0\) such that, for all
\(\vartheta,\vartheta'\in\Theta_{\rm ac}\),
\[
\begin{aligned}
 \max_{s\in\cS}
 \norm{\pi_\vartheta(\cdot\mid s)
       -\pi_{\vartheta'}(\cdot\mid s)}_1
 &\le L_\pi\norm{\vartheta-\vartheta'},\\
 \norm{c_\vartheta-c_{\vartheta'}}_\infty
 &\le L_c\norm{\vartheta-\vartheta'},\\
 \norm{P_\vartheta-P_{\vartheta'}}_\infty
 &\le L_P\norm{\vartheta-\vartheta'}.
\end{aligned}
\]
Recall that \(\norm{\cdot}_\infty\) denotes the induced matrix infinity
norm, equivalently the maximum absolute row-sum norm. Consequently, when
\(\norm{V}_\infty,\norm{V'}_\infty\le R\),
\[
 \norm{P_\vartheta V-P_{\vartheta'}V'}_\infty
 \le \norm{V-V'}_\infty
      +L_PR\norm{\vartheta-\vartheta'}.
\]
Substitution in the definitions of \(h_{\rm ac}\) and \(f_{\rm ac}\)
therefore gives finite constants \(L_h(R),L_f(R)\), independent of \(s\),
such that
\[
\begin{aligned}
 \norm{h_{\rm ac}(s,V,\vartheta)
       -h_{\rm ac}(s,V',\vartheta')}_\infty
 &\le L_h(R)\{\norm{V-V'}_\infty
                   +\norm{\vartheta-\vartheta'}\},\\
 \norm{f_{\rm ac}(s,V,\vartheta)
       -f_{\rm ac}(s,V',\vartheta')}
 &\le L_f(R)\{\norm{V-V'}_\infty
                   +\norm{\vartheta-\vartheta'}\}.
\end{aligned}
\]
Thus Assumption~\ref{assum:lipschitz} holds on every bounded critic set; in
particular, the invariant cube contains every actual \(V_n\), while the
Bellman resolvent gives \(\sup_\vartheta\norm{V_\vartheta}_\infty\le B_V\).

On the invariant cube \(\norm{V}_\infty\le B_V\), the definitions give
\[
\norm{h_{\rm ac}(s,V,\vartheta)}_\infty\le B_V,
\qquad
\norm{f_{\rm ac}(s,V,\vartheta)}
\le C_{\max}+(1+\gamma_{\rm disc})B_V=2B_V,
\]
uniformly over \(s\in\cS\) and \(\vartheta\in\Theta_{\rm ac}\).
These are the required deterministic bounded-set drift envelopes. The definitions of \(M^V_{n+1}\) and
\(M^\vartheta_{n+1}\), together with \(|C_{n+1}|\le C_{\max}\), separately
give the bounded martingale-noise estimates in
Assumption~\ref{assum:1_2}.  The remaining boundedness and martingale
hypotheses follow from the invariant cube and the stated
controlled-kernel assumptions. To prove~\eqref{eq:actor_critic_coverage},
observe that full support of the softmax policy gives
\[
 P_\vartheta(s,s')>0
 \quad\Longleftrightarrow\quad
 P(s'\mid s,a)>0\quad\text{for some }a\in\cA.
\]
Thus the strongly connected graph assumed in the main paper is the support
graph of every \(P_\vartheta\). Hence \(P_\vartheta\) is irreducible and
\(d_\vartheta(s)>0\) for every \(s\in\cS\) and
\(\vartheta\in\Theta_{\rm ac}\). If \(\vartheta_k\to\vartheta\), every
subsequential limit of \(d_{\vartheta_k}\) is stationary for
\(P_\vartheta\), by continuity of \(\vartheta\mapsto P_\vartheta\).
Uniqueness of the stationary distribution therefore gives
\(d_{\vartheta_k}\to d_\vartheta\). Thus
\(\vartheta\mapsto\min_{s\in\cS}d_\vartheta(s)\) is continuous and
strictly positive. Compactness of \(\Theta_{\rm ac}\) proves
\eqref{eq:actor_critic_coverage}. The separately assumed uniform
hitting-time bound provides the uniform Poisson-equation bounds required
by the abstract theorem. For the actor calculation, define the reduced cost
\[
 K_{s,a}(\vartheta)
 :=c(s,a)+\gamma_{\rm disc}
   \sum_{s'}P(s'\mid s,a)V_\vartheta(s')-V_\vartheta(s).
\]
It satisfies
\(\sum_a\pi_\vartheta(a\mid s)K_{s,a}(\vartheta)=0\).~At the fast fixed point, the averaged slow field is
\begin{equation}
\label{eq:actor_critic_weighted_field}
 F_{\rm ac}(\vartheta)
 =-\sum_{s,a}q_\vartheta(s,a)e_{s,a}K_{s,a}(\vartheta),
 \qquad
 q_\vartheta(s,a):=d_\vartheta(s)\pi_\vartheta(a\mid s).
\end{equation}
Here \(q_\vartheta(s,a)\) is the stationary activation frequency of coordinate
\((s,a)\).~Put
\[
 \pi_{\min}^{\rm ac}:=\frac{e^{-2B_{\rm ac}}}{|\cA|},
 \qquad
 q_{\min}^{\rm ac}:=d_{\min}^{\rm ac}\pi_{\min}^{\rm ac}.
\]
Then \(\pi_\vartheta(a\mid s)\ge\pi_{\min}^{\rm ac}\) and
\(q_\vartheta(s,a)\ge q_{\min}^{\rm ac}\).

\emph{Lyapunov decrease.} Let
\[
L_{\rm ac}(\vartheta):=\sum_{s\in\cS}V_\vartheta(s).
\]

Encode the coordinates blocked by the box projection by setting
\(\chi_{s,a}(\vartheta)=0\) when either
\(\vartheta(s,a)=B_{\rm ac}\) and \(K_{s,a}(\vartheta)<0\), or
\(\vartheta(s,a)=-B_{\rm ac}\) and \(K_{s,a}(\vartheta)>0\), and let
\(\chi_{s,a}(\vartheta)=1\) otherwise. Thus, almost everywhere along the
projected actor ODE,
\[
\dot\vartheta(s,a)
=
-q_\vartheta(s,a)\chi_{s,a}(\vartheta)K_{s,a}(\vartheta).
\]

Along any absolutely continuous solution of this ODE, differentiating
the softmax map gives
\[
\dot\pi_\vartheta(a\mid s)
=
\pi_\vartheta(a\mid s)
\left\{
\dot\vartheta(s,a)
-
\sum_b\pi_\vartheta(b\mid s)\dot\vartheta(s,b)
\right\}.
\]
Consequently, using
\(\sum_a\pi_\vartheta(a\mid s)K_{s,a}(\vartheta)=0\), we obtain
\[
\begin{aligned}
\sum_a K_{s,a}(\vartheta)\dot\pi_\vartheta(a\mid s)
&=
\sum_a\pi_\vartheta(a\mid s)K_{s,a}(\vartheta)
       \dot\vartheta(s,a) \\
&=
-\sum_a\pi_\vartheta(a\mid s)q_\vartheta(s,a)
       \chi_{s,a}(\vartheta)K_{s,a}(\vartheta)^2
\le 0.
\end{aligned}
\]
Applying \citet[Lemma~5.4]{konda1999actor} to the induced policy-space
vector field, as in their proof of Lemma~5.10, and summing over \(s\)
therefore gives
\begin{align}
\frac{d}{dt}J_{B_{\rm ac}}(\vartheta(t))
&=
\frac{d}{dt}L_{\rm ac}(\vartheta(t))
\notag\\
&\le
-\sum_{s,a}\pi_\vartheta(a\mid s)q_\vartheta(s,a)
 \chi_{s,a}(\vartheta)K_{s,a}(\vartheta)^2
\notag\\
&\le
-\pi_{\min}^{\rm ac}q_{\min}^{\rm ac}
 \sum_{s,a}\chi_{s,a}(\vartheta)K_{s,a}(\vartheta)^2.
\label{eq:actor_critic_energy}
\end{align}

\emph{Control through the Bellman residual.}
To express the preceding Lyapunov-decrease estimate in terms of
\(J_{B_{\rm ac}}\), we first bound \(J_{B_{\rm ac}}\) by the Bellman
residual
\(\norm{V_\vartheta-T_{B_{\rm ac}}V_\vartheta}_\infty\),
and then bound that residual by the square root of the active
squared-gap sum. The set of policies generated by the logit box is
given by
\[
 \mathcal P_{B_{\rm ac}}
 :=\left\{\mu\in\Delta(\cA):
 \mu(a)\le e^{2B_{\rm ac}}\mu(b),\ a,b\in\cA\right\}.
\]
This is exactly the softmax image of the logit box.  The forward inclusion
follows from the logit ratios.  Conversely, the ratio constraints force every
\(\mu(a)>0\) and make the range of \(\{\log\mu(a):a\in\cA\}\) at most
\(2B_{\rm ac}\).  Hence
\[
 \eta(a):=\log\mu(a)
 -\frac12\left\{\max_b\log\mu(b)+\min_b\log\mu(b)\right\}
 \in[-B_{\rm ac},B_{\rm ac}]
\]
and
\(e^{\eta(a)}/\sum_{b\in\cA}e^{\eta(b)}=\mu(a)\) for every
\(a\in\cA\), which proves the reverse inclusion.
The restricted Bellman operator is
\[
 [T_{B_{\rm ac}}V](s)
 :=\min_{\mu\in\mathcal P_{B_{\rm ac}}}
 \sum_a\mu(a)\left\{
 c(s,a)+\gamma_{\rm disc}\sum_{s'}P(s'\mid s,a)V(s')
 \right\},
\]
and \(V_{B_{\rm ac}}^\star\) is its fixed point.  Fix \(s\), and for
\(\eta\in[-B_{\rm ac},B_{\rm ac}]^{|\cA|}\) define
\[
 N_s(\eta):=\sum_a e^{\eta(a)}K_{s,a}(\vartheta),
 \qquad
 Z_s(\eta):=\sum_a e^{\eta(a)}.
\]
Because \(N_s(\vartheta(s,\cdot))=0\), and every coordinate blocked
by the projection is already at its beneficial endpoint,
\[
 N_s(\eta)
 \ge-(e^{B_{\rm ac}}-e^{-B_{\rm ac}})
 \sum_a\chi_{s,a}(\vartheta)|K_{s,a}(\vartheta)|.
\]
Indeed, in the difference
\(N_s(\eta)-N_s(\vartheta(s,\cdot))\), every blocked coordinate contributes
nonnegatively, while the sum of the remaining contributions is bounded below
by the right-hand side above.
Because the softmax image of the logit box is exactly the restricted policy
class,
\[
 0\le
 V_\vartheta(s)-[T_{B_{\rm ac}}V_\vartheta](s)
 =-\min_{\eta\in[-B_{\rm ac},B_{\rm ac}]^{|\cA|}}
   \frac{N_s(\eta)}{Z_s(\eta)}
 \le
 \frac{e^{2B_{\rm ac}}-1}{|\cA|}
 \sum_a\chi_{s,a}(\vartheta)|K_{s,a}(\vartheta)|.
\]
Indeed, \(N_s(\vartheta(s,\cdot))=0\), so the minimum is nonpositive,
which gives the left inequality.  The preceding lower bound on
\(N_s(\eta)\), together with
\(Z_s(\eta)\ge|\cA|e^{-B_{\rm ac}}\), gives the right inequality.
Since \(T_{B_{\rm ac}}\) is a \(\gamma_{\rm disc}\)-contraction in the sup
norm,
\[
 \norm{V_\vartheta-V_{B_{\rm ac}}^\star}_\infty
 \le
 \frac{1}{1-\gamma_{\rm disc}}
 \norm{V_\vartheta-T_{B_{\rm ac}}V_\vartheta}_\infty.
\]
Using the componentwise inequality
\(V_\vartheta\ge V_{B_{\rm ac}}^\star\), we therefore have
\[
\begin{aligned}
 J_{B_{\rm ac}}(\vartheta)
 &=\sum_{s\in\cS}
   \{V_\vartheta(s)-V_{B_{\rm ac}}^\star(s)\}\\
 &\le |\cS|\norm{V_\vartheta-V_{B_{\rm ac}}^\star}_\infty\\
 &\le \frac{|\cS|}{1-\gamma_{\rm disc}}
   \norm{V_\vartheta-T_{B_{\rm ac}}V_\vartheta}_\infty\\
 &\le
 \frac{|\cS|(e^{2B_{\rm ac}}-1)}
      {|\cA|(1-\gamma_{\rm disc})}
 \max_{s\in\cS}\sum_a
 \chi_{s,a}(\vartheta)|K_{s,a}(\vartheta)|.
\end{aligned}
\]
Finally, Cauchy--Schwarz and \(\chi_{s,a}\in\{0,1\}\) bound the preceding
maximum by
\(\sqrt{|\cA|}\{\sum_{s,a}\chi_{s,a}(\vartheta)
K_{s,a}(\vartheta)^2\}^{1/2}\).  Therefore,
\begin{equation}
\label{eq:actor_critic_residual}
 J_{B_{\rm ac}}(\vartheta)
 \le A_{B_{\rm ac}}
 \left\{\sum_{s,a}\chi_{s,a}(\vartheta)
 K_{s,a}(\vartheta)^2\right\}^{1/2},
 \qquad
 A_{B_{\rm ac}}
 :=\frac{|\cS|(e^{2B_{\rm ac}}-1)}
         {(1-\gamma_{\rm disc})\sqrt{|\cA|}}.
\end{equation}
Combining~\eqref{eq:actor_critic_energy} and
\eqref{eq:actor_critic_residual} yields
\begin{equation}
\label{eq:actor_critic_quadratic_decay}
 \frac{d}{dt}J_{B_{\rm ac}}(\vartheta(t))
 \le-\kappa_{B_{\rm ac}}^{\rm ac}
       J_{B_{\rm ac}}(\vartheta(t))^2
 \quad\text{for almost every }t,
 \qquad
 \kappa_{B_{\rm ac}}^{\rm ac}
 :=\frac{\pi_{\min}^{\rm ac}q_{\min}^{\rm ac}}
          {A_{B_{\rm ac}}^2}>0.
\end{equation}
The map \(t\mapsto J_{B_{\rm ac}}(\vartheta(t))\) is absolutely
continuous.~If its initial value is positive, applying the chain rule
to its reciprocal while it remains positive and integrating
\eqref{eq:actor_critic_quadratic_decay} gives the following bound; if it
reaches zero, nonnegativity and monotonicity keep it there.~Thus,
\begin{equation}
\label{eq:actor_critic_inverse_time}
 J_{B_{\rm ac}}(\vartheta(t))
 \le\frac{J_{B_{\rm ac}}(\vartheta(0))}
 {1+\kappa_{B_{\rm ac}}^{\rm ac}t
 J_{B_{\rm ac}}(\vartheta(0))}.
\end{equation}

\emph{Equilibrium set and finite-time convergence of the actor ODE.}
To verify the uniform one-block Lyapunov contraction condition required
by Theorem~\ref{th:iter_rate}, we first identify the equilibrium set of
the projected actor ODE and then show that every trajectory reaches this
set within a uniform finite time.~Define
\[
 K_{s,a}^\star
 :=c(s,a)+\gamma_{\rm disc}\sum_{s'}P(s'\mid s,a)
 V_{B_{\rm ac}}^\star(s')-V_{B_{\rm ac}}^\star(s)
\]
and
\[
\begin{aligned}
 \mathcal E_{B_{\rm ac}}
 :=\{\vartheta\in\Theta_{\rm ac}:\;&
 \vartheta(s,a)=B_{\rm ac}\text{ if }K_{s,a}^\star<0,\\
 &\vartheta(s,a)=-B_{\rm ac}\text{ if }K_{s,a}^\star>0\}.
\end{aligned}
\]
Coordinates with \(K_{s,a}^\star=0\) are unrestricted.  For each state
\(s\) and \(\eta\in[-B_{\rm ac},B_{\rm ac}]^{|\cA|}\), define
\[
 N_s^\star(\eta):=\sum_a e^{\eta(a)}K_{s,a}^\star.
\]
Because \(V_{B_{\rm ac}}^\star\) is the fixed point of the restricted
Bellman operator and the softmax image of the logit box is exactly
\(\mathcal P_{B_{\rm ac}}\),
\[
 0=[T_{B_{\rm ac}}V_{B_{\rm ac}}^\star](s)
   -V_{B_{\rm ac}}^\star(s)
 =\min_{\eta\in[-B_{\rm ac},B_{\rm ac}]^{|\cA|}}
   \frac{N_s^\star(\eta)}{Z_s(\eta)}.
\]
The claim \(J_{B_{\rm ac}}(\vartheta)=0\iff
\vartheta\in\mathcal E_{B_{\rm ac}}\) identifies both the hitting target and,
as shown below, the actor ODE equilibria. First suppose
\(J_{B_{\rm ac}}(\vartheta)=0\). Since
\(V_\vartheta-V_{B_{\rm ac}}^\star\) is componentwise nonnegative, this
implies \(V_\vartheta=V_{B_{\rm ac}}^\star\).  The policy-evaluation
fixed-point equation then gives
\[
 N_s^\star(\vartheta(s,\cdot))
 =Z_s(\vartheta(s,\cdot))
 \bigl\{[T_\vartheta V_{B_{\rm ac}}^\star](s)
        -V_{B_{\rm ac}}^\star(s)\bigr\}
 =0.
\]
If \(K_{s,a}^\star>0\) but \(\vartheta(s,a)>-B_{\rm ac}\), lowering only
that logit makes \(N_s^\star\) strictly negative.  Since \(Z_s>0\), this
contradicts the preceding restricted Bellman minimum.  Likewise, if
\(K_{s,a}^\star<0\) but \(\vartheta(s,a)<B_{\rm ac}\), raising only that
logit makes \(N_s^\star\) strictly negative.  Hence
\(\vartheta\in\mathcal E_{B_{\rm ac}}\).

Conversely, suppose \(\vartheta\in\mathcal E_{B_{\rm ac}}\).  The endpoint
signs imply, for every feasible \(\eta\),
\[
 N_s^\star(\eta)-N_s^\star(\vartheta(s,\cdot))
 =\sum_a\bigl(e^{\eta(a)}-e^{\vartheta(s,a)}\bigr)K_{s,a}^\star
 \ge0.
\]
The restricted Bellman minimum is zero, so every feasible numerator is
nonnegative.  Moreover, compactness gives a minimizer
\(\bar\eta_s\) with \(N_s^\star(\bar\eta_s)=0\).  The preceding comparison,
together with the feasibility of \(\vartheta(s,\cdot)\), therefore gives
\(N_s^\star(\vartheta(s,\cdot))=0\).  Thus
\[
 [T_\vartheta V_{B_{\rm ac}}^\star](s)
 -V_{B_{\rm ac}}^\star(s)
 =\frac{N_s^\star(\vartheta(s,\cdot))}
        {Z_s(\vartheta(s,\cdot))}
 =0.
\]
Since this holds for every state, \(V_{B_{\rm ac}}^\star\) is the fixed
point of \(T_\vartheta\).  Uniqueness of that fixed point gives
\(V_\vartheta=V_{B_{\rm ac}}^\star\), and hence
\begin{equation}
\label{eq:actor_critic_zero_set}
 J_{B_{\rm ac}}(\vartheta)=0
 \quad\Longleftrightarrow\quad
 \vartheta\in\mathcal E_{B_{\rm ac}}.
\end{equation}
If \(\vartheta\in\mathcal E_{B_{\rm ac}}\), the equality
\(V_\vartheta=V_{B_{\rm ac}}^\star\) gives
\(K_{s,a}(\vartheta)=K_{s,a}^\star\), and the endpoint conditions give
\(\chi_{s,a}(\vartheta)K_{s,a}(\vartheta)=0\) for every \((s,a)\).
Conversely, if these products vanish, the residual bound gives
\(J_{B_{\rm ac}}(\vartheta)=0\), so~\eqref{eq:actor_critic_zero_set} yields
\(\vartheta\in\mathcal E_{B_{\rm ac}}\).  Since \(q_\vartheta(s,a)>0\), the
products vanish exactly when the projected actor velocity vanishes. Thus \(\mathcal E_{B_{\rm ac}}\), equivalently the zero set of
\(J_{B_{\rm ac}}\), is precisely the equilibrium set of the projected
actor ODE.

If \(K_{s,a}^\star=0\) for every \((s,a)\), then
\(\mathcal E_{B_{\rm ac}}=\Theta_{\rm ac}\),
\(J_{B_{\rm ac}}\) is identically zero, and we set
\(T_{\rm hit}:=0\).  The finite-hitting property is then immediate.
Henceforth suppose that the nonzero restricted-gap set is nonempty, and let
\[
 \Delta_{B_{\rm ac}}
 :=\min_{K_{s,a}^\star\ne0}|K_{s,a}^\star|>0,
 \qquad
 \eta_{B_{\rm ac}}
 :=\frac{\Delta_{B_{\rm ac}}}{2(1+\gamma_{\rm disc})}.
\]
Since \(V_\vartheta-V_{B_{\rm ac}}^\star\) is componentwise
nonnegative,
\[
 \norm{V_\vartheta-V_{B_{\rm ac}}^\star}_\infty
 \le J_{B_{\rm ac}}(\vartheta),
\]
and hence
\[
 |K_{s,a}(\vartheta)-K_{s,a}^\star|
 \le(1+\gamma_{\rm disc})J_{B_{\rm ac}}(\vartheta).
\]
By~\eqref{eq:actor_critic_inverse_time}, every ODE trajectory enters
the sublevel \(J_{B_{\rm ac}}\le\eta_{B_{\rm ac}}\) within time at most
\[
 T_{\rm ent}:=
 \frac{1}{\kappa_{B_{\rm ac}}^{\rm ac}\eta_{B_{\rm ac}}},
\]
with entry time zero if it starts in that sublevel.  This sublevel is
invariant because \(J_{B_{\rm ac}}\) is
nonincreasing.  Thereafter every nonzero \(K_{s,a}^\star\) has its final
sign and
\(|K_{s,a}(\vartheta)|\ge\Delta_{B_{\rm ac}}/2\).  Each corresponding
logit therefore moves toward its required endpoint at speed at least
\(q_{\min}^{\rm ac}\Delta_{B_{\rm ac}}/2\), and reaches it within an
additional time at most
\[
 \frac{4B_{\rm ac}}
      {q_{\min}^{\rm ac}\Delta_{B_{\rm ac}}}.
\]
It remains there because the sign is preserved.  In the nondegenerate case
one may therefore take
\[
 T_{\rm hit}
 :=
 T_{\rm ent}
 +\frac{4B_{\rm ac}}
       {q_{\min}^{\rm ac}\Delta_{B_{\rm ac}}}
 <\infty.
\]
Together with the choice \(T_{\rm hit}:=0\) in the degenerate case, in
either case
\begin{equation}
\label{eq:actor_critic_finite_hit}
 J_{B_{\rm ac}}(\mathsf S_t(\vartheta))=0,
 \qquad \vartheta\in\Theta_{\rm ac},\quad t\ge T_{\rm hit}.
\end{equation}
For any block horizon \(T\ge T_{\rm hit}\) and any
\(\rho_T\in(0,1)\), the uniform one-block contraction condition follows.

\emph{Approximation error from the bounded-logit restriction.}
It remains to bound the difference between the optimal value
\(V_{B_{\rm ac}}^\star\) over policies representable by logits in
\([-B_{\rm ac},B_{\rm ac}]^{|\cA|}\) and the unrestricted optimal value
\(V^\star\).~For each state \(s\), choose an unrestricted optimal action
\(a^\star(s)\), and let \(\widehat\vartheta\in\Theta_{\rm ac}\) assign
logit \(B_{\rm ac}\) to \(a^\star(s)\) and logit \(-B_{\rm ac}\) to every
other action. Under \(\pi_{\widehat\vartheta}\), the total probability of
the other actions is at most \((|\cA|-1)e^{-2B_{\rm ac}}\). Bellman
optimality and
\(\norm{V^\star}_\infty\le C_{\max}/(1-\gamma_{\rm disc})\) give, for
every \(s\in\cS\) and \(a\in\cA\),
\[
0\le
c(s,a)+\gamma_{\rm disc}\sum_{s'\in\cS}P(s'\mid s,a)V^\star(s')
-V^\star(s)
\le \frac{2C_{\max}}{1-\gamma_{\rm disc}},
\]
with equality on the left when \(a=a^\star(s)\). Consequently,
\[
\norm{T_{\widehat\vartheta}V^\star-V^\star}_\infty
\le
\frac{2(|\cA|-1)C_{\max}e^{-2B_{\rm ac}}}
     {1-\gamma_{\rm disc}}.
\]
Since \(T_{\widehat\vartheta}\) is a
\(\gamma_{\rm disc}\)-contraction, the fixed-point residual inequality
gives
\[
\norm{V_{\widehat\vartheta}-V^\star}_\infty
\le
\frac{1}{1-\gamma_{\rm disc}}
\norm{T_{\widehat\vartheta}V^\star-V^\star}_\infty.
\]
Finally, optimality over the unrestricted and restricted policy classes
gives
\(V^\star\le V_{B_{\rm ac}}^\star\le V_{\widehat\vartheta}\)
componentwise, and therefore
\[
 \norm{V_{B_{\rm ac}}^\star-V^\star}_\infty
 \le\frac{2(|\cA|-1)C_{\max}e^{-2B_{\rm ac}}}
          {(1-\gamma_{\rm disc})^2}.
\]
Summing over states proves the restriction-bias bound in the corollary.

It remains to apply Proposition~\ref{prop:face_rates}. The Bellman
resolvent and smoothness of the softmax map on the compact box show that
$J_{B_{\rm ac}}$ is Lipschitz. Taking $T=\max\{1,T_{\rm hit}\}$ and
$\rho_T=e^{-T}$, \eqref{eq:actor_critic_finite_hit} verifies the one-block
Lyapunov decrease condition needed for initial convergence.
The equilibrium set $\mathcal E_{B_{\rm ac}}$ is a box face, and
$V_\vartheta=V_{B_{\rm ac}}^\star$ throughout that face.
For a coordinate fixed on the face, the outward sign is
$-\operatorname{sgn}(K_{s,a}^\star)$, so its signed averaged drift satisfies
\[
 q_\vartheta(s,a)|K_{s,a}^\star|
 \ge q_{\min}^{\rm ac}\Delta_{B_{\rm ac}}>0.
\]
If no coordinate is fixed, the outward-drift condition is vacuous and
$J_{B_{\rm ac}}\equiv0$.
All other model, Lipschitz, and contraction assumptions were verified above.
Both schedules in~\eqref{eq:face_steps} have $0<\alpha_n\le1$, so the
critic cube remains invariant for either choice. Proposition~\ref{prop:face_rates},
with common fast target $V_{B_{\rm ac}}^\star$, gives the component bounds
and the joint rate~\eqref{eq:actor_critic_rate}.
Together with the restriction-bias estimate already proved, this establishes
Corollary~\ref{coro:actor_critic_rate}.
\end{proof}

\section{Extension: Projection of the Fast-Scale Iterate}
\label{sec:ec_fast_projection_extension}

This section gives an optional extension of the framework in which the fast
iterate is projected at every iteration.  The main-paper recursion and the
results proved for it are not changed.  Instead, throughout this section we
replace~\eqref{def:x_update} by
\begin{equation}
x_{n+1}
=
\operatorname{proj}_{\cX}
\left(
x_n+\alpha_n
\{h(Y_n,x_n,\te_n)-x_n+M_{n+1}\}
\right),
\qquad x_1\in\cX,
\label{eq:ec_projected_fast_recursion}
\end{equation}
while retaining the projected slow recursion~\eqref{def:te_update}.
Here \(\cX\subset\bR^{d_x}\) is fixed: it does not depend on \(n\) or
on \(\te\).  Consequently, \(x_n\in\cX\) for every \(n\), and the
fast-stability Assumption~\ref{assum:xn_bound} is no longer needed.

This extension is useful when boundedness of the unprojected fast iterates
is not available a priori or would require a separate application-specific
stability proof. The main paper assumes such boundedness, whereas the
modified recursion enforces it by projecting the fast iterate at every step
onto a prescribed compact feasible region.~Subject to the Euclidean contraction condition on $\mathcal X$ and the
bounded-preprojection-noise condition stated below,
the extension then gives
the corresponding high-probability path-tracking and almost-sure
last-iterate guarantees.

\subsection{Euclidean projection and contraction}

Throughout this section, the fast-coordinate norm is the Euclidean norm.
Let \(\cX\subset\bR^{d_x}\) be a nonempty compact convex polyhedron, and
write \(P_{\cX}:=\operatorname{proj}_{\cX}\) for its Euclidean projection.
We impose the averaged-contraction condition in the Euclidean norm and only
on \(\cX\):
\begin{equation}
\norm{\havg(x;\te)-\havg(x';\te)}
\le
\gamma\norm{x-x'},
\qquad x,x'\in\cX,\quad \te\in\Te.
\label{eq:ec_projected_fast_contraction}
\end{equation}
Throughout this section, \eqref{eq:ec_projected_fast_contraction} replaces
the global contraction requirement in Assumption~\ref{assum:contraction}.
Assumption~\ref{assum:lipschitz} is retained with both
fast-state arguments restricted to \(\cX\): there exist finite constants
\(L_f^{\cX}\), \(L_{h,x}^{\cX}\), and \(L_{h,\te}^{\cX}\) such that,
uniformly over \(y\in\cY\), \(x,\widetilde x\in\cX\), and
\(\te,\widetilde\te\in\Te\),
\[
\norm{f(y,x,\te)-f(y,\widetilde x,\widetilde\te)}
\le
L_f^{\cX}
\bigl(
\norm{x-\widetilde x}
+
\norm{\te-\widetilde\te}
\bigr),
\]
and
\[
\norm{h(y,x,\te)-h(y,\widetilde x,\widetilde\te)}
\le
L_{h,x}^{\cX}\norm{x-\widetilde x}
+
L_{h,\te}^{\cX}\norm{\te-\widetilde\te}.
\]
For every
\(v\in D_{BV}([0,\infty);\bR^{d_x})\) with \(v(0)\in\cX\),
let \(\Gamma_{\cX}(v)\) denote the projected component of the
Skorokhod problem on \(\cX\), as in
Definition~\ref{def:skorokhod_problem}, with
\((\Te,d)\) replaced by \((\cX,d_x)\).
Theorem~\ref{th:skrohod_lip} ensures that
\(\Gamma_{\cX}\) is well defined.
\begin{proposition}[Projection properties and contraction of $\mathcal T_{s,\theta}$]
\label{prop:ec_fast_projection_geometry}
Under~\eqref{eq:ec_projected_fast_contraction}, the following properties
hold.
\begin{enumerate}[(i)]
\item The Euclidean projection is nonexpansive:
\[
\norm{P_{\cX}(u)-P_{\cX}(v)}^2
\le
\left\langle P_{\cX}(u)-P_{\cX}(v),u-v\right\rangle,
\qquad u,v\in\bR^{d_x}.
\]
Consequently,
\(\norm{P_{\cX}(u)-P_{\cX}(v)}\le\norm{u-v}\).
\item The Skorokhod map \(\Gamma_{\cX}\) is Lipschitz: there exists
\(L_{\mathrm{Sk}}^x<\infty\), independent of the time horizon, such that
\[
\sup_{0\le t\le H}
\norm{\Gamma_{\cX}(v)(t)-\Gamma_{\cX}(\widetilde v)(t)}
\le
L_{\mathrm{Sk}}^x
\sup_{0\le t\le H}
\norm{v(t)-\widetilde v(t)}
\]
for every \(H<\infty\) and every
\[
v,\widetilde v\in D_{BV}([0,\infty);\bR^{d_x})
\quad\text{with}\quad
v(0),\widetilde v(0)\in\cX.
\]
\item For \(s\in[0,1]\), define
\[
\mathcal T_{s,\te}(x)
:=
P_{\cX}\bigl((1-s)x+s\havg(x;\te)\bigr).
\]
Then
\begin{equation}
\norm{\mathcal T_{s,\te}(x)-\mathcal T_{s,\te}(x')}
\le
\{1-(1-\gamma)s\}\norm{x-x'},
\qquad x,x'\in\cX.
\label{eq:ec_relaxed_projected_contraction}
\end{equation}
In particular, \(P_{\cX}\circ\havg(\cdot;\te)\) is a
\(\gamma\)-contraction on \(\cX\).
\end{enumerate}
\end{proposition}

\begin{proof}
Let \(p=P_{\cX}(u)\) and \(q=P_{\cX}(v)\). The two projection
variational inequalities give
\[
\langle u-p,q-p\rangle\le0,
\qquad
\langle v-q,p-q\rangle\le0.
\]
Adding them yields
\(\norm{p-q}^2\le\langle p-q,u-v\rangle\), which proves part~(i).
Part~(ii) follows by applying Theorem~\ref{th:skrohod_lip} to the
polyhedron \(\cX\). Finally, part~(i) and
\eqref{eq:ec_projected_fast_contraction} give
\[
\begin{aligned}
\norm{\mathcal T_{s,\te}(x)-\mathcal T_{s,\te}(x')}
&\le
\norm{(1-s)(x-x')
+s\{\havg(x;\te)-\havg(x';\te)\}}\\
&\le\{1-(1-\gamma)s\}\norm{x-x'},
\end{aligned}
\]
which proves part~(iii).
\end{proof}

All remaining standing assumptions are imposed on
\(\cY\times\cX\times\Te\). In particular, we retain
Assumptions~\ref{assum:markov_noise}, \ref{assum:lipschitz_kernel},
\ref{assum:1_2}, \ref{assum:lipschitz}, and
\ref{assum:hitting_time}, with the fast-coordinate bounds interpreted in the
Euclidean norm; Assumption~\ref{assum:contraction} is replaced by
\eqref{eq:ec_projected_fast_contraction}. Let
$R_{\cX}:=\sup_{x\in\cX}\norm{x}<\infty$.  Finiteness of $\cY$,
compactness, and Assumption~\ref{assum:lipschitz} give a finite
$\mathsf B_{\cX}$ bounding both drift maps on $\cY\times\cX\times\Te$, while
Assumption~\ref{assum:1_2} gives the preprojection-noise bound
$N_{\rm noise}^{\cX}:=\frakc_1+\frakc_2R_{\cX}$.  Set
\begin{equation}
A_x^{\cX}:=R_{\cX}+\mathsf B_{\cX}+N_{\rm noise}^{\cX},\qquad
A_\te^{\cX}:=\mathsf B_{\cX}+N_{\rm noise}^{\cX}.
\label{eq:ec_projected_fast_uniform_constants}
\end{equation}
Then, almost surely,
\begin{align}
&\sup_{x\in\cX,\te\in\Te,y\in\cY}
\bigl\{
\norm{x}+\norm{h(y,x,\te)}
\bigr\}
+\sup_n\norm{M_{n+1}}\le A_x^{\cX},\notag\\
&\norm{x_{n+1}-x_n}\le A_x^{\cX}\alpha_n,
\qquad
\norm{\te_{n+1}-\te_n}\le A_\te^{\cX}\beta_n.
\label{eq:ec_projected_fast_increment_bounds}
\end{align}
The first increment bound follows from \(x_n=P_{\cX}(x_n)\) and
Proposition~\ref{prop:ec_fast_projection_geometry}(i). Thus fast
projection bounds the state but does not by itself bound the
preprojection noise. The pathwise bounded-noise part of
Assumption~\ref{assum:1_2} is still required.

\subsection{Constrained equilibrium and projected fast ODE}
For \(x\in\cX\) and \(v\in\bR^{d_x}\), define the fast-set analogue
of~\eqref{def:pi} by
\begin{equation}
\Pi_{\cX}(x,v)
:=
\lim_{\eta\downarrow0}
\frac{P_{\cX}(x+\eta v)-x}{\eta}.
\label{eq:ec_projected_fast_direction}
\end{equation}
Let
\(L_{\bar h,\te}\) be a uniform Lipschitz constant of
\(\havg(x;\cdot)\) on \(\Te\), uniform in \(x\in\cX\).~Such a
constant follows from Assumptions~\ref{assum:lipschitz_kernel},
\ref{assum:lipschitz}, and~\ref{assum:hitting_time}, exactly as in
Lemma~\ref{lemma:lipschitz_stat_dist}, because \(\cX\) is compact.
\begin{proposition}[Constrained fast equilibrium and exponential stability]
\label{prop:ec_projected_fast_equilibrium}
Under the contraction condition~\eqref{eq:ec_projected_fast_contraction}, for every
$\theta\in\Theta$, the map
$P_{\mathcal X}\circ h^{(\mathrm{av.})}(\cdot;\theta)$ has a unique
fixed point in $\mathcal X$. Denote it by
$x_{\mathcal X}^{\star}(\theta)$; equivalently,
\begin{equation}
x_{\mathcal X}^{\star}(\theta)
=
P_{\mathcal X}\!\left(
h^{(\mathrm{av.})}
\bigl(x_{\mathcal X}^{\star}(\theta);\theta\bigr)
\right).
\label{eq:constrained-fast-fixed-point}
\end{equation}
Moreover,
\begin{equation}
\norm{x_{\cX}^\star(\te)-x_{\cX}^\star(\te')}
\le
\frac{L_{\bar h,\te}}{q}\norm{\te-\te'},
\qquad \te,\te'\in\Te.
\label{eq:ec_projected_fast_fp_lipschitz}
\end{equation}
For frozen \(\te\), the projected fast ODE
\begin{equation}
\dot z(t)
=
\Pi_{\cX}
\bigl(z(t),\havg(z(t);\te)-z(t)\bigr)
\label{eq:ec_projected_fast_ode}
\end{equation}
has \(x_{\cX}^\star(\te)\) as its unique equilibrium, and its flow
\(\mathsf S_t^{x,\te}\) satisfies
\begin{equation}
\norm{
\mathsf S_t^{x,\te}(z)-
\mathsf S_t^{x,\te}(z')}
\le
e^{-qt}\norm{z-z'},
\qquad z,z'\in\cX,\quad t\ge0.
\label{eq:ec_projected_fast_ode_contraction}
\end{equation}
In particular,
\[
\norm{\mathsf S_t^{x,\te}(z)-x_{\cX}^\star(\te)}
\le e^{-qt}\norm{z-x_{\cX}^\star(\te)}.
\]
\end{proposition}

\begin{proof}
Proposition~\ref{prop:ec_fast_projection_geometry}(iii) and Banach's
fixed-point theorem prove existence and uniqueness. Comparing the two
fixed-point identities at \(\te\) and \(\te'\) gives
\[
\begin{aligned}
\norm{x_{\cX}^\star(\te)-x_{\cX}^\star(\te')}
&\le
\gamma
\norm{x_{\cX}^\star(\te)-x_{\cX}^\star(\te')}
+L_{\bar h,\te}\norm{\te-\te'},
\end{aligned}
\]
which proves~\eqref{eq:ec_projected_fast_fp_lipschitz}.

The projection characterization also shows that
\(x_{\cX}^\star(\te)\) is a fixed point of
\(\mathcal T_{s,\te}\) for every \(s\in(0,1]\): indeed,
\(x=P_{\cX}(\havg(x;\te))\) is equivalent to
\(\havg(x;\te)-x\) belonging to the normal cone at \(x\), and that
cone is positively homogeneous.

Let
\[
N_{\cX}(z)
:=
\left\{\eta\in\bR^{d_x}:
\langle\eta,y-z\rangle\le0\text{ for every }y\in\cX\right\}
\]
be the Euclidean outward normal cone. For fixed \(\te\), the map
\( x \mapsto \havg(x;\te)-x\) is Lipschitz on \(\cX\) with
constant at most \(1+\gamma\). Hence the projected ODE is well posed by
the same projected-dynamical-systems argument as in
Lemma~\ref{lemma:ode_soln_unique_exist}. The projection characterization
shows that the fixed point \(x_{\cX}^\star(\te)\) is an equilibrium of the
ODE. For two solutions, the tangent--normal decomposition gives, for almost
every \(t\),
\[
\dot z=\havg(z;\te)-z-\eta,
\qquad
\dot z'=\havg(z';\te)-z'-\eta',
\]
where
\(\eta\in N_{\cX}(z)\) and \(\eta'\in N_{\cX}(z')\).
Here the normals are outward, so \(-\eta\) and \(-\eta'\) are the
corresponding inward correction terms. Since the normal-cone map is
monotone,
\[
\langle\eta-\eta',z-z'\rangle\ge0.
\]
Therefore,
\[
\begin{aligned}
\frac{d}{dt}\frac12\norm{z-z'}^2
&=
\left\langle
z-z',
\havg(z;\te)-\havg(z';\te)
\right\rangle\\
&\quad
-\norm{z-z'}^2
-\langle z-z',\eta-\eta'\rangle\\
&\le
-(1-\gamma)\norm{z-z'}^2
=
-q\norm{z-z'}^2.
\end{aligned}
\]
Gronwall's inequality proves
\eqref{eq:ec_projected_fast_ode_contraction}. It also proves uniqueness of
the equilibrium and the final displayed bound in the proposition.
\end{proof}

The constrained target need not equal the fixed point used in the
unprojected framework. For comparison with that framework, suppose
additionally that the original global Assumption~\ref{assum:contraction}
holds, and let \(x^\star(\te)\) denote its unique fixed point. If
\begin{equation}
x^\star(\te)\in\cX,
\qquad \te\in\Te,
\label{eq:ec_projection_target_compatibility}
\end{equation}
then \(x_{\cX}^\star(\te)=x^\star(\te)\), so the projection is only a
safeguard and the reduced slow ODE is unchanged. Without this additional
global assumption and compatibility condition, the extension instead
analyzes the constrained reduced field
\begin{equation}
\favg_{\cX}(\te;\te')
:=
\sum_{y\in\cY}\mu^{(\te')}(y)
f\bigl(y,x_{\cX}^\star(\te),\te\bigr),
\label{eq:ec_projected_fast_reduced_field}
\end{equation}
and the slow projected ODE
\begin{equation}
\dot z
=
\Pi_{\Te}
\bigl(z,\favg_{\cX}(z;z)\bigr).
\label{eq:ec_projected_fast_reduced_ode}
\end{equation}
The field \(\te\mapsto\favg_{\cX}(\te;\te)\) is Lipschitz by
\eqref{eq:ec_projected_fast_fp_lipschitz} and the standing assumptions;
hence the same projected-dynamical-systems argument as in
Lemma~\ref{lemma:ode_soln_unique_exist} makes this ODE well posed.
Let \(\mathsf S_t^{\cX}\) denote the flow of
\eqref{eq:ec_projected_fast_reduced_ode}, and let
\begin{equation}
\mathcal E\mathcal Q_{\Te}^{\cX}
:=
\left\{z\in\Te:
\Pi_{\Te}\bigl(z,\favg_{\cX}(z;z)\bigr)=\mathbf 0
\right\}
\label{eq:ec_projected_fast_equilibrium_set}
\end{equation}
be its equilibrium set.
In this section, the constrained analogues of
Assumptions~\ref{assum:ode_stability},
\ref{assum:one_block_decrease}, and
\ref{assum:one_block_contraction} mean those assumptions with
\(\mathsf S_s\), \(\eq\), and \(\favg\) replaced, respectively, by
\(\mathsf S_s^{\cX}\),
\(\mathcal E\mathcal Q_{\Te}^{\cX}\), and \(\favg_{\cX}\).
When the constrained analogue of
Assumption~\ref{assum:one_block_contraction} is invoked, denote its
one-block contraction factor by \(\rho_T^{\cX}\in(0,1)\).
If $\mathcal X$ fails to contain $x^\star(\theta)$ for some
$\theta\in\Theta$, the projected recursion targets
$x_{\mathcal X}^\star(\theta)$ instead, and the resulting reduced slow
field, limiting ODE, and equilibrium set may differ from those of the
unprojected recursion.
\subsection{Fixed fast-time blocks and the moving-parameter comparison}

The projected-fast comparison cannot be run over an entire slow-time
block by a direct Gronwall argument: the fast-time duration of such a block
is of order \(\alpha_{n_m}/\beta_{n_m}\), which diverges.  We therefore
partition each slow block into subblocks having a fixed duration in the
fast clock.

Fix a fast-time horizon \(T_x>0\).~On the slow block
\(\cK_m\), put \(\nu_{m,0}:=n_m\), and, while
\(\nu_{m,j}<n_{m+1}\), define
\begin{equation}
\nu_{m,j+1}
:=
\min\left\{
n_{m+1},
\inf\left\{n>\nu_{m,j}:
\sum_{r=\nu_{m,j}}^{n-1}\alpha_r\ge T_x
\right\}
\right\}.
\label{eq:ec_local_fast_blocks}
\end{equation}
Here, \(m\) indexes the slow-time blocks and \(j\) indexes the fast-time
subblocks within \(\cK_m\). The index \(\nu_{m,j}\) is the left
endpoint of the \(j\)-th subblock, while \(\nu_{m,j+1}\) is its right
endpoint, chosen when the accumulated fast-clock time first reaches
\(T_x\), or as \(n_{m+1}\) if the slow block ends first.
Let \(N_m^{\rm sub}\) be the number of resulting subblocks. For
\(j=0,\ldots,N_m^{\rm sub}-1\), define the discrete index set
\[
I_{m,j}
:=
\{r\in\bN:\nu_{m,j}\le r<\nu_{m,j+1}\},
\qquad
H_{m,j}^x:=\sum_{r\in I_{m,j}}\alpha_r.
\]
Every subblock except possibly the last one satisfies
\begin{equation}
T_x\le H_{m,j}^x<T_x+\alpha_{\nu_{m,j}},
\label{eq:ec_fast_block_overshoot}
\end{equation}
and the last subblock satisfies the corresponding upper bound. The arguments below use only the bounds in~\eqref{eq:ec_fast_block_overshoot} on the accumulated fast time of each subblock.
For integers $r<s$, call the index set $\left\{r,r+1,\ldots,s-1\right\}$ admissible if
\[
\sum_{i=r}^{s-1}\alpha_i\le T_x+\alpha_r,
\]
and call it full if the same sum is also at least \(T_x\). Thus every local
subblock is admissible, and every one except possibly the last is full.
The terms admissible and full will also be used for the discrete index sets \(\{\nu_j,\nu_j+1,\ldots,\nu_{j+1}-1\}\) in the trajectory-wide fast-clock partition defined later in~\eqref{eq:ec_projected_fast_global_blocks}.
Unlike the local subblocks \(I_{m,j}\), these blocks partition the entire
iteration suffix and are not restarted or truncated at the slow-block
endpoints \(n_{m+1}\).

For each admissible block with integer endpoints \(r<s\), define a fresh
block-restarted, moving-parameter averaged recursion
\(\{\widetilde x_k\}_{k=r}^{s}\) by
\begin{align}
\widetilde x_r&:=x_r,\notag\\
\widetilde x_{k+1}
&:=
P_{\cX}\left(
\widetilde x_k+
\alpha_k\{\havg(\widetilde x_k;\te_k)-\widetilde x_k\}
\right),
\qquad r\le k<s.
\label{eq:ec_projected_fast_auxiliary}
\end{align}
The dependence of \(\widetilde x_k\) on the selected block endpoints
\(r,s\) is suppressed. In particular, for the local subblock \(I_{m,j}\),
we take \(r=\nu_{m,j}\) and \(s=\nu_{m,j+1}\), and restart the auxiliary
recursion from \(x_{\nu_{m,j}}\). The auxiliary recursion uses the actual moving parameter \(\te_k\). 

We retain the fast Poisson solution from
Definition~\ref{def:poisson_fast}, now restricted to
\((x,\te)\in\cX\times\Te\), rather than
\((x,\te)\in\bR^{d_x}\times\Te\). With
\(M''_{n+1}\) as in~\eqref{def:M''_n}, set
\[
\Xi_{n+1}:=M_{n+1}+M''_{n+1},
\qquad
\Delta^h_{n+1}
:=
u_h^{(x_n,\te_n)}(Y_n)-u_h^{(x_n,\te_n)}(Y_{n+1}).
\]
Adding and subtracting \(u_h^{(x_n,\te_n)}(Y_{n+1})\) in the Poisson
equation, as in Lemma~\ref{lemma:9}, gives, pathwise,
\begin{equation}
h(Y_n,x_n,\te_n)-\havg(x_n;\te_n)+M_{n+1}
=\Xi_{n+1}+\Delta^h_{n+1}.
\label{eq:ec_projected_fast_poisson_decomposition}
\end{equation}
The sequence \(\{\Xi_{n+1}\}\) is a
martingale-difference sequence with respect to \(\{\cF_n\}\). Since
\(\cY\) is finite and \(\cX\times\Te\) is compact, the retained
bounded-noise and Lipschitz assumptions, together with
Assumption~\ref{assum:hitting_time}, give deterministic constants
\(C_\Xi,U_h^{\cX},C_u<\infty\) such that the first bound below holds
almost surely for every \(n\ge1\), whereas the remaining two bounds hold
for every \((x,\te),(x',\te')\in\cX\times\Te\):
\begin{align}
\norm{\Xi_{n+1}}&\le C_\Xi,\notag\\
\sup_{y\in\cY}\norm{u_h^{(x,\te)}(y)}&\le U_h^{\cX},\notag\\
\sup_{y\in\cY}
\norm{u_h^{(x,\te)}(y)-u_h^{(x',\te')}(y)}
&\le C_u\{\norm{x-x'}+\norm{\te-\te'}\}.
\label{eq:ec_projected_fast_poisson_uniform}
\end{align}
For \(r\le k\le s\), define
\begin{align}
\mathcal N_{r,k}^x
&:=\sum_{i=r}^{k-1}\alpha_i\Xi_{i+1},
\label{eq:ec_projected_fast_martingale}\\
\mathcal R_{r,k}^x
&:=\sum_{i=r}^{k-1}\alpha_i\Delta^h_{i+1},
\label{eq:ec_projected_fast_residual}
\end{align}
and, for \(\delta>0\), let
\begin{equation}
\mathcal G_{r,s}^{x,\mathrm{proj}}(\delta)
:=
\left\{
\max_{r\le k\le s}\norm{\mathcal N_{r,k}^x}<\delta
\right\}.
\label{eq:ec_projected_fast_micro_good_event}
\end{equation}

\begin{lemma}[Fast martingale and Poisson residual on a fast-time block]
\label{lem:ec_projected_fast_driver_control}
There are constants \(c_x,C_x^{\rm res}\in(0,\infty)\), independent of
\(r,s\), such that every admissible block with endpoints \(r<s\) satisfies
\begin{align}
\bP\left(
[\mathcal G_{r,s}^{x,\mathrm{proj}}(\delta)]^c
\right)
&\le
2d_x(s-r)
\exp\left\{-\frac{c_x\delta^2}{\alpha_r}\right\},
\label{eq:ec_projected_fast_micro_tail}\\
\max_{r\le k\le s}\norm{\mathcal R_{r,k}^x}
&\le C_x^{\rm res}\alpha_r.
\label{eq:ec_projected_fast_residual_bound}
\end{align}
\end{lemma}

\begin{proof}
Coordinatewise Azuma--Hoeffding, followed by a union bound over the
coordinates and the at most \(s-r\) partial sums, gives
\[
\bP\left(
\max_{r\le k\le s}\norm{\mathcal N_{r,k}^x}\ge\delta
\right)
\le
2d_x(s-r)
\exp\left\{
-\frac{\delta^2}
{2\kappa_{d_x}^2C_\Xi^2
\sum_{i=r}^{s-1}\alpha_i^2}
\right\}.
\]
By admissibility and monotonicity of the step sizes,
\[
\sum_{i=r}^{s-1}\alpha_i^2
\le
\alpha_r\sum_{i=r}^{s-1}\alpha_i
\le(T_x+\alpha_1)\alpha_r.
\]
This proves~\eqref{eq:ec_projected_fast_micro_tail}.

For the residual, abbreviate
\(U_i(y):=u_h^{(x_i,\te_i)}(y)\).  Summation by parts gives, for
\(r<k\le s\),
\begin{align*}
\mathcal R_{r,k}^x
&=
\alpha_rU_r(Y_r)-\alpha_{k-1}U_{k-1}(Y_k)\\
&\quad+
\sum_{i=r+1}^{k-1}
\{\alpha_iU_i(Y_i)-\alpha_{i-1}U_{i-1}(Y_i)\}.
\end{align*}
Using~\eqref{eq:ec_projected_fast_increment_bounds} and
\eqref{eq:ec_projected_fast_poisson_uniform}, its norm is at most a
constant times
\[
\alpha_r
+\sum_{i=r+1}^{s-1}|\alpha_i-\alpha_{i-1}|
+\sum_{i=r}^{s-1}\alpha_i^2
+\sum_{i=r}^{s-1}\alpha_i\beta_i.
\]
The first two terms are \(O(\alpha_r)\), because \(\alpha_i\) is
decreasing.  The third term was bounded above.  Finally,
\(\beta_i\le\alpha_i\) under Condition~\ref{cond:hyper_1}, so the last
term is also \(O(\alpha_r)\).  This proves
\eqref{eq:ec_projected_fast_residual_bound}.
\end{proof}

\begin{lemma}[Projected comparison on a fast-time block]
\label{lem:ec_projected_fast_skorokhod_comparison}
There is \(C_{T_x}<\infty\), independent of \(r,s\), such that, for
every \(r<s\) satisfying the admissibility condition above, on
\(\mathcal G_{r,s}^{x,\mathrm{proj}}(\delta)\),
\begin{equation}
\max_{r\le k\le s}
\norm{x_k-\widetilde x_k}
\le
C_{T_x}(\delta+\alpha_r).
\label{eq:ec_projected_fast_skorokhod_comparison}
\end{equation}
\end{lemma}

\begin{proof}
Let \(F_x(x,\te):=\havg(x;\te)-x\). On the selected block, embed the
projected stochastic recursion~\eqref{eq:ec_projected_fast_recursion}
and the block-restarted averaged recursion
\eqref{eq:ec_projected_fast_auxiliary} as right-continuous
step paths on the local fast-time grid
\(t_{r,k}:=\sum_{i=r}^{k-1}\alpha_i\). Their unconstrained drivers are
\begin{align*}
v_{r,k}
&:=x_r+\sum_{i=r}^{k-1}\alpha_i
\{F_x(x_i,\te_i)+\Xi_{i+1}+\Delta^h_{i+1}\},\\
\widetilde v_{r,k}
&:=x_r+\sum_{i=r}^{k-1}\alpha_iF_x(\widetilde x_i,\te_i).
\end{align*}
For \(r\le k<s\) and \(t\in[t_{r,k},t_{r,k+1})\), set
\[
x^\circ(t):=x_k,\qquad
\widetilde x^\circ(t):=\widetilde x_k,\qquad
v^\circ(t):=v_{r,k},\qquad
\widetilde v^\circ(t):=\widetilde v_{r,k},
\]
and extend all four paths constantly with their corresponding index-\(s\)
values for \(t\ge t_{r,s}\).
At each jump, the normal correction of the Skorokhod problem is precisely
the correction made by Euclidean projection. Uniqueness of the
Skorokhod problem therefore gives
\(x^\circ=\Gamma_{\cX}(v^\circ)\) and
\(\widetilde x^\circ=\Gamma_{\cX}(\widetilde v^\circ)\).

Let \(L_{F,x}\) be the fast-state Lipschitz constant of \(F_x\), and
write
\(D_k:=\max_{r\le i\le k}\norm{x_i-\widetilde x_i}\).
Proposition~\ref{prop:ec_fast_projection_geometry}(ii),
Lemma~\ref{lem:ec_projected_fast_driver_control}, and the good event give
\[
D_k
\le
L_{\mathrm{Sk}}^x
\left\{
L_{F,x}\sum_{i=r}^{k-1}\alpha_iD_i
+\delta+C_x^{\rm res}\alpha_r
\right\}.
\]
Discrete Gronwall and
\(\sum_{i=r}^{s-1}\alpha_i\le T_x+\alpha_1\) prove
\eqref{eq:ec_projected_fast_skorokhod_comparison}.
\end{proof}

For \(r\le k\le s\), set
\[
\mathcal W_{r,k}^x
:=
\prod_{i=r}^{k-1}(1-q\alpha_i),
\qquad \mathcal W_{r,r}^x:=1.
\]

\begin{lemma}[Moving constrained equilibrium]
\label{lem:ec_projected_fast_moving_equilibrium}
Let
\(L_{fp}^{\cX}:=L_{\bar h,\te}/q\). On every admissible fast-time
interval,
\begin{equation}
\norm{\widetilde x_k-x_{\cX}^\star(\te_k)}
\le
\mathcal W_{r,k}^x
\norm{x_r-x_{\cX}^\star(\te_r)}
+\frac{L_{fp}^{\cX}A_\te^{\cX}}{q}
\frac{\beta_r}{\alpha_r},
\qquad r\le k\le s.
\label{eq:ec_projected_fast_moving_equilibrium}
\end{equation}
\end{lemma}

\begin{proof}
For \(r\le i<s\), the auxiliary recursion
\eqref{eq:ec_projected_fast_auxiliary} can be written as
\[
\widetilde x_{i+1}
=
\mathcal T_{\alpha_i,\te_i}(\widetilde x_i).
\]
As observed in the proof of
Proposition~\ref{prop:ec_projected_fast_equilibrium},
\(x_{\cX}^{\star}(\te_i)\) is a fixed point of every relaxed map
\(\mathcal T_{a,\te_i}\), \(a\in(0,1]\). Hence
\[
\mathcal T_{\alpha_i,\te_i}
\bigl(x_{\cX}^{\star}(\te_i)\bigr)
=
x_{\cX}^{\star}(\te_i).
\]
Adding and subtracting \(x_{\cX}^{\star}(\te_i)\), and then using the
relaxed contraction~\eqref{eq:ec_relaxed_projected_contraction}, the
equilibrium Lipschitz bound~\eqref{eq:ec_projected_fast_fp_lipschitz},
and the slow-increment bound
\eqref{eq:ec_projected_fast_increment_bounds}, gives
\[
\begin{aligned}
\norm{\widetilde x_{i+1}-x_{\cX}^{\star}(\te_{i+1})}
&\le
\norm{
\mathcal T_{\alpha_i,\te_i}(\widetilde x_i)
-
\mathcal T_{\alpha_i,\te_i}
\bigl(x_{\cX}^{\star}(\te_i)\bigr)
}
+
\norm{
x_{\cX}^{\star}(\te_i)
-
x_{\cX}^{\star}(\te_{i+1})
}\\
&\le
(1-q\alpha_i)
\norm{\widetilde x_i-x_{\cX}^{\star}(\te_i)}
+
L_{fp}^{\cX}\norm{\te_{i+1}-\te_i}\\
&\le
(1-q\alpha_i)
\norm{\widetilde x_i-x_{\cX}^{\star}(\te_i)}
+
L_{fp}^{\cX}A_\te^{\cX}\beta_i.
\end{aligned}
\]
Iterate this inequality.  Since
\(\beta_i/\alpha_i\le\beta_r/\alpha_r\) for \(i\ge r\), the forcing
sum is at most
\[
L_{fp}^{\cX}A_\te^{\cX}\frac{\beta_r}{\alpha_r}
\sum_{i=r}^{k-1}\alpha_i
\prod_{\ell=i+1}^{k-1}(1-q\alpha_\ell)
\le
\frac{L_{fp}^{\cX}A_\te^{\cX}}{q}
\frac{\beta_r}{\alpha_r},
\]
where the last sum telescopes after multiplication by \(q\).
\end{proof}

Combining the preceding two lemmas gives, on
\(\mathcal G_{r,s}^{x,\mathrm{proj}}(\delta)\),
\begin{equation}
\begin{aligned}
\norm{x_k-x_{\cX}^\star(\te_k)}
&\le
\mathcal W_{r,k}^x
\norm{x_r-x_{\cX}^\star(\te_r)}\\
&\quad+
C_{T_x}
\left(
\delta+\alpha_r+\frac{\beta_r}{\alpha_r}
\right),
\qquad r\le k\le s.
\end{aligned}
\label{eq:ec_projected_fast_micro_tracking}
\end{equation}
If the interval is full, then
\begin{equation}
\mathcal W_{r,s}^x
\le
\exp\left\{-q\sum_{i=r}^{s-1}\alpha_i\right\}
\le e^{-qT_x}
=:\rho_x<1.
\label{eq:ec_projected_fast_micro_contraction}
\end{equation}

\subsection{Slow-block interface and high-probability event}

For the remainder of this section, define the projected-fast tracking error by
\[
e_n^{(x)}:=\norm{x_n-x_{\cX}^\star(\te_n)},
\qquad n\ge1.
\]

Define the projected-fast event on the \(m\)th slow block by
\begin{equation}
\mathcal H_{x,m}^{\rm proj}(\delta)
:=
\bigcap_{j=0}^{N_m^{\rm sub}-1}
\mathcal G_{\nu_{m,j},\nu_{m,j+1}}^{x,\mathrm{proj}}(\delta).
\label{eq:ec_projected_fast_slow_block_event}
\end{equation}
Thus, \(\mathcal H_{x,m}^{\rm proj}(\delta)\) is the event that, on every
fast-time subblock of the \(m\)th slow block, the maximum accumulated
martingale fluctuation is strictly less than \(\delta\), with the
accumulation restarted at each subblock entrance. On this event,
\eqref{eq:ec_projected_fast_micro_tracking} can be applied successively
across all fast-time subblocks of the slow block.

\begin{proposition}[Projected-fast input to the slow proof]
\label{prop:ec_projected_fast_slow_interface}
For every fixed slow-block horizon \(T\), there are constants
\(c_{x,T},C_{x,T}^{\rm proj}\in(0,\infty)\) such that, for all \(m\ge0\) and
\(\delta>0\),
\begin{equation}
\bP\left(
[\mathcal H_{x,m}^{\rm proj}(\delta)]^c
\right)
\le
2d_xK_m
\exp\left\{-\frac{c_{x,T}\delta^2}{\alpha_{n_m}}\right\}.
\label{eq:ec_projected_fast_slow_block_tail}
\end{equation}
On \(\mathcal H_{x,m}^{\rm proj}(\delta)\),
\begin{equation}
\sum_{r=n_m}^{n_{m+1}-1}
\beta_r
\norm{x_r-x_{\cX}^\star(\te_r)}
\le
C_{x,T}^{\rm proj}
\left(
\delta+\alpha_{n_m}
+\frac{\beta_{n_m}}{\alpha_{n_m}}
\right).
\label{eq:ec_projected_fast_weighted_interface}
\end{equation}
\end{proposition}

\begin{proof}
Applying the union bound to the subblock failure events in
\eqref{eq:ec_projected_fast_slow_block_event} and using
\eqref{eq:ec_projected_fast_micro_tail}, together with
\[
\sum_{j=0}^{N_m^{\rm sub}-1}
(\nu_{m,j+1}-\nu_{m,j})=K_m,
\qquad
\alpha_{\nu_{m,j}}\le\alpha_{n_m},
\]
gives~\eqref{eq:ec_projected_fast_slow_block_tail}, after decreasing
the constant in the exponent if necessary.
We next prove~\eqref{eq:ec_projected_fast_weighted_interface}, which
bounds the cumulative fast-tracking error weighted by the slow step
sizes. Recall \(g_m^{(x)}(\delta)\) from
\eqref{def:block_fast_envelope}; its three terms dominate the
corresponding terms on later subblocks. Iterating the endpoint form of
\eqref{eq:ec_projected_fast_micro_tracking} over the full fast subblocks
and using~\eqref{eq:ec_projected_fast_micro_contraction} gives
\[
e_{\nu_{m,j}}^{(x)}
\le
\rho_x^j e_{n_m}^{(x)}+C_{T_x}g_m^{(x)}(\delta),
\]
up to a change in \(C_{T_x}\). Since
\(x_{n_m},x_{\cX}^\star(\te_{n_m})\in\cX\), compactness of \(\cX\) gives
a deterministic bound on \(e_{n_m}^{(x)}\).

On an admissible subblock with integer endpoints \(r<s\), monotonicity of \(\beta_i/\alpha_i\) and the
telescoping identity used in the proof of
Lemma~\ref{lem:ec_projected_fast_moving_equilibrium} yield
\[
\sum_{i=r}^{s-1}\beta_i\mathcal W_{r,i}^x
\le
\frac{\beta_r}{q\alpha_r}.
\]
Also,
\[
\sum_{i=r}^{s-1}\beta_i
\le
\frac{\beta_r}{\alpha_r}H_{m,j}^x
\le C_{T_x}\frac{\beta_r}{\alpha_r}.
\]
Apply these bounds to~\eqref{eq:ec_projected_fast_micro_tracking} and
sum over \(j\). By the preceding entrance-error bound, the part of
\(e_{\nu_{m,j}}^{(x)}\) inherited from the initial slow-block error is
bounded by \(\rho_x^j e_{n_m}^{(x)}\), and hence decays geometrically across
the fast subblocks. The remaining terms are controlled using
\[
\sum_{j=0}^{N_m^{\rm sub}-1}
\frac{\beta_{\nu_{m,j}}}{\alpha_{\nu_{m,j}}}
\le C_T
\]
because each full subblock has fast duration at least \(T_x\) and the total
slow duration is at most \(T+\beta_{n_m}\). We verify this bound explicitly.
Fixed fast duration and the power-step
formulas make \((N_0+i)/(N_0+r)\) uniformly bounded on each admissible block with endpoints $r<s$, 
so \(\beta_i/\alpha_i\) and \(\beta_r/\alpha_r\) are uniformly comparable there.
Consequently, on each full subblock,
\[
\beta_r/\alpha_r \le \frac{C}{T_x}\sum_{i=r}^{s-1}\beta_i.
\]
Summing over the full subblocks gives a constant times the slow duration;
for the possibly short final subblock, monotonicity gives
\(\beta_r/\alpha_r \le \beta_{n_m}/\alpha_{n_m} \). Moreover,
\(\sum_{j=0}^{N_m^{\rm sub}-1}\rho_x^j \frac{\beta_{\nu_{m,j}}}{\alpha_{\nu_{m,j}}} \le \frac{\beta_{n_m}}{\alpha_{n_m}(1-\rho_x)}\). Thus the initial
transient is multiplied by \(\beta_{n_m}/\alpha_{n_m}\), already a term in
\(g_m^{(x)}(\delta)\), and no factor \(N_m^{\rm sub}\) is lost. This proves
\eqref{eq:ec_projected_fast_weighted_interface}.
\end{proof}
Equation~\eqref{eq:ec_projected_fast_weighted_interface} is the constrained analogue of~\eqref{eq:weighted-fast-input-section6}, which controls the weighted fast-tracking term in the original--auxiliary comparison~\eqref{ineq:te_tte_block}.
After establishing the constrained slow Poisson and martingale bounds
below, the proofs of Propositions~\ref{prop:te_te1}
and~\ref{prop:te4_z} carry over under the following
substitutions:
\begin{equation}
x^\star(\te) \longmapsto x_{\cX}^\star(\te),
\qquad
\favg \longmapsto \favg_{\cX},
\qquad
\cG_x^{(m)}(\delta)
\longmapsto
\mathcal H_{x,m}^{\rm proj}(\delta).
\label{eq:ec_projected_fast_substitutions}
\end{equation}
To make the slow-coordinate transfer explicit, define the function describing the slow Markov noise
\begin{equation}
g_{f,\cX}^{(\te)}(y)
:=
f\bigl(y,x_{\cX}^\star(\te),\te\bigr)
-\favg_{\cX}(\te;\te),
\qquad y\in\cY,
\label{eq:ec_projected_fast_slow_forcing}
\end{equation}
and its return-time Poisson solution
\begin{equation}
u_{f,\cX}^{(\te)}(y)
:=
\bE^{(\te)}\left[
\sum_{r=0}^{\tau-1}g_{f,\cX}^{(\te)}(Y_r)
\,\middle|\,Y_0=y
\right],
\label{eq:ec_projected_fast_slow_poisson}
\end{equation}
where \(\tau\) is the return time used in
Definition~\ref{def:poisson_slow}.  For \(n\ge1\), set
\begin{equation}
M_{n+1}^{\prime\prime\prime,\cX}
:=
u_{f,\cX}^{(\te_n)}(Y_{n+1})
-
\sum_{y\in\cY}p^{(\te_n)}(Y_n,y)
u_{f,\cX}^{(\te_n)}(y).
\label{eq:ec_projected_fast_slow_mds}
\end{equation}
Then \(\{M_{n+1}^{\prime\prime\prime,\cX}\}\) is a
martingale-difference sequence and
\begin{equation}
g_{f,\cX}^{(\te_n)}(Y_n)
=
u_{f,\cX}^{(\te_n)}(Y_n)
-u_{f,\cX}^{(\te_n)}(Y_{n+1})
+M_{n+1}^{\prime\prime\prime,\cX}.
\label{eq:ec_projected_fast_slow_poisson_decomposition}
\end{equation}
By compactness of \(\cX\) and \(\Te\),
\eqref{eq:ec_projected_fast_fp_lipschitz}, and the same killed-resolvent
argument used for Lemma~\ref{lemma:Poisson}, there are finite deterministic
constants \(U_{f,\cX}\), \(C_{f,\cX}\), and \(C_{\mathrm{mg}}^{(\theta,\cX)}\) such that,
for every \(\te,\te'\in\Te\),
\begin{equation}
\sup_{\vartheta\in\Te,\,y\in\cY}\norm{u_{f,\cX}^{(\vartheta)}(y)} \le U_{f,\cX},
\qquad
\sup_{y\in\cY}\norm{u_{f,\cX}^{(\te)}(y)-u_{f,\cX}^{(\te')}(y)}
\le C_{f,\cX}\norm{\te-\te'},
\label{eq:ec_projected_fast_slow_poisson_bounds}
\end{equation}
and, for every \(n\ge1\),
\(\norm{M'_{n+1}+M_{n+1}^{\prime\prime\prime,\cX}}
\le C_{\mathrm{mg}}^{(\theta,\cX)}\) almost surely.  Define
\begin{equation}
\mathcal G_{\te,\cX}^{(m)}(\delta_s)
:=
\left\{
\max_{1\le k\le K_m}
\norm{
\sum_{\ell=0}^{k-1}\beta_{n_m+\ell}
\bigl(M'_{n_m+\ell+1}
+M_{n_m+\ell+1}^{\prime\prime\prime,\cX}\bigr)
}<\delta_s
\right\}.
\label{eq:ec_projected_fast_slow_good_event}
\end{equation}
The coordinatewise Azuma--Hoeffding argument gives
\begin{equation}
\bP\left(
[\mathcal G_{\te,\cX}^{(m)}(\delta_s)]^c
\right)
\le
2d_{\te}\sum_{k=1}^{K_m}
\exp\left\{
-\frac{\delta_s^2}
{2\kappa_{d_{\te}}^2(C_{\mathrm{mg}}^{(\theta,\cX)})^2
\{b(n_m)-b(n_m+k)\}}
\right\}.
\label{eq:ec_projected_fast_slow_tail}
\end{equation}

\begin{proposition}[Localized slow-ODE tracking with both projections]
\label{prop:ec_projected_fast_tube}
Fix a slow-block horizon \(T\). There is
\(C_{B,T}^{\cX}<\infty\) such that, for every \(m\ge0\) and
\(\delta_f,\delta_s>0\), on
\[
\mathcal H_{x,m}^{\rm proj}(\delta_f)
\cap
\mathcal G_{\te,\cX}^{(m)}(\delta_s),
\]
\begin{equation}
\sup_{0\le t\le H_m}
\norm{\te^{(m),\circ}(t)-z_{\cX}^{(T_m)}(t)}
\le
C_{B,T}^{\cX}
\left(
\delta_f+\delta_s+\alpha_{n_m}
+\frac{\beta_{n_m}}{\alpha_{n_m}}
\right),
\label{eq:ec_projected_fast_tube}
\end{equation}
where \(z_{\cX}^{(T_m)}\) solves
\eqref{eq:ec_projected_fast_reduced_ode} from \(\te_{n_m}\). Moreover, the
failure probability of this event is at most the sum of the right-hand
sides of \eqref{eq:ec_projected_fast_slow_block_tail} and
\eqref{eq:ec_projected_fast_slow_tail}.
\end{proposition}

\begin{proof}
Use \eqref{eq:ec_projected_fast_slow_poisson_decomposition} in the original
slow-driver decomposition. Bounds
\eqref{eq:ec_projected_fast_fp_lipschitz} and
\eqref{eq:ec_projected_fast_slow_poisson_bounds} give the same Poisson
residual, auxiliary-recursion, and Skorokhod estimates as in
Section~\ref{sec:analysis_te_n}, with
\(x^\star\), \(\favg\), and the stepsize-weighted fast input replaced
according to
\eqref{eq:ec_projected_fast_substitutions}. Proposition
\ref{prop:ec_projected_fast_slow_interface} supplies that stepsize-weighted
input. The same summation-by-parts calculation gives a deterministic
constrained slow-Poisson residual bounded by
\(C\{\beta_{n_m}+b(n_m)-b(n_{m+1})\}=O(\beta_{n_m})\).
The union bound gives the stated failure probability.
\end{proof}

For a rate-specific choice, set
\begin{equation}
L_m^{x,\mathrm{proj}}
:=
\log\{4d_xK_m(m+2)^2\},
\qquad
\delta_{\rm dev}^{(x,\mathrm{proj})}(m)
:=
\left(
\frac{\alpha_{n_m}L_m^{x,\mathrm{proj}}}{c_{x,T}}
\right)^{1/2}.
\label{eq:ec_projected_fast_rate_threshold}
\end{equation}
Then~\eqref{eq:ec_projected_fast_slow_block_tail} gives
\[
\bP\left(
[\mathcal H_{x,m}^{\rm proj}
(\delta_{\rm dev}^{(x,\mathrm{proj})}(m))]^c
\right)
\le\frac1{2(m+2)^2}.
\]
Combine this event with
\(\mathcal G_{\te,\cX}^{(m)}
(\delta_{\rm dev}^{(\sm,\cX)}(m))\), where
\[
\delta_{\rm dev}^{(\sm,\cX)}(m)
:=
\kappa_{d_{\te}} C_{\mathrm{mg}}^{(\theta,\cX)}
\sqrt{2D_mL_m^\te}.
\]
Equations~\eqref{eq:ec_projected_fast_slow_tail} and
\eqref{eq:ec_projected_fast_slow_block_tail}, followed by
Borel--Cantelli, give the same eventual pathwise slow-block event as in
Section~\ref{sec:conv_rate}, now with
\begin{equation}
d_{m,\cX}^{\rm rate}
:=
\delta_{\rm dev}^{(x,\mathrm{proj})}(m)
+\delta_{\rm dev}^{(\sm,\cX)}(m)
+\alpha_{n_m}
+\frac{\beta_{n_m}}{\alpha_{n_m}}.
\label{eq:ec_projected_fast_rate_envelope}
\end{equation}
On the resulting event,
Proposition~\ref{prop:ec_projected_fast_tube} gives
\begin{equation}
\sup_{0\le t\le H_m}
\norm{\te^{(m),\circ}(t)-z_{\cX}^{(T_m)}(t)}
\le C_{B,T}^{\cX}d_{m,\cX}^{\rm rate},
\label{eq:ec_projected_fast_slow_ode_tracking}
\end{equation}

\subsection{Almost-sure rates, including the harmonic slow step}

The preceding local construction is enough for the slow proof. For
completeness, a global partition of the fast clock gives the iteration-wise
fast rate. Under the shifted-step condition~\eqref{cond:N0:alpha2}, put
\(\nu_0:=1\); under the finite-prefix alternative, put \(\nu_0:=n_0\), where
\(n_0\) is the suffix index defined in Appendix~\ref{sec:N0_cond}. Define
\begin{equation}
\nu_{j+1}
:=
\inf\left\{n>\nu_j:
\sum_{r=\nu_j}^{n-1}\alpha_r\ge T_x
\right\}.
\label{eq:ec_projected_fast_global_blocks}
\end{equation}
Apply~\eqref{eq:ec_projected_fast_micro_tracking} on these successive
blocks with
\[
\delta_j=C_\delta
\sqrt{\alpha_{\nu_j}\log(N_0+\nu_j)},
\]
where \(C_\delta\) is chosen so that \(c_xC_\delta^2>1\). Since
\(N_0+\nu_j\asymp(j+1)^{1/(1-\fraka)}\), integral comparison for the
power steps gives
\(\nu_{j+1}-\nu_j=O((N_0+\nu_j)^{\fraka})\). Hence the right-hand side of
\eqref{eq:ec_projected_fast_micro_tail} is at most
\[
C(N_0+\nu_j)^{\fraka-c_xC_\delta^2},
\]
which is summable in \(j\). Therefore,
eventually almost surely,
\begin{equation}
e_{\nu_{j+1}}^{(x)}
\le
\rho_xe_{\nu_j}^{(x)}
+C_{T_x}\left\{
\sqrt{\alpha_{\nu_j}\log(N_0+\nu_j)}
+\alpha_{\nu_j}
+\frac{\beta_{\nu_j}}{\alpha_{\nu_j}}
\right\}.
\label{eq:ec_projected_fast_endpoint_recursion}
\end{equation}
A geometric convolution of this recursion has the same order as its
right-hand envelope. On a global fast block,
\((N_0+n)/(N_0+\nu_j)\) is uniformly bounded and tends to one; applying
the admissible-interval version of
\eqref{eq:ec_projected_fast_micro_tracking} to the final partial block
therefore gives
\begin{equation}
\norm{x_n-x_{\cX}^\star(\te_n)}
=
O_{\rm a.s.}\left(
(N_0+n)^{-r_x(\fraka,\frakb)}\sqrt{\log(N_0+n)}
\right).
\label{eq:ec_projected_fast_as_rate}
\end{equation}

The next theorem records the complete transfer to the slow results.

\begin{theorem}[Rates with projection on both time scales]
\label{th:ec_projected_fast_rates}
Let \(\cX\) be a nonempty compact convex polyhedron and equip the fast
coordinate with the Euclidean norm. Suppose
Condition~\ref{cond:hyper_1} and the standing controlled-Markov,
Euclidean contraction, Lipschitz, bounded-noise, and hitting-time
assumptions hold on \(\cY\times\cX\times\Te\). Use either the
shifted-step condition~\eqref{cond:N0:alpha2} or the finite-prefix alternative
of Appendix~\ref{sec:N0_cond}. Let \(J\) be a Lipschitz Lyapunov function
satisfying the constrained analogue of
Assumption~\ref{assum:ode_stability} for
\eqref{eq:ec_projected_fast_reduced_ode}. Assumption
\ref{assum:xn_bound} is not required.
Then~\eqref{eq:ec_projected_fast_as_rate} holds.  The following slow
conclusions also hold.
\begin{enumerate}[(i)]
\item Under the constrained analogue of the general one-block decrease
condition in Assumption~\ref{assum:one_block_decrease},
\[
J(\te_n)\longrightarrow0,
\qquad
\operatorname{dist}
\bigl(\te_n,\mathcal E\mathcal Q_{\Te}^{\cX}\bigr)\longrightarrow0
\quad\text{almost surely},
\]
where \(\mathcal E\mathcal Q_{\Te}^{\cX}\) is defined in
\eqref{eq:ec_projected_fast_equilibrium_set}.
\item Under the hypotheses of part~(i), suppose additionally that
there exist \(p_J>1\), \(\kappa_J>0\), and \(u_J>0\) such that
\[
\mathcal D_J(u)\ge\kappa_Ju^{p_J},
\qquad 0\le u\le u_J.
\]
If \(\frakb<1\), then
\begin{equation}
J(\te_n)
=O_{\rm a.s.}\!\left(
(N_0+n)^{-(1-\frakb)/(p_J-1)}
+(N_0+n)^{-r_x(\fraka,\frakb)/p_J}
\{\log(N_0+n)\}^{1/(2p_J)}
\right).
\label{eq:ec_projected_fast_power_rate}
\end{equation}
If \(\frakb=1\), then
\begin{equation}
J(\te_n)
=O_{\rm a.s.}\!\left(
\{\log(N_0+n)\}^{-1/(p_J-1)}
\right).
\label{eq:ec_projected_fast_power_rate_bone}
\end{equation}
Thus the harmonic rate is \(O_{\rm a.s.}(1/\log n)\) when \(p_J=2\).
\item Under the constrained analogue of the linear one-block contraction
condition in Assumption~\ref{assum:one_block_contraction}, if
\(\frakb<1\), then
\begin{equation}
J(\te_n)
=O_{\rm a.s.}\left(
(N_0+n)^{-r_x(\fraka,\frakb)}\sqrt{\log(N_0+n)}
\right).
\label{eq:ec_projected_fast_contraction_rate}
\end{equation}
If \(\frakb=1\), put
\[
\lambda_T^{\cX}
:=
-\frac{\log\rho_T^{\cX}}{T}>0.
\]
Then
\begin{equation}
J(\te_n)
=
\begin{cases}
O_{\rm a.s.}\!\left((N_0+n)^{-r_x(\fraka,1)}\sqrt{\log(N_0+n)}\right),
& \lambda_T^{\cX} >r_x(\fraka,1),\\[1mm]
O_{\rm a.s.}\!\left((N_0+n)^{-r_x(\fraka,1)}
\{\log(N_0+n)\}^{3/2}\right),
& \lambda_T^{\cX} =r_x(\fraka,1),\\[1mm]
O_{\rm a.s.}\!\left((N_0+n)^{-\lambda_T^{\cX} }\right),
&0< \lambda_T^{\cX} <r_x(\fraka,1).
\end{cases}
\label{eq:ec_projected_fast_contraction_rate_bone}
\end{equation}
At \(\frakb=1\) and \(\fraka=2/3\),
\begin{equation}
\norm{x_n-x_{\cX}^\star(\te_n)}
=O_{\rm a.s.}\left(
(N_0+n)^{-1/3}\sqrt{\log(N_0+n)}
\right),
\label{eq:ec_projected_fast_optimized_bone}
\end{equation}
and the slow rate is given by
\eqref{eq:ec_projected_fast_contraction_rate_bone} with
\(r_x(2/3,1)=1/3\).
\end{enumerate}
If, in addition, the original global
Assumption~\ref{assum:contraction} holds and
\eqref{eq:ec_projection_target_compatibility} holds, then
\(\mathcal E\mathcal Q_{\Te}^{\cX}=\eq\), and every target and reduced ODE in the theorem is
the same as in the main paper.
The shifted-step-size and finite-prefix alternatives of
Appendix~\ref{sec:N0_cond} carry over without change: the global fast-time
construction begins at \(1\) in the former case and at \(n_0\) in the latter.
\end{theorem}

\begin{proof}
By the block-geometry lemmas,
\eqref{eq:ec_projected_fast_rate_threshold} and the constrained slow
threshold satisfy
\[
\delta_{\rm dev}^{(x,\mathrm{proj})}(m)
=O\left(
(N_0+n_m)^{-\fraka/2}
\sqrt{\log(N_0+n_m)}
\right),
\]
and
\[
\delta_{\rm dev}^{(\sm,\cX)}(m)
=O\left(
(N_0+n_m)^{-\frakb/2}
\sqrt{\log(N_0+n_m)}
\right).
\]
Therefore
\[
d_{m,\cX}^{\rm rate}
=O\left(
(N_0+n_m)^{-r_x(\fraka,\frakb)}\sqrt{\log(N_0+n_m)}
\right).
\]
Propositions~\ref{prop:ec_projected_fast_slow_interface} and
\ref{prop:ec_projected_fast_tube} provide, respectively, the
stepsize-weighted fast input and the localized slow-ODE tracking estimate used in the
block-endpoint proofs.  Those deterministic proofs use only these estimates,
the Lipschitz property of \(J\), and the stated Lyapunov decrease or
contraction of the relevant flow. Repeating them with the constrained
analogues defined above gives all the stated slow conclusions, including the
harmonic branch. Finally,
\eqref{eq:ec_projected_fast_as_rate} follows from
\eqref{eq:ec_projected_fast_endpoint_recursion}; setting
\((\fraka,\frakb)=(2/3,1)\) gives
\eqref{eq:ec_projected_fast_optimized_bone}.
\end{proof}

\section{Auxiliary Proofs Used for Main-Paper Section~\ref{sec:faster_iterations}}
\label{sec:aux_app}

\begin{lemma}
\label{lemma:poisson_sensitivity_1}
The following bounds hold.
\begin{enumerate}[(i)]
\item \label{poisson_sens_I}
For every $k\ge1$,
\[
\norm{x_{k+1}-x_k}\le A_x\alpha_k,
\]
and
\[
\norm{\te_{k+1}-\te_k}\le A_\te\beta_k.
\]
\item \label{poisson_h_bound}
The solution~\eqref{def:poisson_soln_h} of the fast Poisson equation satisfies
\[
\norm{u_h^{(x,\te)}(y)}\le 2\frakc_5\mathsf B_6,\qquad x\in\bR^{d_x},\ \te\in\Te,\ y\in\cY,
\]
whenever $\norm{x}\le\frakc_6$.
\item \label{bound:m_2_prime}
The martingale difference sequence defined in~\eqref{def:M''_n} satisfies
\[
\norm{M''_{n+1}}\le 4\frakc_5\mathsf B_6,\qquad n\ge1.
\]
\item \label{lemma:B2.3}
Define
\[
C_{h,\mathrm{Pois.}}
:=
\max\left\{
2\frakc_5\lhx,
\frakc_5\left[
2\lhte+
\bigl(\lmu+2\lp\frakc_5\bigr)
\mathsf B_6
\right]
\right\}.
\]
For all $x,x'\in\bR^{d_x}$ with $\norm{x},\norm{x'}\le\frakc_6$, all $\te,\te'\in\Te$, and all $y\in\cY$,
\[
\norm{u^{(x,\te)}_h(y)-u^{(x',\te')}_h(y)}
\le \chpoi \left(\norm{x-x'}+\norm{\te-\te'}\right).
\]
\end{enumerate}
\end{lemma}

\begin{proof}
For part (i), the update equations, Assumptions~\ref{assum:1_2}
and~\ref{assum:xn_bound}, and the definition of $\mathsf B_6$ give
\[
\begin{aligned}
\norm{x_{k+1}-x_k}
&\le A_x\alpha_k,\\
\norm{\te_{k+1}-\te_k}
&\le A_\te\beta_k,
\end{aligned}
\]
where the second inequality also uses nonexpansiveness of projection.

For parts (ii)--(iv), write
\[
g_{x,\te}(y):=h(y,x,\te)-h^{(\mathrm{av.})}(x;\te).
\]
When $\norm{x}\le \frakc_6$, $\sup_y\norm{g_{x,\te}(y)}\le2\mathsf B_6$. The
return-time representation \eqref{def:poisson_soln_h} and Assumption~\ref{assum:hitting_time} therefore imply
\[
\norm{u_h^{(x,\te)}(y)}\le2\frakc_5\mathsf B_6,
\]
proving part (ii). Part (iii) follows from \eqref{def:M''_n}, since
\[
\norm{M''_{n+1}}
\le2\sup_y\norm{u_h^{(x,\te)}(y)}
\le4\frakc_5\mathsf B_6.
\]

It remains to prove part (iv). Let
$\cY_0:=\cY\setminus\{y\ust\}$. The claim is immediate if
$\cY_0$ is empty. For vector-valued functions on $\cY_0$, use
$\norm{v}_{\infty,\mathrm{vec}}:=\max_y\norm{v(y)}$; then
\[
\norm{Av}_{\infty,\mathrm{vec}}
\le\norm{A}_\infty\norm{v}_{\infty,\mathrm{vec}}
\]
for every scalar matrix $A$, with no dimension-dependent factor.

Let $Q^{(\te)}$ be the restriction of $p^{(\te)}$ to
$\cY_0\times\cY_0$. The hitting-time identity gives
\begin{equation}
\begin{aligned}
Z^{(\te)}
&:=\sum_{n\ge0}(Q^{(\te)})^n
=(I-Q^{(\te)})^{-1},\\
\sup_{\te\in\Te}\norm{Z^{(\te)}}_\infty&\le \frakc_5.
\end{aligned}
\label{eq:fast_killed_resolvent_bound}
\end{equation}
Assumption~\ref{assum:hitting_time} makes $y\ust$ positive recurrent and gives
$\mu^{(\te)}(y\ust)>0$. Since
$\sum_y\mu^{(\te)}(y)g_{x,\te}(y)=0$, Lemma~\ref{lemma:finite_state_cycle_formula}, applied
coordinatewise to the return-time representation \eqref{def:poisson_soln_h}, gives
$u_h^{(x,\te)}(y\ust)=0$. First-step decomposition then yields
\begin{equation}
u_h^{(x,\te)}=Z^{(\te)}g_{x,\te}
\quad\text{on }\cY_0.
\label{eq:fast_poisson_resolvent_representation}
\end{equation}

Set $\Delta_x:=\norm{x-x'}$ and $\Delta_\te:=\norm{\te-\te'}$.
Assumption~\ref{assum:lipschitz}, the bound $\mathsf B_6$, and
Lemma~\ref{lemma:lipschitz_stat_dist} give
\[
\norm{h^{(\mathrm{av.})}(x;\te)-h^{(\mathrm{av.})}(x';\te')}
\le \lhx\Delta_x+(\lhte+\lmu\mathsf B_6)\Delta_\te,
\]
and hence
\begin{equation}
\norm{g_{x,\te}-g_{x',\te'}}_{\infty,\mathrm{vec}}
\le2\lhx\Delta_x+(2\lhte+\lmu\mathsf B_6)\Delta_\te.
\label{eq:fast_centered_forcing_sensitivity}
\end{equation}
\begin{samepage}
Also, $\norm{g_{x',\te'}}_{\infty,\mathrm{vec}}\le2\mathsf B_6$.
Since
\[
\norm{Q^{(\te)}-Q^{(\te')}}_\infty\le \lp\Delta_\te,
\]
the resolvent identity and \eqref{eq:fast_killed_resolvent_bound} yield
\[
\norm{Z^{(\te)}-Z^{(\te')}}_\infty
\le \frakc_5^2\lp\Delta_\te.
\]
\end{samepage}
Using \eqref{eq:fast_poisson_resolvent_representation}--\eqref{eq:fast_centered_forcing_sensitivity},
\[
\begin{aligned}
\sup_{y\in\cY_0}
\norm{u_h^{(x,\te)}(y)-u_h^{(x',\te')}(y)}
\le{}&2\frakc_5\lhx\Delta_x\\
&+\frakc_5\bigl[2\lhte+(\lmu+2\lp\frakc_5)\mathsf B_6\bigr]\Delta_\te.
\end{aligned}
\]
Both solutions vanish at $y\ust$, so this bound holds on all of $\cY$.
Taking the larger coefficient gives
$\chpoi$, proving part (iv).
\end{proof}

\section{Auxiliary Results Used for Main-Paper Section~\ref{sec:analysis_te_n}}
\begin{lemma}\label{lemma:Poisson}
Consider the Poisson equation~\eqref{def:poisson_f}. The following bounds hold.
\begin{enumerate}[(i)]
\item \label{lemma:poisson_part_1}
For every $\te\in\Te$ and $y\in\cY$,
\[
\norm{u^{(\te)}_f(y)}\le 2\frakc_5\mathsf B_8.
\]
\item \label{lemma:poisson_2}
For every $\te,\te'\in\Te$ and $y\in\cY$,
\[
\norm{u^{(\te)}_f(y)-u^{(\te')}_f(y)}\le \cfpoi \norm{\te-\te'},
\]
where
\[
C_{f,\mathrm{Pois.}}
:=
\frakc_5\left[
2L_f(L_{fp}+1)
+
\bigl(L_\mu+2L_p\frakc_5\bigr)
\mathsf B_8
\right].
\]
\item \label{lemma:Poisson_3}
The martingale difference sequence defined in~\eqref{def:M_3'} satisfies
\[
\norm{(M''')^{(m)}_{n+1}}\le 4\frakc_5\mathsf B_8.
\]
\end{enumerate}
\end{lemma}

\begin{proof}
Set
\[
g_\te(y):=f(y,x^\star(\te),\te)-f^{(\mathrm{av.})}(\te;\te).
\]
The definition of $\mathsf B_8$ gives $\sup_y\norm{g_\te(y)}\le2\mathsf B_8$. Hence the
return-time representation and Assumption~\ref{assum:hitting_time} imply
\[
\norm{u_f^{(\te)}(y)}\le2\frakc_5\mathsf B_8,
\]
which proves part (i).

For part (ii), let $\Delta_\te:=\norm{\te-\te'}$. Assumption~\ref{assum:lipschitz}
and Lemma~\ref{lemma:lipschitz_stat_dist} give
\[
\sup_y\norm{
f(y,x^\star(\te),\te)
-f(y,x^\star(\te'),\te')}
\le L_f(\lfp+1)\Delta_\te.
\]
Adding and subtracting the stationary average with measure $\mu^{(\te)}$ gives
\[
\norm{f^{(\mathrm{av.})}(\te;\te)
-f^{(\mathrm{av.})}(\te';\te')}
\le\bigl[L_f(\lfp+1)+\lmu\mathsf B_8\bigr]\Delta_\te.
\]
Use the vector-valued sup norm introduced in the proof of Lemma~\ref{lemma:poisson_sensitivity_1}.
The bounded-set estimate above gives
\begin{equation}
\norm{g_{\te'}}_{\infty,\mathrm{vec}}\le2\mathsf B_8.
\label{eq:slow_centered_forcing_bound}
\end{equation}
Combining the preceding two Lipschitz bounds gives
\begin{equation}
\norm{g_\te-g_{\te'}}_{\infty,\mathrm{vec}}
\le\bigl[2L_f(\lfp+1)+\lmu\mathsf B_8\bigr]\Delta_\te.
\label{eq:slow_centered_forcing_sensitivity}
\end{equation}
These are the only bounds on this function needed.

If $\cY_0:=\cY\setminus\{y\ust\}$ is empty, the claim is
immediate. Otherwise, use the killed kernels $Q^{(\te)}$ and resolvents
$Z^{(\te)}$ constructed in the proof of Lemma~\ref{lemma:poisson_sensitivity_1}. Assumption~\ref{assum:hitting_time}
and Lemma~\ref{lemma:finite_state_cycle_formula}, applied coordinatewise to the return-time representation
\eqref{def:poisson_f_soln}, give $u_f^{(\te)}(y\ust)=0$; first-step decomposition then gives
\[
u_f^{(\te)}=Z^{(\te)}g_\te
\quad\text{on }\cY_0,\qquad
\norm{Z^{(\te)}}_\infty\le \frakc_5.
\]
As in that proof,
\[
\norm{Z^{(\te)}-Z^{(\te')}}_\infty
\le \frakc_5^2\lp\Delta_\te.
\]
Therefore, adding and subtracting $Z^{(\te)}g_{\te'}$ and using
\eqref{eq:slow_centered_forcing_bound}--\eqref{eq:slow_centered_forcing_sensitivity},
\[
\sup_{y\in\cY}
\norm{u_f^{(\te)}(y)-u_f^{(\te')}(y)}
\le \frakc_5\Bigl[2L_f(\lfp+1)+(\lmu+2\lp\frakc_5)\mathsf B_8\Bigr]\Delta_\te.
\]
The coefficient is $\cfpoi$, proving part (ii). Finally,
\eqref{def:M_3'} and part (i) give
\[
\norm{(M''')_{n+1}^{(m)}}
\le2\sup_y\norm{u_f^{(\te)}(y)}
\le4\frakc_5\mathsf B_8,
\]
which proves part (iii).
\end{proof}

\begin{lemma}\label{lemma:mu_p_g}
For all $\te,\te',\tte\in\Te$, the following hold:
\begin{enumerate}[(i)]
\item
\[
\norm{\favg(\te';\te)-\favg(\tte;\te)}
\le \lf (\lfp +1)\norm{\te'-\tte}.
\]
\item \label{lemma:mu_p_g-ii}
\[
\norm{\favg(\te;\te)-\favg(\te';\te')}
\le \lfavg\norm{\te-\te'},
\]
where
\[
\lfavg
:=
L_f(L_{fp}+1)
+
L_\mu\mathsf B_8.
\]
\item \label{lemma:mu_p_g-iii}
\[
\norm{\favg(\te;\te')-\favg(\te;\tte)}
\le \lmu\mathsf B_8\norm{\te'-\tte}.
\]
\end{enumerate}
\end{lemma}

\begin{proof}
For part~(i),
\[
\begin{aligned}
\norm{\favg(\te';\te)-\favg(\tte;\te)}
&\le \sum_{y\in\cY}\mu\ute(y)\norm{f(y,x\ust(\te'),\te')-f(y,x\ust(\tte),\tte)}\\
&\le \lf \left(\norm{x\ust(\te')-x\ust(\tte)}+\norm{\te'-\tte}\right)\\
&\le \lf (\lfp +1)\norm{\te'-\tte}.
\end{aligned}
\]
For part~(ii), add and subtract $\sum_{y\in\cY}\mu\ute(y)f(y,x\ust(\te'),\te')$ and use part~(i), Lemma~\ref{lemma:lipschitz_stat_dist}, and the bounded-set estimate
\[
\norm{f(y,x\ust(\te'),\te')}\le \mathsf B_8.
\]
Part~(iii) follows directly from Lemma~\ref{lemma:lipschitz_stat_dist} and the same bounded-set estimate.
\end{proof}

\begin{lemma}[Residual term in the Poisson decomposition]
\label{lemma:adhoc_1}
For every $1 \le n\le K_m$,
\[
\left\|\sum_{k=0}^{n-1}\beta\um_k\left(u^{(\te\um_k)}_f(Y\um_k)-u^{(\te\um_k)}_f(Y\um_{k+1})\right)\right\|
\le R_m,
\]
where
\begin{equation}
R_m =3C_P\beta_{n_m}+C_L\br{b(n_m)-b(n_{m+1})},\label{def:R_m}
\end{equation}
with
\[
C_P:=2\frakc_5\mathsf B_8,\qquad
C_L:=\cfpoi(\frakc_1+\frakc_2\frakc_6+\mathsf B_6)=\cfpoi A_\te.
\]
Here \(C_{f,\mathrm{Pois.}}\) is the constant defined in
part~(\ref{lemma:poisson_2}) of Lemma~\ref{lemma:Poisson}.
\end{lemma}
\begin{proof}[Proof of Lemma~\ref{lemma:adhoc_1}]
The sum can be written as follows:
\[
\begin{aligned}
&\sum_{k=0}^{n-1}\beta\um_k\left(u^{(\te\um_k)}_f(Y\um_k)-u^{(\te\um_k)}_f(Y\um_{k+1})\right) \\
&\quad=\beta\um_0 u^{(\te\um_0)}_f(Y\um_0)
-\beta\um_{n-1} u^{(\te\um_{n-1})}_f(Y\um_n)\\
&\qquad+\sum_{k=1}^{n-1}\Bigl(
\beta\um_k u^{(\te\um_k)}_f(Y\um_k)
-\beta\um_{k-1}u^{(\te\um_{k-1})}_f(Y\um_k)\Bigr).
\end{aligned}
\]
The two boundary terms are bounded by
\[
\beta^{(m)}_0C_P+\beta^{(m)}_{n-1}C_P\le2C_P\beta_{n_m}.
\]
For the summation term, write
\[
\begin{aligned}
&\beta\um_k u^{(\te\um_k)}_f(Y\um_k)
-\beta\um_{k-1}u^{(\te\um_{k-1})}_f(Y\um_k)\\
&\quad=
\beta\um_{k-1}\bigl\{u^{(\te\um_k)}_f(Y\um_k)
-u^{(\te\um_{k-1})}_f(Y\um_k)\bigr\}\\
&\qquad+(\beta\um_k-\beta\um_{k-1})u^{(\te\um_k)}_f(Y\um_k).
\end{aligned}
\]
By part~(\ref{lemma:poisson_2}) of Lemma~\ref{lemma:Poisson} and the
update bound for $\te\um_k$,
\[
\norm{u^{(\te\um_k)}_f(Y\um_k)-u^{(\te\um_{k-1})}_f(Y\um_k)}
\le
\beta\um_{k-1} C_L.
\]
Also, $\norm{u^{(\te^{(m)}_k)}_f(Y^{(m)}_k)}\le C_P$ and
\[
\sum_{k=1}^{n-1}|\beta^{(m)}_k-\beta^{(m)}_{k-1}|\le\beta^{(m)}_0=\beta_{n_m}.
\]
Combining these estimates gives
\[
\left\|\sum_{k=0}^{n-1}\beta^{(m)}_k\left(u^{(\te^{(m)}_k)}_f(Y^{(m)}_k)-u^{(\te^{(m)}_k)}_f(Y^{(m)}_{k+1})\right)\right\|
\le
3C_P\beta_{n_m}+C_L\{b(n_m)-b(n_{m+1})\}=R_m.
\]
\end{proof}

\section{Strong monotonicity and exponential stability of projected ODEs}
\label{app:strong_monotonicity}

\begin{definition}
\label{def:strong_monotone}
Let \(K\subset\bR^d\) be a nonempty closed convex set and let
\(F:K\to\bR^d\). Define
\[
\Pi_K(x,v):=\lim_{h\downarrow0}
\frac{\operatorname{proj}_K(x+hv)-x}{h},
\qquad x\in K,\quad v\in\bR^d,
\]
and consider the projected dynamical system
\[
\dot x=\Pi_K(x,F(x)).
\]
The vector field \(F\) is strongly monotone at \(x^\star\in K\) if
there exists \(\mu_{\rm sm}>0\) such that
\[
\left\langle
F(x)-F(x^\star),
x-x^\star
\right\rangle
\le
-\mu_{\rm sm}\norm{x-x^\star}^{2},
\qquad
x\in K.
\]
\end{definition}
\begin{theorem}
\label{th:monotone}
Let \(K\) be a nonempty compact convex polyhedron and let
\(F:K\to\mathbb R^d\) be Lipschitz. Suppose that \(x^\star\in K\)
is an equilibrium of
\[
\dot x=\Pi_K(x,F(x))
\]
and that \(F\) is strongly monotone at \(x^\star\), in the sense of
Definition~\ref{def:strong_monotone}, with constant
\(\mu_{\rm sm}>0\). Then the projected ODE is well posed and every
solution satisfies
\[
\norm{x(t)-x^\star}
\le
e^{-\mu_{\rm sm}t}\norm{x(0)-x^\star},
\qquad t\ge0.
\]
In particular, \(x^\star\) is the unique globally exponentially stable
equilibrium.
\end{theorem}
\begin{proof}
Set
\[
F_{\rm ext}(z):=F(\operatorname{proj}_K z),
\qquad z\in\mathbb R^d.
\]
Since Euclidean projection is nonexpansive, \(F_{\rm ext}\) is globally
Lipschitz; since \(K\) is compact, it is also bounded. Thus
Assumption~1 of~\citet{dupuis1993dynamical} holds. Taking
\(x_0^n=x_0\), \(b_i^n=F_{\rm ext}\), and \(a_i^n=1/n\) in their
Theorem~3 gives existence, while their Theorem~2 gives uniqueness.

Since \(x^\star\) is an equilibrium, it is a constant solution.
Repeating the comparison calculation leading to equation~(15) in the
proof of \citet[Theorem~2]{dupuis1993dynamical}, with their drift
\(b=F_{\rm ext}\) and with the two solutions \(x(\cdot)\) and
\(x^\star\), gives, for almost every \(t\ge0\),
\[
\frac12\frac{\mathrm d}{\mathrm dt}
\norm{x(t)-x^\star}^{2}
\le
\left\langle
F(x(t))-F(x^\star),x(t)-x^\star
\right\rangle
\le
-\mu_{\rm sm}\norm{x(t)-x^\star}^{2}.
\]
Gronwall's inequality gives
\[
\norm{x(t)-x^\star}
\le
e^{-\mu_{\rm sm}t}\norm{x(0)-x^\star},
\qquad t\ge0.
\]
Applying this estimate to any constant equilibrium trajectory also
proves that \(x^\star\) is the unique equilibrium.
\end{proof}

\begin{corollary}[Lyapunov and one-block consequences]
\label{cor:strong_monotone_consequences}
Under the hypotheses of Theorem~\ref{th:monotone}, let
\(\mathsf S_t^F\) denote the projected-ODE flow and define
\[
J(x):=\norm{x-x^\star}.
\]
Then \(J\) is a \(1\)-Lipschitz Lyapunov function and, for every \(T>0\),
\[
J(\mathsf S_T^F(x))
\le
\rho_TJ(x),
\qquad
\rho_T:=e^{-\mu_{\rm sm}T}\in(0,1),
\qquad x\in K.
\]
Thus the uniform one-block Lyapunov contraction condition holds for
every \(T>0\).
\end{corollary}
\begin{lemma}
\label{lemma:str_mono}
Let \(K\) be a nonempty compact convex polyhedron. Consider the projected ODE
\[
\dot x=\Pi_K\bigl(x,\mathcal T_{\rm ctr}(x)-x\bigr),
\]
where \(\mathcal T_{\rm ctr}:K\to K\) is a contraction with contraction factor
\(\gamma\in[0,1)\). If \(x^\star\) is an equilibrium, then
\(\mathcal T_{\rm ctr}(x^\star)=x^\star\), the vector field
\[
F_{\rm ctr}(x):=\mathcal T_{\rm ctr}(x)-x
\]
is strongly monotone at \(x^\star\), and
Corollary~\ref{cor:strong_monotone_consequences} applies.
\end{lemma}

\begin{proof}
Since \(\mathcal T_{\rm ctr}(x^\star)\in K\), the vector
\[
\mathcal T_{\rm ctr}(x^\star)-x^\star
\]
belongs to the tangent cone of \(K\) at \(x^\star\). Because
\(x^\star\) is an equilibrium,
\[
\Pi_K\bigl(x^\star,\mathcal T_{\rm ctr}(x^\star)-x^\star\bigr)=0.
\]
It follows that
\[
\mathcal T_{\rm ctr}(x^\star)=x^\star.
\]

For every \(x\in K\),
\[
\begin{aligned}
\left\langle
F_{\rm ctr}(x)-F_{\rm ctr}(x^\star),
x-x^\star
\right\rangle
&=
\left\langle
\mathcal T_{\rm ctr}(x)-\mathcal T_{\rm ctr}(x^\star),
x-x^\star
\right\rangle
-
\norm{x-x^\star}^{2}\\
&\le
\norm{\mathcal T_{\rm ctr}(x)-\mathcal T_{\rm ctr}(x^\star)}
\norm{x-x^\star}
-
\norm{x-x^\star}^{2}\\
&\le
-(1-\gamma)\norm{x-x^\star}^{2}.
\end{aligned}
\]
Thus \(F_{\rm ctr}\) is strongly monotone at \(x^\star\), with constant
\(1-\gamma\). The conclusion now follows from
Corollary~\ref{cor:strong_monotone_consequences}.
\end{proof}

\section{Skorokhod problem}\label{sec:skorokhod}

We recall the Skorokhod problem and the Lipschitz property of its solution map. This property is used to convert bounds between unconstrained driving paths into bounds between the corresponding projected paths.

Let $D([0,\infty);\bR^d)$ denote the space of right-continuous functions with left limits from $[0,\infty)$ to $\bR^d$. If $\upsilon:[0,\infty)\to\bR^d$, let $|\upsilon|(T)$ denote its total variation on $[0,T]$. Let $D_{BV}([0,\infty);\bR^d)$ be the set of all $\upsilon\in D([0,\infty);\bR^d)$ such that $\upsilon$ has bounded variation on every compact interval. Let \(\partial\Theta\) denote the boundary of \(\Theta\).
For $x\in\partial\Te$, define the set of unit inward normals
\[
n_{\Te}(x):=\left\{\nu\in\bR^d:\norm{\nu}=1,\ \langle \nu,y-x\rangle\ge0\ \text{for all }y\in\Te\right\}.
\]
For $x\in\operatorname{int}(\Te)$, set $n_{\Te}(x):=\varnothing$.

\begin{definition}[Skorokhod problem]\label{def:skorokhod_problem}
Let $\upsilon\in D([0,\infty);\bR^d)$ with $\upsilon(0)\in\Te$ be given. A pair $(\phi,\eta_{sk})$ solves the Skorokhod problem with respect to $\Te$ if, for all $t\ge0$,
\begin{enumerate}[(i)]
\item $\phi(t)=\upsilon(t)+\eta_{sk}(t)$ and $\phi(0)=\upsilon(0)$;
\item $\phi(t)\in\Te$;
\item $|\eta_{sk}|(t)<\infty$;
\item $|\eta_{sk}|(t)=\int_{(0,t]}\mathbbm{1}_{\partial\Te}(\phi(s))\,d|\eta_{sk}|(s)$;
\item there exists a measurable function $\nu:[0,\infty)\to\bR^d$ such that $\nu(s)\in n_{\Te}(\phi(s))$ for $d|\eta_{sk}|$-a.e. $s$, and
\[
\eta_{sk}(t)=\int_{(0,t]}\nu(s)\,d|\eta_{sk}|(s).
\]
\end{enumerate}
\end{definition}

The path $\upsilon$ is called the unconstrained driving path, while $\phi$ is the projected path. When the solution is unique, write $\Gamma(\upsilon):=\phi$ for its projected component. The following is~\citet[Theorem~1]{dupuis1993dynamical} in the present notation.

\begin{theorem}[Lipschitz property of the Skorokhod map]\label{th:skrohod_lip}
Let $\Te\subset\bR^d$ be a convex polyhedron. For every $\upsilon\in D_{BV}([0,\infty);\bR^d)$ with $\upsilon(0)\in\Te$, there exists a unique solution to the Skorokhod problem. Moreover, the corresponding map $\Gamma$ is Lipschitz, i.e., there exists a constant \(\LS>0\) such that, for every
\(T<\infty\) and every
\(\upsilon_1,\upsilon_2\in D_{BV}([0,\infty);\bR^d)\) with
\(\upsilon_1(0),\upsilon_2(0)\in\Te\),
\[
\sup_{0\le t\le T}
\norm{
\Gamma(\upsilon_1)(t)-\Gamma(\upsilon_2)(t)
}
\le
\LS
\sup_{0\le t\le T}
\norm{\upsilon_1(t)-\upsilon_2(t)}.
\]
\end{theorem}

Whenever a block step path \(q^{(m),\circ}\), defined in
Section~\ref{sec:notation}, is used as an argument of \(\Gamma\), extend it
constantly beyond the block endpoint:
\[
q^{(m),\circ}(r):=q_{K_m}^{(m)},\qquad r\ge H_m.
\]
This convention does not affect any estimate on \([0,H_m]\).

\begin{lemma}[Step-path Skorokhod comparison on a block]
\label{lemma:skor_1}
Fix \(m\ge0\). Let
\[
\te^{(m),\circ},\quad
\tte^{(m),\circ},\quad
v^{(m),\circ},\quad
\tv^{(m),\circ}
\]
be the right-continuous block step embeddings of the original slow sequence
and of the sequences in~\eqref{def:te_n_prime}--
\eqref{def:vtilde_unconstrained}, using the convention of
Section~\ref{sec:notation}. Then
\[
\te^{(m),\circ}=\Gamma\bigl(v^{(m),\circ}\bigr),
\qquad
\tte^{(m),\circ}=\Gamma\bigl(\tv^{(m),\circ}\bigr).
\]
The Lipschitz constant \(\LS\) depends only on the geometry of \(\Te\),
not on \(m\) or the driving paths. Consequently, for every
\(0\le n\le K_m\),
\[
\max_{0\le k\le n}
\norm{\te\um_k-\tte\um_k}
\le
\LS
\max_{0\le k\le n}
\norm{v\um_k-\tv\um_k}.
\]
\end{lemma}

\begin{proof}
For the first recursion, set \(\eta^{(m)}_0:=0\) and
\(\eta^{(m)}_k:=\te^{(m)}_k-v^{(m)}_k\). For \(0\le k<K_m\), write
\[
\Delta v^{(m)}_{k+1}:=v^{(m)}_{k+1}-v^{(m)}_k,
\qquad
\Delta\eta^{(m)}_{k+1}:=\eta^{(m)}_{k+1}-\eta^{(m)}_k.
\]
Then, at the grid time
\(t^{(m)}(k+1)\),
\[
\te^{(m)}_{k+1}
=\operatorname{proj}_{\Te}\!\left(\te^{(m)}_k+\Delta v^{(m)}_{k+1}\right),
\qquad
\Delta\eta^{(m)}_{k+1}
=\te^{(m)}_{k+1}-\left(\te^{(m)}_k+\Delta v^{(m)}_{k+1}\right).
\]
Because \(\Te\) is closed and convex and
\(\te^{(m)}_{k+1}=\operatorname{proj}_{\Te}
(\te^{(m)}_k+\Delta v^{(m)}_{k+1})\), the first-order optimality
condition for the Euclidean projection yields, for every \(y\in\Te\),
\[
\left\langle
\te^{(m)}_{k+1}
-\left(\te^{(m)}_k+\Delta v^{(m)}_{k+1}\right),
y-\te^{(m)}_{k+1}
\right\rangle\ge0.
\]
The first argument is \(\Delta\eta^{(m)}_{k+1}\). Hence
\[
\left\langle\Delta\eta^{(m)}_{k+1},
y-\te^{(m)}_{k+1}\right\rangle\ge0,
\qquad y\in\Te.
\]
Thus every nonzero regulator jump is an inward normal at
\(\te^{(m)}_{k+1}\), which must then lie on \(\partial\Te\); its normalized
direction belongs to \(n_{\Te}(\te^{(m)}_{k+1})\). Because the block has
only finitely many jumps and both step embeddings are constant between grid
points, the step regulator
\(\eta^{(m),\circ}:=\te^{(m),\circ}-v^{(m),\circ}\) has finite variation
supported on times at which \(\te^{(m),\circ}\in\partial\Te\). Hence
\[
\nu^{(m)}_{k+1}:=
\frac{\Delta\eta^{(m)}_{k+1}}{\norm{\Delta\eta^{(m)}_{k+1}}}
\]
at each nonzero jump defines a measurable direction process with
\[
\eta^{(m),\circ}(t)
=\int_{(0,t]}\nu^{(m)}(s)\,\mathrm d|\eta^{(m),\circ}|(s).
\]
Consequently,
\[
\left(\te^{(m),\circ},\te^{(m),\circ}-v^{(m),\circ}\right)
\]
solves the c\`adl\`ag Skorokhod problem driven by \(v^{(m),\circ}\), and
uniqueness gives
\(\te^{(m),\circ}=\Gamma(v^{(m),\circ})\). The same argument yields
\(\widetilde\te^{(m),\circ}=\Gamma(\widetilde v^{(m),\circ})\).
Applying Theorem~\ref{th:skrohod_lip} on \([0,t^{(m)}(n)]\) therefore gives
\[
\begin{aligned}
\max_{0\le k\le n}\norm{\te^{(m)}_k-\widetilde\te^{(m)}_k}
&=\sup_{0\le t\le t^{(m)}(n)}
\norm{\te^{(m),\circ}(t)-\widetilde\te^{(m),\circ}(t)}\\
&\le \LS\sup_{0\le t\le t^{(m)}(n)}
\norm{v^{(m),\circ}(t)-\widetilde v^{(m),\circ}(t)}\\
&=\LS\max_{0\le k\le n}
\norm{v^{(m)}_k-\widetilde v^{(m)}_k}.
\end{aligned}
\]
\end{proof}

\begin{lemma}[Unconstrained path for the projected ODE]
\label{lemma:skorohod_4}
Fix \(\tau\in\bR_+\). There exists an absolutely continuous path
\(w^{(\tau)}:[0,\infty)\to\bR^{d_{\te}}\) satisfying
\[
w^{(\tau)}(0)=z^{(\tau)}(0)=\te^{\circ}(\tau)
\]
and
\[
\dot w^{(\tau)}(t)
=\favg\bigl(z^{(\tau)}(t);z^{(\tau)}(t)\bigr)
\]
for Lebesgue-a.e. \(t\ge0\), such that
\[
z^{(\tau)}=\Gamma\bigl(w^{(\tau)}\bigr).
\]
\end{lemma}

\begin{proof}
This follows from the Skorokhod representation of projected dynamical
systems in~\citet[Theorem~2]{dupuis1993dynamical}.
\end{proof}

\section{Step-size-related results}
\label{sec:ec_stepsize}
\begin{lemma}
\label{lemma:alpha2_chi_bound}
Let
\[
\alpha_n=(n+\no)^{-\fraka},
\qquad
0<\fraka<1,
\]
and define
\[
\chi(m,n)
:=
\begin{cases}
\displaystyle
\prod_{\ell=m}^{n}(1-\alpha_\ell),
& m\le n,\\[1ex]
1,
& m>n.
\end{cases}
\]
Suppose that
\[
(\no+1)^{1-\fraka}\ge4.
\]
Then, for every \(n\ge1\),
\[
\sum_{k=1}^{n}
\alpha_k^2\chi(k+1,n)
\le
2\alpha_n,
\]
and
\[
\sum_{k=1}^{n}
\alpha_k^2\chi^2(k+1,n)
\le
2\alpha_n.
\]
\end{lemma}

\begin{proof}
Let
\[
S_n:=\sum_{k=1}^n\alpha_k^2\chi(k+1,n).
\]
Then $S_1=\alpha_1^2\le2\alpha_1$ and, for $n\ge2$,
\[
S_n=\alpha_n^2+(1-\alpha_n)S_{n-1}.
\]
Assume inductively that $S_{n-1}\le2\alpha_{n-1}$ and set
$s_n:=n+\no$. Since $0<\fraka<1$,
\[
\frac{\alpha_{n-1}}{\alpha_n}
=\left(\frac{s_n}{s_n-1}\right)^{\fraka}
\le\frac{s_n}{s_n-1}\le1+\frac2{s_n}.
\]
Therefore,
\[
\begin{aligned}
S_n
&\le\alpha_n^2+2(1-\alpha_n)\alpha_n
\left(1+\frac2{s_n}\right)\\
&\le2\alpha_n-\alpha_n^2+\frac{4\alpha_n}{s_n}
\le2\alpha_n,
\end{aligned}
\]
where the last inequality uses
$4/s_n\le s_n^{-\fraka}=\alpha_n$, which follows from
$(\no+1)^{1-\fraka}\ge4$. This proves the first bound.
Because $0\le\chi(k+1,n)\le1$,
\[
\sum_{k=1}^n\alpha_k^2\chi^2(k+1,n)
\le\sum_{k=1}^n\alpha_k^2\chi(k+1,n)
\le2\alpha_n,
\]
proving the second.
\end{proof}

\begin{lemma}\label{lemma:weight_variation_chi}
For $\alpha_n=(n+\no)^{-\fraka}$ and $0<\fraka<1$, for every $n\ge2$,
\[
\sum_{k=2}^{n}
|\alpha_k\chi(k+1,n)-\alpha_{k-1}\chi(k,n)|
\le\alpha_n.
\]
Consequently,
\[
\left\|\sum_{k=2}^{n}
\{\alpha_k\chi(k+1,n)-\alpha_{k-1}\chi(k,n)\}
u^{(x_{k-1},\te_{k-1})}_h(Y_k)\right\|
\le2\frakc_5\mathsf B_6\alpha_n,
\]
almost surely.
\end{lemma}

\begin{proof}
Let $A_k:=\alpha_k\chi(k+1,n)$. Since
$\chi(k,n)=(1-\alpha_k)\chi(k+1,n)$,
\[
A_k-A_{k-1}
=\{\alpha_k-\alpha_{k-1}(1-\alpha_k)\}\chi(k+1,n).
\]
With $s=k+\no$, the inequality
$\alpha_k\ge\alpha_{k-1}(1-\alpha_k)$ is equivalent to
\[
\left(\frac{s-1}{s}\right)^{\fraka}\ge1-s^{-\fraka},
\]
which follows from subadditivity of $x\mapsto x^{\fraka}$ on $\bR_+$.
Thus $A_k-A_{k-1}\ge0$, and
\[
\sum_{k=2}^{n}|A_k-A_{k-1}|=A_n-A_1\le\alpha_n.
\]
The final bound follows from part~(\ref{poisson_h_bound}) of
Lemma~\ref{lemma:poisson_sensitivity_1}.
\end{proof}

\section{Miscellaneous results}\label{sec:ec_misc}

\begin{theorem}[Azuma--Hoeffding]\label{th:azuma}
Suppose \(\{X_k\}_{k\ge0}\) is a scalar martingale and
\(\{c_k\}_{k\ge1}\) is a deterministic nonnegative sequence such that
\(|X_k-X_{k-1}|\le c_k\) almost surely. Then, for every
\(N\in\bN\) and every \(\eps>0\),
\[
\bP(|X_N-X_0|\ge\eps)\le 2\exp\left(-\frac{\eps^2}{2\sum_{k=1}^{N}c_k^2}\right).
\]
\end{theorem}

\begin{lemma}[Well-posedness of the projected ODE]\label{lemma:ode_soln_unique_exist}
For every $z_0\in\Te$, the projected ODE
\[
\dot z(t)=\Pi_{\Te}\left(z(t),\favg(z(t);z(t))\right),\qquad z(0)=z_0,
\]
has a unique solution, and this solution remains in $\Te$ for all $t\ge0$.
\end{lemma}

\begin{proof}
Lemma~\ref{lemma:lipschitz_stat_dist} gives
$\sup_{\te\in\Te}\norm{x^\star(\te)}\le\frakc_8$.  Hence
Assumption~\ref{assum:lipschitz}, applied on the ball of radius
$\frakc_8$, makes the diagonal field
$F(\te):=\favg(\te;\te)$ Lipschitz on $\Te$.  Define
$F_{\rm ext}(z):=F(\operatorname{proj}_{\Te}z)$.  Nonexpansiveness of
projection makes $F_{\rm ext}$ globally Lipschitz; compactness of $\Te$
makes it bounded, and it agrees with $F$ on $\Te$.  Existence follows
from~\citet[Theorem~3]{dupuis1993dynamical}, while the Skorokhod
representation and uniqueness follow
from~\citet[Theorem~2]{dupuis1993dynamical}. Hence the projected ODE is
well posed and its solution remains in $\Te$.
\end{proof}

\begin{lemma}[Perturbed power-drift recursion]
\label{lem:nonlinear_block_recursion}
Let \(p>1\), \(\kappa>0\), \(\zeta>0\), \(\ell\ge0\), and
\(C_e\in[0,\infty)\). Let
\(\{u_m\}\) and \(\{e_m\}\) be nonnegative sequences such that
\(u_m\to0\) and, for all sufficiently large \(m\),
\[
u_{m+1}\le u_m-\kappa u_m^p+e_m,
\]
and
\[
e_m
\le
C_e(m+2)^{-\zeta}\{\log(m+2)\}^{\ell}.
\]
Then there exists \(C<\infty\) such that, for all sufficiently large
\(m\),
\[
u_m
\le
C
\left\{
(m+2)^{-1/(p-1)}
+
(m+2)^{-\zeta/p}\{\log(m+2)\}^{\ell/p}
\right\}.
\]
\end{lemma}

\begin{proof}
Set
\[
w_m:=(m+2)^{-1/(p-1)},
\qquad
\eta_m:=(m+2)^{-\zeta/p}\{\log(m+2)\}^{\ell/p},
\]
and \(h_m:=w_m+\eta_m\). These sequences are eventually decreasing.
Because
\[
\frac{1}{p-1}+1=\frac{p}{p-1},
\]
the mean-value theorem gives, eventually,
\[
0\le w_m-w_{m+1}\le C_w w_m^p.
\]
It also gives
\[
0\le\eta_m-\eta_{m+1}
\le C_\eta\frac{\eta_m}{m+2}.
\]
If \(\zeta\le p/(p-1)\), then
\[
\frac{\eta_m}{m+2}=O(\eta_m^p),
\]
where the logarithmic factor at equality is harmless because
\(\ell\ge0\). If \(\zeta>p/(p-1)\), then
\[
\frac{\eta_m}{m+2}=O(w_m^p).
\]
Consequently, for some \(C_h<\infty\),
\begin{equation}
0\le h_m-h_{m+1}
\le
C_h(w_m^p+\eta_m^p)
\label{ineq:power_barrier_decrement}
\end{equation}
eventually.

Choose \(R>0\) so that \(\kappa pR^{p-1}\le1/2\). Since \(u_m\to0\)
and \(h_m\to0\), choose \(A_0\ge1\) and then a sufficiently large
index \(M\) such that
\[
\kappa A_0^{p-1}\ge2C_h,
\qquad
\kappa A_0^p\ge2C_e,
\qquad
u_M\le R/2,
\qquad
A_0h_M\le R/2.
\]
Set
\[
A:=\max\left\{A_0,\frac{u_M}{h_M}\right\},
\qquad
v_m:=Ah_m.
\]
Then \(u_M\le v_M\le R/2\), and \(v_m\le R/2\) for every
\(m\ge M\). Since \((w_m+\eta_m)^p\ge w_m^p+\eta_m^p\),
\eqref{ineq:power_barrier_decrement} and the choice of \(A\) give
\[
\begin{aligned}
v_m-\kappa v_m^p+e_m
&\le
Ah_m-\kappa A^p(w_m^p+\eta_m^p)+C_e\eta_m^p\\
&\le
Ah_m-AC_h(w_m^p+\eta_m^p)\\
&\le
Ah_{m+1}=v_{m+1}.
\end{aligned}
\]
The map \(F(u):=u-\kappa u^p\) is nondecreasing on \([0,R]\).
Therefore, if \(u_m\le v_m\), then
\[
u_{m+1}
\le F(u_m)+e_m
\le F(v_m)+e_m
\le v_{m+1}.
\]
Induction proves \(u_m\le v_m\) for every \(m\ge M\), which is the
claimed bound.
\end{proof}

\begin{lemma}[Geometric convolution of the blockwise rate envelope]
\label{lem:geometric-convolution}
Suppose \(1/2<\frakb<1\). Fix \(\rho\in(0,1)\) and \(r>0\), and set
\[
\varphi_m
:=
(N_0+n_m)^{-r}
\sqrt{\log(N_0+n_m)}.
\]
Then, for every \(m_0\ge 0\), there exist \(m_1\ge m_0\) and
\(C_{\mathrm{geo}}<\infty\) such that, for every \(m\ge m_1\),
\begin{equation}
\sum_{j=m_0}^{m-1}
\rho^{m-1-j}\varphi_j
\le
C_{\mathrm{geo}}\varphi_m.
\label{eq:geometric-convolution}
\end{equation}
Moreover,
\begin{equation}
\rho^{m-m_0}
\le
C_{\mathrm{geo}}\varphi_m,
\qquad m\ge m_1.
\label{eq:geometric-transient}
\end{equation}

Consequently, if
\[
d_j\le C_r\varphi_j,
\qquad j\ge m_0,
\]
then
\[
\sum_{j=m_0}^{m-1}
\rho^{m-1-j}d_j
\le
C_rC_{\mathrm{geo}}\varphi_m,
\qquad m\ge m_1.
\]
\end{lemma}

\begin{proof}
By the minimality of $n_m$,
\[
T\le t(n_m)-t(n_{m-1})<T+\beta_{n_m-1}\le T+\beta_{n_0}.
\]
This holds for $m\ge1$.
Summing over the first $m$ blocks gives
\begin{equation}
mT\le t(n_m)-t(n_0)\le m(T+\beta_{n_0}),
\qquad m\ge1.
\label{eq:block-time-growth}
\end{equation}
Integral comparison for
$t(n)=\sum_{k=1}^{n-1}(\no+k)^{-\frakb}$, together with \eqref{eq:block-time-growth},
yields constants $c_{T,-},c_{T,+}>0$ and $M<\infty$ such that
\begin{equation}
c_{T,-}(m+1)^{1/(1-\frakb)}
\le \no+n_m
\le c_{T,+}(m+1)^{1/(1-\frakb)},
\qquad m\ge M.
\label{eq:block-index-growth}
\end{equation}
Set
\[
p:=\frac{r}{1-\frakb},
\qquad
\psi_m:=(m+2)^{-p}\sqrt{\log(m+2)}.
\]
Equation \eqref{eq:block-index-growth} gives constants $c_-,c_+>0$ such that
\begin{equation}
c_-\psi_m\le\varphi_m\le c_+\psi_m,
\qquad m\ge M.
\label{eq:envelope-comparison}
\end{equation}
For $m\ge j\ge M$,
\[
\frac{\psi_j}{\psi_m}
=\left(\frac{m+2}{j+2}\right)^p
\sqrt{\frac{\log(j+2)}{\log(m+2)}}
\le\left(\frac{m+2}{j+2}\right)^p
\le(m-j+1)^p.
\]
Hence, by \eqref{eq:envelope-comparison},
\begin{equation}
\varphi_j\le C(m-j+1)^p\varphi_m,
\qquad m\ge j\ge M.
\label{eq:backward-envelope}
\end{equation}
Putting $k=m-1-j$ in \eqref{eq:backward-envelope} gives
\[
\sum_{j=M}^{m-1}\rho^{m-1-j}\varphi_j
\le C\varphi_m\sum_{k=0}^{\infty}\rho^k(k+2)^p
\le C_{\rm geo}\varphi_m,
\]
because $\rho\in(0,1)$. The finitely many terms with $m_0\le j<M$ are
$O(\rho^{m-M})=O(\varphi_m)$, since an exponential sequence decays faster
than a polynomial--logarithmic one. This proves \eqref{eq:geometric-convolution}. The same comparison
gives $\rho^{m-m_0}=O(\varphi_m)$, proving \eqref{eq:geometric-transient}. Finally, if
$d_j\le C_r\varphi_j$, multiply the first bound by $C_r$ to obtain the
stated consequence.
\end{proof}

\begin{lemma}[Block estimates for nonharmonic slow steps]
\label{lem:block_geometry}
Suppose \(1/2<\frakb<1\). There exist \(C_T<\infty\) and
\(m_T<\infty\) such that, for every
\(m\ge m_T\),
\[
D_m\le (T+1)\beta_{n_m},
\qquad
K_m\le C_T(N_0+n_m)^{\frakb},
\]
and
\[
L_m^x+L_m^\te
\le C_T\log(N_0+n_m).
\]
\end{lemma}

\begin{proof}
Set \(X_m:=\no+n_m\). Minimality of \(n_{m+1}\) gives
\[
a(n_m,n_{m+1})
<T+\beta_{n_{m+1}-1}\le T+\beta_{n_0}\le T+1,
\]
and hence
\[
D_m
\le\beta_{n_m}\sum_{k=n_m}^{n_{m+1}-1}\beta_k
\le(T+1)\beta_{n_m}.
\]
Let
\[
\overline K_m
:=
\left\lceil 2^{\frakb}T X_m^{\frakb}\right\rceil.
\]
Because \(\frakb<1\), eventually \(\overline K_m\le X_m\), and then
\[
t(n_m+\overline K_m)-t(n_m)
\ge
\overline K_m(X_m+\overline K_m)^{-\frakb}
\ge
2^{-\frakb}\overline K_m X_m^{-\frakb}
\ge T.
\]
Minimality therefore yields
\[
K_m=n_{m+1}-n_m
\le \overline K_m
\le C_T X_m^{\frakb}.
\]
Finally, \eqref{eq:block-index-growth} implies \(m+2\le C_TX_m^{1-\frakb}\) eventually,
so
\[
\begin{aligned}
L_m^x+L_m^\te
&=\log\{4d_xK_m(m+2)^2\}
 +\log\{4d_\te K_m(m+2)^2\}\\
&\le C+2\log K_m+4\log(m+2)
\le C_T\log X_m.
\end{aligned}
\]
Increasing \(m_T\) and \(C_T\), if necessary, proves the result.
\end{proof}

\begin{lemma}[Block estimates for harmonic slow steps]
\label{lem:harmonic_block_geometry}
Suppose \(\frakb=1\), and set
\[
X_m:=N_0+n_m,
\qquad
A_T:=e^T.
\]
Then, for every \(m\ge0\),
\[
(A_T-1)(X_m-1)\le K_m<(A_T-1)X_m+1.
\]
Moreover, there exists \(\xi_T>0\) such that
\[
X_m=\xi_Te^{mT}+O_T(1)
\qquad\text{as }m\to\infty.
\]
Finally, there exist \(C_T<\infty\) and a deterministic integer
\(m_T\ge0\) such that, for every \(m\ge m_T\),
\[
D_m\le(T+1)X_m^{-1},
\qquad
L_m^x+L_m^\te\le C_T\log X_m,
\]
and, whenever \(n_m\le n<n_{m+1}\),
\[
X_m\le N_0+n<X_{m+1}<e^TX_m+1\le(e^T+1)X_m.
\]
\end{lemma}

\begin{proof}
Minimality of \(n_{m+1}\) gives
\[
T
\le
\sum_{r=0}^{K_m-1}\frac1{X_m+r}
<
T+\frac1{X_{m+1}-1}.
\]
Integral comparison, applied also to the same sum with its last term
removed, gives
\[
T\le\log\frac{X_m+K_m-1}{X_m-1},
\qquad
\log\frac{X_m+K_m-1}{X_m}<T,
\]
which proves the asserted bounds on \(K_m\). Hence
\[
X_{m+1}=e^TX_m+\eta_m,
\qquad
-(e^T-1)\le\eta_m<1.
\]
Thus
\[
e^{-mT}X_m
=
X_0+\sum_{j=0}^{m-1}e^{-(j+1)T}\eta_j
\longrightarrow\xi_T,
\]
and the tail of this absolutely convergent series gives
\(X_m=\xi_Te^{mT}+O_T(1)\). Also,
\(X_{m+1}-1\ge e^T(X_m-1)\), so \(\xi_T\ge X_0-1>0\).

Since \(\beta_k\le X_m^{-1}\) on block \(m\),
\[
D_m
\le
X_m^{-1}\sum_{k=n_m}^{n_{m+1}-1}\beta_k
\le
(T+1)X_m^{-1}.
\]
The preceding estimates give \(K_m=O_T(X_m)\) and
\(m=O_T(\log X_m)\). Substitution in the definitions of
\(L_m^x\) and \(L_m^\te\) proves the logarithmic bound. The final
within-block comparison follows from the upper bound on \(K_m\).
\end{proof}

\begin{lemma}[Geometric convolution for harmonic slow steps]
\label{lem:harmonic_geometric_convolution}
Suppose \(\frakb=1\). Fix \(\rho\in(0,1)\), \(r>0\), and
\(\ell\ge0\), and set
\[
\varphi_m
:=
(N_0+n_m)^{-r}\{\log(N_0+n_m)\}^{\ell},
\qquad
\lambda_{\rho,T}:=-\frac{\log\rho}{T}.
\]
When \(\rho=\rho_T\), the quantity
\(\lambda_{\rho,T}\) coincides with
\(\lambda_T\) in
\eqref{def:effective_block_exponent}.
For every fixed \(m_0\), as \(m\to\infty\),
\[
\rho^{m-m_0}
+
\sum_{j=m_0}^{m-1}\rho^{m-1-j}\varphi_j
=
\begin{cases}
O(\varphi_m),
&\lambda_{\rho,T}>r,\\[1mm]
O\!\left((N_0+n_m)^{-r}
\{\log(N_0+n_m)\}^{\ell+1}\right),
&\lambda_{\rho,T}=r,\\[1mm]
O\!\left((N_0+n_m)^{-\lambda_{\rho,T}}\right),
&0<\lambda_{\rho,T}<r.
\end{cases}
\]
The same bounds hold with \(\varphi_j\) replaced by any nonnegative
sequence satisfying \(d_j\le C\varphi_j\).
\end{lemma}
\begin{proof}
Lemma~\ref{lem:harmonic_block_geometry} gives
\(N_0+n_m\asymp e^{mT}\) and
\(\log(N_0+n_m)\asymp m+1\). Therefore,
\[
\varphi_m\asymp e^{-rTm}(m+1)^\ell.
\]
Up to constant factors, the convolution is
\[
\sum_{j=m_0}^{m-1}
\rho^{m-1-j}e^{-rTj}(j+1)^\ell.
\]

If \(\lambda_{\rho,T}>r\), equivalently
\(\rho<e^{-rT}\), factor out \(e^{-rT(m-1)}\) and sum the
convergent geometric series in
\[
(\rho e^{rT})^{m-1-j}.
\]
If \(\lambda_{\rho,T}=r\), equivalently
\(\rho=e^{-rT}\), sum \((j+1)^\ell\).
If \(0<\lambda_{\rho,T}<r\), equivalently
\(\rho>e^{-rT}\), factor out \(\rho^{m-1}\) and sum the
convergent series in
\[
\left(\frac{e^{-rT}}{\rho}\right)^j(j+1)^\ell.
\]
These give, respectively,
\[
O(e^{-rTm}m^\ell),\qquad
O(e^{-rTm}m^{\ell+1}),\qquad
O(\rho^m).
\]
The initial transient has the same or smaller order, and converting
back from \(m\) proves the result.
\end{proof}

\begin{lemma}[Finite-state Kac cycle identity]
\label{lemma:finite_state_cycle_formula}
Let \(P\) be a transition matrix on a finite state space \(\cY\), with
unique stationary distribution \(\mu\). Fix \(y^\star\in\cY\), and let
\((Y_n)_{n\ge0}\) be the homogeneous Markov chain with transition matrix
\(P\), started from \(Y_0=y^\star\). Let \(\bE_{y^\star}\) denote
expectation under this law, and set
\[
\tau_{y^\star}:=\inf\{n\ge1:Y_n=y^\star\}.
\]
If \(\bE_{y^\star}[\tau_{y^\star}]<\infty\), then, for every integer
\(d_g \ge1\) and every \(g:\cY\to\bR^{d_g}\),
\[
\bE_{y^\star}
\left[
\sum_{n=0}^{\tau_{y^\star}-1}g(Y_n)
\right]
=
\frac{1}{\mu(y^\star)}
\sum_{y\in\cY}\mu(y)g(y).
\]
\end{lemma}

\begin{proof}
For \(j\in\cY\), define the finite cycle occupation measure
\[
\nu(j):=\bE_{y\ust}
\left[\sum_{n=0}^{\tau_{y\ust}-1}\mathbbm{1}\{Y_n=j\}\right].
\]
It is finite because
\[
\sum_{j\in\cY}\nu(j)
=\bE_{y\ust}[\tau_{y\ust}]<\infty.
\]
The Markov property and \(Y_0=Y_{\tau_{y\ust}}=y\ust\) give, for every
\(j\),
\[
\sum_{i\in\cY}\nu(i)P(i,j)
=\bE_{y\ust}
\left[\sum_{n=1}^{\tau_{y\ust}}\mathbbm{1}\{Y_n=j\}\right]
=\nu(j).
\]
Thus \(\nu\) is a finite invariant measure. Uniqueness of the stationary
distribution implies \(\nu=c\mu\), while the first-return property gives
\(\nu(y\ust)=1\), and hence \(c=1/\mu(y\ust)\). Therefore
\[
\bE_{y\ust}
\left[\sum_{n=0}^{\tau_{y\ust}-1}g(Y_n)\right]
=\sum_{j\in\cY}\nu(j)g(j)
=\frac1{\mu(y\ust)}\sum_{j\in\cY}\mu(j)g(j),
\]
as claimed.
\end{proof}

\section{Step-size shift and the starting index of the analysis}
\label{sec:N0_cond}
This appendix describes two ways to ensure that the step-size
conditions required by the analysis hold. We can choose \(N_0\)
sufficiently large and begin the analysis at \(n=1\), or take
\(N_0=1\) and begin the analysis at a sufficiently large index
\(n_0\). In both cases, the algorithm runs from \(n=1\).

\subsection{Choosing \(N_0\) to begin the analysis at \(n=1\)}
\label{subsec:baseline}
This is the convention already adopted throughout the main paper. 
We first give a sufficient lower bound on \(N_0\) for the
analysis to apply from \(n=1\). For the blockwise fast analysis in
Section~\ref{sec:faster_iterations}, choose \(N_0\) so that
\[
(N_0+1)^{1-\fraka}\ge4.
\]
Equivalently, it is sufficient to take
\[
N_0
\ge
\left\lceil
4^{1/(1-\fraka)}
\right\rceil-1.
\]
The named domination inequality~\eqref{cond:N0:dom} then holds
automatically because \(A_x\ge A_\te\),
\(C_\te=\chpoi A_\te\), and \(\beta_1\le\alpha_1\).
The small-step condition is used in the local martingale and
Poisson-residual estimates. No additional infinite-horizon condition involving a fixed
fast-martingale threshold is needed: the block-dependent thresholds
will be chosen later so that their failure probabilities are
summable.

\subsection{Beginning the analysis at \(n=n_0\) when \(N_0=1\)}
\label{subsec:prefix}
Alternatively, take \(N_0=1\) and define
\[
n_0:=\min\{n\ge1:(n+1)^{1-\fraka}\ge4\}.
\]
The algorithm runs from \(n=1\), while the blockwise analysis
begins at \(n=n_0\), using the iterates generated by the
algorithm up to that point. We retain the original iteration
index and slow-time grid. For every \(n\ge n_0\),
\[
\alpha_n=(n+1)^{-\fraka},
\qquad
\beta_n=(n+1)^{-\frakb},
\qquad
(n+1)^{1-\fraka}\ge4.
\]
We now establish the analogue of Lemma~\ref{lemma:alpha2_chi_bound}
for sums beginning at \(n_0\).
For every \(n\ge n_0\),
\begin{equation}
\sum_{k=n_0}^{n}\alpha_k^2\chi(k+1,n)
\le2\alpha_n,
\qquad
\sum_{k=n_0}^{n}\alpha_k^2\chi^2(k+1,n)
\le2\alpha_n.
\label{eq:burnin_suffix_convolution}
\end{equation}
Indeed, at \(n=n_0\), the first sum is \(\alpha_{n_0}^2\le2\alpha_{n_0}\).
If the first inequality holds with terminal index \(n-1\), the product
structure of \(\chi\) gives
\[
\begin{aligned}
\sum_{k=n_0}^{n}\alpha_k^2\chi(k+1,n)
&=\alpha_n^2+(1-\alpha_n)
  \sum_{k=n_0}^{n-1}\alpha_k^2\chi(k+1,n-1)\\
&\le\alpha_n^2+2(1-\alpha_n)\alpha_{n-1}.
\end{aligned}
\]
Moreover,
\[
\frac{\alpha_{n-1}}{\alpha_n}
=\left(\frac{n+1}{n}\right)^{\fraka}
\le\frac{n+1}{n}
\le1+\frac{2}{n+1},
\]
and \(4/(n+1)\le\alpha_n\) because
\((n+1)^{1-\fraka}\ge4\). Consequently,
\[
\sum_{k=n_0}^{n}\alpha_k^2\chi(k+1,n)
\le
2\alpha_n-\alpha_n^2+\frac{4\alpha_n}{n+1}
\le2\alpha_n.
\]
This proves the first inequality in
\eqref{eq:burnin_suffix_convolution} by induction. The second follows from
\(0\le\chi^2(k+1,n)\le\chi(k+1,n)\).

Every block entrance satisfies \(n_m\ge n_0\). Therefore, whenever the
fast-tail or Poisson-residual proof invokes
Lemma~\ref{lemma:alpha2_chi_bound}, its relevant block-local sum is bounded
instead by the corresponding suffix sum in
\eqref{eq:burnin_suffix_convolution}. For example, for
\(n\ge n_m\),
\[
\sum_{k=n_m}^{n}\alpha_k^2\chi^2(k+1,n)
\le
\sum_{k=n_0}^{n}\alpha_k^2\chi^2(k+1,n)
\le2\alpha_n
\le2\alpha_{n_m},
\]
and the same argument applies without the square on \(\chi\). The weight-variation bound from Lemma~\ref{lemma:weight_variation_chi} is likewise valid on the suffix. 

Under the large-shift convention used in the main paper, inequality~\eqref{cond:N0:dom}, together with the monotonicity of \(\beta_n/\alpha_n\), is used in Appendix~\ref{sec:appendix_proofs}, in the proof of Proposition~\ref{prop:block_fast_pois_residual}, immediately after the summation-by-parts decomposition of \eqref{eq:block_fast_poisson_residual}, to bound the Poisson-solution increments. When \(N_0=1\), the same inequality holds for every
\(n\ge n_0\), since
$$
C_\te\beta_n
=\chpoi A_\te\beta_n
\le \chpoi A_x\alpha_n,
\qquad n\ge n_0,
$$
using \(A_\te\le A_x\) and \(\beta_n\le\alpha_n\).

Finally, the original slow-time grid is retained. The block construction
gives
\[
T
\le t(n_m)-t(n_{m-1})
<T+\beta_{n_m-1}
\le T+\beta_{n_0},
\]
and hence
\[
mT
\le t(n_m)-t(n_0)
\le m(T+\beta_{n_0}).
\]
Beginning the block construction at \(t(n_0)\) changes only
the constants in the block-geometry and geometric-convolution
bounds, for both \(\frakb<1\) and \(\frakb=1\).
Thus, every blockwise result assuming~\eqref{cond:N0:alpha2}
applies to the blocks beginning at \(n_0\), with
\eqref{eq:burnin_suffix_convolution} providing the required
step-size estimate. The corresponding almost-sure rates hold
for all sufficiently large original indices \(n\).
Since \(N_0=1\), these rates are expressed in terms of \(n+1\),
with the same polynomial and logarithmic exponents.
\clearpage

\section[
  Rates under contraction of the reduced slow map
  (Proof of Proposition~\ref{prop:boundary_contraction_rate})
]{Rates under contraction of the reduced slow map
  (Proof of Proposition~\ref{prop:boundary_contraction_rate})}
\label{sec:ec_boundary_contraction}

This appendix proves Proposition~\ref{prop:boundary_contraction_rate}.
We first establish Lemma~\ref{lem:bc_projected_averaging}, which
bounds the distance between a noisy projected recursion and a
projected auxiliary recursion driven by the same input.

The next lemma assumes that the noise \(\xi_{n+1}\) admits
the decomposition~\eqref{eq:bc_noise_decomposition} into a
martingale-difference term \(m_{n+1}\), a difference
\(R_n-R_{n+1}\) of bounded correction terms, and a small
remainder \(r_n\). These auxiliary sequences are introduced
only for the analysis and are not computed by the algorithm.
When applying the lemma to the slow recursion in the proof of
Proposition~\ref{prop:boundary_contraction_rate}, we take
\[
 \xi_{n+1}
 =f(Y_n,x_n,\theta_n)
  -\sum_{y\in\cY}\mu^{(\theta_n)}(y)f(y,x_n,\theta_n)
  +M'_{n+1}.
\]
Here \(\xi_{n+1}\) combines the Markov noise
with the original martingale-difference noise \(M'_{n+1}\).
We construct \(m_{n+1}\), \(R_n\), and \(r_n\) using the
Poisson equation, as shown below; the remainder \(r_n\)
accounts for changes in the current pair \((x_n,\theta_n)\).

\begin{lemma}\label{lem:bc_projected_averaging}
Let \(\Theta\) be compact and convex, and consider adapted sequences
\[
 \theta_{n+1}=\proj[(1-\beta_n)\theta_n+\beta_nG_n+\beta_n\xi_{n+1}],
 \qquad
 z_{n+1}=\proj[(1-\beta_n)z_n+\beta_nG_n],
\]
where \(G_n\) is \(\cF_n\)-measurable, \(\theta_1,z_1\in\Theta\), and
\(\beta_n=B/t_n\), with \(t_n=N_0+n\), \(B>1/2\), and
\(0<\beta_n\le1\). Suppose that \(G_n\) and \(\xi_{n+1}\) are
uniformly bounded by deterministic constants and
\begin{equation}
 \xi_{n+1}=m_{n+1}+R_n-R_{n+1}+r_n,
 \qquad
 \bE[m_{n+1}\mid\cF_n]=0,
 \qquad
 \norm{m_{n+1}}\le K_{\mathrm{md}},
 \label{eq:bc_noise_decomposition}
\end{equation}
where \(R_n\) is adapted and, for finite deterministic constants
\(P,C_r\ge0\), \(\norm{R_n}\le P\) and
\(\norm{r_n}^2\le C_r\beta_n\) almost surely for every \(n\).
Here \(K_{\mathrm{md}}:=K_\xi+2P+\sqrt{C_r}\), where
\(K_\xi<\infty\) is a deterministic uniform bound on
\(\norm{\xi_{n+1}}\). The bound on \(m_{n+1}\) follows from
\[
 \begin{aligned}
 \norm{m_{n+1}}
 &=\norm{\xi_{n+1}-R_n+R_{n+1}-r_n}\\
 &\le K_\xi+2P+\sqrt{C_r\beta_n}
 \le K_{\mathrm{md}},
 \end{aligned}
\]
using \(\beta_n\le1\). In particular, \(K_{\mathrm{md}}\) is
independent of \(n\) and of the block index. Then
\begin{equation}
 \norm{\theta_n-z_n}
 =O_{\rm a.s.}\!\left(\sqrt{\frac{\log t_n}{t_n}}\right).
 \label{eq:bc_projected_averaging_rate}
\end{equation}
All norms in this lemma are Euclidean.
\end{lemma}

\begin{proof}
Nonexpansiveness of projection and boundedness of \(\xi_{n+1}\) give
\begin{equation}
 \begin{aligned}
 \norm{\theta_{n+1}-z_{n+1}}^2
 &\le (1-\beta_n)^2\norm{\theta_n-z_n}^2\\
 &\quad+2\beta_n(1-\beta_n)
       \langle\theta_n-z_n,\xi_{n+1}\rangle+C\beta_n^2.
 \end{aligned}
 \label{eq:bc_energy_projection}
\end{equation}
Here and below, \(C\) denotes a finite deterministic constant whose
value may increase between inequalities. Since projection fixes every
point of \(\Theta\), boundedness of the inputs also yields
\[
 \norm{\theta_{n+1}-\theta_n}
 +\norm{z_{n+1}-z_n}\le C\beta_n.
\]

To handle the difference \(R_n-R_{n+1}\) in
\eqref{eq:bc_noise_decomposition}, we work directly with the squared
shifted distance \(\norm{\theta_n-z_n+\beta_nR_n}^2\).
Expand this square at indices \(n\) and \(n+1\), and substitute
\eqref{eq:bc_noise_decomposition} into~\eqref{eq:bc_energy_projection}.
The terms involving \(R_n-R_{n+1}\) cancel up to an error bounded
by \(C\beta_n^2\), giving
\begin{equation}
 \begin{aligned}
 &\norm{\theta_{n+1}-z_{n+1}+\beta_{n+1}R_{n+1}}^2\\
 &\quad\le(1-\beta_n)^2\norm{\theta_n-z_n+\beta_nR_n}^2\\
 &\qquad+2\beta_n(1-\beta_n)
                 \langle\theta_n-z_n,m_{n+1}\rangle\\
 &\qquad+2\beta_n(1-\beta_n)
                 \langle\theta_n-z_n,r_n\rangle+C\beta_n^2.
 \end{aligned}
 \label{eq:bc_corrected_energy}
\end{equation}
Indeed, the uncancelled terms are
\[
 \begin{aligned}
 &2(1-\beta_n)\beta_n^2\langle\theta_n-z_n,R_n\rangle\\
 &\quad+2\beta_{n+1}
   \langle(\theta_{n+1}-\theta_n)-(z_{n+1}-z_n),R_{n+1}\rangle\\
 &\quad+2\{\beta_{n+1}-(1-\beta_n)\beta_n\}
                      \langle\theta_n-z_n,R_{n+1}\rangle\\
 &\quad+\beta_{n+1}^2\norm{R_{n+1}}^2
             -(1-\beta_n)^2\beta_n^2\norm{R_n}^2.
 \end{aligned}
\]
Their total absolute value is \(O(\beta_n^2)\), by the preceding
increment bound, compactness of \(\Theta\), boundedness of \(R_n\),
and \(\beta_{n+1}-(1-\beta_n)\beta_n=O(\beta_n^2)\).

Fix \(c\) with \(1<c<2B\). Using
\[
 \norm{\theta_n-z_n}
 \le\norm{\theta_n-z_n+\beta_nR_n}+P\beta_n
\]
and \(\norm{r_n}^2\le C_r\beta_n\), Young's inequality gives
\[
 \begin{aligned}
 &2\beta_n(1-\beta_n)
       |\langle\theta_n-z_n,r_n\rangle|\\
 &\quad\le(2-c/B)\beta_n
       \norm{\theta_n-z_n+\beta_nR_n}^2+C\beta_n^2.
 \end{aligned}
\]
Since \(\beta_n=B/t_n\), inequality~\eqref{eq:bc_corrected_energy}
therefore becomes
\begin{equation}
 \begin{aligned}
 &\norm{\theta_{n+1}-z_{n+1}+\beta_{n+1}R_{n+1}}^2\\
 &\quad\le
 \left(1-\frac{c}{t_n}+\frac{B^2}{t_n^2}\right)
       \norm{\theta_n-z_n+\beta_nR_n}^2\\
 &\qquad+2\beta_n(1-\beta_n)
       \langle\theta_n-z_n,m_{n+1}\rangle+\frac{C}{t_n^2}.
 \end{aligned}
 \label{eq:bc_energy_recursion}
\end{equation}
Conditional Hoeffding's inequality, together with
\(\norm{\theta_n-z_n}^2
 \le2\norm{\theta_n-z_n+\beta_nR_n}^2+2P^2\beta_n^2\), gives
\begin{equation}
 \begin{aligned}
 &\bE\!\left[
  \exp\{2s\beta_n(1-\beta_n)
                \langle\theta_n-z_n,m_{n+1}\rangle\}
  \,\middle|\,\cF_n\right]\\
 &\quad\le\exp\!\left\{
   4K_{\mathrm{md}}^2s^2\beta_n^2
   \bigl(\norm{\theta_n-z_n+\beta_nR_n}^2+P^2\beta_n^2\bigr)
   \right\},\qquad s\in\bR.
 \end{aligned}
 \label{eq:bc_conditional_mgf}
\end{equation}

Choose \(d>0\) so that \(4B^2K_{\mathrm{md}}^2d<c-1\).
Exponentiate~\eqref{eq:bc_energy_recursion} after multiplying by
\(dt_{n+1}\), and apply~\eqref{eq:bc_conditional_mgf} with
\(s=dt_{n+1}=d(t_n+1)\). The coefficient of
\(dt_n\norm{\theta_n-z_n+\beta_nR_n}^2\) in the resulting
exponent is
\[
 \begin{aligned}
 &\frac{t_n+1}{t_n}
       \left(1-\frac{c}{t_n}+\frac{B^2}{t_n^2}\right)
 +4B^2K_{\mathrm{md}}^2d\frac{(t_n+1)^2}{t_n^3}\\
 &\qquad=1-\frac{c-1-4B^2K_{\mathrm{md}}^2d}{t_n}
   +O(t_n^{-2}).
 \end{aligned}
\]
The remaining terms in that exponent are bounded by \(C/t_n\).
Consequently, there exists \(\delta>0\) such that, for all sufficiently
large \(n\),
\[
 \begin{aligned}
 &\bE\!\left[
 e^{dt_{n+1}\norm{\theta_{n+1}-z_{n+1}+\beta_{n+1}R_{n+1}}^2}
 \,\middle|\,\cF_n\right]\\
 &\quad\le e^{C/t_n}
 \left(e^{dt_n\norm{\theta_n-z_n+\beta_nR_n}^2}\right)^{1-\delta/t_n},
 \qquad 0<1-\delta/t_n<1.
 \end{aligned}
\]
Taking expectations and applying Jensen's inequality to the concave
power \(u\mapsto u^{1-\delta/t_n}\) yields
\[
 \begin{aligned}
 &\log\bE\!\left[
   e^{dt_{n+1}\norm{\theta_{n+1}-z_{n+1}+\beta_{n+1}R_{n+1}}^2}
   \right]\\
 &\quad\le\left(1-\frac{\delta}{t_n}\right)
   \log\bE\!\left[e^{dt_n\norm{\theta_n-z_n+\beta_nR_n}^2}\right]
   +\frac{C}{t_n}.
 \end{aligned}
\]
The squared shifted distance is bounded by a deterministic constant,
so its exponential moment is finite at every fixed index. Starting at
an index where the preceding recursion holds, induction bounds the
logarithm of the exponential moment by the larger of its initial value
and \(C/\delta\). Including the finite prefix therefore gives
\[
 \sup_{n\ge1}\bE\!\left[
 e^{dt_n\norm{\theta_n-z_n+\beta_nR_n}^2}\right]<\infty.
\]
Markov's inequality now implies
\[
 \bP\!\left(
 \norm{\theta_n-z_n+\beta_nR_n}^2
 \ge K\frac{\log t_n}{t_n}\right)\le Ct_n^{-dK}.
\]
Choose \(K>1/d\) and apply Borel--Cantelli. Finally,
\[
 \norm{\theta_n-z_n}
 \le\norm{\theta_n-z_n+\beta_nR_n}+\frac{PB}{t_n}
 =O_{\rm a.s.}\!\left(\sqrt{\frac{\log t_n}{t_n}}\right),
\]
which proves~\eqref{eq:bc_projected_averaging_rate}.
\end{proof}

\begin{proof}[Proof of Proposition~\ref{prop:boundary_contraction_rate}]

The proof proceeds in four steps:
\begin{enumerate}[(i)]
\item
We identify the projected equilibrium and establish almost-sure
convergence of the original iterates under the logarithmically
separated step sizes.

\item
We construct Poisson decompositions for the fast and slow noise.
For the fast recursion, we form the recursive weighted average
\(U_n\) of the fast noise and analyze \(p_n=x_n-U_n\).
For the slow recursion, we construct the projected auxiliary
sequence \(z_n\) and apply
Lemma~\ref{lem:bc_projected_averaging} to bound
\(\norm{\theta_n-z_n}\). These auxiliary sequences are defined
in~\eqref{eq:bc_auxiliaries} and are used only in the proof.

\item
We show that the auxiliary slow iterate eventually remains in
the face of \(\Theta\) maximizing the inner product with the
averaged drift at the equilibrium. On this face, that drift
drops out of the projection objective, yielding a sharper bound
on the auxiliary iterate's increments.

\item
We use fast contraction and reduced slow-map contraction to
derive coupled error bounds. An induction gives the rate
\(O_{\rm a.s.}(\log(N_0+n)/\sqrt{N_0+n})\), which we then
transfer from the auxiliary sequences to the original iterates.
\end{enumerate}

We use the original fast norm and Euclidean slow norm. Define the full
averaged slow drift and its associated map by
\[
 F(x,\theta)=\sum_{y\in\cY}\mu^{(\theta)}(y)f(y,x,\theta),
 \qquad G(x,\theta)=\theta+F(x,\theta).
\]
Thus \(F(x^\star(\theta),\theta)=\favg(\theta;\theta)\) and
\(G(x^\star(\theta),\theta)=\overline G(\theta)\).
Constants below are deterministic unless stated otherwise; their
values may increase from one estimate to the next.

\emph{Projected equilibrium and initial convergence.}
The map \(\proj\circ\overline G:\Theta\to\Theta\) is a contraction,
so it has a unique fixed point \(\theta^\star\). Put
\(c=1-\rho_s>0\) and
\(v=F(x^\star(\theta^\star),\theta^\star)\). With the outward normal
cone convention
\[
 N_\Theta(\theta)=
 \{w:\langle w,y-\theta\rangle\le0\text{ for every }y\in\Theta\},
\]
the projection characterization gives
\begin{equation}
 v\in N_\Theta(\theta^\star),\qquad
 \theta^\star=\proj[\theta^\star+\beta v]\quad(\beta\ge0).
 \label{eq:bc_equilibrium_normal}
\end{equation}
At a boundary equilibrium, 
the averaged slow drift \(v\) may
be nonzero and point outward. 
Projection nevertheless leaves
\(\theta^\star\) fixed, as expressed
in~\eqref{eq:bc_equilibrium_normal}.
Contraction of \(\overline G\)
implies
\[
 \langle\theta-\theta^\star,
      \favg(\theta;\theta)-v\rangle
 \le-c\norm{\theta-\theta^\star}^2.
\]
Both \(v\) and the normal correction in the projected ODE have
nonpositive contributions to the derivative of squared distance to
\(\theta^\star\). Hence its flow satisfies
\begin{equation}
 \norm{\mathsf S_t\theta-\theta^\star}
 \le e^{-ct}\norm{\theta-\theta^\star},\qquad t\ge0.
 \label{eq:bc_ode_contraction}
\end{equation}
In particular, the equilibrium set is \(\eq=\{\theta^\star\}\),
and distance to \(\theta^\star\) is a global Lipschitz Lyapunov
function with a uniform one-block contraction.

The initial-convergence argument in Section~\ref{sec:ec_face_rates},
through~\eqref{eq:face_initial_fast_tracking} and the subsequent
slow-block comparison, applies to the present schedule. That part of
the argument uses the standing assumptions, the step-size relations,
and Lyapunov decrease; it does not use the box-face or common-target
conditions. In the present case
\(\beta_n/\alpha_n=B/(A\log t_n)\to0\),
\(t_n\alpha_n\to\infty\), and both step-size sums diverge.
The endpoint ratios on fixed fast-time blocks tend to one, while
those on fixed slow-time blocks remain bounded. Thus the same
Poisson and martingale comparisons, together with
\eqref{eq:bc_ode_contraction}, give
\begin{equation}
 \norm{x_n-x^\star(\theta_n)}
 =O_{\rm a.s.}\!\left(\sqrt{\alpha_n\log t_n}
                       +\frac{\beta_n}{\alpha_n}\right)\to0,
 \qquad \theta_n\to\theta^\star
 \quad\text{almost surely}.
 \label{eq:bc_initial_convergence}
\end{equation}
The bounds in~\eqref{eq:bc_initial_convergence} follow from
the block comparison estimates applied with the step sizes
in~\eqref{eq:bc_steps}.

\emph{Noise averaging within the feasible set.}
Center both noise streams at the actual pair \((x_n,\theta_n)\):
\[
 \begin{aligned}
 \xi^x_{n+1}
 &=h(Y_n,x_n,\theta_n)-\havg(x_n;\theta_n)+M_{n+1},\\
 \xi^\theta_{n+1}
 &=f(Y_n,x_n,\theta_n)-F(x_n,\theta_n)+M'_{n+1}.
 \end{aligned}
\]
Both admit~\eqref{eq:bc_noise_decomposition}, with bounded martingale
differences and correctors and with \(\norm{r_n}\le C\alpha_n\).
To see this for either observation, let \(u_{x,\theta}\) solve its
Poisson equation centered at its full average, and put
\(u_n=u_{x_n,\theta_n}\). Then take
\[
 \begin{aligned}
 R_n&=u_n(Y_n),\\
 m_{n+1}&=M^a_{n+1}+u_n(Y_{n+1})-\sum_{y\in\cY}p^{(\theta_n)}(Y_n,y)u_n(y),\\
 r_n&=u_{n+1}(Y_{n+1})-u_n(Y_{n+1}),
 \end{aligned}
\]
where \(M^a\) denotes \(M\) or \(M'\), respectively.
The transition rule for \(Y_{n+1}\) makes \(m_{n+1}\) a martingale
difference with respect to the original filtration. The Poisson
equation verifies the decomposition exactly, including its index
\(Y_n\). Lemma~\ref{lemma:poisson_sensitivity_1} supplies the
boundedness and sensitivity estimates for \(h\). To obtain them for
\(f\), put \(g_{x,\theta}(y)=f(y,x,\theta)-F(x,\theta)\).
On a bounded fast-state set this function is bounded and satisfies
\[
 \sup_y\norm{g_{x,\theta}(y)-g_{x',\theta'}(y)}
 \le C(\norm{x-x'}+\norm{\theta-\theta'}),
\]
by update Lipschitzness and Lemma~\ref{lemma:lipschitz_stat_dist}.
It is centered under \(\mu^{(\theta)}\). The return-time solution
therefore vanishes at the reference state and equals
\(Z^{(\theta)}g_{x,\theta}\) off that state, where
\(Z^{(\theta)}\) is the killed resolvent in
\eqref{eq:fast_killed_resolvent_bound}. Its uniform bound and the
resolvent identity give
\[
 \sup_y\norm{u_{x,\theta}(y)}\le C,
 \qquad
 \sup_y\norm{u_{x,\theta}(y)-u_{x',\theta'}(y)}
 \le C(\norm{x-x'}+\norm{\theta-\theta'}).
\]
The actual increments are \(O(\alpha_n)\) and \(O(\beta_n)\), so
the sensitivity estimate gives the stated remainder bound. No
independence between \(M^a_{n+1}\) and \(Y_{n+1}\) is needed.

Define the recursive weighted average \(U_n\) of the fast noise
and the projected slow auxiliary sequence \(z_n\) by
\begin{equation}
 \begin{aligned}
 U_1&=0,&
 U_{n+1}&=(1-\alpha_n)U_n+\alpha_n\xi^x_{n+1},& p_n&=x_n-U_n,\\
 z_1&=\theta_1,&
 z_{n+1}&=\proj[(1-\beta_n)z_n+\beta_nG(x_n,\theta_n)],&
 \Delta_n&=\theta_n-z_n.
 \end{aligned}
 \label{eq:bc_auxiliaries}
\end{equation}
For completeness, the deterministic averaging weights
\[
 w_{n,i}=\alpha_i\prod_{j=i+1}^{n-1}(1-\alpha_j),\qquad 1\le i<n,
\]
satisfy \(\sum_{i<n}w_{n,i}^2=O(\alpha_n)\) and
\(\sum_{i<n}w_{n,i}\alpha_i=O(\alpha_n)\).
These follow by applying a multiple of \(\alpha_n\) as a
supersolution to the respective recursions
\(v_{n+1}=(1-\alpha_n)^2v_n+\alpha_n^2\) and
\(b_{n+1}=(1-\alpha_n)b_n+\alpha_n^2\), since
\(\alpha_n-\alpha_{n+1}=o(\alpha_n^2)\).
The weights are eventually nondecreasing in \(i\), because
\(1/\alpha_{i+1}-1/\alpha_i\to0\); each fixed initial weight
decays faster than every inverse power of \(t_n\).
Summation by parts therefore bounds the coboundary contribution by
\(O(\alpha_n)\), as does the remainder contribution.
Applying the Azuma--Hoeffding inequality coordinatewise to
\(\sum_{i<n}w_{n,i}m_{i+1}\), using \(\sum_{i<n}w_{n,i}^2=O(\alpha_n)\), 
gives
\[
 \mathbb{P}\!\left(
  \left\|\sum_{i<n}w_{n,i}m_{i+1}\right\|_2
  >K\sqrt{\alpha_n\log t_n}
 \right)\le C t_n^{-2}
\]
for sufficiently large deterministic \(K\).
Since these probabilities are summable, 
the Borel--Cantelli lemma bounds the martingale contribution by
\(O_{\rm a.s.}(\sqrt{\alpha_n\log t_n})\).
Norm equivalence gives the same order in the original fast norm.

The original slow recursion for \(\theta_n\) and the recursion
for the projected auxiliary sequence \(z_n\)
in~\eqref{eq:bc_auxiliaries} have the shared-input
form required by Lemma~\ref{lem:bc_projected_averaging}.
Its remainder condition holds because \(\alpha_n^2=O(\beta_n)\).
Consequently,
\begin{equation}
 \norm{U_n}=O_{\rm a.s.}\!\left(\frac{\log t_n}{\sqrt{t_n}}\right),
 \qquad
 \norm{\Delta_n}=O_{\rm a.s.}\!\left(\sqrt{\frac{\log t_n}{t_n}}\right).
 \label{eq:bc_filter_bounds}
\end{equation}
The inputs and actual iterates, together with the auxiliary
sequences \(U_n\), \(p_n\), and \(z_n\), are bounded by
deterministic constants.
Thus all subsequent Lipschitz estimates may
be taken on one enlarged deterministic fast-state ball.

\emph{The face reached by the auxiliary.}
Set
\[
 S=\operatorname*{argmax}_{y\in\Theta}\langle v,y\rangle,
 \qquad e_n=F(x_n,\theta_n)-v+\Delta_n.
\]
Equations~\eqref{eq:bc_initial_convergence} and~\eqref{eq:bc_filter_bounds}
give \(z_n\to\theta^\star\), \(e_n\to0\), and
\[
 z_{n+1}=\proj[z_n+\beta_n(v+e_n)].
\]
Suppose first that \(S\ne\Theta\). A polyhedron has only finitely
many distinct normal cones, all closed. Therefore
\[
 \delta=\min\{\operatorname{dist}(v,C):
     C=N_\Theta(y)\text{ for some }y\in\Theta,\ v\notin C\}>0.
\]
Work after \(\norm{e_n}\le\delta/4\). If \(z_{n+1}\notin S\),
let \(r=(z_{n+1}-z_n)/\beta_n\). Projection optimality gives
\(v+e_n-r\in N_\Theta(z_{n+1})\), a cone excluding \(v\),
so \(\norm{r}\ge3\delta/4\). Testing the same normal inequality
against \(z_n\) yields
\[
 \langle v,r\rangle
 \ge\norm{r}^2-\norm{e_n}\norm{r}
 \ge3\delta^2/8.
\]
Every step ending outside \(S\) therefore increases
\(\langle v,z_n\rangle\) by at least \(3\delta^2\beta_n/8\).
Since \(z_n\in\Theta\) and \(\Theta\) is compact,
\(\langle v,z_n\rangle\) is bounded above by
\(\max_{y\in\Theta}\langle v,y\rangle<\infty\).
If \(z_n\) remained outside \(S\), the preceding increment
bound and \(\sum_n\beta_n=\infty\) would imply
\(\langle v,z_n\rangle\to\infty\), a contradiction.
Thus \(z_n\) enters \(S\) after finitely many iterations.
After entry, a step leaving
\(S\) would require a strict increase from its maximum, which is
impossible. Hence the auxiliary eventually remains in \(S\).

For \(z_n,z_{n+1}\in S\), restrict the projection objective to
\(S\). The term involving \(v\) is constant there, giving
\begin{equation}
z_{n+1}=\operatorname{proj}_S[z_n+\beta_ne_n],
\qquad \norm{z_{n+1}-z_n}\le\beta_n\norm{e_n}.
\label{eq:bc_face_movement}
\end{equation}
If \(S=\Theta\), the above relations hold from the start.
The argument showing that \(z_n\) eventually enters and remains
in \(S\) also applies when \(\Theta\) is lower-dimensional.
If \(v\ne0\) is orthogonal to every difference \(y-y'\), with
\(y,y'\in\Theta\), then \(\langle v,y\rangle\) is constant
over \(\Theta\), so \(S=\Theta\).
Only the auxiliary sequence \(z_n\) must eventually remain
in \(S\); the noisy iterates \(\theta_n\) need not do so.
No strict complementarity is required:
\(\theta^\star\) may lie on the relative boundary of \(S\).
Moreover, the fast equilibrium \(x^\star(\theta)\) need not
be constant as \(\theta\) varies over \(S\).

\emph{Coupled contraction estimates.}
Recall \(q=1-\gamma>0\), and let \(K\) be a Lipschitz constant of
\(x^\star\), and let \(L_g\) bound the Lipschitz dependence of
\(G\) on its fast argument. Write
\[
e_n^{(x)}:=\norm{p_n-x^\star(z_n)},\qquad
e_n^{(\theta)}:=\norm{z_n-\theta^\star},\qquad
\rho_n=\norm{U_n}+\norm{\Delta_n},\qquad
s_n=\frac{\log t_n}{\sqrt{t_n}}.
\]
By~\eqref{eq:bc_filter_bounds}, \(\rho_n=O_{\rm a.s.}(s_n)\).
Subtracting the update for \(U_{n+1}\)
in~\eqref{eq:bc_auxiliaries} from the update for \(x_{n+1}\)
cancels the noise term \(\alpha_n\xi^x_{n+1}\).
Using \(p_n=x_n-U_n\), we obtain
\[
 p_{n+1}=(1-\alpha_n)p_n
          +\alpha_n\havg(x_n;\theta_n).
\]
Since \(x_n-p_n=U_n\) and \(\theta_n-z_n=\Delta_n\),
Lipschitz continuity gives, for a deterministic constant \(L\),
\[
 \begin{aligned}
 \norm{\havg(x_n;\theta_n)-\havg(p_n;z_n)}
 &\le L\bigl(\norm{U_n}+\norm{\Delta_n}\bigr)
 =L\rho_n,\\
 \norm{G(x_n,\theta_n)-G(p_n,z_n)}
 &\le L\bigl(\norm{U_n}+\norm{\Delta_n}\bigr)
 =L\rho_n.
 \end{aligned}
\]
Also,
\[
 \norm{e_n}\le L_g e_n^{(x)} +(1+\rho_s) e_n^{(\theta)} +L\rho_n,
\]
because \(\favg(z_n;z_n)-v
=\overline G(z_n)-\overline G(\theta^\star)-(z_n-\theta^\star)\).
Fast contraction and~\eqref{eq:bc_face_movement} consequently give,
eventually almost surely,
\begin{equation}
 \begin{aligned}
 e^{(x)}_{n+1}\le{}&(1-q\alpha_n+KL_g\beta_n)e^{(x)}_n
     +K(1+\rho_s)\beta_n e_n^{(\theta)} \\
 &\quad +(L\alpha_n+KL\beta_n)\rho_n.
 \end{aligned}
 \label{eq:bc_fast_error}
\end{equation}
For the slow error, compare the auxiliary update with
\(\proj[(1-\beta_n)\theta^\star+\beta_n\overline G(\theta^\star)]
=\theta^\star\), using~\eqref{eq:bc_equilibrium_normal}.
Nonexpansiveness and contraction of \(\overline G\) give
\begin{equation}
 e_{n+1}^{(\theta)} \le(1-c\beta_n)e^{(\theta)}_n
                  +L_g\beta_n e^{(x)}_n +L\beta_n\rho_n.
 \label{eq:bc_slow_error}
\end{equation}
Thus the comparison retains the equilibrium normal \(v\) without
assuming that the averaged drift vanishes at \(\theta^\star\).

To close these inequalities, note that
\[
 \frac{s_{n+1}}{s_n}
 =1-\frac1{2t_n}+\frac1{t_n\log t_n}+O(t_n^{-2}).
\]
Choose \(K_x>2L/q\), and then choose \(K_\theta\) so large that
\[
 c-\frac{L_gK_x+L}{K_\theta}>\frac1{2B};
\]
this is possible since \(Bc\ge2\). On each almost-sure path,
choose a sufficiently large index \(N\) after face identification
and all eventual estimates. A finite random \(C\) can be chosen
so that \(\rho_n\le Cs_n\) for \(n\ge N\),
\(e^{(x)}_N\le CK_xs_N\), and
\(e^{(\theta)}_N\le CK_\theta s_N\).
Assuming the latter two bounds at index \(n\),
\eqref{eq:bc_slow_error} proves the slow bound at \(n+1\).
For~\eqref{eq:bc_fast_error}, the decrement below \(CK_xs_n\),
divided by \(Cs_n\), is at least
\[
 \alpha_n(qK_x-L)
 -\beta_n\{KL_gK_x+K(1+\rho_s)K_\theta+KL\}.
\]
Since \(\beta_n/\alpha_n\to0\) and \(t_n\alpha_n\to\infty\),
this eventually exceeds \(K_x(1-s_{n+1}/s_n)\), proving the
fast bound at \(n+1\). Induction gives
\( e^{(x)}_n +  e_n^{(\theta)}=O_{\rm a.s.}(s_n)\).
Finally,
\[
 \norm{\theta_n-\theta^\star}\le e_n^{(\theta)} +\norm{\Delta_n},
 \qquad
 \norm{x_n-x^\star(\theta_n)}
 \le e^{(x)}_n +\norm{U_n}+K\norm{\Delta_n}.
\]
Together with~\eqref{eq:bc_filter_bounds}, these inequalities prove
\eqref{eq:bc_rate}. Lipschitzness of \(x^\star\) gives the same bound
for \(\norm{x_n-x^\star(\theta^\star)}\); Lipschitzness of \(J\)
and \(J(\theta^\star)=0\) give it for \(\lvert J(\theta_n)\rvert\).
\end{proof}
\par

\clearpage

\section[Stochastic saddle-point learning without strong convexity--concavity]{Stochastic saddle-point learning without strong convexity--concavity}
\label{sec:ec_saddle_learning}

Consider the problem
\begin{equation}
 \min_{u\in U}\max_{v\in V}\mathcal L(u,v),
 \qquad
 \mathcal L(u,v)=\varphi(u)+u^{\mathsf T}Av-\psi(v),
 \label{eq:sp_problem}
\end{equation}
where $U\subset\bR^{d_u}$ and $V\subset\bR^{d_v}$ are nonempty compact
convex polyhedra, $A\in\bR^{d_u\times d_v}$, and $\varphi$ and $\psi$
are convex differentiable functions with Lipschitz gradients on
neighborhoods of $U$ and $V$. Write $\theta=(u,v)$ and
$\Theta=U\times V$. Let $\theta^\star=(u^\star,v^\star)$ be a saddle
point, so that
\[
 \mathcal L(u^\star,v)\le\mathcal L(u^\star,v^\star)
 \le\mathcal L(u,v^\star),\qquad u\in U,\ v\in V.
\]
\begin{assumption}[Radial power curvature]
\label{assum:sp_power_curvature}
There exist \(q>2\) and \(c_{\rm sp}>0\) such that
\begin{equation}
 \begin{aligned}
 &\langle u-u^\star,
   \nabla\varphi(u)-\nabla\varphi(u^\star)\rangle
 +\langle v-v^\star,
   \nabla\psi(v)-\nabla\psi(v^\star)\rangle\\
 &\hspace{30mm}\ge
 c_{\rm sp}\norm{\theta-\theta^\star}^{q},
 \qquad \theta=(u,v)\in\Theta.
 \end{aligned}
 \label{eq:sp_power_curvature}
\end{equation}
\end{assumption}
All norms in this section are Euclidean. Condition~\eqref{eq:sp_power_curvature} permits curvature to vanish at the saddle and does not require strong convexity–concavity. Convexity–concavity alone, however, does not guarantee convergence of the descent–ascent ODE, which can cycle even in bilinear games.

Define the descent--ascent field
\begin{equation}
 F_{\rm sp}(u,v)
 :=\bigl(-\nabla\varphi(u)-Av,\ A^{\mathsf T}u-\nabla\psi(v)\bigr).
 \label{eq:sp_field}
\end{equation}
Suppose the finite controlled Markov chain satisfies
Assumptions~\ref{assum:markov_noise}, \ref{assum:lipschitz_kernel},
and~\ref{assum:hitting_time}. Let $\widehat F_{\rm sp}(y,\theta)$ be
bounded and Lipschitz in $\theta$, uniformly in $y$, with stationary mean
\begin{equation}
 \sum_{y\in\cY}\mu^{(\theta)}(y)\widehat F_{\rm sp}(y,\theta)
 =F_{\rm sp}(\theta),\qquad \theta\in\Theta.
 \label{eq:sp_oracle}
\end{equation}
Consider
\begin{equation}
 \begin{aligned}
 x_{n+1}&=(1-\alpha_n)x_n+
                 \alpha_n\widehat F_{\rm sp}(Y_n,\theta_n),\\
 \theta_{n+1}&=\operatorname{proj}_{\Theta}
                    (\theta_n+\beta_nx_n).
 \end{aligned}
 \label{eq:sp_recursion}
\end{equation}
The fast variable estimates the entire saddle field; both primal and dual
variables evolve on the slow time scale. Initialize $\theta_1\in\Theta$
and $\norm{x_1}\le R$ almost surely, where $R$ is a deterministic bound
also satisfying
$\sup_{y,\theta}\norm{\widehat F_{\rm sp}(y,\theta)}\le R$.

\begin{proposition}[Saddle-point learning under power curvature]
\label{prop:sp_power_rate}
Under the preceding conditions, the reduced projected ODE has the unique
equilibrium $\theta^\star$, and
$J(\theta):=\norm{\theta-\theta^\star}$ satisfies
Assumption~\ref{assum:one_block_decrease} with power $p_J=q-1$.
For recursion~\eqref{eq:sp_recursion}, choose
\begin{equation}
 \alpha_n=(N_0+n)^{-\fraka},\qquad
 \beta_n=(N_0+n)^{-\frakb},\qquad
 \fraka=\frac{2(q-1)}{4q-5},\quad
 \frakb=\frac{3(q-1)}{4q-5},
 \label{eq:sp_steps}
\end{equation}
with $N_0$ satisfying~\eqref{cond:N0:alpha2}. Then
\begin{equation}
 \begin{aligned}
 \norm{\theta_n-\theta^\star}
 &=O_{\rm a.s.}\!\left(
 (N_0+n)^{-\frac{1}{4q-5}}
 \{\log(N_0+n)\}^{\frac{1}{2(q-1)}}\right),\\
 \norm{x_n-F_{\rm sp}(\theta_n)}
 &=O_{\rm a.s.}\!\left(
 (N_0+n)^{-\frac{q-1}{4q-5}}\sqrt{\log(N_0+n)}\right).
 \end{aligned}
 \label{eq:sp_rates}
\end{equation}
The saddle may lie on the boundary of $\Theta$; no strict-complementarity
condition is imposed.
\end{proposition}

\begin{proof}
In~\eqref{def:x_update}--\eqref{def:te_update}, take
$h(y,x,\theta)=\widehat F_{\rm sp}(y,\theta)$,
$f(y,x,\theta)=x$, and $M_{n+1}=M'_{n+1}=0$. By~\eqref{eq:sp_oracle},
the averaged fast map is constant in $x$:
\[
 h^{(\mathrm{av.})}(x;\theta)=F_{\rm sp}(\theta),
 \qquad x^\star(\theta)=F_{\rm sp}(\theta).
\]
Thus Assumption~\ref{assum:contraction} holds with contraction factor
zero. The oracle regularity and $f(y,x,\theta)=x$ verify
Assumption~\ref{assum:lipschitz}. Since $0<\alpha_n\le1$,
\[
 \norm{x_{n+1}}\le(1-\alpha_n)\norm{x_n}+\alpha_nR\le R
\]
whenever $\norm{x_n}\le R$. This proves
Assumption~\ref{assum:xn_bound}; the martingale assumption holds
trivially. The controlled-chain assumptions were imposed above.

The reduced slow field is $F_{\rm sp}$. Let $\theta(t)$ solve its
projected ODE, and use the outward normal cone
\[
 N_\Theta(\theta)
 :=\{w:\langle w,z-\theta\rangle\le0\text{ for every }z\in\Theta\}.
\]
The saddle inequalities imply
$F_{\rm sp}(\theta^\star)\in N_\Theta(\theta^\star)$.
Almost everywhere,
\[
 \dot\theta(t)=F_{\rm sp}(\theta(t))-\eta(t),
 \qquad \eta(t)\in N_\Theta(\theta(t)).
\]
For $d=\theta(t)-\theta^\star$, monotonicity of the normal cone gives
$\langle d,\eta(t)-F_{\rm sp}(\theta^\star)\rangle\ge0$.
Consequently, cancellation of the bilinear terms yields
\begin{equation}
 \begin{aligned}
 \frac{d}{dt}\frac{\norm{d}^{2}}2
 &\le\langle d,F_{\rm sp}(\theta(t))-F_{\rm sp}(\theta^\star)\rangle\\
 &=-\langle u-u^\star,\nabla\varphi(u)-\nabla\varphi(u^\star)\rangle
   -\langle v-v^\star,\nabla\psi(v)-\nabla\psi(v^\star)\rangle\\
 &\le-c_{\rm sp}\norm{d}^{q}.
 \end{aligned}
 \label{eq:sp_lyapunov}
\end{equation}
Applying the same inequality to a constant equilibrium trajectory shows
that every equilibrium equals $\theta^\star$. The function $J$ is
therefore $1$-Lipschitz and has exactly the required zero set. Wherever
$J(\theta(t))>0$,~\eqref{eq:sp_lyapunov} gives
\[
 \frac{d}{dt}J(\theta(t))\le-c_{\rm sp}J(\theta(t))^{q-1}.
\]
If a trajectory reaches $\theta^\star$, it remains there by uniqueness of
the projected flow. Integrating the preceding inequality gives, for each
fixed $T>0$,
\begin{equation}
 J(\mathsf S_T(\theta))
 \le J(\theta)
 \bigl[1+c_{\rm sp}(q-2)T J(\theta)^{q-2}\bigr]^{-1/(q-2)}.
 \label{eq:sp_block}
\end{equation}
In particular, $J$ is a Lyapunov function, and
\[
 \mathcal D_J(u)
 :=u-u\bigl[1+c_{\rm sp}(q-2)Tu^{q-2}\bigr]^{-1/(q-2)}
\]
is continuous, vanishes at zero, and satisfies $0<\mathcal D_J(u)\le u$
for $u>0$. Moreover,
\[
 \lim_{u\downarrow0}\frac{\mathcal D_J(u)}{u^{q-1}}=c_{\rm sp}T>0.
\]
This verifies the uniform one-block decrease and its power lower bound
with $p_J=q-1$. Since $q>2$, the exponents in~\eqref{eq:sp_steps}
satisfy $1/2<\fraka<\frakb<1$.
Corollary~\ref{coro:optimized_power_rate} now gives~\eqref{eq:sp_rates}.
\end{proof}

\paragraph{A family with vanishing curvature.}
Let $0\in U\times V$ and take
\begin{equation}
 \varphi(u)=\frac{\norm{u}^{q}}q,\qquad
 \psi(v)=\frac{\norm{v}^{q}}q,\qquad q>2,
 \label{eq:sp_power_example}
\end{equation}
with arbitrary dimensions and arbitrary $A$. The origin is a saddle,
and
\[
 \langle u,\nabla\varphi(u)\rangle+
 \langle v,\nabla\psi(v)\rangle
 =\norm{u}^{q}+\norm{v}^{q}
 \ge2^{1-q/2}\norm{(u,v)}^{q}.
\]
Thus~\eqref{eq:sp_power_curvature} holds. The Hessians of both penalties
vanish at zero, so on sets containing a neighborhood of zero neither
strong convexity nor strong concavity is available. The proposition also
applies when zero lies on the boundary. For example, $q=4$ gives
$\fraka=6/11$, $\frakb=9/11$, and slow-iterate rate
$O_{\rm a.s.}((N_0+n)^{-1/11}\{\log(N_0+n)\}^{1/6})$.
These are the rates supplied by the general power-decrease theorem;
no optimality claim is made.

\begin{remark}[Comparison with existing frameworks]
\label{rem:sp_ac_frameworks}
The obstruction in these applications is in the reduced slow dynamics;
the fast averaged map is contractive. The comparisons below concern the
hypotheses of the cited results, rather than possible extensions of their
proof techniques.

For~\eqref{eq:sp_power_example}, suppose $0$ is interior to $\Theta$
and $A\ne0$. Then
\[
 DF_{\rm sp}(0)=
 R_A:=\begin{pmatrix}0&-A\\ A^{\mathsf T}&0\end{pmatrix}.
\]
If $\sigma>0$ is a nonzero singular value of $A$, the derivative of the
forward map $\theta\mapsto\theta+hF_{\rm sp}(\theta)$ has eigenvalues
$1\pm ih\sigma$, whose moduli exceed one for every $h>0$.
Projection is inactive in a neighborhood of zero for each fixed $h$.
Hence even the projected forward map is not locally nonexpansive in any
fixed norm. This rules out the reduced-map contraction hypotheses
of~\citet{chandak2026Ok,chandakhaquebambos2025} and the nonexpansive-map
hypothesis of~\citet{chandak2025non} for this family.
The nonconvex-gradient extension in~\citet{chandak2025non} does not
apply either: $R_A$ is nonsymmetric, so $F_{\rm sp}$ cannot be the
negative Euclidean gradient of a scalar objective near zero. Its minimax
specialization, which uses strong concavity in the fast player, also
differs from~\eqref{eq:sp_recursion}, where the fast variable estimates
the field and neither player requires strong curvature.

The contractive-map analysis of~\citet{chen2020finite} likewise does
not cover this reduced map. The more general drift condition in
\citet[Assumption~2.2]{chen2022finite} requires uniform quadratic
dissipation toward an unprojected root. Here
\[
 \langle\theta,F_{\rm sp}(\theta)\rangle
 =-\norm{u}^{q}-\norm{v}^{q},
\]
which cannot be bounded above by $-c\norm{\theta}^{2}$ with $c>0$
near zero. Nor can a positive-definite quadratic change of Lyapunov
function restore local exponential dissipation: the linearization
$R_A$ has nonzero purely imaginary eigenvalues. The power-decrease
condition~\eqref{eq:sp_block} is sufficient for our result.

For the actor--critic recursion in Section~\ref{sec:applications},
the lack of a nonexpansive slow map is already visible in a one-state,
two-action MDP. Take costs $0$ and $1$, deterministic return to the same
state, $B_{\rm ac}=1$, and write
$p=(1+e^{-(\vartheta_1-\vartheta_2)})^{-1}$. At the exact critic,
the averaged actor field is
\[
 F_{\rm ac}(\vartheta)=\bigl(p(1-p),-p(1-p)\bigr).
\]
Its Jacobian has eigenvalue
$2p(1-p)(1-2p)>0$ in direction $(1,-1)$ when $p<1/2$.
At $\vartheta=(-1/2,1/2)$, both this point and
$\vartheta+F_{\rm ac}(\vartheta)$ are interior to the logit box.
Thus $\vartheta\mapsto\operatorname{proj}_{\Theta_{\rm ac}}
(\vartheta+F_{\rm ac}(\vartheta))$ is not locally nonexpansive in any fixed
norm. The same obstruction persists for sufficiently small positive
forward steps. At the constrained optimum $(1,-1)$, the actor field is
nonzero and points outward, although the projected velocity is zero.
Consequently, neither a reduced-map contraction/nonexpansiveness
condition nor an unprojected-root condition follows from the hypotheses
of our actor--critic application. The unprojected slow-gradient result
in~\citet{chandak2025non} does not provide a rate to this constrained
equilibrium. In the full application, the controlled Markov sampling
and value-based Lyapunov function must also be retained.

The actor--critic result concerns precisely the global-clock recursion
stated in Section~\ref{sec:applications}, motivated by
\citet[Algorithm~3]{konda1999actor}; it does not extend the present rate
claim to the original visit-count or delayed implementation. Neither
projection nor Markov noise alone distinguishes our analysis from all
earlier work. The point of these two applications is that projected
slow-ODE Lyapunov decrease can hold when the slow-map assumptions in
the comparison results fail.
\end{remark}
\end{document}